\documentclass[9pt]{amsart}
\usepackage{amsmath}
\usepackage{amsxtra}
\usepackage{amscd}
\usepackage{amsthm}
\usepackage{amsfonts}
\usepackage{amssymb}
\usepackage{eucal}
\usepackage[all]{xy}
\usepackage{graphicx}

\newtheorem{cor}[subsubsection]{Corollary}
\newtheorem{lem}[subsubsection]{Lemma}
\newtheorem{prop}[subsubsection]{Proposition}

\newtheorem{conj}[subsubsection]{Conjecture}
\newtheorem{thm}[subsubsection]{Theorem}

\theoremstyle{definition}

\theoremstyle{remark}
\newtheorem{rem}[subsubsection]{Remark}

\newcommand{\thmref}[1]{Theorem~\ref{#1}}

\newcommand{\secref}[1]{Sect.~\ref{#1}}
\newcommand{\lemref}[1]{Lemma~\ref{#1}}
\newcommand{\propref}[1]{Proposition~\ref{#1}}
\newcommand{\corref}[1]{Corollary~\ref{#1}}
\newcommand{\conjref}[1]{Conjecture~\ref{#1}}

\numberwithin{equation}{section}

\newcommand{\nc}{\newcommand}
\nc{\renc}{\renewcommand}
\nc{\ssec}{\subsection}
\nc{\sssec}{\subsubsection}
\nc{\on}{\operatorname}

\nc\ol{\overline}
\nc\wt{\widetilde}
\nc\tboxtimes{\wt{\boxtimes}}
\nc\tstar{\wt{\star}}
\nc{\alp}{a}

\nc{\ZZ}{{\mathbb Z}}
\nc{\NN}{{\mathbb N}}
\nc{\OO}{{\mathbb O}}
\renc{\SS}{{\mathbb S}}
\nc{\DD}{{\mathbb D}}
\nc{\GG}{{\mathbb G}}

\nc{\Fq}{{\mathbb F}_q}
\nc{\Fqb}{\ol{{\mathbb F}_q}}
\nc{\Ql}{\ol{{\mathbb Q}_\ell}}
\nc{\id}{\text{id}}
\nc\X{\mathcal X}

\nc{\Hom}{\on{Hom}}
\nc{\Lie}{\on{Lie}}
\nc{\Loc}{\on{Loc}}
\nc{\Pic}{\on{Pic}}
\nc{\Bun}{\on{Bun}}
\nc{\IC}{\on{IC}}
\nc{\Fls}{\on{Fl}^{\frac{\infty}{2}}}
\nc{\ICs}{\on{IC}^{\frac{\infty}{2}}}
\nc{\ICsd}{\overset{\bullet}{\on{IC}}{}^{\frac{\infty}{2}}}
\nc{\ICsdt}{\overset{\bullet}{\on{IC}}{}^{\frac{\infty}{2},\tau}}
\nc{\ICsl}{\on{IC}^{\lambda+\frac{\infty}{2}}}
\nc{\ICslm}{\on{IC}^{\lambda+\frac{\infty}{2},-}}
\nc{\ICsm}{\on{IC}^{-,\frac{\infty}{2}}}
\nc{\Aut}{\on{Aut}}
\nc{\rk}{\on{rk}}
\nc{\Sh}{\on{Sh}}
\nc{\Perv}{\on{Perv}}
\nc{\pos}{{\on{pos}}}
\nc{\Conv}{\on{Conv}}
\nc{\Sph}{\on{Sph}}
\nc{\Sat}{\on{Sat}}
\nc{\Sym}{\on{Sym}}
\nc{\BunBb}{\overline{\Bun}_B}
\nc{\BunNb}{\overline{\Bun}_N}
\nc{\BunTb}{\overline{\Bun}_T}
\nc{\BunBbm}{\overline{\Bun}_{B^-}}
\nc{\BunBbel}{\overline{\Bun}_{B,el}}
\nc{\BunBbmel}{\overline{\Bun}_{B^-,el}}
\nc{\Buno}{\overset{o}{\Bun}}
\nc{\BunPb}{{\overline{\Bun}_P}}
\nc{\BunBM}{\Bun_{B(M)}}
\nc{\BunBMb}{\overline{\Bun}_{B(M)}}
\nc{\BunPbw}{{\widetilde{\Bun}_P}}
\nc{\BunBP}{\widetilde{\Bun}_{B,P}}
\nc{\GUb}{\overline{G/U}}
\nc{\GUPb}{\overline{G/U(P)}}

\nc{\Hhom}{\underline{\on{Hom}}}
\nc\syminfty{\on{Sym}^{\infty}}
\nc\lal{\ol{\kappa_x}}
\nc\xl{\ol{x}}
\nc\thl{\ol{\theta}}
\nc\nul{\ol{\nu}}
\nc\mul{\ol{\mu}}
\nc\Sum\Sigma
\nc{\oX}{\overset{o}{X}{}}
\nc{\hl}{\overset{\leftarrow}h{}}
\nc{\hr}{\overset{\rightarrow}h{}}
\nc{\M}{{\mathcal M}}
\nc{\N}{{\mathcal N}}
\nc{\F}{{\mathcal F}}
\nc{\D}{{\mathcal D}}
\nc{\Q}{{\mathcal Q}}
\nc{\Y}{{\mathcal Y}}
\nc{\G}{{\mathcal G}}
\nc{\E}{{\mathcal E}}
\nc{\CalC}{{\mathcal C}}
\nc\Dh{\widehat{\D}}

\nc{\C}{{\mathcal C}}
\nc{\K}{{\mathcal K}}
\renewcommand{\H}{{\mathcal H}}

\nc{\T}{{\mathcal T}}
\nc{\V}{{\mathcal V}}
\renc{\P}{{\mathcal P}}
\nc{\A}{{\mathcal A}}
\nc{\B}{{\mathcal B}}
\nc{\U}{{\mathcal U}}

\nc{\Gr}{{\on{Gr}}}

\nc{\frn}{{\check{\mathfrak u}(P)}}

\nc{\fC}{\mathfrak C}
\nc{\p}{\mathfrak p}
\nc{\q}{\mathfrak q}
\nc\f{{\mathfrak f}}

\nc{\qo}{{\mathfrak q}}
\nc{\po}{{\mathfrak p}}
\nc{\s}{{\mathfrak s}}
\nc\w{\text{w}}

\renewcommand{\mod}{{\on{-mod}}}
\newcommand{\comod}{{\on{-comod}}}

\nc\mathi\iota
\nc\Spec{\on{Spec}}
\nc\Mod{\on{Mod}}
\nc{\tw}{\widetilde{\mathfrak t}}
\nc{\pw}{\widetilde{\mathfrak p}}
\nc{\qw}{\widetilde{\mathfrak q}}
\nc{\jw}{\widetilde j}

\nc{\grb}{\overline{\Gr}}
\nc{\I}{\mathcal I}

\nc{\kappach}{{\check\kappa_x}}
\nc{\Lambdach}{{\check\Lambda}{}}
\nc{\much}{{\check\mu}}
\nc{\omegach}{{\check\omega}}
\nc{\nuch}{{\check\nu}}
\nc{\etach}{{\check\eta}}
\nc{\alphach}{{\checka}}
\nc{\oblvtach}{{\check\oblvta}}
\nc{\pich}{{\check\pi}}
\nc{\ch}{{\check h}}

\nc{\Hb}{\overline{\H}}

\nc{\BA}{{\mathbb{A}}}
\nc{\BC}{{\mathbb{C}}}
\nc{\BE}{{\mathbb{E}}}
\nc{\BF}{{\mathbb{F2}}}
\nc{\BG}{{\mathbb{G}}}
\nc{\BM}{{\mathbb{M}}}
\nc{\BO}{{\mathbb{O}}}
\nc{\BD}{{\mathbb{D}}}
\nc{\BN}{{\mathbb{N}}}
\nc{\BP}{{\mathbb{P}}}
\nc{\BQ}{{\mathbb{Q}}}
\nc{\BR}{{\mathbb{R}}}
\nc{\BV}{{\mathbb{V}}}
\nc{\BW}{{\mathbb{W}}}
\nc{\BZ}{{\mathbb{Z}}}
\nc{\BS}{{\mathbb{S}}}

\nc{\CA}{{\mathcal{A}}}
\nc{\CB}{{\mathcal{B}}}

\nc{\CE}{{\mathcal{E}}}
\nc{\CF}{{\mathcal{F2}}}
\nc{\CG}{{\mathcal{G}}}
\nc{\CH}{{\mathcal{H}}}

\nc{\CL}{{\mathcal{L}}}
\nc{\CC}{{\mathcal{C}}}
\nc{\CM}{{\mathcal{M}}}
\nc{\CN}{{\mathcal{N}}}
\nc{\cCN}{\check{{\mathcal{N}}}}
\nc{\CK}{{\mathcal{K}}}
\nc{\CO}{{\mathcal{O}}}
\nc{\CP}{{\mathcal{P}}}
\nc{\CQ}{{\mathcal{Q}}}
\nc{\CR}{{\mathcal{R}}}
\nc{\CS}{{\mathcal{S}}}
\nc{\CT}{{\mathcal{T}}}
\nc{\CU}{{\mathcal{U}}}
\nc{\CV}{{\mathcal{V}}}
\nc{\CW}{{\mathcal{W}}}
\nc{\CX}{{\mathcal{X}}}
\nc{\CY}{{\mathcal{Y}}}
\nc{\CZ}{{\mathcal{Z}}}
\nc{\CI}{{\mathcal{I}}}
\nc{\CJ}{{\mathcal{J}}}

\nc{\csM}{{\check{\mathcal A}}{}}
\nc{\oM}{{\overset{\circ}{\mathcal M}}{}}
\nc{\obM}{{\overset{\circ}{\mathbf M}}{}}
\nc{\oCA}{{\overset{\circ}{\mathcal A}}{}}
\nc{\obA}{{\overset{\circ}{\mathbf A}}{}}
\nc{\ooM}{{\overset{\circ}{M}}{}}
\nc{\osM}{{\overset{\circ}{\mathsf M}}{}}
\nc{\vM}{{\overset{\bullet}{\mathcal M}}{}}
\nc{\nM}{{\underset{\bullet}{\mathcal M}}{}}
\nc{\oD}{{\overset{\circ}{\mathcal D}}{}}
\nc{\obD}{{\overset{\circ}{\mathbf D}}{}}
\nc{\oA}{{\overset{\circ}{\mathbb A}}{}}
\nc{\op}{{\overset{\bullet}{\mathbf p}}{}}
\nc{\cp}{{\overset{\circ}{\mathbf p}}{}}
\nc{\oU}{{\overset{\bullet}{\mathcal U}}{}}
\nc{\oZ}{{\overset{\circ}{\mathcal Z}}{}}
\nc{\ofZ}{{\overset{\circ}{\mathfrak Z}}{}}
\nc{\oF}{{\overset{\circ}{\fF}}}

\nc{\fa}{{\mathfrak{a}}}
\nc{\fb}{{\mathfrak{b}}}
\nc{\fd}{{\mathfrak{d}}}
\nc{\ff}{{\mathfrak{f}}}
\nc{\fg}{{\mathfrak{g}}}
\nc{\fgl}{{\mathfrak{gl}}}
\nc{\fh}{{\mathfrak{h}}}
\nc{\fj}{{\mathfrak{j}}}
\nc{\fk}{{\mathfrak{k}}}
\nc{\fl}{{\mathfrak{l}}}
\nc{\fm}{{\mathfrak{m}}}
\nc{\fn}{{\mathfrak{n}}}
\nc{\fu}{{\mathfrak{u}}}
\nc{\fp}{{\mathfrak{p}}}
\nc{\fr}{{\mathfrak{r}}}
\nc{\fs}{{\mathfrak{s}}}
\nc{\ft}{{\mathfrak{t}}}
\nc{\fz}{{\mathfrak{z}}}
\nc{\fsl}{{\mathfrak{sl}}}
\nc{\hsl}{{\widehat{\mathfrak{sl}}}}
\nc{\hgl}{{\widehat{\mathfrak{gl}}}}
\nc{\hg}{{\widehat{\mathfrak{g}}}}
\nc{\hf}{{\widehat{\mathfrak{f}}}}
\nc{\chg}{{\widehat{\mathfrak{g}}}{}^\vee}
\nc{\hn}{{\widehat{\mathfrak{n}}}}
\nc{\chn}{{\widehat{\mathfrak{n}}}{}^\vee}

\nc{\fA}{{\mathfrak{A}}}
\nc{\fB}{{\mathfrak{B}}}
\nc{\fD}{{\mathfrak{D}}}
\nc{\fE}{{\mathfrak{E}}}
\nc{\fF}{{\mathfrak{F2}}}
\nc{\fG}{{\mathfrak{G}}}
\nc{\fK}{{\mathfrak{K}}}
\nc{\fL}{{\mathfrak{L}}}
\nc{\fM}{{\mathfrak{M}}}
\nc{\fN}{{\mathfrak{N}}}
\nc{\fP}{{\mathfrak{P}}}
\nc{\fU}{{\mathfrak{U}}}
\nc{\fV}{{\mathfrak{V}}}
\nc{\fZ}{{\mathfrak{Z}}}

\nc{\ba}{{\mathbf{a}}}
\nc{\bb}{{\mathbf{b}}}
\nc{\bc}{{\mathbf{c}}}
\nc{\bd}{{\mathbf{d}}}
\nc{\bbf}{{\mathbf{f}}}
\nc{\be}{{\mathbf{e}}}
\nc{\bi}{{\mathbf{i}}}
\nc{\bj}{{\mathbf{j}}}
\nc{\bn}{{\mathbf{n}}}
\nc{\bo}{{\mathbf{o}}}
\nc{\bp}{{\mathbf{p}}}
\nc{\bq}{{\mathbf{q}}}
\nc{\bu}{{\mathbf{u}}}
\nc{\bv}{{\mathbf{v}}}
\nc{\bx}{{\mathbf{x}}}
\nc{\bs}{{\mathbf{s}}}
\nc{\by}{{\mathbf{y}}}
\nc{\bw}{{\mathbf{w}}}
\nc{\bA}{{\mathbf{A}}}
\nc{\bK}{{\mathbf{K}}}
\nc{\bB}{{\mathbf{B}}}
\nc{\bF}{{\mathbf{F2}}}
\nc{\bC}{{\mathbf{C}}}
\nc{\bG}{{\mathbf{G}}}
\nc{\bD}{{\mathbf{D}}}
\nc{\bE}{{\mathbf{E}}}
\nc{\bH}{{\mathbf{H}}}
\nc{\bI}{{\mathbf{I}}}
\nc{\bM}{{\mathbf{M}}}
\nc{\bN}{{\mathbf{N}}}
\nc{\bO}{{\mathbf{O}}}
\nc{\bV}{{\mathbf{V}}}
\nc{\bW}{{\mathbf{W}}}
\nc{\bX}{{\mathbf{X}}}
\nc{\bZ}{{\mathbf{Z}}}
\nc{\bS}{{\mathbf{S}}}

\nc{\sA}{{\mathsf{A}}}
\nc{\sB}{{\mathsf{B}}}
\nc{\sC}{{\mathsf{C}}}
\nc{\sD}{{\mathsf{D}}}
\nc{\sF}{{\mathsf{F}}}
\nc{\sG}{{\mathsf{G}}}
\nc{\sH}{{\mathsf{H}}}
\nc{\sK}{{\mathsf{K}}}
\nc{\sM}{{\mathsf{M}}}
\nc{\sO}{{\mathsf{O}}}
\nc{\sW}{{\mathsf{W}}}
\nc{\sQ}{{\mathsf{Q}}}
\nc{\sP}{{\mathsf{P}}}
\nc{\sZ}{{\mathsf{Z}}}
\nc{\sfi}{{\mathsf{i}}}
\nc{\sr}{{\mathsf{r}}}
\nc{\bk}{{\mathsf{k}}}
\nc{\sg}{{\mathsf{g}}}
\nc{\sff}{{\mathsf{f}}}
\nc{\sfe}{{\mathsf{e}}}
\nc{\sfj}{{\mathsf{j}}}
\nc{\sfb}{{\mathsf{b}}}
\nc{\sfc}{{\mathsf{c}}}
\nc{\sd}{{\mathsf{d}}}
\nc{\sv}{{\mathsf{v}}}

\nc{\BK}{{\bar{K}}}

\nc{\tA}{{\widetilde{\mathbf{A}}}}
\nc{\tB}{{\widetilde{\mathcal{B}}}}
\nc{\tg}{{\widetilde{\mathfrak{g}}}}
\nc{\tG}{{\widetilde{G}}}
\nc{\TM}{{\widetilde{\mathbb{M}}}{}}
\nc{\tO}{{\widetilde{\mathsf{O}}}{}}
\nc{\tU}{{\widetilde{\mathfrak{U}}}{}}
\nc{\TZ}{{\tilde{Z}}}
\nc{\tx}{{\tilde{x}}}
\nc{\tbv}{{\tilde{\bv}}}
\nc{\tfP}{{\widetilde{\mathfrak{P}}}{}}
\nc{\tz}{{\tilde{\zeta}}}
\nc{\tmu}{{\tilde{\mu}}}

\nc{\urho}{\underline{\pi}}
\nc{\uB}{\underline{B}}
\nc{\uC}{{\underline{\mathbb{C}}}}
\nc{\ui}{\underline{i}}
\nc{\uj}{\underline{j}}
\nc{\ofP}{{\overline{\mathfrak{P}}}}
\nc{\oB}{{\overline{\mathcal{B}}}}
\nc{\og}{{\overline{\mathfrak{g}}}}
\nc{\oI}{{\overline{I}}}

\nc{\eps}{\varepsilon}
\nc{\hrho}{{\hat{\pi}}}

\nc{\one}{{\mathbf{1}}}
\nc{\two}{{\mathbf{t}}}

\nc{\Rep}{{\mathop{\operatorname{\rm Rep}}}}
\nc{\Tot}{{\mathop{\operatorname{\rm Tot}}}}
\nc{\Ker}{{\mathop{\operatorname{\rm Ker}}}}
\nc{\Hilb}{{\mathop{\operatorname{\rm Hilb}}}}
\nc{\End}{{\mathop{\operatorname{\rm End}}}}
\nc{\Ext}{{\mathop{\operatorname{\rm Ext}}}}
\nc{\CHom}{{\mathop{\operatorname{{\mathcal{H}}\it om}}}}
\nc{\GL}{{\mathop{\operatorname{\rm GL}}}}
\nc{\gr}{{\mathop{\operatorname{\rm gr}}}}
\nc{\Id}{{\mathop{\operatorname{\rm Id}}}}
\nc{\de}{{\mathop{\operatorname{\rm def}}}}
\nc{\length}{{\mathop{\operatorname{\rm length}}}}
\nc{\supp}{{\mathop{\operatorname{\rm supp}}}}

\nc{\Cliff}{{\mathsf{Cliff}}}
\nc{\Fl}{\on{Fl}}
\nc{\Fib}{{\mathsf{Fib}}}
\nc{\Coh}{{\on{Coh}}}
\nc{\QCoh}{{\on{QCoh}}}
\nc{\IndCoh}{{\on{IndCoh}}}
\nc{\FCoh}{{\mathsf{FCoh}}}

\nc{\reg}{{\text{\rm reg}}}

\nc{\cplus}{{\mathbf{C}_+}}
\nc{\cminus}{{\mathbf{C}_-}}
\nc{\cthree}{{\mathbf{C}_*}}
\nc{\Qbar}{{\bar{Q}}}
\nc\Eis{\on{Eis}}
\nc\Eisb{\ol\Eis{}}
\nc\Eisr{\on{Eis}^{rat}{}}
\nc\wh{\widehat}
\nc{\Def}{\on{Def_{\check{\fb}}(E)}}
\nc{\barZ}{\overline{Z}{}}
\nc{\barbarZ}{\overline{\barZ}{}}
\nc{\barpi}{\overline\iota}
\nc{\barbarpi}{\overline\barpi}
\nc{\barpip}{\overline\iota{}^+}
\nc{\barpim}{\overline\iota{}^-}

\nc{\fq}{\mathfrak q}

\nc{\fqb}{\ol{\fq}{}}
\nc{\fpb}{\ol{\fp}{}}
\nc{\fpr}{{\fp^{rat}}{}}
\nc{\fqr}{{\fq^{rat}}{}}

\nc{\hattimes}{\wh\otimes}

\nc{\bh}{{\bar{h}}}
\nc{\bOmega}{{\overline{\Omega(\check \fn)}}}

\nc{\seq}[1]{\stackrel{#1}{\sim}}

\nc{\cT}{{\check{T}}}
\nc{\cG}{{\check{G}}}
\nc{\cM}{{\check{M}}}
\nc{\cB}{{\check{B}}}
\nc{\cN}{{\check{N}}}

\nc{\ct}{{\check{\mathfrak t}}}
\nc{\cg}{{\check{\fg}}}
\nc{\cb}{{\check{\fb}}}
\nc{\cn}{{\check{\fn}}}

\nc{\cLambda}{{\check\Lambda}}

\nc{\cla}{{\check\kappa_x}}
\nc{\cmu}{{\check\mu}}
\nc{\clambda}{{\check\lambda}}
\nc{\cnu}{{\check\nu}}
\nc{\ceta}{{\check\eta}}

\nc{\DefbE}{{\on{Def}_{\cB}(E_\cT)}}

\nc{\imathb}{{\ol{\imath}}}
\nc{\rlr}{\overset{\longrightarrow}{\underset{\longrightarrow}\longleftarrow}}

\nc{\KG}{K\backslash G}
\nc{\comult}{{co\text{-}mult}}
\nc{\counit}{{co\text{-}unit}}
\nc{\uHom}{{\underline{\Maps}}}
\nc{\dgSch}{\on{Sch}}
\nc{\Sch}{\on{Sch}}
\nc{\affdgSch}{\on{Sch}^{\on{aff}}}
\nc{\affSch}{\on{Sch}^{\on{aff}}}
\nc{\Groupoids}{\on{Grpd}}
\nc{\inftygroup}{\on{Spc}}
\nc{\inftyCat}{\infty\on{-Cat}}
\nc{\StinftyCat}{\inftyCat^{\on{St}}}
\nc{\MoninftyCat}{\infty\on{-Cat}^{\on{Mon}}}
\nc{\SymMoninftyCat}{\infty\on{-Cat}^{\on{SymMon}}}
\nc{\SymMonStinftyCat}{\on{DGCat}^{\on{SymMon}}}
\nc{\MonStinftyCat}{\on{DGCat}^{\on{Mon}}}
\nc{\inftystack}{\on{Stk}}
\nc{\inftystackalg}{Stk^{1\text{-}alg}}
\nc{\inftyprestack}{\on{PreStk}}
\nc{\inftydgnearstack}{\on{NearStk}}
\nc{\inftydgstack}{\on{Stk}}
\nc{\inftydgstackalg}{DGStk^{1\text{-}alg}}
\nc{\inftydgprestack}{\on{PreStk}}
\nc{\dgindSch}{\on{indSch}}
\nc{\indSch}{{}^{\on{cl}}\!\on{indSch}}
\nc{\infSch}{\on{infSch}}
\nc{\dr}{{\on{dR}}}

\nc{\mmod}{{\on{-}\!{\mathbf{mod}}}}

\nc{\starr}{\text{\dh}}
\nc{\Spectra}{\on{Spectra}}
\nc{\Crys}{\on{Crys}}
\nc{\oblv}{{\mathbf{oblv}}}
\nc{\ind}{{\mathbf{ind}}}
\nc{\coind}{{\mathbf{coind}}}
\nc{\inv}{{\mathbf{inv}}}
\nc{\triv}{{\mathbf{triv}}}
\nc{\CMaps}{{\mathcal Maps}}
\nc{\Maps}{\on{Maps}}
\nc{\bMaps}{\mathbf{Maps}}
\nc{\Grid}{\on{Grid}}
\nc{\hGrid}{\on{Grid}^{\geq\,\on{dgnl}}}
\nc{\Diag}{\on{Diag}}
\nc{\bDelta}{\mathbf{\Delta}}
\nc{\tCateg}{(\infty\on{-2)-Cat}}
\nc{\ul}{\underline}
\nc{\Seg}{\on{Seq}}
\nc{\biSeg}{\on{bi-Seq}}
\nc{\triSeg}{\on{tri-Seq}}
\nc{\quadSeg}{\on{quad-Seq}}
\nc{\nSeg}{\on{n-Seq}}
\nc{\Segm}{\on{Seg}^{\on{mkd}}}
\nc{\fLm}{\fL^{\on{mkd}}}
\nc{\inftyCatm}{\inftyCat^{\on{mkd}}}
\nc{\Blocks}{\mathbf{Blocks}}
\nc{\Snakes}{\mathbf{Snakes}}
\nc{\bifL}{\on{bi-}\!\fL}
\nc{\Sets}{\on{Sets}}
\nc{\Ran}{\on{Ran}}
\nc{\Vect}{\on{Vect}}
\nc{\Shv}{\on{Shv}}
\nc{\unn}{\mathbf{union}}
\nc{\Spc}{\on{Spc}}
\nc{\ppart}{(\!(t)\!)}
\nc{\qqart}{[\![t]\!]}
\nc{\Dmod}{\on{D-mod}}
\nc{\cD}{\mathcal D}
\nc{\ocD}{\overset{\circ}{\cD}}
\nc{\sfp}{\mathsf{p}}
\nc{\sfq}{\mathsf{q}}
\nc{\DGCat}{\on{DGCat}}
\renc{\det}{\on{det}}
\nc{\Conf}{\on{Conf}}
\nc{\Whit}{\on{Whit}}
\nc{\Reg}{\on{Reg}}
\nc{\Res}{\on{Res}}
\nc{\BunNbox}{(\overline\Bun_N^{\omega^\rho})_{\infty\cdot x}} 
\nc{\BunNmbox}{(\overline\Bun_{N^-}^{\omega^\rho})_{\infty\cdot x}}
\nc{\bHecke}{\overset{\bullet}{\on{Hecke}}}
\nc{\Hecke}{\on{Hecke}}
\nc{\bCZ}{\ol\CZ}
\nc{\oCZ}{\overset{\circ}\CZ} 
\nc{\boCZ}{\ol{\oCZ}}
\nc{\sotimes}{\overset{!}\otimes}
\nc{\semiinf}{{\frac{\infty}{2}}}
\nc{\coInd}{\on{coInd}}
\nc{\bCM}{\overset{\bullet}\CM{}}
\nc{\bCF}{\overset{\bullet}\CF{}}
\nc{\SI}{\on{SI}}
\nc{\KL}{\on{KL}}
\nc{\htt}{\wh{\mathfrak{t}}}
\nc{\Ind}{\on{Ind}}
\nc{\Bl}{\mathsf{Bl}}

\title{A conjectural extension of the Kazhdan-Lusztig equivalence} 

\author{Dennis Gaitsgory}

\dedicatory{To Masaki Kashiwara, with admiration} 

\date{\today}

\begin{document} 

\begin{abstract}
A theorem of Kazhdan and Lusztig establishes an equivalence between the category of
$G(\CO)$-integrable representations of the Kac-Moody algebra $\hg_{-\kappa}$ at a negative level $-\kappa$
and the category $\Rep_q(G)$ of (algebraic) representations of the ``big" (a.k.a. Lusztig's) quantum group.
In this paper we propose a conjecture that describes the category of Iwahori-integrable Kac-Moody modules.
The corresponding object on the quantum group side, denoted $\Rep^{\on{mxd}}_q(G)$, involves Lusztig's version 
of the quantum group for the Borel and the De Concini-Kac version for the negative Borel.  
\end{abstract} 

\maketitle

\tableofcontents

\section*{Introduction}

\ssec{What is this paper about?}

\sssec{}

In their series of papers \cite{KL}, D.~Kazhdan and G.~Lusztig established an equivalence between the (abelian) category
$$\KL(G,-\kappa):=\hg\mod_{-\kappa}^{G(\CO)}$$
of $G(\CO)$-integrable modules over the affine Kac-Moody Lie algebra at a negative level $-\kappa$ and the (abelian) category 
$\Rep_q(G)$ of integrable (=algebraic) representations of the ``big" quantum group, whose quantum parameter $q$ is related
to $\kappa$ via formula \eqref{e:q and kappa}.

\medskip

This paper addresses the following natural question: if we enlarge the category $\hg\mod_{-\kappa}^{G(\CO)}$ to 
$\hg\mod_{-\kappa}^I$, i.e., if we relax $G(\CO)$-integrability to Iwahori-integrability, what would the corresponsing
category on the quantum group side be? 

\sssec{}

We propose a conjectural answer to this question: namely, we define a version of the category of modules over the
quantum group, denoted $\Rep^{\on{mxd}}_q(G)$, see \secref{ss:mixed} for the actual definition. 
Our \conjref{c:main} says that there is supposed to be an equivalence 
\begin{equation} \label{e:main intro}
\sF_{-\kappa}:\hg\mod_{-\kappa}^I\simeq \Rep^{\on{mxd}}_q(G).
\end{equation} 

\sssec{}

Let us explain the main points of difference/similarity between $\Rep^{\on{mxd}}_q(G)$ and the usual category $\Rep_q(G)$:

\begin{itemize}

\item In both versions, the vector space underlying an object of our category of representations is $\cLambda$-graded,
where $c\Lambda$ is the weight lattice of $G$, and carries an action of the \emph{positive part of Lusztig's quantum group},
$U_q^{\on{Lus}}(N)$;

\item In the case of $\Rep_q(G)$, our vector space carries an \emph{integrable} action also of the \emph{negative part of Lusztig's quantum group},
$U_q^{\on{Lus}}(N^-)$, while in the case of $\Rep^{\on{mxd}}_q(G)$ it carries a \emph{not necessarily integrable} action of the 
\emph{negative part of the De Concini-Kac version of the quantum group}, $U_q^{\on{DK}}(N^-)$.

\end{itemize}

\medskip

Note that when $q$ is \emph{not} a root of unity, there is no difference between the Lusztig and De Concini-Kac versions. So in this
case, the difference is in the integrability condition with respect to $U_q(N^-)$. In other words, for $q$ not a root of unity, our 
$\Rep^{\on{mxd}}_q(G)$ is just the quantum category $\CO$.

\medskip

However, when $q$ is a root of unity, $U_q^{\on{DK}}(N^-)$ and $U_q^{\on{Lus}}(N^-)$ are truly different, so in this case
$\Rep^{\on{mxd}}_q(G)$ is emphatically \emph{not} the quantum category $\CO$.

\sssec{}

Here is one caveat regarding the proposed equivalence \eqref{e:main intro}: 

\medskip

As was just mentioned, the original equivalence of \cite{KL}
\begin{equation} \label{e:original intro}
\sF_{-\kappa}:\hg\mod_{-\kappa}^{G(\CO)}\simeq \Rep_q(G)
\end{equation} 
is an exact equivalence of abelian categories. Hence, it induces an equivalence of their derived categories,
preserving the t-structures.  

\medskip

From now on, when we say ``category" we will mean a triangulated category, or even
more precisely, a DG category. If $\CC$ is such a category, and if it is equipped with a t-structure, 
we will write $\CC^\heartsuit$ for the heart of the t-structure. So let us read \eqref{e:original intro}  
as an equivalence of derived categories; the original equivalence at the abelian level is
$$(\hg\mod_{-\kappa}^{G(\CO)})^\heartsuit \simeq (\Rep_q(G))^\heartsuit.$$

\medskip

Now, the point here is that the extended equivalence \eqref{e:main intro} only holds at the level of 
DG categories. I.e., it is \emph{not} t-exact with respect to the natural t-structures that exist on
both sides. 

\begin{rem}
That said, one can try to mimic the construction of \cite[Sect. 2]{FG2} to define a new t-structure on 
$\hg\mod_{-\kappa}^I$ so that the equivalence \eqref{e:main intro} becomes t-exact. This is an interesting
problem, but we will not pursue it in this paper. See, however, \secref{sss:semi-inf categ} below. 
\end{rem} 

\sssec{}

Let us mention one curious feature of the equivalence \eqref{e:main intro}. 

\medskip

Recall that the Kazhdan-Lusztig equivalence \eqref{e:original intro} sends the standard objects
$$\BV_{-\kappa}^\clambda \in \hg\mod_{-\kappa}^{G(\CO)}, \quad \clambda\in \cLambda^+$$
(affine Weyl modules) to the standard objects
$$\CV_q^\clambda\in \Rep_q(G), \quad \clambda\in \cLambda^+$$
(quantum Weyl modules).

\medskip

Now, the category $\Rep^{\on{mxd}}_q(G)$ contains a naturally defined family of standard objects,
denoted $\BM^\clambda_{q,\on{mxd}}$, where $\clambda$ is a weight of $G$. These are modules induced from
characters of the quantum Borel.

\medskip

Under the equivalence \eqref{e:main intro}, the objects $\BM^\clambda_{q,\on{mxd}}$ \emph{do not} correspond to
the affine Verma modules. Rather, they go over the \emph{Wakimoto modules}\footnote{Our conventions regarding Wakimoto
modules are different from those in most places in the literature such as \cite{Fr2} abd \cite{FG2}, see Remark \ref{r:other Wak}},
denoted $\BW^\clambda_{-\kappa}$.

\medskip

Let us mention that if $\kappa$ is irrational, then $\BW^\clambda_{-\kappa}$ is actually isomorphic to the affine Verma module
$\BM^\clambda_{-\kappa}$. If $\kappa$ is rational, the isomorphism still holds for $\clambda$ dominant, but not otherwise.
For example, if $\clambda$ is sufficiently anti-dominant, the Wakimoto module $\BW^\clambda_{-\kappa}$ is isomorphic to
the \emph{dual affine Verma module} $\BM^{\vee,\clambda}_{-\kappa}$. 

\ssec{Where did the motivation come from?}  \label{ss:motiv 1}

The motivation for guessing the equivalence \eqref{e:main intro} came from multiple sources. 

\medskip

One source of motivation is the author's desire to re-prove the original Kazhdan-Lusztig equivalence \eqref{e:original intro}
by a ``more algebraic method".

\sssec{}

We recall that the statement of the equivalence \eqref{e:original intro} in \cite{KL} is not as mere abelian categories,
but as \emph{braided monoidal} abelian categories. In fact, the proof of the equivalence in \cite{KL} uses
this additional structure in a most essential way. 

\medskip

One can interpret the braided monoidal structure on $\KL(G,-\kappa)$ as a structure of \emph{de Rham factorization category},
and the braided monoidal structure on $\Rep_q(G)$ as a structure of \emph{Betti factorization factorization category}. From this
point of view, the equivalence \eqref{e:original intro} should read that these two structures match up under the (appropriately defined)
Riemann-Hilbert functor that maps Betti factorization categories to de Rham factorization categories. 

\sssec{}

Now, a structure of factorization category in either of the two contexts is a complicated piece of data. However, there is one
case when a factorization category can be described succinctly: namely, when the factorization category in question is that
of \emph{factorization modules} for a \emph{factorization algebra} (say, within another factorization category, but one which
is easier to understand). 

\medskip

The ``trouble with" the original Kazhdan-Lusztig equivalence \eqref{e:original intro} is that the categories involved are 
\emph{not} of this form. 

\sssec{}

By contrast, the factorization category corresponding to $\Rep^{\on{mxd}}_q(G)$ is (more or less tautologically) equivalent 
to that of factorization modules. 

\medskip

The ambient factorization category in question is the factorization category corresponding
to the braided monoidal category $\Rep_q(T)$ of representations of the quantum torus. The factorization algebra in question,
denoted $\Omega_q^{\on{Lus}}$, is the Koszul dual of the Hopf algebra $U^{\on{Lus}}_q(N)\in \Rep_q(T)$. This actually explains the 
appearance of the De Concini-Kac version: it enters as the dual Hopf algebra of $U^{\on{Lus}}_q(N)$.

\sssec{}  \label{sss:semi-inf categ}

Let us now look at the left-hand side of the proposed equivalence \eqref{e:main intro}. 

\medskip

The Iwahori subgroup $I$ is not a factorizable object. However, the following result was proved by S.~Raskin
(see, e.g., \cite[Sect. 5]{Ga2} for a proof): for a category $\CC$ equipped with an action of $G(\CK)$ (see \secref{sss:action at level kappa}
for what this means), there is a canonical equivalence
$$\CC^I\simeq \CC^{N(\CK)\cdot T(\CO)};$$
here the superscript indicates taking the equivariant category with respect to the corresponding subgroup. 

\medskip

Hence, we can interpret the category $\hg\mod_{-\kappa}^I$ as 
\begin{equation} \label{e:semi-inf categ}
\hg\mod_{-\kappa}^{N(\CK)\cdot T(\CO)}.
\end{equation} 

This paves a way to using factorization methods, as the category \eqref{e:semi-inf categ} admits a natural factorization structure.

\medskip

If we could prove that a certain explicit functor from $\hg\mod_{-\kappa}^{N(\CK)\cdot T(\CO)}$ to the category of factorization modules over 
a De Rham version $\Omega_{-\kappa}^{\on{Lus}}$ of $\Omega_q^{\on{Lus}}$ was an equivalence, that would establish the equivalence
\eqref{e:main intro}. And having \eqref{e:main intro}, one can hope to be able to extract the original equivalence \eqref{e:original intro}. 

\sssec{}

Now, an equivalence between $\hg\mod_{-\kappa}^{N(\CK)\cdot T(\CO)}$ and 
$\Omega_{-\kappa}^{\on{Lus}}\mod^{\on{Fact}}$ as factorization categories may or may not be too much to hope for. 

\medskip

However, such an equivalence is supposed to take place at the level of fibers (this is what our Conjecture \eqref{e:main intro} says)
and also on sufficiently large subcategories, which should be enough to deduce the original equivalence \eqref{e:original intro}. 

\ssec{Motivation from local geometric Langlands}

There is yet another aspect to the above story, which has to do with local geometric Langlands.

\sssec{}

One of the key conjectures (proposed in 2008 by J.~Lurie and the author) is that the category $\KL(G,-\kappa)$ is supposed to be
equivalent (as a factorization category) to the twisted Whittaker category $\Whit_{-\check\kappa}(\Gr_{\cG})$ of the affine
Grassmannian of the Langlands dual group (here $\check\kappa$ is the level for $\cG$ dual to the level $\kappa$ for $G$,
see \secref{sss:duality for levels}):
\begin{equation} \label{e:FLE}
\hg\mod_{-\kappa}^{G(\CO)} \simeq \Whit_{-\check\kappa}(\Gr_{\cG})
\end{equation} 

This conjecture is called the \emph{Fundamental Local Equivalence}, or the FLE.  

\medskip

The trouble with proving the FLE is essentially the same one as with the original Kazhdan-Lusztig equivalence 
\eqref{e:original intro}: the two sides are some complicated factorization categories, yet we must relate them, based 
just on the combinatorial information that the groups $G$ and $\cG$ are mutually dual.

\sssec{}

However, just as in \secref{ss:motiv 1}, one can have a better chance to first prove the Iwahori version of the FLE, namely, 
an equivalence
\begin{equation} \label{e:Iw FLE}
\hg\mod_{-\kappa}^I \simeq \Whit_{-\check\kappa}(\on{Fl}^{\on{aff}}_{\cG}),
\end{equation} 
and then bootstrap from it the original FLE \eqref{e:FLE}.

\sssec{}

In a subsequent publication, the author is planning to record the properties of the conjectural equivalence 
\eqref{e:Iw FLE}: the behavior of the standard and costandard objects, compatibility with duality, etc. 

\ssec{Representation-theoretic motivation}

Finally, there is a purely representation-theoretic piece of motivation for the conjectural equivalence
\eqref{e:main intro}\footnote{It originated in discussions between S.~Arkhipov, R.~Bezrukavnikov, M.~Finkelberg, I.~Mikovi\'c and the author
some 20 years ago.}. 

\sssec{}

We have the equivalence established in the paper \cite{ABG} that says that a regular block $\Bl(\Rep^{\on{sml}}_q(G))$
of the category of modules over the \emph{small} quantum group is equivalent to the category 
$$\IndCoh(\{0\}\underset{\cg}\times \wt\cg)$$
of ind-coherent sheaves on the \emph{derived Spring fiber} over $0$ for the Langlands dual Lie algebra. 

\sssec{}

Let now $\chi$ be a point of the spectrum of the ``$q$-center" of the De Concini-Kac algebra. We can form
the corresponding category
$$\Rep^{\chi}_q(G)$$
(so that for $\chi=0$ we recover $\Rep^{\on{sml}}_q(G)$), and consider its regular block $\Bl(\Rep^{\chi}_q(G))$.

\medskip

By analogy with representations in positive characteristic, one conjectured an equivalence
\begin{equation} \label{e:chi intro}
\Bl(\Rep^{\chi}_q(G))\simeq \IndCoh(\{\chi\}\underset{\cg}\times \wt\cg),
\end{equation} 
generalizing the equivalence of \cite{ABG}. 

\sssec{}

Now, as we shall see in \secref{s:coh}, our conjectural equivalence \eqref{e:main intro}, when restricted to a regular block gives an equivalence
$$\Bl(\Rep_q^{\on{mxd}}(G))\simeq \IndCoh((\cn\underset{\cg}\times \wt\cg)/\cB),$$
which is a version of \eqref{e:chi intro} in the family $\chi\in \cn/\on{Ad}(\cB)$. 

\ssec{What is done in this paper?}

We now proceed to review the actual mathematical content of the paper. 

\sssec{}

This paper centers around \conjref{c:main}, which proclaims the existence of an equivalence 
\eqref{e:main intro}. However, the statement of the conjecture is both preceded and followed 
by some 45 pages of mathematical text. 

\medskip

In fact, this paper is divided into 3 parts. In the first part, we recall some facts pertaining 
to the affine category $\CO$. In the second part, we review various
versions of the category of modules over the quantum group. Neither Part I nor Part II contain
substantial original results. In the third part, after stating
\conjref{c:main}, we run some consistency checks and derive some consequences.

\sssec{}  \label{sss:duality for levels}

Part 3 is the core of this paper, in which we study the affine algebra vs quantum group relationship. Here we need to take our field
of coefficients $k$ to be $\BC$. 

\medskip

Our quantum parameter is the quadratic form $q$ on $\cLambda$ with coefficients in $k^\times$, related to $\kappa$ by the formula 
\begin{equation} \label{e:q and kappa}
q(\clambda):=\on{exp}(2\cdot \pi\cdot i\cdot \frac{\check\kappa(\clambda,\clambda)}{2}),
\end{equation} 
where $\check\kappa$ is as in \secref{sss:level} below.

\medskip

Note that the quadratic form $q$ comes as restriction to the diagonal of the symmetric bilinear form $b'$ on $\cLambda$ with coefficients in $k^\times$
equal to
\begin{equation} \label{e:b and kappa}
b'(\clambda_1,\clambda_2):=\on{exp}(2\cdot \pi\cdot i\cdot \frac{\check\kappa(\clambda_1,\clambda_2)}{2}).
\end{equation} 

\sssec{}

After stating \conjref{c:main}, we do the following:

\medskip

\noindent--We note that \conjref{c:main} implies that the original Kazhdan-Lusztig equivalence \eqref{e:original intro} satisfies
\begin{equation}  \label{e:Weyl intro}
\sF_{-\kappa}(\BV^\clambda_{-\kappa}):=\CV_q^\clambda,\quad \clambda\in \cLambda
\end{equation}
where $\BV^\clambda_{-\kappa}\in \KL(G,-\kappa)$ is the (derived) Weyl module, defined by
$$\BV^\clambda_{-\kappa}:=\on{Av}_!^{G(\CO)/I}(\BW_{-\kappa}^\clambda),$$
and $\CV^\clambda_q\in \Rep_q(G)$ is the (derived) Weyl module, defined by
$$\CV^\clambda_{-\kappa}:=\ind^{\Rep_q(G)}_{\Rep_q(B)}(k^\clambda).$$

\medskip

\noindent--We prove the isomorphism \eqref{e:Weyl intro} unconditionally (i.e., without assuming \conjref{c:main}).
This occupies most of \secref{s:small cohomology}. We should mention that this isomorphism is essentially 
equivalent to the main result of \cite{Liu}. 

\medskip

\noindent--We give an expression for the functor
\begin{equation} \label{e:small intro}
\Rep_q(G)\to \Vect, \quad \CM\mapsto \on{C}^\cdot(u_q(N),\CM)^\clambda, \quad \clambda\in \cLambda
\end{equation} 
in terms of the original Kazhdan-Lusztig equivalence \eqref{e:original intro}. This is done in \secref{ss:small cohomology}. 

\medskip

\noindent--Assuming \conjref{c:main}, we show that the functor
\begin{equation} \label{e:DK intro}
\Rep^{\on{mxd}}_q(G)\to \Vect, \quad \CM\mapsto \on{C}^\cdot(U^{\on{DK}}_q(N^-),\CM)^\clambda, \quad \clambda\in \cLambda
\end{equation} 
corresponds under the equivalence \eqref{e:main intro} to the functor
$$\hg\mod_{-\kappa}^I\to \Vect, \quad \CM\mapsto \on{C}^\semiinf(\fn^-(\CK),\CM)^\clambda,$$
see \secref{sss:seminf la} for the notation. This is done in \secref{s:DK cohomology}. 

\medskip

\noindent--Assuming \conjref{c:main}, we show that under the equivalence
$$\Bl(\Rep^{\on{mxd}}_q(G))\simeq \IndCoh((\wt{\check\CN}\underset{\cg}\times \wt\cg)/\cG)$$
(normalized as in \secref{sss:coh norm}), the functor \eqref{e:DK intro} for $\clambda=\clambda_0$ corresponds to the functor 
of restriction to the big Schubert cell 
$$\on{pt}/\cT\simeq (\cG/\cB\times \cG/\cB)^o/\cG\subset (\wt{\check\CN}\underset{\cg}\times \wt\cg)/\cG,$$
followed by the functor of $\cT$-invariants. This is done in \secref{ss:DK cohomology}.

\medskip

\noindent--Finally, we show that the duality equivalences
$$(\Rep_q(G))^\vee \simeq \Rep_{q^{-1}}(G) \text{ and } (\Rep^{\on{mxd}}_q(G))^\vee \simeq \Rep^{\on{mxd}}_{q^{-1}}(G)$$
and
$$\KL(G,-\kappa)^\vee\simeq \KL(G,\kappa) \text{ and } (\hg\mod_{-\kappa}^I)^\vee \simeq \hg\mod_\kappa^I$$
combined with the equivalence(s) $\sF_{-\kappa}$ induce an equivalence
$$\sF_\kappa:\KL(G,\kappa)\simeq \Rep_{q^{-1}}(G),$$
and, assuming \conjref{c:main}, also an equivalence
$$\sF_\kappa:  \hg\mod_\kappa^I\simeq \Rep^{\on{mxd}}_{q^{-1}}(G).$$

We study properties of these dual equivalences parallel to ones at the negative level, listed above. 

\sssec{What is \emph{not} done in this paper?}

There are two main themes that could have been part of this paper but that are not:

\medskip

One is the discussion of factorization (see \secref{ss:motiv 1}). The other is the relationship of the categories
appearing in \eqref{e:main intro} to the Whittaker category $\Whit_{-\check\kappa}(\on{Fl}^{\on{aff}}_{\cG})$.

\ssec{Structure of the paper}

We now proceed to describing the contents of the paper section-by-section.

\sssec{}

In \secref{s:recall} we recollect some basic facts pertaining to the category of Kac-Moody modules:
the definition is not completely straightforward as it involves a \emph{renormalization procedure}
designed to make the category compactly generated. We explain that the categories at opposite levels
are in relation of \emph{duality}, see \secref{sss:duality DG categ} for what this means. 

\medskip

In \secref{s:Wak} we introduce Wakimoto modules. We first introduce ``the true", i.e., semi-infinite
Wakimoto modules $\BW_\kappa^{\clambda,\semiinf}$; they are as induced from the loop subalgebra 
$\fb(\CK)\subset \fg(\CK)$. The key feature of $\BW^{\clambda,\semiinf}_\kappa$
is that it belongs to the category $\hg\mod_\kappa^{N(\CK)\cdot T(\CO)}$, see \secref{sss:semi-inf categ}. We then
show that the usual Wakimoto module $\BW^\clambda_\kappa$ can be obtained from $\BW^{\clambda,\semiinf}_\kappa$ 
by the procedure of averaging with respect to the Iwahori group.
We study the pattern of convolution of Wakimoto modules with the standard objects 
$J_\mu$ in the category $\Dmod_{-\kappa}(\on{Fl}_G^{\on{aff}})^I$. We show that $$J_\mu\star \BW^\clambda_\kappa\simeq 
\BW^{\clambda+\mu}_\kappa.$$ (This result had been previously obtained in \cite{FG2}.) 

\medskip

In \secref{s:Verma} we study the relationship between Wakimoto modules and Verma modules. 
We recall that these two classes of modules coincide when the level $\kappa$ is irrational.
When the level $\kappa$ is negative, we re-prove the theorem of \cite{Fr1} which says that 
the Wakimoto module $\BW_\kappa^\clambda$ is isomorphic to the affine Verma module 
$\BM_\kappa^\clambda$ if $\clambda$ is dominant, and to the dual affine Verma module 
$\BM_\kappa^{\vee,\clambda}$ if $\clambda$ is sufficiently anti-dominant. Whereas the original
proof in \cite{Fr1} relies on the analysis of singular vectors, our proof uses the Kashiwara-Tanisaki
localization theorem.  

\sssec{}

In \secref{s:quant alg} we introduce the general framework that most versions of the category
of modules over the quantum group fit into. Namely, we start with the datum of quadratic form
$q$ on the weight lattice with values in $k^\times$ and attach to it the braided monoidal category
$\Rep_q(T)$ that we think of as the category of representations of the quantum torus. Given
a Hopf algebra $A$ in $\Rep_q(T)$, we consider the category $A\mod(\Rep_q(T))_{\on{loc.nilp}}$
of locally nilpotent $A$-modules and its (relative to $\Rep_q(T)$) Drinfeld's center, denoted
$Z_{\on{Dr},\Rep_q(T)}(A\mod(\Rep_q(T))_{\on{loc.nilp}})$. We study the basic properties of
such categories: t-structures, standard and costandard objects, etc. 

\medskip

In \secref{s:quant grp} we specialize to the case of Hopf algebras $A$ relevant to quantum groups. 
The main example is $A=U_q^{\on{Lus}}(N)$. The resulting category 
$Z_{\on{Dr},\Rep_q(T)}(A\mod(\Rep_q(T))_{\on{loc.nilp}})$ is our $\Rep_q^{\on{mxd}}(G)$. 
In this section we also recall the definition of the category of algebraic (a.k.a. locally finite or
integrable) modules of Lusztig's quantum group, denoted $\Rep_q(G)$. We note that this
category does \emph{not} fit into the pattern of $Z_{\on{Dr},\Rep_q(T)}(A\mod(\Rep_q(T))_{\on{loc.nilp}})$,
which makes it more difficult and more interesting to study.

\medskip

In \secref{s:small} we specialize to the case when $q$ takes values in the group of roots of unity, and
consider the case of the Hopf algebra $A=u_q(N)$, the positive part of the ``small" quantum group.
We introduce the corresponding category $\Rep_q^{\on{sml,grd}}(G):=Z_{\on{Dr},\Rep_q(T)}(A\mod(\Rep_q(T))_{\on{loc.nilp}})$,
and study how it is related to the category $\Rep_q(G)$. We note that there are three versions of $\Rep_q^{\on{sml,grd}}(G)$
that differ from each other by renormalization (i.e., which objects are declared compact); it is important to keep track of these distinctions, 
for otherwise various desired equivalences would not hold ``as-is". 

\medskip

In \secref{s:1/2} we introduce yet another \emph{non-standard} version of the category of modules over the quantum group,
denoted $\Rep_q^{\frac{1}{2}}(G)$: it corresponds to having as the positive part Lusztig's algebra $U_q^{\on{Lus}}(N)$ and as
the negative part the ``small" version $u_q(N^-)$. There are three versions of this category that differ from each other
by renormalization, and they serve as intermediaries between $\Rep_q^{\on{sml,grd}}(G)$ and $\Rep_q^{\on{mxd}}(G)$.

\medskip

In \secref{s:q duality} we prove that the categories $\Rep_q^{\on{mxd}}(G)$ and $\Rep_{q^{-1}}^{\on{mxd}}(G)$ are each other's duals.
The corresponding duality functor $\BD^{\on{can}}$ sends the standard object $\BM^\clambda_{q,\on{mxd}}$ to
$\BM^{-\clambda-2\check\rho}_{q^{-1},\on{mxd}}[d]$. In addition, we introduce a contragredient duality functor $\BD^{\on{contr}}$
from (a certain subcategory containing all compact objects of) $\Rep_q^{\on{mxd}}(G)$
to the category $\Rep_{q^{-1}}^{\wt{\on{mxd}}}(G)$ (in which the roles of Lusztig's version and the De Concini-Kac version are swapped). We show
that the functors $\BD^{\on{can}}$ and $\BD^{\on{contr}}$ differ by a ``long intertwining functor"
$$\Upsilon: \Rep_q^{\on{mxd}}(G)\to \Rep_q^{\wt{\on{mxd}}}(G).$$

\sssec{}

In \secref{s:conj} we formulate our \conjref{c:main}, which states the existence of the equivalence satisfying certain properties. 
We then run some initial consistency checks. We also formulate a version of \conjref{c:main} for the positive level, obtained
from the initial one by duality. 

\medskip

In \secref{s:small cohomology} we prove \thmref{t:Weyl}, which states the existence of the isomorphism \eqref{e:Weyl intro}.
The key idea in the proof is to identify the object in the category
\begin{equation} \label{e:1/2 intro}
\Rep(\cB)\underset{\Rep(\cG)}\otimes \KL(G,-\kappa) 
\end{equation} 
that corresponds to
the \emph{baby Verma module}, considered as an object of the category
$$\Rep_q^{\frac{1}{2}}(G)\simeq \Rep(\cB)\underset{\Rep(\cG)}\otimes \Rep_q(G).$$
It turns out that such a description is essentially equivalent to the description of the functor
$$\KL(G,-\kappa)\to \Vect$$
corresponding to \eqref{e:small intro} via the Kazhdan-Lusztig equivalence \eqref{e:original intro}. 
The description of the sought-for object in \eqref{e:1/2 intro} is closely related to a certain
geometric object which was recently introduced in \cite{Ga2} under the name {\it semi-infinite
intersection cohomology sheaf}. 

\medskip

In \secref{s:AB action} we introduce an additional requirement that the conjectural equivalence \eqref{e:main intro}
is supposed to satisfy. Namely, it is supposed to be compatible with the actions on the two sides of the
monoidal category $\QCoh(\fn/\on{Ad}(B_H)$, where $\QCoh(\fn/\on{Ad}(B_H)$ acts on $\hg\mod_{-\kappa}^I$ via
the Arkhipov-Bezrukavnikov functor $$\QCoh(\fn/\on{Ad}(B_H)\to \Dmod_{-\kappa}(\on{Fl}_G^{\on{aff}})^I.$$ 
We will see that this compatibility gives a conceptual explanation of the identification of the object in
\eqref{e:1/2 intro} corresponding to the baby Verma module. 

\medskip

In \secref{s:DK cohomology} we will show that the functor $\hg\mod_{-\kappa}^I\to \Vect$ corresponding under the
conjectural equivalence \eqref{e:main intro} to the functor \eqref{e:DK intro}, is given by \emph{semi-infinite cohomology}
with respect to $\fn^-(\CK)$. 

\medskip

Finally, in \secref{s:coh}, we will explain how the equivalence \eqref{e:main intro} leads to the identification 
of a regular block $\Bl(\Rep^{\on{mxd}}_q(G))$ of $\Rep^{\on{mxd}}_q(G)$ with the category 
$\IndCoh((\wt{\check\CN}\underset{\cg}\times \wt\cg)/\cG)$ of ind-coherent sheaves on the Steinberg
stack $(\wt{\check\CN}\underset{\cg}\times \wt\cg)/\cG$. 

\ssec{Conventions}

\sssec{Ground field}

In this paper we will be working over a ground field $k$, assumed algebraically closed and of characteristic $0$. 
All algebro-geometric objects in this paper will be schemes (or more generally, prestacks) over $k$. 

\sssec{Reductive groups}

We let $G$ be a reductive group over $k$. We denote by $B$ a (chosen) Borel subgroup in $G$ and by $T$ its Cartan
quotient.  We let
$$\fg\supset \fb\twoheadrightarrow \ft$$
denote their respective Lie algebras. 

\medskip

Deviating slightly from the usual conventions, we will denote by $\Lambda$ the \emph{coweight} lattice of $T$,
and by $\cLambda$ the weight lattice. 

\sssec{The level}  \label{sss:level}

By a \emph{level} we will mean  a $W$-invariant symmetric bilinear form $\kappa$ on $\ft$, or which is the same, 
a $W$-invariant symmetric bilinear form on the coweight lattice $\Lambda$ with coefficients in $k$.

\medskip

We will assume that $\kappa$ is non-degenerate, i.e., it defines an isomorphism $\ft\to \ft^\vee=:\check\ft$. We let
$\check\kappa$ denote the resulting symmetric bilinear form on $\ft^\vee$, or which is the same, a $W$-invariant symmetric bilinear 
form on the weight lattice $\cLambda$ with coefficients in $k$.

\medskip

To $\kappa$ we attach an $\on{Ad}(G)$-invariant symmetric bilinear form on $\fg$ so that its restriction to $\ft$ equals
$$\kappa+\kappa_{\on{crit}},$$
where 
$$\kappa_{\on{crit}}=-\frac{\kappa_{\on{Kil}}}{2},$$
where $\kappa_{\on{Kil}}$ is the Killing form of the adjoint action of $\ft$ on $\fg$. 

\begin{rem}

Note that the level $\kappa=0$ corresponds to the form $\kappa_{\on{crit}}$ on $\fg$; it is called the critical level. Our
convention of shifting the level for $\fg$ by the critical is an affine version of the ``$\rho$-shift" in the usual representation theory of $\fg$. 

\end{rem} 

\sssec{DG categories}  \label{sss:DG categories}

The object of study of this paper is DG categories over $k$.
This automatically puts us in the context of 
higher algebra, developed in \cite{Lu}.  We refer the reader to \cite[Chapter 1]{GR} for a user guide. 

\medskip

There are a few things we need to mention about DG categories, viewed both intrinsically and extrinsically. 

\medskip 

When we say ``DG category" we will assume it to be \emph{cocomplete}, unless explicitly stated otherwise.
When talking about a functor between two DG categories, we will always mean an \emph{exact} functor
(i.e., a functor preserving finite colimits). When the DG categories in question are cocomplete, we will
assume our functor to be \emph{continuous} (i.e., preserving infinite direct sums, which, given exactness, 
is equivalent to preserving filtered colimits, and in fact all colimits). 

\medskip

Most of the DG categories we will encounter (and ones that we will end up working with) are \emph{compactly generated}.
If $\CC$ is such a category, we will denote by $\CC_c$ its full (but \emph{not} cocomplete) subcategory consisting
of compact objects. 

\medskip

Vice versa, starting from a non-cocomplete category $\CC_0$, one can produce a cocomplete DG category by the procedure
of \emph{ind-completion}; the resulting cocomplete category $\CC$ will be denoted $\on{IndCompl}(\CC_0)$. The category
$\CC$ is compactly generated and $\CC_0\subset \CC_c$. For any compactly generated $\CC$, we have $\CC\simeq \on{IndCompl}(\CC_c)$.
 
\medskip

Given a DG category $\CC$ one can talk about a t-structure on $\CC$. We will denote by $\CC^{\leq 0}$ (resp., $\CC^{\geq 0}$)
the full subcategory of connective (resp., coconnetive) objects. We let $\CC^\heartsuit:=\CC^{\leq 0}\cap \CC^{\geq 0}$ denote the
heart of the t-structure; this is an abelian category. We let $\CC^{<\infty}$ (resp., $\CC^{>\infty}$) the full subcategory consisting of
\emph{eventually connective} (resp., \emph{eventually coconnective}) objects. 

\sssec{Tensor product of DG categories}

The $(\infty,1)$-category $\on{DGCat}$ of DG categories carries a symmetric monoidal structure, denoted 
$\CC_1,\CC_2\mapsto \CC_1\otimes \CC_2$.

\medskip

The unit object for this symmetric monoidal structure is the category $\Vect$ of chain complexes of vector spaces. 

\medskip

As a basic example, if $\CC_i=A_i\mod$ for an associative algebra $A_i$, then
$$\CC_1\otimes \CC_2\simeq (A_1\otimes A_2)\mod.$$

\sssec{Duality for DG categories}  \label{sss:duality DG categ} 

In particular, it makes sense to talk about a DG category $\CC$ being \emph{dualizable}. For
a functor $F:\CC_1\to \CC_2$ between two dualizable DG categories we let 
$F^\vee:\CC_2^\vee\to \CC_1^\vee$ denote the dual functor. 

\medskip

If $\CC$ is compactly generated, it is dualizable and we have 
$$\CC^\vee\simeq \on{IndCompl}((\CC_c)^{\on{op}}).$$

\sssec{Algebra in DG categories}

The symmetric monoidal structure on $\on{DGCat}$ allows ``to do algebra" in the world of DG categories.
In particular, given a monoidal DG category $\CC$ (i.e., an associative algebra object in $\on{DGCat}$)
and its right and left module categories $\CC_1$ and $\CC_2$, respectively, it makes sense to talk about
$$\CC_1\underset{\CC}\otimes \CC_2\in \on{DGCat}.$$

\sssec{(Weak) actions of groups on categories}

We have the symmetric monoidal functor
$$(\Sch_{\on{f.t.}})_{/k}\to \on{DGCat}, \quad S\mapsto \QCoh(S), \quad (S_1\overset{f}\to S_2) \mapsto \QCoh(S_1)\overset{f_*}\longrightarrow \QCoh(S_2).$$

In particular, if $H$ is an algebraic group, the DG category $\QCoh(H)$ has a structure of monoidal category under convolution. We a category
acted on \emph{weakly} by $H$ we will mean a module category over $\QCoh(H)$. 

\medskip

For such a module category $\CC$, we set
$$\inv_H(\CC):=\on{Funct}_{H\mmod}(\Vect,\CC).$$

For $\CC=\Vect$, we have
$$\inv_H(\Vect)\simeq \Rep(H),$$
and for any $\CC$, the category $\inv_H(\CC)$ is naturally a module category over $\Rep(H)$. 

\medskip

It is shown in \cite[Theorem 2.5.5]{Ga3}, the above functor
$$H\mmod\to \Rep(H)\mmod, \quad \CC\mapsto \inv_H(\CC)$$
is an equivalence, with the inverse given by
$$\CC'\mapsto \Vect\underset{\Rep(H)}\otimes \CC'.$$

\begin{rem}
In this paper we will also encounter the notion of \emph{strong} action of an algebraic group 
(rather, group ind-scheme) on a DG category. We defer the discussion of this notion until 
\secref{ss:action}.
\end{rem} 

\ssec{Acknowledgements}

It is an honor to dedicate this paper to Masaki Kashiwara, especially as this paper deals with areas
where he has made crucial contributions: quantum groups and representations of affine Kac-Moody algebras.

\medskip

I am grateful to S.~Arkhipov, R.~Bezrukavnikov, M.~Finkelberg, E.~Frenkel, D.~Kazhdan and I.~Mirkovi\'c,
discussions with whom have shaped my thinking about the subjects treated in this paper over the years. 

\medskip

I am grateful to J.~Lurie for introducing me to the world of higher categories, without which this project
could never take off. 

\medskip

I am grateful to P.~Etingof for providing me with some very useful references. 

\medskip

The author's research is supported by NSF grant DMS-1063470. He has also received support from 
ERC grant 669655. 

\bigskip

\bigskip

\centerline{\bf Part I: Iwahori-intergrable Kac-Moody representations}

\bigskip

\section{Modules over the Kac-Moody algebra: recollections}  \label{s:recall}

\ssec{Definition of the category of modules}   \label{ss:g-mod}

One of the two primary objects of study in this paper is the (DG) category of modules over the Kac-Moody algebra 
at a given level $\kappa$, denoted $\hg\mod_\kappa$. In this subsection we recall its definition and discuss some
basic properties.

\sssec{}   \label{sss:g-mod}

The version of the DG category $\hg\mod_\kappa$ that we will use was defined in \cite[Sect. 23.1]{FG3}. It involves a renormalization
procedure.

\medskip

Namely, we start with the abelian category $(\hg\mod_\kappa)^\heartsuit$ and consider the usual derived category
$D((\hg\mod_\kappa)^\heartsuit)$ (by which we mean the corresponding DG category, following \cite[Sect. 1.3.2]{Lu}). 

\medskip

For every congruence subgroup $K_i\subset G(\CO)$, we consider the corresponding induced module
\begin{equation} \label{e:induced from congr}
\on{Ind}_{\fk_i}^{\hg_\kappa}(k)\in (\hg\mod_\kappa)^\heartsuit.
\end{equation} 

We let $\hg\mod_\kappa$ be the ind-completion of the full (but not cocomplete)
subcategory of $D((\hg\mod_\kappa)^\heartsuit)$ generated under finite colimits by the objects 
\eqref{e:induced from congr}. By construction, the objects \eqref{e:induced from congr} form a set
of compact generators of $\hg\mod_\kappa$. 

\medskip

Ind-completing the tautological embedding, we obtain a functor
\begin{equation} \label{e:ren KM}
\fs:\hg\mod_\kappa\to D((\hg\mod_\kappa)^\heartsuit).
\end{equation}

It is shown in {\it loc.cit.} that $\hg\mod_\kappa$ carries a unique t-structure, for which the functor
\eqref{e:ren KM} is t-exact and defines an equivalence of the corresponding eventually coconnective
subcategories, i.e.,
$$(\hg\mod_\kappa)^{>\infty}\to D^+((\hg\mod_\kappa)^\heartsuit).$$

Note, however, that the functor $\fs$ is \emph{not} conservative. In particular, the category
$\hg\mod_\kappa$ is \emph{not} left-separated in its t-structure. 

\sssec{}   \label{sss:g-modK} 

Let $K$ be a subgroup of finite-codimension in $G(\CO)$. The category $\hg\mod_\kappa$ has a version, 
denoted $\hg\mod_\kappa^K$ and defined as follows. 

\medskip

We start with the abelian category $(\hg\mod_\kappa^K)^\heartsuit$ of modules for the Harish-Chandra pair
$(\hg_\kappa,K)$, and consider its derived category $D((\hg\mod_\kappa)^\heartsuit)$.

\medskip

The renormalization procedure is done with respect to modules of the form
$$\on{Ind}_{\fk}^{\hg_\kappa}(V),\quad V\in (\Rep(K)_{\on{f.d.}})^\heartsuit.$$

\medskip

When $K$ is pro-unipotent, the object $\on{Ind}_{\fk}^{\hg_\kappa}(k)$ is a compact generator of 
$\hg\mod_\kappa^K$. 

\sssec{}

The cases of particular interest for us are when $K=G(\CO)$ and $K=I$, the Iwahori subgroup.

\medskip

We denote
\begin{equation} \label{e:Weyl mod}
\BV^\clambda_\kappa:=\on{Ind}_{\fg(\CO)}^{\hg_\kappa}(V^\clambda)\in \hg\mod_\kappa^{G(\CO)},  \quad \clambda\in \cLambda^+,
\end{equation} 
where $V^\clambda$ denotes the irreducible representation of $G$ with highest weight $\clambda$, viewed as a representation of
$G(\CO)$ via the evaluation map $G(\CO)\to G$. 

\medskip

Denote also 
$$\BM_\kappa^\clambda:=\on{Ind}^{\hg_\kappa}_{\on{Lie}}(k^\clambda)\in \hg\mod_\kappa^I, \quad \clambda\in \cLambda,$$
where $k^\clambda$ is one-dimensional representation of $T$ corresponding to the character $\clambda$,
viewed as a representation of $I$ via $I\to T$. 

\medskip

The Weyl modules $\BV^\clambda_\kappa$ (resp., the Verma modules $\BM_\kappa^\clambda$) form a set of compact generators
for $\hg\mod_\kappa^{G(\CO)}$ (resp., $\hg\mod_\kappa^I$). 

\sssec{}  \label{sss:averaging}

Let us be given a pair of subgroups $K'\subset K''$. In this case we have the (obvious) forgetful functor
$$\oblv_{K''/K'}:\hg\mod_\kappa^{K''}\to \hg\mod_\kappa^{K'},$$
which admits a right adjoint, denoted
$$\on{Av}^{K''/K'}_*:\hg\mod_\kappa^{K'}\to \hg\mod_\kappa^{K''}.$$

If the quotient $K''/K'$ is homologically contractible, i.e., if $k\simeq H_{\on{dR}}(K''/K')$
(e.g., both $K'$ and $K''$ are pro-unipotent), then the functor
$\oblv_{K''/K'}$ is fully faithful. 

\sssec{}

One shows that the naturally defined functor
$$\underset{i}{\on{colim}}\, \hg\mod_\kappa^{K_i}\to \hg\mod_\kappa$$
is an equivalence. 

\medskip

From here we obtain that for any $K\subset G(\CO)$ as above, there is an adjoint pair
$$\oblv_K:\hg\mod_\kappa^K\rightleftarrows \hg\mod_\kappa:\on{Av}^K_*.$$

\sssec{} \label{sss:proper averaging}

Assume for a moment that $K''/K'$ is a proper scheme (e.g., $K''=G(\CO)$ and $K'=I$). Then the
functor $\oblv_{K''/K'}$ also admits a \emph{left} adjoint, denoted $\on{Av}^{K''/K'}_!$. 

\medskip

However, Verdier duality on $K''/K'$ implies that we have a canonical isomorphism
$$\on{Av}^{K''/K'}_!\simeq \on{Av}^{K''/K'}_*[2(\dim(K''/K'))],$$
see \cite[Sect. 22.10]{FG2}. 

\ssec{Categories with Kac-Moody group actions}    \label{ss:action}

In certain places in this paper, it will be convenient to place the pattern described in \secref{ss:g-mod}
in the more general context of \emph{categories equipped with an action of $G(\CK)$}. 

\medskip

The material in this subsection is drawn from \cite[Sect. 22]{FG2}.

\sssec{}

We introduce the category $\Dmod_\kappa(G(\CK))$ of D-modules 
on the loop group $G(\CK)$ as follows.

\medskip

We start with the abelian category $(\Dmod_\kappa(G(\CK)))^\heartsuit$ defined in 
\cite[Sect. 21]{FG2}, and apply the renormalization procedure with respect to the following class
of objects: 
$$\on{Ind}^{\Dmod_\kappa(G(\CK))}_{\IndCoh(G(\CK))}(\CF),\quad \CF\in (\Coh(G(\CK)))^\heartsuit,$$
where $\on{Ind}^{\Dmod_\kappa(G(\CK))}_{\IndCoh(G(\CK))}$ is the left adjoint of the forgetful functor
$$(\Dmod_\kappa(G(\CK)))^\heartsuit\to (\IndCoh(G(\CK)))^\heartsuit$$
(one can show that the above functor $\on{Ind}^{\Dmod_\kappa(G(\CK))}_{\IndCoh(G(\CK))}$ is exact).

\medskip

Convolution defines on $\Dmod_\kappa(G(\CK))$ a structure of monoidal DG category, see \cite[Sect. 22.3]{FG2}.

\sssec{}  \label{sss:action at level kappa}

By a (DG) category $\CC$ acted on (strongly) by $G(\CK)$ at level $\kappa$ we will mean a module category 
over $\Dmod_\kappa(G(\CK))$. 

\medskip

In particular, such $\CC$ acquires a (strong) action of the group-scheme $G(\CO)$. For every subgroup $K\subset G(\CO)$
of finite codimension, we can consider the corresponding equivariant category $\CC^K$. It is acted on by the corresponding Hecke
category $\Dmod_\kappa(G(\CK)/K))^K$.

\medskip

An example of a category acted on (strongly) by $G(\CK)$ at level $\kappa$ at level $\kappa$ is $\CC:=\hg\mod_\kappa$.
The corresponding categories $\CC^K$ are what we have denoted earlier by $\hg\mod_\kappa^K$

\medskip

The categories $\CC^K$ for different $K$ are related by the functors $(\oblv,\on{Av}_*)$ as in Sects.
\ref{sss:averaging} and \ref{sss:proper averaging}. 

\sssec{}

Let $\CC$ be a category acted (strongly) on by $G(\CK)$ at level $\kappa$, and let $H\subset G(\CK)$ be a
group indsubscheme, such that the Kac-Moody extension $\hg_\kappa$ splits over its Lie algebra. 

\medskip

Then it makes sense to consider the equivariant category $\CC^H$. 

\medskip

A case of particular interest for us is when $H=N(\CK)\cdot T(\CO)$. Explicitly, if we write 
$N(\CK)$ as a union of group subschemes $N_\alpha$ that are invariant under conjugation 
by $T(\CO)$, then each $\CC^{N_\alpha\cdot T(\CO)}$ is a full subcategory in $\CC^{T(\CO)}$ 
(due to pro-unipotence), and we have
$$\CC^{N(\CK)\cdot T(\CO)}=\underset{\alpha}\cap\, \CC^{N_\alpha\cdot T(\CO)},$$
which is also a full subcategory in $\CC^{T(\CO)}$. 

\ssec{Duality}

It turns out that the categories of Kac-Moody modules at opposite levels are related to each other
by the procedure of \emph{duality of DG categories}, see \secref{sss:duality DG categ}. This 
often provides a convenient tool for the study of their representation theory.  

\medskip

The material in this subsection does not admit adequate references in the published literature.

\sssec{}

We start with the discussion of the finite-dimensional $\fg$. Let $A$ be an associative algebra and 
let $A^{\on{rev}}$ be the algebra with reversed multiplication.  Then the categories 
$$A\mod \text{ and } A^{\on{rev}}\mod$$
are related by
$$(A\mod)^\vee\simeq  A^{\on{rev}}\mod.$$

The pairing 
$$A\mod\otimes A^{\on{rev}}\mod\to \Vect$$
is given by
$$M_1,M_2\mapsto M_1\underset{A}\otimes M_2.$$

The dualizing object in $A\mod\otimes A^{\on{rev}}\mod$ is provided by $A$, viewed as a bimodule.

\medskip

Taking $A=U(\fg)$, and identifying $U(\fg)^{\on{rev}}\simeq U(\fg)$ via the antipode map, we obtain a
canonical identification
\begin{equation} \label{e:duality fin dim}
(\fg\mod)^\vee\simeq \fg\mod.
\end{equation} 

The corresponding pairing
$$\fg\mod\otimes \fg\mod\to \Vect$$
is given by
$$M_1,M_2\mapsto \CHom_{\fg\mod}(k,\CM_1\otimes \CM_2).$$

\medskip

The duality \eqref{e:duality fin dim} induces a duality
$$(\fg\mod^K)^\vee\simeq  \fg\mod^K$$
for any subgroup $K\subset G$, so that the functors
$$\oblv_K:\fg\mod^K\to \fg\mod:\on{Av}^K_*$$
satisfy
$$(\oblv_K)^\vee\simeq \on{Av}^*_K \text{ and } (\on{Av}^*_K)^\vee\simeq \oblv_K.$$

\medskip

Let $\BD$ denote the resulting contravariant auto-equivalence
$$(\fg\mod^K)_c\to (\fg\mod^K)_c.$$

\medskip

Taking $K=G$, the corresponding duality on $\fg\mod^G=\Rep(G)$ is just contragredient duality,
i.e., 
$$\BD(V)\simeq V^\vee.$$

\medskip

For $K=B$, we have
$$\BD(M^\clambda)\simeq M^{-\clambda-2\check\rho}[d].$$

\sssec{}

We now discuss the Kac-Moody case. 

\medskip

Let $-\kappa$ be the opposite level of $\kappa$. A key feature of the category $\Dmod_\kappa(G(\CK))$
is that it admits a global sections functor
$$\Gamma(G(\CK),-): \Dmod_\kappa(G(\CK))\to \hg\mod_\kappa\otimes \hg\mod_{-\kappa},$$
see \cite[Sect. 21]{FG2}.

\medskip

For a subgroup $K\subset G(\CO)$ of finite codimension, consider the object
$$\delta_{G(\CK),K}:=\on{Ind}^{\Dmod_\kappa(G(\CK))}_{\IndCoh(G(\CK))}(\CO_{K})\in \Dmod_\kappa(G(\CK)),$$
where $\CO_{K}\in \Coh(G(\CK))$ is the structure sheaf of $K$. 

\medskip

Consider the object
$$\Gamma(G(\CK),\delta_{G(\CK),K})\in \hg\mod_\kappa\otimes \hg\mod_{-\kappa}.$$

It naturally upgrades to an object of $\hg\mod^K_\kappa\otimes \hg\mod^K_{-\kappa}$. Denote
$$\on{unit}_K:=\Gamma(G(\CK),\delta_{G(\CK),K})[-\dim(G(\CO)/K)].$$

\sssec{}

The following assertion is proved in the same way as \cite[Theorem 6.2]{AG2}:

\begin{prop}  \label{p:duality K}
The object $\on{unit}_K$ defines the co-unit of a duality. 
\end{prop} 

Thus, we obtain that the categories $\hg\mod_\kappa^K$ and $\hg\mod_{-\kappa}^K$ are canonically mutually dual. 
We let $$\BD:(\hg\mod_\kappa^K)_c\to (\hg\mod_{-\kappa}^K)_c$$
denote the resulting contravariant equivalence. 

\medskip

We will denote the resulting pairing
$$\hg\mod^K_\kappa\otimes \hg\mod^K_{-\kappa}\to \Vect$$
by $\langle-,-\rangle_{\hg\mod_\kappa^K}$. 

\begin{rem}
The pairing $\langle-,-\rangle_{\hg\mod_\kappa^K}$ is explicitly given by the operation \emph{semi-infinite Tor} relative to $K$. Alternatively,
we can take $\langle-,-\rangle_{\hg\mod_\kappa^K}$ as the definition of this semi-infinite Tor functor. 
\end{rem}

\sssec{}

By unwinding the definitions, we obtain the following:

\begin{lem}  
For $V\in \Rep(K)_{\on{fin.dim}}$, we have a canonical isomorphism
$$\BD(\on{Ind}_{\fk}^{\hg_\kappa}(V))\simeq \on{Ind}_{\fk}^{\hg_{-\kappa}}(V^\vee\otimes \det(\fg(\CO)/\fk)).$$
\end{lem}

In the above formula, the line $\det(\fg(\CO)/\fk)$ is, according to our conventions, placed in cohomological degree
$-\dim(G(\CO)/K)$, and is regarded as a one-dimensional representation of $K$. 

\medskip

In particular, for $K=G(\CO)$, we have
$$\BD(\on{Ind}_{\fg(\CO)}^{\hg_\kappa}(V))\simeq \on{Ind}_{\fg(\CO)}^{\hg_{-\kappa}}(V^\vee)$$
and for $K=I$, we have
\begin{equation} \label{e:dual of Verma}
\BD(\BM^\clambda_\kappa)\simeq \BM_{-\kappa}^{-\clambda-2\check\rho}[d].
\end{equation} 

\sssec{}   \label{sss:Av adj}

Let $K'\subset K''$ be a pair of subgroups. Then the functors
$$\oblv_{K''/K'}:\hg\mod_\kappa^{K''}\rightleftarrows \hg\mod_\kappa^{K'}:\on{Av}_*^{K''/K'}$$
and
$$\oblv_{K''/K'}:\hg\mod_{-\kappa}^{K''}\rightleftarrows \hg\mod_{-\kappa}^{K'}:\on{Av}_*^{K''/K'}$$
are related by
$$(\oblv_{K''/K'})^\vee \simeq \on{Av}_*^{K''/K'} \text{ and } (\on{Av}_*^{K''/K'})^\vee\simeq \oblv_{K''/K'},$$
with respect to the identifications
$$(\hg\mod_\kappa^{K''})^\vee \simeq \hg\mod_{-\kappa}^{K''} \text{ and } 
(\hg\mod_\kappa^{K'})^\vee \simeq \hg\mod_{-\kappa}^{K'}.$$ 

\sssec{}  \label{sss:duality}

The compatible family of equivalences
$$(\hg\mod_\kappa^{K_i})^\vee \simeq \hg\mod_{-\kappa}^{K_i}$$ 
for congruence subgroups induces an equivalence
$$(\hg\mod_\kappa)^\vee \simeq \hg\mod_{-\kappa},$$
uniquely characterized by the property that for every $K$ as above, the functors
$$\oblv_K:\hg\mod_\kappa^K \rightleftarrows \hg\mod_\kappa: \on{Av}^K_*$$
and 
$$\oblv_K:\hg\mod_{-\kappa}^K \rightleftarrows \hg\mod_{-\kappa}: \on{Av}^K_*$$
are obtained from each other as duals.

\medskip

We denote the corresponding pairing
$$\hg\mod_\kappa\otimes \hg\mod_{-\kappa}\to \Vect$$
by $\langle-,-\rangle_{\hg\mod_\kappa}$.

\medskip

For any \emph{group-scheme}, $H\subset G(\CK)$, the pairing $\langle-,-\rangle_{\hg\mod_\kappa}$ induces a perfect pairing
$$\langle-,-\rangle_{\hg\mod^H_\kappa}: \hg\mod_\kappa^H\otimes \hg\mod^H_{-\kappa}\to \Vect,$$
so that the functors $\oblv_H$ and $\on{Av}^H_*$ are mutually dual as before. 

\ssec{BRST}  \label{ss:BRST}

We will now discuss the functor of BRST reduction (a.k.a. semi-infinite cohomology) with respect to the loop algebra $\fn(\CK)$.
It appears prominently in this paper as it is closely related to Wakimoto modules. 

\sssec{}

Consider the topological Lie algebra $\fn(\CK)$, and consider the category $\fn(\CK)\mod$, defined in the same 
manner as in \secref{sss:g-mod}.

\medskip

We define the functor
$$\on{C}^\semiinf(\fn(\CK),-):\hg\mod_\kappa\to \Vect$$
as follows.

\medskip

Write $\fn(\CK)=\underset{\alpha}\cup\, N_\alpha$. Note that for $N_{\alpha_1}\subset N_{\alpha_2}$ we we have a canonically defined natural transformation
$$\on{C}^\cdot(\fn_{\alpha_1},-)\to \on{C}^\cdot(\fn_{\alpha_2},-\otimes \det(\fn_{\alpha_2}/\fn_{\alpha_1})).$$

\medskip

Assume without loss of generality that $\fn(\CO)\subset \fn_\alpha$ for
all $\alpha$. The assignment
$$\alpha\mapsto \on{C}^\cdot(\fn_{\alpha},-\otimes \det(n_\alpha/\fn(\CO))$$
is a functor from the category of indices $\alpha$ to that of functors $\hg\mod_\kappa\to \Vect$. 

\medskip

We set
$$\on{C}^\semiinf(\fn(\CK),-):=\underset{\alpha}{\on{colim}}\,  \on{C}^\cdot(\fn_{\alpha},-\otimes \det(n_\alpha/\fn(\CO)).$$

\sssec{}  \label{sss:Miura}

Consider now the topological Lie algebra $\fb(\CK)$. According to \cite[Sect. 2.7.7]{BD}, there exists a canonically defined central
extension
\begin{equation} \label{e:Tate}
0\to k\to \wh\ft_{\on{Tate}}\to \ft(\CK)\to 0
\end{equation} 
with the following property:

\medskip

The composite functor
$$\fb(\CK)\mod \to \fn(\CK)\mod\overset{\on{C}^\semiinf(\fn(\CK),-)}\longrightarrow \Vect$$
canonically lifts to a functor
\begin{equation} \label{e:BRST functor} 
\on{BRST}_\fn:\fb(\CK)\mod\to \wh\ft_{\on{Tate}}\mod
\end{equation}

Moreover, the above functor \eqref{e:BRST functor} has a structure of functor between categories
acted on by $B(\CO)$. 

\begin{rem}
A remarkable feature of the extension \eqref{e:Tate} is that it \emph{does} not admit a canonical
splitting as vector spaces. In fact, it does not admit a splitting which is $\on{Aut}(\CO)$-invariant. 
\end{rem}

\sssec{}

Since \eqref{e:Tate} is a central extension of the abelian Lie algebra $\ft(\CK)$, it corresponds to a skew
symmetric form on $\ft(\CK)$. It is shown in \cite[Sect. 2.8.17]{BD} that this form equals the usual Kac-Moody
cocycle corresponding to the symmetric bilinear form on $\ft$ equal to 
$$\frac{\kappa_{\on{Kil}}}{2}|_{\ft}=-\kappa_{\on{crit}}|_\ft.$$

\medskip

Given a symmetric bilinear form $\kappa$ on $\ft$, let $\htt'_\kappa$ be the central extension 
$$0\to k\to \htt'_\kappa\to \ft(\CK)\to 0$$
equal to the Baer sum of the restriction of 
$$\hg_\kappa|_{\ft(\CK)}:=\hg_\kappa\underset{\fg(\CK)}\times \htt(\CK)$$
and $\wh\fh_{\on{Tate}}$.

\medskip

Note that due to our conventions of shifting the level for $\fg$ by the critical value, the extension $\htt'_\kappa$ is the Baer
sum of $\htt_\kappa$ and an abelian (but not canonically split) extension of $\ft(\CK)$. 

\medskip

We obtain that the composite functor
$$\hg_\kappa\mod \to \hg_\kappa|_{\fb(\CK)}\mod\to \fn(\CK)\mod\to \Vect$$
can be canonically lifted to a functor
\begin{equation} \label{sss:BRST on g}
\on{BRST}_\fn:\hg_\kappa\mod\to \htt'\mod_\kappa.
\end{equation} 

\section{Semi-infinite cohomology and Wakimoto modules}  \label{s:Wak}

Wakimoto modules for the Kac-Moody Lie algebra $\hg_\kappa$, first introduced in \cite{FF}, 
play a central role in this paper. 

\medskip

In this section we will reinterpret the construction of Wakimoto modules by first
defining \emph{semi-infinite} Wakimoto modules (as modules in a certain precise sense
(co)induced from loops into the Borel subalgebra), and then applying the functor of Iwahori-averaging. 

\ssec{The functor of semi-infinite cohomology}

In this subsection we will discuss a ``hands-on" way to describe a certain variant of the functor of semi-infinite cohomology
$$\on{BRST}_\fn: \hg\mod_\kappa\to \htt'\mod_\kappa,$$
see \eqref{sss:BRST on g}. 

\sssec{}   \label{sss:seminf la} 

The above functor $\on{BRST}_\fn$ is a functor between categories acted on by $B(\CO)$. In particular, it gives rise to the (same-named) functor
\begin{equation} \label{e:BRST with T}
\on{BRST}_\fn: \hg\mod^{T(\CO)}_\kappa\to \htt'\mod^{T(\CO)}_\kappa.%=:\KL(T)'_\kappa.
\end{equation} 

\medskip

Fix $\clambda\in \cLambda$, and let
$$\on{C}^\semiinf(\fn(\CK),-)^\clambda:\hg\mod^{T(\CO)}_\kappa\to \Vect$$
be the functor equal to the composite of $\on{BRST}_\fn$ of \eqref{e:BRST with T} and the functor
\begin{equation} \label{e:lambda comp}
\htt'\mod^{T(\CO)}_\kappa\to \Vect, \quad \CM\mapsto \CHom_{\Rep(T(\CO))}(k^\clambda,\CM).
\end{equation}

\medskip

We will now describe the functor $\on{C}^\semiinf(\fn(\CK),-)^\clambda$ explicitly as a colimit.

\begin{rem}
Recall that when we write $\hg\mod_\kappa$, we take into account the critical shift (so that the symmetric bilinear
form on $\fg$ corresponding to $\kappa$ is one whose restriction to $\ft$ is 
$\kappa+\kappa_{\on{crit}}$). Recall that we are assuming that $\kappa$ is non-degenerate. 
In this case, the category
$$\KL(T)_\kappa':=\htt'\mod^{T(\CO)}_\kappa$$
is semi-simple and is equivalent to $\Rep(T)\simeq \Vect^{\cLambda}$ by means of
$$k^\clambda\mapsto \Ind_{\ft(\CO)}^{\htt'_\kappa}(k^\clambda),$$
with the functors \eqref{e:lambda comp} providing the inverse. Hence, in this case, the collection of
the functors $\on{C}^\semiinf(\fn(\CK),-)^\clambda$ recovers the functor \eqref{e:BRST with T}.
\end{rem} 

\sssec{}  \label{sss:family of sgrps}

Let $A$ be a filtered set, and let 
$$\alpha\in I_\alpha$$
be a family of pro-solvable subgroups of $G(\CK)$ with the following properties:

\begin{itemize}

\item $T(\CO)\subset I_\alpha$ and each $I_\alpha$ admits a triangular decomposition in the sense that the multiplication map defines an isomorphism of schemes
$$I^+_\alpha\times T(\CO)\times I^-_\alpha\to I_\alpha,$$
where
$$I^+_\alpha:=I_\alpha\cap N(\CK), \quad I^-_\alpha:=I_\alpha\cap N^-(\CK).$$

\item For each $\alpha$, we have $N(\CO)\subset I^+_{\alpha}$, for 
$\alpha_1<\alpha_2$, we have $I^+_{\alpha_1}\subset I^+_{\alpha_2}$ and
$$\underset{\alpha\in A}\cup\, I^+_\alpha=N(\CK).$$

\item For each $\alpha$, we have $N^-(\CO)\supset I^-_{\alpha}$, for 
$\alpha_1<\alpha_2$, we have $I^-_{\alpha_1}\supset I^-_{\alpha_2}$ and
$$\underset{\alpha\in A}\cap\, I^-_\alpha=\{1\}.$$

\end{itemize}

\sssec{}  \label{sss:ex family}

An example of such a family of subgroups is provided by taking $A=\Lambda^+$ with the order relation given by
$$\mu_1<\mu_2\, \Leftrightarrow\, \mu_2-\mu_1\in \Lambda^+,$$
and setting 
$$I_\mu:=\on{Ad}_{t^{-\mu}}(I).$$

\sssec{}

It follows that for every $\alpha$ we have a canonical (surjective) homomorphism
$I_\alpha\to T$, whose kernel, denoted $\overset{\circ}I_\alpha$, is the pro-unipotent radical of $I_\alpha$, i.e.,
$$I_\alpha\simeq T\ltimes \overset{\circ}I_\alpha.$$

\sssec{}

Note the restriction of the Kac-Moody extension $\hg_\kappa$ to each $\on{Lie}(I_\alpha)$ admits a canonical splitting
fixed by the condition that it agrees with the given splitting over $\ft(\CO)\subset \fg(\CO)$. These splittings agree under the
inclusions
$$\on{Lie}(I_{\alpha_1})\supset \on{Lie}(I_{\alpha_1}\cap I_{\alpha_2})\subset \on{Lie}(I_{\alpha_2}), \quad \alpha_1,\alpha_2\in A.$$

\medskip

In particular, for every $\CM\in \hg\mod_\kappa$ and $\alpha\in A$, it makes sense to consider
$$\on{C}^\cdot(\on{Lie}(\overset{\circ}{I}_{\alpha}),\CM) \in \ft\mod,$$
and for a pair of indices $\alpha_1,\alpha_2\in A$ we have the natural maps of $\ft$-modules 
$$\on{C}^\cdot(\on{Lie}(\overset{\circ}{I}_{\alpha_1}),\CM)\to 
\on{C}^\cdot(\on{Lie}(\overset{\circ}{I}_{\alpha_1}\cap \overset{\circ}{I}_{\alpha_2}),\CM)\leftarrow 
\on{C}^\cdot(\on{Lie}(\overset{\circ}{I}_{\alpha_2}),\CM).$$

\medskip

In what follows, for $\CM\in \hg\mod_\kappa^T$ and $\clambda\in \cLambda$, we will write
$$\on{C}^\cdot(\on{Lie}(\overset{\circ}{I}_{\alpha}),\CM)^\clambda:=
\CHom_{\Rep(T)}\left(k^\clambda, \on{C}^\cdot(\on{Lie}(\overset{\circ}{I}_{\alpha}),\CM)\right).$$

\sssec{}  \label{sss:degrees}

Note that for every pair of indices $\alpha_1<\alpha_2$ and $\CM\in \hg\mod_\kappa$, we have canonical maps 
$$\on{C}^\cdot(\on{Lie}(\overset{\circ}{I}_{\alpha_1}),\CM) \to \on{C}^\cdot(\on{Lie}(\overset{\circ}{I}_{\alpha_1}\cap \overset{\circ}{I}_{\alpha_2}),\CM)\to
\on{C}^\cdot(\on{Lie}(\overset{\circ}{I}_{\alpha_2}),\CM\otimes \ell_{\alpha_1,\alpha_2}),$$
where $\ell_{\alpha_1,\alpha_2}$ is the graded line 
$$\det(\on{Lie}(I^+_{\alpha_2})/\on{Lie}(I^+_{\alpha_1})),$$
equipped with the natural $T$-action.

\medskip

Here and hereafter, for a finite-dimensional vector space $V$, its determinant line $\det(V)$ will be placed in cohomological degree $-\dim(V)$. 

\sssec{}

Let $\ell_\alpha$ be the graded line $\det(\on{Lie}(I^+_\alpha)/\fn(\CO))$, equipped 
with the natural action of $T$. 

\medskip

Thus, for $\alpha_1<\alpha_2$ and $\CM\in \hg\mod_\kappa$, we have a canonical map
$$\on{C}^\cdot(\on{Lie}(\overset{\circ}{I}_{\alpha_1}),\CM \otimes \ell_{\alpha_1})\to
\on{C}^\cdot(\on{Lie}(\overset{\circ}{I}_{\alpha_2}),\CM \otimes \ell_{\alpha_2})$$
of $\ft$-modules.  The assignment
$$\alpha\mapsto \on{C}^\cdot(\on{Lie}(\overset{\circ}{I}_\alpha),\CM \otimes \ell_\alpha)^\clambda$$
upgrades to a functor from the category of indices $A$ to $\Vect$. Set 
\begin{equation} \label{e:semiinf as colim}
\hg\mod_\kappa^T\to \Vect, \quad \CM\mapsto \underset{\alpha}{\on{colim}}\, \on{C}^\cdot(\on{Lie}(\overset{\circ}{I}_\alpha),\CM \otimes \ell_\alpha)^\clambda.
\end{equation}

It is easy to see that the above construction is canonically independent of the choice of the family $\alpha\mapsto I_\alpha$.  

\medskip

The following results by unfolding the definitions:

\begin{prop}  \label{p:semiinf}
For $\CM\in \hg\mod_\kappa^{T(\CO)}$ there exists a canonical isomorphism 
$$\on{C}^\semiinf(\fn(\CK),\CM)^\clambda \simeq 
\underset{\alpha}{\on{colim}}\, \on{C}^\cdot(\on{Lie}(\overset{\circ}{I}_\alpha),\CM \otimes \ell_\alpha)^\clambda.$$
\end{prop}

\ssec{The semi-infinite Wakimoto module}

In this subsection we will define the \emph{semi-infinite} Wakimoto modules, denoted $\BW^{\clambda,\semiinf}_\kappa$
for $\clambda\in \cLambda$. 

\sssec{}

Let $\alpha\mapsto I_\alpha$ be as in \secref{sss:family of sgrps}. Let $\ell^-_\alpha$ denote the graded line
$$\det(\on{Lie}(I^-)/\on{Lie}(I^-_\alpha)).$$

\medskip

For each $\alpha$ consider the object
$$\Ind_{\on{Lie}(I_\alpha)}^{\hg_\kappa}(k^\clambda\otimes \ell^-_\alpha)\in \hg\mod_\kappa.$$

Note that it naturally belongs to
$$\hg\mod_\kappa^{I_\alpha}\subset \hg\mod_\kappa^{I^+_\alpha\cdot T(\CO)}\subset \hg\mod_\kappa^{T(\CO)}.$$

Note also that it is concentrated in cohomological degree $-\dim(I^-/I^-_\alpha)$, which is the 
cohomological degree of $\ell^-_\alpha$. 

\sssec{}  \label{sss:trans map}

We claim that for $\alpha_1\leq \alpha_2$ we have a naturally defined map
$$\Ind_{\on{Lie}(I_{\alpha_1})}^{\hg_\kappa}(k^\clambda\otimes \ell^-_{\alpha_1})\to
\Ind_{\on{Lie}(I_{\alpha_2})}^{\hg_\kappa}(k^\clambda\otimes \ell^-_{\alpha_2}).$$

Indeed, this map equals the composition
$$\Ind_{\on{Lie}(I_{\alpha_1})}^{\hg_\kappa}(k^\clambda\otimes \ell^-_{\alpha_1}) \to 
\Ind_{\on{Lie}(I_{\alpha_1}\cap I_{\alpha_2})}^{\hg_\kappa}(k^\clambda\otimes \ell^-_{\alpha_2}) \to
\Ind_{\on{Lie}(I_{\alpha_2})}^{\hg_\kappa}(k^\clambda\otimes \ell^-_{\alpha_2}),$$
where the second arrow comes by applying $\Ind_{\on{Lie}(I_{\alpha_2})}^{\hg_\kappa}$ to 
the counit of the adjunction 
$$\Ind_{\on{Lie}(I_{\alpha_1}\cap I_{\alpha_2})}^{\on{Lie}(I_{\alpha_2})}\circ \Res_{\on{Lie}(I_{\alpha_1}\cap I_{\alpha_2})}^{\on{Lie}(I_{\alpha_2})}\to \on{Id},$$
and the first arrow comes by applying $\Ind_{\on{Lie}(I_{\alpha_1})}^{\hg_\kappa}$ to the natural transformation
$$\on{Id} \otimes \,\ell^-_{\alpha_1,\alpha_2} \to
\Ind_{\on{Lie}(I_{\alpha_1}\cap I_{\alpha_2})}^{\on{Lie}(I_{\alpha_1})}\circ \Res_{\on{Lie}(I_{\alpha_1}\cap I_{\alpha_2})}^{\on{Lie}(I_{\alpha_1})},$$
where 
$$\ell^-_{\alpha_1,\alpha_2}:=\det(\on{Lie}(I^-_{\alpha_1})/\on{Lie}(I^-_{\alpha_1}\cap I^-_{\alpha_2}))^{\otimes -1}.$$

\sssec{}

We define
$$\BW^{\clambda,\semiinf}_\kappa:=\underset{\alpha}{\on{colim}}\, \Ind_{\on{Lie}(I_\alpha)}^{\hg_\kappa}(k^\clambda\otimes \ell^-_\alpha)\in \hg\mod_\kappa^{T(\CO)}.$$

\medskip

By construction, $\BW^{\clambda,\semiinf}_\kappa$ belongs to the full subcategory $\hg\mod_\kappa^{N(\CK)\cdot T(\CO)}\subset \hg\mod_\kappa^{T(\CO)}$, i.e., 
it is equivariant with respect to any group-subscheme of $N(\CK)$.

\medskip

It is easy to see that $\BW^{\clambda,\semiinf}_\kappa$ is canonically independent of the choice of the family of the subgroups
$\alpha\mapsto I_\alpha$. 

\begin{rem}
A feature of $\BW^{\clambda,\semiinf}_\kappa$ is that, when viewed as an object of $\hg\mod_\kappa^{T(\CO)}$,
it is \emph{infinitely connective}, i.e., is cohomologically $\leq -n$ for any $n$ with respect to the natural
t-structure on $\hg\mod_\kappa^{T(\CO)}$. In other words, all of its cohomologies with respect to this t-structure are zero. 
\end{rem}

\ssec{Relation to duality}

As we shall presently see, the semi-infinite Wakimoto modules defined above are precisely the objects that represent the BRST
functor. 

\sssec{}

Recall that we have a canonical pairing
$$\langle-,-\rangle_{\hg\mod^{T(\CO)}}:\hg\mod_\kappa^{T(\CO)}\otimes \hg\mod_{-\kappa}^{T(\CO)}\to \Vect,$$
see \secref{sss:duality}.

\medskip

We claim:

\begin{prop}  \label{p:semiinf with Wak}
The functor
$$\langle \BW^{\clambda,\semiinf}_\kappa,\CM\rangle_{\hg\mod^{T(\CO)}}: \hg\mod_{-\kappa}^{T(\CO)}\to \Vect$$
is canonically isomorphic to the functor
$$\on{C}^\semiinf(\fn(\CK),\CM\otimes \det(\fn^-)^{\otimes -1})^{-\clambda}.$$ 
\end{prop}

Note that in the above formula, the graded line $\det(\fn^-)^{\otimes -1}$ is placed in the cohomological degree $\dim(\fn^-)=:d:=\dim(\fn)$,
and as a character of $T$, it corresponds to $2\rho$.  In other words, of we trivialize $\det(\fn^-)$ as a line, we have:
$$\on{C}^\semiinf(\fn(\CK),\CM\otimes \det(\fn^-)^{\otimes -1})^{-\clambda}\simeq 
\on{C}^\semiinf(\fn(\CK),\CM)^{-\clambda-2\check\rho}[-d].$$

\begin{rem}
Note that \propref{p:semiinf with Wak} is an affine analog of the following isomorphism of functors for the category $\fg\mod^T$:
$$\langle M^\clambda,\CM\rangle_{\fg\mod^T}\simeq \on{C}^\cdot(\fn,\CM \otimes \det(\fn^-)^{\otimes -1})^{-\clambda}
\simeq \on{C}^\cdot(\fn,\CM)^{-\clambda-2\check\rho}[-d],$$
where $M^\clambda\in \fg\mod^B$ is the Verma module, i.e., $M^\clambda=\Ind_\fb^\fg(k^\clambda)$.
\end{rem}

\begin{proof}[Proof of \propref{p:semiinf with Wak}]

The proof is tautological from the following canonical identification:

\medskip

Let $I'$ be a solvable group sub-scheme of $G(\CK)$ that contains $T\subset T(\CO)$ as a maximal torus.
Let $\overset{\circ}I{}'$ denote the unipotent radical of $I'$, so that $I'\simeq T\ltimes \overset{\circ}I{}'$. 
Note that the Kac-Moody extension $\hg_\kappa$, as well as $\hg\mod_{-\kappa}$, is equipped with splittings over 
$\on{Lie}(I')$, compatible with the given splitting over $T(\CO)$.

\medskip

Note also that the splitting of the extension $\hg_{\on{Kil}}$ endows the determinant line
$$\on{det.rel.}(\fg(\CO),\on{Lie}(I')):=\det(\fg(\CO)/\fg(\CO)\cap \on{Lie}(I')) \otimes \det(\on{Lie}(I')/\fg(\CO)\cap \on{Lie}(I'))^{\otimes -1}$$
with an action of $I'$. The corresponding character of $T$ corresponds to the natural action of $T$ on $\on{det.rel.}(\fg(\CO),\on{Lie}(I'))$
(note $T$ is contained in $G(\CO)\cap I'$). 

\medskip

Let $\CM$ be an object of $\hg\mod_{-\kappa}^T$, and let $\CN$ be an object of $\on{Lie}(I')\mod^T$. Then we have a canonical isomorphism:
$$\langle \CN\otimes \on{det.rel.}(\fg(\CO),\on{Lie}(I')),\CM\rangle_{\on{Lie}(I')\mod^T} \simeq
\langle \Ind_{I'}^{\hg\mod_\kappa}(\CN),\CM\rangle_{\hg\mod^{T(\CO)}}.$$

\end{proof}

\ssec{The usual Wakimoto modules}

In this section we will define the ``usual'' Wakimoto modules and relate them to the semi-infinite Wakimoto modules defined earlier. 

\sssec{}

Recall the averaging functor
$$\on{Av}^{I/T}_*:\hg\mod_\kappa^T\to \hg\mod_\kappa^I.$$

We set
\begin{equation} \label{e:def Wak}
\BW_\kappa^\clambda:=\on{Av}^{I/T}_*(\BW^{\clambda,\semiinf}_\kappa).
\end{equation} 

\begin{rem}
A remarkable feature of $\BW_\kappa^\clambda$, not obvious from the above definition, 
is that it belongs to the heart of the t-structure on 
$\hg\mod_\kappa^I$.
\end{rem}

\sssec{}

From \propref{p:semiinf with Wak} and \secref{sss:Av adj} we obtain:

\begin{cor}  \label{c:semiinf with Wak I}
There exists a canonical isomorphism of functors $\hg\mod_{-\kappa}^I\to \Vect$
$$\langle \BW^{\clambda}_\kappa,\CM\rangle_{\hg\mod^I} \simeq 
\on{C}^\semiinf(\fn(\CK),\CM\otimes \det(\fn^-)^{\otimes -1})^{-\clambda}.$$ 
\end{cor} 

This corollary shows that $\BW_\kappa^\clambda$, defined by formula \eqref{e:def Wak}, indeed agrees
with the usual definition of Wakimoto modules, defined e.g. in \cite[Sect. 11]{FG2}.

\begin{rem} \label{r:other Wak}
Our conventions regarding what we call ``the usual" Wakimoto modules are slightly different from those in \cite{FG2}. 

\medskip

Namely, the most standard Wakimoto module is what in {\it loc.cit.} was denoted $\BW^{w_0}_{\kappa,\clambda}$;
for example for $\clambda=0$, it has a structure of chiral/vertex operator algebra. Our $\BW^{\clambda}_\kappa$
is what in \cite{FG2} is denoted $\BW^1_{\kappa,\clambda}$. 

\medskip

Note, however, that $\BW^{w_0}_{\kappa,\clambda}$ and
$\BW^1_{\kappa,\clambda}$ can be easily related:
$$\BW^{w_0}_{\kappa,\clambda}\simeq j_{w_0,*}\star \BW^1_{\kappa,\clambda} \text{ and }
\BW^1_{\kappa,\clambda}\simeq j_{w_0,!}\star \BW^{w_0}_{\kappa,\clambda}.$$
\end{rem}  

\sssec{}

Here is one of the standard properties of Wakimoto modules:

\begin{prop}  \label{p:neg semiinf of Wak}
The semi-infinite cohomology $\on{C}^{\semiinf}(\fn^-(\CK),\BW^{\clambda}_\kappa)^{\clambda'}$
is given by:
$$
\begin{cases}
&\det(\fn^-)^{\otimes -1} \text{ for }\clambda'=\clambda+2\check\rho \\
&0 \text{ otherwise.}
\end{cases}
$$
\end{prop}

\begin{proof}

For any $\CM\in \fg\mod_\kappa^{T(\CO)}$, we have 
$$\on{C}^{\semiinf}(\fn^-(\CK),\CM)^{\clambda'}\simeq \on{C}^{\semiinf}(\fn^-(\CK),\on{Av}^{I^-}_*(\CM))^{\clambda'}.$$

If $\CM$ was $I^+$-equivariant, the latter further identifies with 
$$\on{C}^{\semiinf}(\fn^-(\CK),\on{Av}^{I/T}_*(\CM))^{\clambda'}.$$

Hence, it suffices to show that
$$\on{C}^{\semiinf}(\fn^-(\CK),\BW^{\clambda,\semiinf}_\kappa)^{\clambda'}=\
\begin{cases}
&\det(\fn^-)^{\otimes -1} \text{ for } \clambda'=\clambda+2\check\rho\\
&0 \text{ otherwise}.
\end{cases}
$$

Let $^-\!\on{Lie}(I_\alpha)$ be the opposite subalgebra to $I_\alpha$, i.e.,
$$^-\!\on{Lie}(I_\alpha)\cap \on{Lie}(I_\alpha)=\ft \text{ and } ^-\!\on{Lie}(I_\alpha)\oplus \on{Lie}(\overset{\circ}{I}_\alpha)=\fg(\CK).$$

Parallel to \propref{p:semiinf}, for $\CM\in \fg\mod_\kappa^{T(\CO)}$, we have:
$$\on{C}^{\semiinf}(\fn^-(\CK),\CM)^{\clambda'}\simeq 
\underset{\alpha}{\on{colim}}\, \on{C}_\cdot({}^-\!\on{Lie}(\overset{\circ}{I}_\alpha),\CM\otimes \det({}^-\!\on{Lie}(I^-_\alpha)/{}^-\!\on{Lie}(I^-))^{\otimes -1}
\otimes \det(\fn^-)^{\otimes -1})^{\clambda'}.$$

Taking 
$$\CM=\BW^{\clambda,\semiinf}_\kappa:=\underset{\alpha}{\on{colim}}\, \Ind_{\on{Lie}(I_\alpha)}^{\hg_\kappa}(k^\clambda\otimes \ell^-_\alpha)\in \hg\mod_\kappa^{T(\CO)},$$
and contracting the double colimit, it suffices to show that for each index $\alpha$, we have
$$\on{C}_\cdot({}^-\!\on{Lie}(\overset{\circ}{I}_\alpha),\Ind_{\on{Lie}(I_\alpha)}^{\hg_\kappa}
(k^\clambda\otimes \det({}^-\!\on{Lie}(I^-_\alpha)/{}^-\!\on{Lie}(I^-)))^{\otimes -1}
\otimes \ell^-_\alpha)^{\clambda'}=
\begin{cases}
&k \text{ for } \clambda'=\clambda'\\
&0 \text{ otherwise}.
\end{cases}
$$

However, this follows from the fact that $\Ind_{\on{Lie}(I_\alpha)}^{\hg_\kappa}(-)$ is free over ${}^-\!\on{Lie}(\overset{\circ}{I}_\alpha)$, while
the lines $\det({}^-\!\on{Lie}(I^-_\alpha)/{}^-\!\on{Lie}(I^-))$ and $\ell^-_\alpha$ are canonically isomorphic. 

\end{proof} 

\ssec{Convolution action on Wakimoto modules}

In this subsection we will assume that our level is integral; we will denote it by $-\kappa$. 
That said, the discussion will go through verbatim for a rational level, see Remark \ref{r:Frob lattice}.

\medskip

We will study the behavior of Wakimoto modules with respect to convolution with certain standard
D-modules on the affine flag scheme.

%\medskip

%The construction in this section depend on the choice of the uniformiser $t\in \CO$. 

\sssec{}

Let us view $\kappa$ (i.e., the opposite of the given level) as a pairing
$$\kappa:\Lambda\otimes \Lambda\to \BZ.$$

Hence, $\kappa$ defines an embedding
$$\Lambda\to \cLambda.$$

\begin{rem}  \label{r:Frob lattice}
If $\kappa$ is rational, rather than integral, in what follows the sublattice $\Lambda\subset \cLambda$ should be replaced by
$$\{\cmu\in \cLambda\,|\, \kappa(\cmu,\clambda)\in \BZ,\, \forall \clambda\in \cLambda\}.$$
\end{rem} 

\sssec{}   \label{sss:omega} 

%\otimes \omega_x^{\frac{\kappa(\mu,\mu)}{2}}.$$
%(without the grading or $T$-action), where $\omega_x$ denotes the cotangent fiber of the formal
%disc $\Spec(k\qqart)$, and $\omega^{\frac{1}{2}}_x$ it chosen square once and for all square root. 

%We extend this
%assignment to all of $\Lambda$ by requiring that 
%$$\fl_{\mu_1+\mu_2}\simeq \fl_{\mu_1}\otimes \fl_{\mu_2}.$$
%Using the uniformiser $t$ and the Killing form, we can $\fn(\CK)/\fn(\CO)$ with the topological dual of $\fn^-(\CO)$.
%This allows to identify $\fl_\mu$ for $\mu\in \Lambda^+$ with
%$$\det(\fn(\CO)/\on{Ad}_{t^\mu}(\fn(\CO))\simeq \det(\on{Ad}_{t^{-\mu}}(\fn(\CO))/\fn(\CO)).$$

Let $t$ be a uniformizer in $\Spec(\CO)$, so that for $\mu\in \Lambda$, we can view $t^\mu$ as a point 
in $T(\CK)$. Note, however, that for $\CM\in \hg\mod_\kappa^{T(\CO)}$, the object
$$t^\mu\cdot \CM\in  \hg\mod_\kappa$$
is canonically independent of the choice of $t$ (this is because the image of $t^\mu$ modulo $T(\CO)$
is independent of the choice of $t$). 

\medskip 

Let  $\omega_x$ denote the cotangent fiber of the formal disc $\Spec(\CO)$. A choice of $t$ trivializes
this line, and a change $t\mapsto f(t)\cdot t$ results in scaling this trivialization by $f(0)$. 

\medskip

For $\mu\in \Lambda^-$ let $\fl_\mu$ denote the line $\det(\on{Lie}(I^-)/\on{Lie}(\on{Ad}_{t^\mu}(I^-)))$, 
considered without the grading or $T$-action. 

\medskip

We claim: 

%\medskip

%The next assertion results from the construction of the semi-infinite Wakimoto modules:

\begin{prop}  \label{p:conv Wak}
For $\mu\in \Lambda^-$ we have a canonical isomorphism. 
$$(t^\mu\cdot \BW^{\clambda,\semiinf}_{-\kappa})\otimes \fl_\mu[-\langle \mu,2\check\rho\rangle]
\simeq \BW^{\clambda+\mu,\semiinf}_{-\kappa} \otimes \omega_x^{\langle \mu,\clambda\rangle}.$$
\end{prop}

\begin{proof}

Follows from the observation that for $\mu\in \Lambda^-$, the action of $T$ on the determinant line of $\on{Lie}(I^-)/\on{Lie}(\on{Ad}_{t^\mu}(I^-))$
is given by the character $\kappa_{\on{crit}}(\mu,-)$.

\end{proof} 

\sssec{}

The (monoidal) category $\Dmod_{-\kappa}(\on{Fl}^{\on{aff}}_G)^I$ of $-\kappa$-twisted D-modules on the affine flag scheme 
$\on{Fl}^{\on{aff}}_G=G(\CK)/I$ acts on $\hg\mod_{-\kappa}^I$ by convolutions.

\medskip

(As $\kappa$ was assumed integral, we can 
identify $\Dmod_{-\kappa}(\on{Fl}^{\on{aff}}_G)^I$ with the non-twisted version $\Dmod(\on{Fl}^{\on{aff}}_G)^I$, but it is more
convenient not to resort to this identification.)

\medskip

For $\mu\in \Lambda^-$ we let 
$$j_{\mu,!}, j_{\mu,*}\in (\Dmod_{-\kappa}(\on{Fl}^{\on{aff}}_G)^I)^\heartsuit$$
denote the corresponding standard (resp., costandard object) corresponding to the coset $I\cdot t^\mu\in G(\CK)/I$,
normalized so that its !-fiber at $t^\mu$ is the line $\fl_\mu$ placed in degree $\langle -\mu,2\check\rho\rangle$.

\medskip

For $\mu_1,\mu_2\in  \Lambda^-$ we have a canonical isomorphism
\begin{equation} \label{e:BMW mult}
j_{\mu_1,*}\star j_{\mu_2,*}\simeq j_{\mu_1+\mu_2,*}\otimes \omega_x^{\kappa(\mu_1,\mu_2)}.
\end{equation} 
%(Note that this is consistent with \propref{p:conv Wak}.)

\sssec{}

It is known after \cite{AB} that the assignment
$$\mu\mapsto  j_{\mu,*}, \quad \mu\in \Lambda^-$$
can be extended to an assignment 
$$\mu\mapsto  J_\mu, \quad \mu\in \Lambda,$$
uniquely determined by the requirement that 
\begin{equation} \label{e:BMW mult bis}
J_{\mu_1}\star J_{\mu_2}\simeq J_{\mu_1+\mu_2}\otimes \omega_x^{\kappa(\mu_1,\mu_2)}.
\end{equation} 

It is also known that for $\mu\in \Lambda^+$, the corresponding object $J_\mu$ is \emph{a} standard object 
on the orbit $I\cdot t^\mu\subset \on{Fl}^{\on{aff}}_G$; denote it by $j_{\mu,!}$. 

\medskip

The !-fiber of $j_{\mu,!}$ at $t^\mu$ is the line 
$$\fl_\mu:=\det(\on{Lie}(I^+)/\on{Lie}(\on{Ad}_{t^\mu}(I^+)))\simeq \fl_{-\mu}^{\otimes -1}\otimes \omega_x^{-\kappa_{\on{crit}}(\mu,\mu)},$$
placed in degree $\langle \mu,2\check\rho\rangle$. Note that for $\mu\in \Lambda^+$, we have
$$j_{-\mu,*}\star j_{\mu,!}\simeq \delta_{1,\Fl}\otimes \omega_x^{-\kappa(\mu,\mu)}\simeq j_{\mu,!}\star j_{-\mu,*}.$$

\medskip

For future reference we introduce the corresponding costandard object $j_{\mu,*}$ normalized so that its 
!-fiber at $t^\mu$ is the above line $\fl_\mu$. 

\medskip

By construction
$$\begin{cases}
&J_\mu\simeq \on{Av}^{I/T}_*(t^\mu\cdot \delta_{\Fl})\otimes \fl_\mu[-\langle \mu,2\check\rho\rangle], \quad \mu\in \Lambda^-;\\
&J_\mu\simeq \on{Av}^{I/T}_!(t^\mu\cdot \delta_{1,\Fl})\otimes \fl_\mu[-\langle \mu,2\check\rho\rangle], \quad \mu\in \Lambda^+.
\end{cases}
$$

From here, for $\CM\in \hg\mod_\kappa^I$, we have:
\begin{equation} \label{e:J as conv}
\begin{cases}
&J_\mu\star \CM\simeq \on{Av}^{I/T}_*(t^\mu\cdot \CM) \otimes \fl_\mu[-\langle \mu,2\check\rho\rangle], \quad \mu\in \Lambda^-;\\
&J_\mu\star \CM\simeq \on{Av}^{I/T}_!(t^\mu\cdot \CM) \otimes \fl_\mu[-\langle \mu,2\check\rho\rangle], \quad \mu\in \Lambda^+.
\end{cases}
\end{equation} 

\sssec{}

From \propref{p:conv Wak} (and \corref{c:semiinf with Wak I})
we obtain:

\begin{cor}  \label{c:conv Wak neg}  Let $\CM\in \hg\mod_{-\kappa}^I$.

\smallskip

\noindent{\em(a)}
For $\mu\in \Lambda^-$ there exists the canonical isomorphism
$$j_{\mu,*}\star \BW_{-\kappa}^\clambda\simeq \BW_{-\kappa}^{\clambda+\mu}\otimes \omega_x^{\langle \mu,\clambda\rangle}.$$

\smallskip

\noindent{\em(b)} 
For $\mu\in \Lambda^+$ there exists the canonical isomorphism
$$\on{C}^\semiinf(\fn(\CK),j_{\mu,*} \star \CM)^\clambda \simeq \on{C}^\semiinf(\fn(\CK),\CM)^{\clambda-\mu} \otimes 
\omega_x^{\langle \mu,\clambda\rangle-\kappa(\mu,\mu)}.$$
\end{cor}

From \corref{c:conv Wak neg}(a) we also obtain:

\begin{cor}  \label{c:conv Wak all}
For any $\mu\in \Lambda$ and any $\clambda\in \cLambda$, there exists a canonical isomorphism
$$J_\mu \star \BW_{-\kappa}^\clambda\simeq \BW_{-\kappa}^{\clambda+\mu}\otimes \omega_x^{\langle \mu,\clambda\rangle}.$$
\end{cor}

\begin{rem}
We can define another family of twisted D-modules
$$'\!J_{\mu,*}:=J_{\mu,*}\otimes \omega_x^{\frac{\kappa(\mu,\mu+2\check\rho)}{2}}, \quad \mu\in \Lambda^-,$$
and for this family we will have
$$'\!J_{\mu_1,*}\star {}'J_{\mu_2,*}\simeq {}'\!J_{\mu_1+\mu_2,*}.$$
Similarly, choosing a compatible family of powers $\omega_x^c$ with $c\in \BQ$, we can define
$$'\BW_{-\kappa}^\clambda:=\BW_{-\kappa}^\clambda\otimes \omega_x^{\frac{\kappa^{-1}(\clambda,\clambda+2\check\rho)}{2}},$$
where $\kappa^{-1}$ is the resulting form $\cLambda\otimes \cLambda\to \BQ$. In this case we have
$$'\!J_\mu \star {}'\BW_{-\kappa}^\clambda\simeq {}'\BW_{-\kappa}^{\clambda+\mu}.$$

In fact, the module $'\BW_{-\kappa}^\clambda$ defined in the above way carries a unique action of (a finite cover of) the group ind-scheme  
$$\on{Aut}(\CO), \quad \on{Lie}(\on{Aut}(\CO))=\on{Span}(t^i\partial_t, i\geq 0)$$
compatible with the $\on{Aut}(\CO)$-action on $\hg_\kappa$, in such a way that its highest weight line, i.e.,  
$\omega_x^{\frac{\kappa^{-1}(\clambda,\clambda+2\check\rho)}{2}}$, is acted on by (a finite cover of) the group-scheme
$$\on{Aut}_0(\Spec(\CO)), \quad \on{Lie}(\on{Aut}_0(\Spec(\CO)))=\on{Span}(t^i\partial_t, i\geq 1)$$
by character $\frac{\kappa^{-1}(\clambda,\clambda+2\check\rho)}{2}$. 

\end{rem} 

\section{Wakimoto modules via Verma modules}  \label{s:Verma}

In this section we will give a more explicit description of Wakimoto modules within the affine category $\CO$, i.e., $\hg\mod_\kappa^I$. 
Namely, we will express them via affine Verma modules.

\ssec{Affine Verma modules}

\sssec{}

The affine Verma module $\BM_\kappa^\clambda\in \hg\mod_\kappa^I$ is defined to be 
$$\BM_\kappa^\clambda:=\Ind_{\on{Lie}(I)}^{\hg_\kappa}(k^\clambda).$$

Note that $\BM_\kappa^\clambda\in \hg\mod_\kappa^I$ co-represents the functor
$$\on{C}^\cdot(\overset{\circ}{I},\CM)^\clambda\simeq \CHom_{\Rep(I)}(k^\clambda,\CM)$$

\sssec{}

Note that by taking $I_\alpha=I$, we obtain a canonical map
$$\BM_\kappa^\clambda\to \BW^{\clambda,\semiinf}_\kappa.$$

From here, by adjunction, we obtain a map in $\hg\mod_\kappa^I$ 
\begin{equation} \label{e:Verma to Wak}
\BM_\kappa^\clambda\to \BW^{\clambda}_\kappa.
\end{equation} 

\ssec{Irrational level}

\sssec{}

We have: 
\begin{prop}  \label{p:Wak irrational}
Let $\kappa$ be irrational. Then the map \eqref{e:Verma to Wak} is an isomorphism for any $\clambda$.
\end{prop} 

\begin{proof}

Follows from the fact that for $\kappa$ irrational, the induction functor
$$\on{Ind}^{\hg_\kappa}_{\fg(\CO)}:\fg\mod^B\to \hg\mod_\kappa^I$$
is an equivalence.

\end{proof} 

\sssec{}

As a corollary, combining \corref{c:semiinf with Wak I} and \eqref{e:dual of Verma}, with we obtain:

\begin{cor}
For irrational $\kappa$ and any $\clambda\in \cLambda$, the functors $\hg\mod_\kappa^I\to \Vect$ 
$$\on{C}^\semiinf(\fn(\CK),\CM)^\clambda \text{ and } \on{C}^\cdot(\overset{\circ}{I},\CM)^\lambda$$
are canonically isomorphic. 
\end{cor}

\ssec{Convolution action on Verma modules}

From now on until the end of this section we will assume that our level is integral; we will denote it by $-\kappa$. 
That said, the discussion will go through verbatim for a negative-rational level, see Remark \ref{r:Frob lattice}.

\sssec{}

As in \secref{sss:trans map}, for any $\clambda\in \cLambda$ and $\mu\in \Lambda^+$ there exists a canonical map in $\hg_{-\kappa}\mod^T$. 
$$\BM_{-\kappa}^{\clambda-\mu} \otimes \omega_x^{\langle -\mu,\clambda\rangle}\to (t^{-\mu}\cdot \BM_{-\kappa}^{\clambda})\otimes 
\fl_{-\mu}[\langle \mu,2\check\rho\rangle].$$

From here, we obtain a map
\begin{equation} \label{e:twist Verma}
\BM_{-\kappa}^{\clambda-\mu} \otimes \omega_x^{\langle -\mu,\clambda\rangle}
\to j_{-\mu,*}\star \BM_{-\kappa}^{\clambda}, \quad \mu\in \Lambda^+.
\end{equation}

We claim:

\begin{prop}  \label{p:Wak as colim Verma}
There exists a canonical isomorphism 
$$\BW^{\clambda}_{-\kappa}\simeq \underset{\mu\in \Lambda^+}{\on{colim}}\, j_{-\mu,*}\star \BM_{-\kappa}^{\clambda+\mu}\otimes 
\omega_x^{\langle \clambda,\mu\rangle+\kappa(\mu,\mu)}.$$
\end{prop}

\begin{proof}

By construction
$$\BW^{\clambda,\semiinf}_{-\kappa}\simeq \underset{\mu\in \Lambda^+}{\on{colim}}\, (t^{-\mu}\cdot \BM_{-\kappa}^{\clambda+\mu})\otimes \fl_{-\mu}\otimes
\omega_x^{\langle \clambda,\mu\rangle+\kappa(\mu,\mu)}[\langle \mu,2\check\rho\rangle].$$

The required isomorphism follows from \eqref{e:J as conv} by applying $\on{Av}^{I/T}_*$.

\end{proof} 

\sssec{}

Combining with \corref{c:semiinf with Wak I}, we obtain: 

\begin{cor}  \label{c:semiinf via Verma}
For $\CM\in \hg_{-\kappa}\mod^I$ and any $\clambda\in \cLambda$, there exists a canonical equivalence
$$\on{C}^\semiinf(\fn(\CK),\CM)^\clambda\simeq \underset{\mu\in \Lambda^+}{\on{colim}}\, 
\CHom_{\hg_{-\kappa}\mod^I}(\BM_{-\kappa}^{\clambda+\mu},j_{\mu,*}\star \CM)\otimes \omega_x^{-\langle \mu,\clambda\rangle}.$$
\end{cor}

\ssec{The negative level case}

In turns out that when the level is rational, the precise relation between Verma modules and Wakimoto modules drastically 
depends on the sign of the level. In this subsection we will specialize to the case when the level $-\kappa$ is negative. 

\sssec{}

We have the following assertion that follows from Kashiwara-Tanisaki localization,
proved in \secref{ss:conv Verma}:

\begin{thm}  \label{t:twist Verma}
If $\clambda-\mu\in \cLambda^+$, then the map \eqref{e:twist Verma} is an isomorphism.
\end{thm} 

From here, using \propref{p:Wak as colim Verma}, we obtain:  
%rederive the following result of \cite{Fr2}:

\begin{cor}  \label{c:Wak=Ver}
For $\clambda\in \cLambda^+$, the map 
$$\BM_{-\kappa}^\clambda\to \BW^{\clambda}_{-\kappa}$$
is an isomorphism.
\end{cor}

\begin{rem}
\corref{c:Wak=Ver} was originally proved in \cite[Theorem 2]{Fr1}, by a different method. 
Note also that \thmref{t:twist Verma} is logically equivalent to \corref{c:Wak=Ver}.
\end{rem} 

\sssec{}

Let us now give an expression for $\BW^\clambda_{-\kappa}$ for $\clambda$ not necessarily dominant. Namely, let $\mu\in \Lambda^+$
be such that $\clambda+\mu\in \cLambda^+$. We have:
$$\BW^\clambda_{-\kappa}\simeq j_{-\mu,*}\star \BW^{\clambda+\mu}_{-\kappa}\otimes \omega_x^{\langle \mu,\clambda\rangle+\kappa(\mu,\mu)},$$
and combining with \thmref{t:twist Verma}, we obtain:
\begin{equation} \label{e:Wak non-dom}
\BW^\clambda_{-\kappa}\simeq j_{-\mu,*}\star \BM^{\clambda+\mu}_{-\kappa}\otimes \omega_x^{\langle \mu,\clambda\rangle+\kappa(\mu,\mu)}.
\end{equation} 

\sssec{}

We have just seen that Wakimoto modules at the negative level are expressible via affine Verma modules by a ``finite" procedure.
By contrast, the expression for functor of semi-infinite cohomology $\on{C}^\semiinf(\fn(\CK),-)^\clambda$ involves a colimit. 
Indeed, by combining \thmref{t:twist Verma} and \corref{c:semiinf via Verma}, we obtain:

\begin{cor}  \label{c:semiinf neg via Wak}
For any $\clambda\in \cLambda$ and $\CM\in \hg\mod_{-\kappa}^I$, there exists a canonical isomorphism
\begin{multline*}
\on{C}^\semiinf(\fn(\CK),\CM)^\clambda\simeq \underset{\mu\in \Lambda^+}{\on{colim}}\, 
\CHom_{\hg\mod_{-\kappa}^I}(\BW_{-\kappa}^\clambda, j_{-\mu,*}\star j_{\mu,*}\star \CM\otimes \omega_x^{\kappa(\mu,\mu)})\simeq \\
\simeq \underset{\mu\in \Lambda^+}{\on{colim}}\, 
\CHom_{\hg\mod_{-\kappa}^I}(\BM_{-\kappa}^\clambda, j_{-\mu,*}\star j_{\mu,*}\star \CM\otimes \omega_x^{\kappa(\mu,\mu)}),
\end{multline*}
where the transition maps are given by the maps
$$j_{-\mu,*}\star j_{\mu,*} \otimes \omega_x^{\kappa(\mu,\mu)} \to j_{-\mu-\mu',*}\star j_{\mu+\mu',*} \otimes \omega_x^{\kappa(\mu+\mu',\mu+\mu')}, \quad \mu'\in \Lambda^+$$
that come from the canonical maps $j_{\mu',!}\to  j_{\mu',*}$.
\end{cor} 

\begin{rem}

Let us also note that the object
$$\underset{\mu\in \Lambda^+}{\on{colim}}\, j_{-\mu,*}\star j_{\mu,*}\star \omega_x^{\kappa(\mu,\mu)}\in \Dmod_{-\kappa}(\on{Fl}^{\on{aff}}_G)^I$$
identifies with
$$\on{Av}^{I/T}_*(\iota_*(\omega_{N(\CK)\cdot 1})),$$
where 
$$N(\CK)\cdot 1 \overset{\iota}\hookrightarrow \on{Fl}^{\on{aff}}_G$$
is the embedding of the $N(\CK)$-orbit of the point $1\in \on{Fl}^{\on{aff}}_G$; here we are using the fact that the twisting corresponding to $\kappa$
is canonically trivialized on $N(\CK)\cdot 1$ so that 
$$\iota_*(\omega_{N(\CK)\cdot 1})\in \Dmod_{-\kappa}(\on{Fl}^{\on{aff}}_G)^{N(\CK)\cdot T(\CO)}$$
makes sense. 

\medskip

So the conclusion of \corref{c:semiinf neg via Wak} can we rewritten as
$$\on{C}^\semiinf(\fn(\CK),\CM)^\clambda \simeq \CHom_{\fg\mod_{-\kappa}^T}(\BW^\clambda_{-\kappa}, \iota_*(\omega_{N(\CK)\cdot 1})\star \CM),$$
where we can further rewrite 
$$\iota_*(\omega_{N(\CK)\cdot 1})\star \CM\simeq \omega_{N(\CK)/N(\CO)}\star \CM$$
in terms of the action of the group $N(\CK)$ on $\hg\mod_\kappa$. 

\end{rem}

\sssec{}

We can now exhibit the collection of objects that are right-orthogonal to the Wakimito modules:

\begin{cor} \label{c:right orth to Wak}
The collection of modules
$$\clambda'\mapsto \underset{\mu\in \Lambda^+}{\on{colim}}\, j_{-\mu,*}\star j_{\mu,*}\star j_{w_0,*}\star \BW_{-\kappa}^{w_0(\clambda')-2\check\rho}
\otimes \omega_x^{\kappa(\mu,\mu)}$$
is right-orthogonal to the collection
$$\clambda\mapsto \BW_{-\kappa}^\clambda.$$
\end{cor} 

\begin{proof}

Follows by combining \propref{p:neg semiinf of Wak} and \corref{c:semiinf neg via Wak}.

\end{proof}

\ssec{The positive level case}

In this subsection we will take our level $\kappa$ to be positive integral. 

\sssec{}

At the positive level, the behavior of Wakimoto modules and the functor of semi-infinite cohomology will be in a certain sense opposite to that
of the negative level: the latter will be co-representable by a compact object, while the former will be (infinite) colimits of Verma modules.

\medskip

Indeed, the latter assertion amounts to the isomorphism 
$$\BW^{\clambda}_{\kappa}\simeq \underset{\mu\in \Lambda^+}{\on{colim}}\, j_{-\mu,*}\star \BM_{\kappa}^{\clambda-\mu}\otimes 
\omega_x^{\langle \clambda,\mu\rangle-\kappa(\mu,\mu)},$$
given by \propref{p:Wak as colim Verma}. 

\sssec{}

For the expression of semi-infinite cohomology we have the following consequence of 
Corollaries \ref{c:semiinf with Wak I} and \ref{c:Wak=Ver}:

\begin{cor}
For $\CM\in \hg\mod_\kappa^I$ and $\clambda\in \cLambda$ such that $-\clambda-2\check\rho\in \cLambda^+$, we
have a canonical isomorphism: 
$$\on{C}^\semiinf(\fn(\CK),\CM)^\clambda \simeq  \CHom_{\hg\mod_\kappa^I}(\BM^\clambda_\kappa,\CM).$$
For $\mu\in \Lambda^+$, we have a canonical isomorphism
$$\on{C}^\semiinf(\fn(\CK),\CM)^{\clambda+\mu}\simeq 
\on{C}^\semiinf(\fn(\CK),j_{\mu,*}\star \CM)^\clambda \otimes \omega_x^{-\langle \mu,\clambda\rangle-\kappa(\mu,\mu)}.$$
\end{cor} 

\begin{rem}
Let $\clambda\in \cLambda$ be such that $-\clambda-2\check\rho \in \cLambda^+$.
Then from the above corollary we obtain an isomorphism
$$j_{\mu,*}\star \BM_\kappa^{\clambda} \simeq \BM_\kappa^{\clambda-\mu}
\otimes \omega_x^{\langle \mu,\clambda\rangle}, \quad \mu\in \Lambda^+$$
The latter isomorphism can be obtained from the Kashiwara-Tanisaki localization at the positive level. 
\end{rem} 

\ssec{Proof of \thmref{t:twist Verma}}  \label{ss:conv Verma}

For the proof of \thmref{t:twist Verma} we will choose a uniformizer $t\in \CO$; so we will trivialize the line $\omega_x$. 

\sssec{}  \label{sss:adm}

First, let us recall the version of the Kasiwara-Tanisaki localization theorem that we will need. 

\medskip

Let us call a weight $\clambda_0\in \cLambda$ \emph{admissible} for $\kappa$ if the following conditions hold:
\begin{equation} \label{e:adm}
\begin{cases}
&\langle \alpha,\clambda_0+\check\rho\rangle\notin \BZ^{>0} \text{ for all positive coroots } \alpha, \\
&\pm \langle\alpha,\clambda_0+\check\rho\rangle +n\cdot \frac{-\kappa(\alpha,\alpha)}{2}\notin \BZ^{>0} \text{ for all positive coroots } \alpha \text{ and all } n\in \BZ^{>0}.
\end{cases}
\end{equation}

\medskip

We consider the twisting $(-\kappa,\clambda_0)$ on the affine flag scheme $\Fl^{\on{aff}_G}$, so that
$$\Gamma(\Fl^{\on{aff}_G},\delta_{1,\Fl})\simeq \BM_{-\kappa}^{\clambda_0}.$$ 

The Kasiwara-Tanisaki theorem of \cite{KT} says that in this case the functor
$$\Gamma(\Fl^{\on{aff}_G},-):\Dmod_{(-\kappa,\clambda_0)}(\Fl^{\on{aff}_G})^I\to \hg\mod_{-\kappa}^I$$
is t-exact, sends standard objects to standard objects, and costandard objects to costandard objects.

\sssec{}

For an element $\clambda \in \cLambda$ we can always find an element 
$$\mu\cdot w\in W^{\on{aff}}:=W\ltimes \Lambda,$$
so that
$$\clambda=\mu+w(\clambda_0+\check\rho)-\check\rho$$
with $\clambda_0$ being $\kappa$-admissible. 

\medskip

Hence,
$$\BM^\clambda_{-\kappa}\simeq \Gamma(\Fl^{\on{aff}_G},j_{\mu\cdot w,!}), \quad j_{\mu\cdot w,!}\in \Dmod_{(-\kappa,\clambda_0)}(\Fl^{\on{aff}_G})^I.$$

Similarly, for any $\nu\in \Lambda$, 
$$\BM^{\clambda+\nu}_{-\kappa}\simeq \Gamma(\Fl^{\on{aff}_G},j_{(\nu+\mu)\cdot w,!}).$$

\medskip

Thus, to prove \thmref{t:twist Verma} it suffices to show that if $\clambda\in \cLambda^+$ and $\nu\in \Lambda^+$, we have
$$j_{\nu,!}\star j_{\mu\cdot w,!}\simeq j_{(\nu+\mu)\cdot w,!}.$$

\medskip

The latter is a combinatorial condition that translates as
$$\ell(\nu)+\ell(\mu\cdot w)=\ell((\nu+\mu)\cdot w), \quad \nu\in \Lambda^+.$$

\medskip

It is equivalent to
\begin{equation} \label{e:length}
\ell(\mu\cdot w)=\langle \mu,2\check\rho\rangle+\ell(w).
\end{equation}

\sssec{}

In order to prove \eqref{e:length} it is easy to see that it suffices to show that the element $\mu\cdot w$
satisfies the following condition for every positive root $\check\alpha$:
\begin{equation} \label{e:pos cond}
\begin{cases}
&\langle \mu, \check\alpha\rangle>0 \text{ if } w^{-1}(\check\alpha) \text{ is positive},\\
&\langle \mu, \check\alpha\rangle\geq 0 \text{ if } w^{-1}(\check\alpha) \text{ is negative}.
\end{cases}
\end{equation}

We claim that \eqref{e:pos cond} is forced by the condition that $\clambda$ is dominant. 

\medskip

Suppose that there exists a positive root $\check\alpha$ for which $\langle \mu, \check\alpha\rangle\leq 0$ and 
$w^{-1}(\check\alpha)=:\beta$ is positive.  We have
$$0<\langle \alpha, \clambda+\check\rho\rangle=
\langle \alpha, w(\clambda_0+\check\rho)\rangle+ \langle\alpha,\mu\rangle=
\langle \beta, \clambda_0+\check\rho \rangle + \langle \mu,\check\alpha \rangle \cdot \frac{\kappa(\alpha,\alpha)}{2}.$$

However, this violates \eqref{e:adm}.

\medskip

Similarly, suppose there exists a positive root $\check\alpha$ for which $\langle \mu, \check\alpha\rangle<0$ 
and $w^{-1}(\check\alpha)=:\beta$ is negative. We have
$$0<\langle \alpha, \clambda+\check\rho\rangle=
\langle \alpha, w(\clambda_0+\check\rho)\rangle+ \langle\alpha,\mu\rangle=
-\langle -\beta, \clambda_0+\check\rho \rangle + \langle \mu,\check\alpha \rangle \cdot \frac{\kappa(\alpha,\alpha)}{2},$$
and this again violates \eqref{e:adm}.

\qed[\thmref{t:twist Verma}]

\sssec{}

For $\clambda\in \cLambda$, let $\BM_{-\kappa}^{\vee,\clambda}\in \hg\mod_{-\kappa}^I$ be the dual Verma module, i.e.,
$$\CHom_{\hg\mod_{-\kappa}^I}(\BM_{-\kappa}^{\clambda'},\BM_{-\kappa}^{\vee,\clambda})=
\begin{cases}
&k \text{ if } \clambda'=\clambda,\\
&0 \text{ otherwise}.
\end{cases}$$

\medskip

Let us note that a similar argument to the one proving \thmref{t:twist Verma}, proves the following assertion:

\medskip

\begin{thm}  \label{t:dual Verma}
Let $\clambda$ be dominant and $\mu\in \Lambda^+$ be such that $-2\check\rho-\clambda+\mu\in \cLambda^+$.
Then 
$$j_{\mu,!}\star \BM_{-\kappa}^{\vee,-2\check\rho-\clambda}\simeq \BM_{-\kappa}^{-2\check\rho-\clambda+\mu}.$$
\end{thm}

From here, we obtain the following result originally proved in \cite[Proposition 6.3 and Remark 6.4]{Fr2} (see also \cite[Proposition 13.2.1]{FG2} for a proof): 

\begin{cor} \label{c:Wakimoto for neg}
For $\clambda\in \cLambda^+$ we have a canonical isomorphism
$$\BW_{-\kappa}^{-2\check\rho-\clambda}\simeq \BM_{-\kappa}^{\vee,-2\check\rho-\clambda}.$$
\end{cor}

\bigskip

\centerline{\bf Part II: Representations of the quantum group}

\bigskip

\section{Quantum algebras}  \label{s:quant alg}

This section is devoted to a review of the basic setting in which quantum groups are constructed. It can be summarized as
follows: we start with the quantum torus, take a Hopf algebra $A$ in the corresponding braided monoidal category, and then
take the (relative) Drinfeld center of the category of $A$-modules. 

\ssec{The quantum torus}

Our approach to the definition of categories of modules for the various versions of the quantum group
is to build them starting from the basic case, that being the case of the quantum torus. 

\sssec{}

We start with the lattice $\cLambda$ and a symmetric bilinear $W$-invariant form, denoted $b'$,
$$\cLambda\otimes \cLambda\to k^\times.$$ 

We denote by $q$ the associated quadratic form $q(\clambda)=b'(\clambda,\clambda)$, and by $b$ the
symmetric bilinear form associated to $q$, i.e.,
$$b(\clambda_1,\clambda_2)=q(\clambda_1+\clambda_2)\cdot q(\clambda_1)^{-1}\cdot q(\clambda_2)^{-1}=b'(\clambda_1,\clambda_2)^2.$$

\medskip

We regard $\Rep(T)$ as a monoidal category, and the form $b'$ defines on $\Rep(T)$ a new braided structure, obtained
by multiplying the tautological one by $b'$: the new braiding 
$$R_{k^{\clambda_1},k^{\clambda_2}}:k^{\clambda_1}\otimes k^{\clambda_2}\to k^{\clambda_2}\otimes k^{\clambda_1}$$
is the identity map times $b'(\clambda_1,\clambda_2)$.

\medskip

Denote the resulting braided monoidal category by $\Rep_q(T)$. 

\sssec{}

Although we will not use this extensively in the current paper, we remark that the braided monoidal category $\Rep_q(T)$ 
carries a ribbon structure, where the ribbon automorphism of the object $k^\clambda\in \Rep_q(T)$ equals 
$$b'(\clambda,\clambda+2\check\rho).$$

The $2\rho$-shift in this formula is necessary in order to make the quantum Hopf algebras in \secref{ss:q grp}
\emph{ribbon-equivariant}. 

\sssec{}
Let $c$ be a compact object of $\Rep_q(T)$. To it we can attach its left monoidal and right monoidal duals,
which are objects equipped with perfect pairings
$$c^{\vee,L}\otimes c\to k \text{ and } c\otimes c^{\vee,R}\to k,$$ 
respectively. 

\medskip

The braided structure on $\Rep_q(T)$ allows us to identify $c^{\vee,L}\simeq c^{\vee,R}$. Multiplying this identification
by the ribbon twist, we can make it compatible with the monoidal structure
$$
\CD
(c_1\otimes c_2)^{\vee,L}  @>>>  (c_2\otimes c_1)^{\vee,R}  \\
@VVV   @VVV  \\
c_2^{\vee,L}  \otimes c_1^{\vee,L}  @>>> c_2^{\vee,R}  \otimes c_1^{\vee,R}
\endCD
$$

Henceforth, we will use the above identification between $c^{\vee,L}$ and $c^{\vee,R}$ and simply write $c^\vee$.

\sssec{}  \label{sss:rev}

Let $\Rep_{q^{-1}}(T)$ denote the braided monoidal category with the inverted $b'$. Note that it can be identified
with $(\Rep_q(T))^{\on{rev-br}}$, the braided monoidal category obtained
from $\Rep_q(T)$ by reversing the braiding. 

%\medskip

%Namely, the identity functor $\Phi$ on 
%$$\Rep_q(T)\simeq \Rep(T)\simeq \Rep_{q^{-1}}(T),$$
%equipped with the tautological natural transformation
%$$\Phi(c_1\otimes c_2) \simeq \Phi(c_2)\otimes \Phi(c_1)$$
%makes the diagram
%$$
%\CD
%\Phi(c_1\otimes c_2) @>>> \Phi(c_2)\otimes \Phi(c_1) \\
%@V{\Phi(R^{-1}_{c_2,c_1})}VV   @VV{R_{\Phi(c_2),\Phi(c_1)}}V  \\
%\Phi(c_2\otimes c_1)   @>>> \Phi(c_1)\otimes \Phi(c_2) 
%\endCD
%$$
%commute. 

\ssec{Quantum Hopf algebras}  \label{ss:Hopf algebras}

The Hopf algebras that appear in this subsection should be thought of as ``positive" parts of
the corresponding versions of the quantum group. 

\sssec{} \label{sss:Hopf algebras}

Note that given a braided monoidal category, it makes sense to talk about Hopf algebras in it. 
Our interest will be Hopf algebras $A$ in $(\Rep_q(T))^\heartsuit$ such that their weight components $A^\clambda$
satisfy
$$A^\clambda=
\begin{cases}
&0 \text{ for } \clambda\notin \cLambda^+ \\
&k \text{ for } \clambda=0 \\
&\text{is finite-dimensional for } \clambda\in \cLambda^+. 
\end{cases}
$$

We will also consider Hopf algebras that satisfy the opposite condition (i.e., replace $\cLambda^+$ by $-\cLambda^+$).

\medskip

For a given $A$, we will consider the DG category $A\mod(\Rep_q(T))$, equipped with the forgetful functor $\oblv_A$ to $\Rep_q(T)$.
Let $\ind_A:\Rep_q(T)\to A\mod(\Rep_q(T))$ denote its left adjoint.

\medskip

The datum of Hopf algebra structure on $A$ is equivalent to a structure of monoidal category on $A\mod(\Rep_q(T))$ so that
the functor $\oblv_A$ is monoidal. 

\sssec{}

Let $A\mod(\Rep_q(T))_{\on{fin.dim}}\subset A\mod(\Rep_q(T))$ denote the full (but not cocomplete) subcategory
of finite-dimensional $A$-modules. We let
$$A\mod(\Rep_q(T))_{\on{loc.nilp}}$$ 
denote the ind-completion of $A\mod(\Rep_q(T))_{\on{fin.dim}}$. 

\medskip

The monoidal structure on $A\mod(\Rep_q(T))$
restricts to one on $A\mod(\Rep_q(T))_{\on{fin.dim}}$, which, in turn, ind-extends to one on $A\mod(\Rep_q(T))_{\on{loc.nilp}}$. 

\medskip

Below we will give a different interpretation of $A\mod(\Rep_q(T))_{\on{loc.nilp}}$. 

\sssec{}

By ind-extending the tautological embedding $A\mod(\Rep_q(T))_{\on{fin.dim}}\hookrightarrow A\mod(\Rep_q(T))$
we obtain a functor
$$A\mod(\Rep_q(T))_{\on{loc.nilp}}\to A\mod(\Rep_q(T)).$$

\medskip

\noindent Warning: in general, this functor does not send compacts to compacts (as $A\mod(\Rep_q(T))_{\on{fin.dim}}$ is not necessarily 
contained in $A\mod(\Rep_q(T))_c$)), nor is it fully faithful. In fact, this functor may fail to even be conservative (this is the case if
$k$ is non-compact as an object of $A\mod(\Rep_q(T))$). 

\medskip

By a slight abuse of notation, we will denote by the same symbol $\oblv_A$ the composite functor
$$A\mod(\Rep_q(T))_{\on{loc.nilp}}\to A\mod(\Rep_q(T)) \overset{\oblv_A}\to \Rep_q(T).$$

Note that the resulting functor
\begin{equation} \label{e:oblv nilp}
\oblv_A:A\mod(\Rep_q(T))_{\on{loc.nilp}}\to \Rep_q(T)
\end{equation} 
is not necessarily conservative. 

\sssec{}  \label{sss:recover from heart mod}

Since $A$ was assumed connective, the category $A\mod(\Rep_q(T))$ carries a t-structure, for which the forgetful functor
$\oblv_A$ is t-exact. The subcategory $A\mod(\Rep_q(T))_{\on{fin.dim}}$ is compatible with this t-structure, and that
in turn induces a t-structure on $A\mod(\Rep_q(T))_{\on{loc.nilp}}$. The functor \eqref{e:oblv nilp} is t-exact, and 
an object of $A\mod(\Rep_q(T))_{\on{loc.nilp}}$ is connective
if and only if its image under \eqref{e:oblv nilp} is connective. 

\medskip

Note, however, that the above t-structure on $A\mod(\Rep_q(T))_{\on{fin.dim}}$ is in general not left separated. 

\medskip

Since $A$ was assumed co-connective (recall that we are assuming that $A$ is actually in the heart), it follows (e.g., using \secref{sss:coind} below) that $(A\mod(\Rep_q(T))_{\on{loc.nilp}})^\heartsuit$
contains enough injective objects $\CI$, such that
$$\CHom_{A\mod(\Rep_q(T))_{\on{loc.nilp}}}(\CM,\CI)$$
is acyclic off degree $0$ for any $\CM\in (A\mod(\Rep_q(T))_{\on{loc.nilp}})^\heartsuit$.

\medskip

This implies that the bounded part of $A\mod(\Rep_q(T))_{\on{loc.nilp}})$, i.e., 
$$(A\mod(\Rep_q(T))_{\on{loc.nilp}})^{< \infty,>-\infty}\subset A\mod(\Rep_q(T))_{\on{loc.nilp}}$$
can be recovered from the abelian category $(A\mod(\Rep_q(T))_{\on{loc.nilp}})^\heartsuit$ as its \emph{bounded derived category}
in the sense of \cite[Sect. 1.3.2]{Lu}. 

\medskip

All of $A\mod(\Rep_q(T))_{\on{loc.nilp}}$ can be recovered as the ind-completion of $(A\mod(\Rep_q(T))_{\on{loc.nilp}}))_c$,
which is a full subcategory in $(A\mod(\Rep_q(T))_{\on{loc.nilp}}))^{<\infty,>\infty}$. 

\ssec{Coalgebras and comodules}

\sssec{} \label{sss:dual Hopf}

Let $A^\vee\in \Rep_q(T)$ be the component-wise monoidal dual of $A$. The Hopf algebra structure on $A$
induces one on $A^\vee$ so that the diagrams
$$
\CD
A\otimes A^\vee \otimes A^\vee   & @>{\on{id}_A\otimes \on{mult}_{A^\vee}}>> & A\otimes A^\vee   \\
@V{\on{comult}_A\otimes \on{id}}VV & &  @VV{\on{pairing}}V   \\
A \otimes A\otimes A^\vee \otimes A^\vee@>{\on{id}_{A}\otimes \on{pairing}\otimes \on{id}_{A^\vee}}>> A\otimes A^\vee @>{\on{pairing}}>> k
\endCD
$$
and
$$
\CD
A \otimes A\otimes A^\vee  & @>{\on{mult}_{A}\otimes \on{id}_{A^\vee}}>> &  A\otimes A^\vee    \\
@V{\on{id}_{A\otimes A}\otimes \on{comult}_{A^\vee}}VV & &  @VV{\on{pairing}}V   \\
A \otimes A\otimes A^\vee \otimes A^\vee @>{\on{id}_{A}\otimes \on{pairing}\otimes \on{id}_{A^\vee}}>> A\otimes A^\vee  @>{\on{pairing}}>> k
\endCD
$$
are commutative. 

\sssec{}  \label{sss:rev algebras}

Note also that if $B$ is a Hopf algebra in a braided monoidal category $\CC$, we can attach to it a Hopf algebra
$B^{\on{rev-comult}}$ in the braided monoidal category $\CC^{\on{rev-br}}$.

\medskip

Namely, we let $B^{\on{rev-comult}}$ be the same as $B$ as an associative algebra, but we set the comultiplication to be  
$$B \overset{\on{comult}}\longrightarrow B\otimes B  \overset{R^{-1}_{B,B}}\longrightarrow B \otimes B.$$

\medskip

Similarly, we can consider the Hopf algebra $B^{\on{rev-mult}}$ in $\CC^{\on{rev-br}}$. It is the same as $B$ as a
co-associative coalgebra, and the multiplication is defined by 
$$B\otimes B \overset{R^{-1}_{B,B}}\longrightarrow B\otimes B \overset{\on{mult}}\to B.$$

\medskip

We note, however, that the antipode on $B$ identifies $B^{\on{rev-comult}}$ and $B^{\on{rev-mult}}$ 
as Hopf algebras. 

\sssec{}

If $B$ is a coalgebra as in \secref{sss:Hopf algebras}, we let
$$B\comod(\Rep_q(T))_{\on{fin.dim}}$$
denote the category of finite-dimensional $B$-comodules.
We let $B\comod(\Rep_q(T))$ denote the ind-completion of $B\comod(\Rep_q(T))_{\on{fin.dim}}$. 

\medskip

We have the forgetful functor
$$\oblv_B:B\comod(\Rep_q(T))\to \Rep_q(T),$$
which admits a continuous \emph{right} adjoint, denoted $\coind_B$. 
The corresponding comonad on $\Rep_q(T)$ is given by tensor product with $B$.

\medskip

Note, however, that the functor $\oblv_B$ is \emph{not} necessarily conservative. 

\sssec{}

We have:

\begin{lem} \label{l:mod vs comod}
There exists a canonical equivalence of monoidal categories 
$$A\mod(\Rep_q(T))_{\on{loc.nilp}}\simeq (A^\vee)^{\on{rev-mult}}\comod(\Rep_{q^{-1}}(T))$$
that commutes with the forgetful functors to
$$\Rep_q(T)\simeq (\Rep_{q^{-1}}(T))^{\on{rev-br}}.$$
\end{lem} 

\sssec{}  \label{sss:coind}

As a consequence, we obtain that the forgetful functor $\oblv_A$ of \eqref{e:oblv nilp}
admits a continuous \emph{right adjoint}, denoted $\coind_A$.

\medskip

The corresponding co-monad on $\Rep_q(T)$ is given by tensor product with $A^\vee$. 

\ssec{Formation of the relative Drinfeld double}

We will now perform the crucial step in the construction of quantum groups: we will add the
``negative part" of the quantum group by considering the (relative) Drinfeld center.  

\sssec{}

Note that tensor product by objects of $\Rep_q(T)$ \emph{on the right} makes the monoidal category $A\mod(\Rep_q(T))_{\on{loc.nilp}}$ into 
a \emph{right module category} for the braided monoidal category $\Rep_q(T)$.

\medskip

In this case, it makes sense to form the \emph{relative Drinfeld double}
$$Z_{\on{Dr},\Rep_q(T)}(A\mod(\Rep_q(T))_{\on{loc.nilp}}).$$

This is a braided monoidal category, universal with respect to the property that it acts on $A\mod(\Rep_q(T))_{\on{loc.nilp}}$
\emph{on the left}, in a way commuting with the right action of $\Rep_q(T)$.

\sssec{}

Explicitly, objects of $Z_{\on{Dr},\Rep_q(T)}(A\mod(\Rep_q(T))_{\on{loc.nilp}})$ are
objects $z\in A\mod(\Rep_q(T))_{\on{loc.nilp}}$ equipped with a system of identifications
\begin{equation} \label{e:Dr center}
z\otimes \CM\to \CM \otimes z, \quad \forall \,\CM\in A\mod(\Rep_q(T))_{\on{loc.nilp}},
\end{equation}
compatible with tensor products of the $\CM$'s, and such that for $\CM$ coming via 
augmentation from an object $c\in \Rep_q(T)$, the resulting map 
$$z\otimes c\to c \otimes z$$
is the braiding $R_{z,c}$ in $\Rep_q(T)$.

\sssec{}  \label{sss:forget from double}

We have the evident (monoidal) conservative forgetful functor
$$\oblv_{\on{Dr}\to A}:Z_{\on{Dr},\Rep_q(T)}(A\mod(\Rep_q(T))_{\on{loc.nilp}})\to A\mod(\Rep_q(T))_{\on{loc.nilp}},$$
which admits a \emph{left} adjoint, denoted $\ind_{A\to \on{Dr}}$. 

\medskip

Via the equivalence of \lemref{l:mod vs comod}, we obtain that $Z_{\on{Dr},\Rep_q(T)}(A\mod(\Rep_q(T))_{\on{loc.nilp}})$
admits also a (monoidal) forgetful functor 
$$\oblv_{\on{Dr}\to A^\vee}:Z_{\on{Dr},\Rep_q(T)}(A\mod(\Rep_q(T))_{\on{loc.nilp}})\to 
(A^\vee)^{\on{rev-mult}}\mod(\Rep_{q^{-1}}(T)),$$
which we further identify with
$$(A^\vee)^{\on{rev-comult}}\mod(\Rep_{q^{-1}}(T)),$$
via the antipode map. 

\medskip

If we disregard the monoidal structures, we can view the latter functor as 
$$Z_{\on{Dr},\Rep_q(T)}(A\mod(\Rep_q(T))_{\on{loc.nilp}})\to 
A^\vee\mod(\Rep_q(T)).$$

\ssec{Basic structures on the category of modules}

\sssec{}

We have the commutative diagram
\begin{equation} \label{e:ind cond0}
\CD
Z_{\on{Dr},\Rep_q(T)}(A\mod(\Rep_q(T))_{\on{loc.nilp}}) @>{\oblv_{\on{Dr}\to A}}>>  A\mod(\Rep_q(T))_{\on{loc.nilp}} \\
@V{\oblv_{\on{Dr}\to A^\vee}}VV   @VV{\oblv_A}V   \\
A^\vee\mod(\Rep_q(T)) @>{\oblv_{A^\vee}}>> \Rep_q(T),
\endCD
\end{equation} 
and the following two commutative diagrams, obtained by passing to left adjoints along the horizontal arrows 
(resp., right adjoints along the vertical arrows):

\begin{equation} \label{e:ind cond1}
\CD
Z_{\on{Dr},\Rep_q(T)}(A\mod(\Rep_q(T))_{\on{loc.nilp}}) @<{\ind_{A\to \on{Dr}}}<<  A\mod(\Rep_q(T))_{\on{loc.nilp}} \\
@V{\oblv_{\on{Dr}\to A^\vee}}VV   @VV{\oblv_A}V   \\
A^\vee\mod(\Rep_q(T))  @<{\ind_{A^\vee}}<< \Rep_q(T),
\endCD
\end{equation} 
and
\begin{equation} \label{e:ind cond2}
\CD
Z_{\on{Dr},\Rep_q(T)}(A\mod(\Rep_q(T))_{\on{loc.nilp}}) @>{\oblv_{\on{Dr}\to A}}>>  A\mod(\Rep_q(T))_{\on{loc.nilp}} \\
@A{\coind_{A^\vee \to \on{Dr}}}AA   @AA{\coind_A}A   \\
A^\vee\mod(\Rep_q(T))  @>{\oblv_{A^\vee}}>> \Rep_q(T).
\endCD
\end{equation} 

\sssec{}   \label{sss:t on double}

The category $Z_{\on{Dr},\Rep_q(T)}(A\mod(\Rep_q(T))_{\on{loc.nilp}})$ carries a t-structure for which the functor
$\oblv_{\on{Dr}\to A}$ is t-exact.

\medskip

It follows from diagrams \eqref{e:ind cond0}, \eqref{e:ind cond1} and \eqref{e:ind cond2} that the functors
$\oblv_{\on{Dr}\to A^\vee}$, $\ind_{A\to \on{Dr}}$ and $\coind_{A^\vee \to \on{Dr}}$ are also t-exact. 

\medskip

Note, however, that the above t-structure on $Z_{\on{Dr},\Rep_q(T)}(A\mod(\Rep_q(T))_{\on{loc.nilp}})$ is in general
\emph{not separated}. In particular, the functor $\oblv_{\on{Dr}\to A^\vee}$ is in general \emph{not} conservative.

\sssec{}

For every $\clambda\in \cLambda$, we have the standard and the costandard objects 
$$\BM_A^\clambda:=\ind_{A\to \on{Dr}}(k^\clambda) \text{ and } \BM_A^{\vee,\clambda}:=\coind_{A^\vee \to \on{Dr}}(k^\clambda)$$
that both lie in the heart of $Z_{\on{Dr},\Rep_q(T)}(A\mod(\Rep_q(T))_{\on{loc.nilp}})$.

\medskip

From the commutative diagrams \eqref{e:ind cond1} and \eqref{e:ind cond2} we obtain that
$$\oblv_{\on{Dr}\to A^\vee}(\BM_A^\clambda)\simeq \ind_{A^\vee}(k^\clambda) \text{ and }
\oblv_{\on{Dr}\to A^\vee}(\BM_A^{\vee,\clambda})\simeq \coind_{A}(k^\clambda).$$

From either of those we obtain:  
$$
\CHom_{Z_{\on{Dr},\Rep_q(T)}(A\mod(\Rep_q(T))_{\on{loc.nilp}})}(\BM_A^\clambda,\BM_A^{\vee,\clambda'})=
\begin{cases}
&k \text{ if } \clambda'=\lambda \\
&0 \text{ otherwise.}
\end{cases}
$$

\sssec{}  \label{sss:when compact}

\medskip

Note that, by construction, the objects $\BM_A^\clambda\in Z_{\on{Dr},\Rep_q(T)}(A\mod(\Rep_q(T))_{\on{loc.nilp}})$ 
are compact and generate $Z_{\on{Dr},\Rep_q(T)}(A\mod(\Rep_q(T))_{\on{loc.nilp}})$. From here it follows that
$$(Z_{\on{Dr},\Rep_q(T)}(A\mod(\Rep_q(T))_{\on{loc.nilp}}))_c\subset (Z_{\on{Dr},\Rep_q(T)}(A\mod(\Rep_q(T))_{\on{loc.nilp}}))^{<\infty ,>-\infty }.$$
We will now give an explicit
description of all compact objects in $Z_{\on{Dr},\Rep_q(T)}(A\mod(\Rep_q(T))_{\on{loc.nilp}})$:

\begin{prop} \label{p:when compact}
 An object $\CM\in (Z_{\on{Dr},\Rep_q(T)}(A\mod(\Rep_q(T))_{\on{loc.nilp}}))^{<\infty ,>-\infty }$ is compact if and only if 
$\oblv_{\on{Dr}\to A^\vee}(\CM)\in A^\vee\mod(\Rep_q(T))$ 
is compact.
\end{prop} 

\noindent Warning: if in the statement of the proposition we omitted the condition that $\CM$ be cohomologically bounded,
the assertion would be false: indeed, the functor $\oblv_{\on{Dr}\to A^\vee}$ is not necessarily conservative. 

\begin{proof}

The ``only if" direction is clear, since the functor $\oblv_{\on{Dr}\to A^\vee}$ preserves compactness, being the left adjoint
of a continuous functor.

\medskip

For the ``if" direction, let us start with an object in 
$$\CM\in (Z_{\on{Dr},\Rep_q(T)}(A\mod(\Rep_q(T))_{\on{loc.nilp}}))^{<\infty ,>-\infty }$$ whose
image in $A^\vee\mod(\Rep_q(T))$ along $\oblv_{\on{Dr}\to A^\vee}$ is compact. A standard argument shows that there exists a fiber
sequence
$$\CM'\to \CM''\to \CM',$$
where $\CM''\in (Z_{\on{Dr},\Rep_q(T)}(A\mod(\Rep_q(T))_{\on{loc.nilp}}))_c$ and 
$\CM'\in  (Z_{\on{Dr},\Rep_q(T)}(A\mod(\Rep_q(T))_{\on{loc.nilp}}))^\heartsuit$ is such that its image
in $A^\vee\mod(\Rep_q(T))$ along $\oblv_{\on{Dr}\to A^\vee}$ is finitely generated and free. 

\medskip

It is easy to
see that such $\CM'$ admits a filtration 
$$0=\CM'_0\subset \CM'_1\subset ...\subset \CM'_n=\CM',$$
where each successive quotient is generated under the action of $A^\vee$ by a homogeneous element (with respect
to the grading by $\cLambda$).  
However, such an element is necessarily annihilated by the action of $A$. Hence, we obtain that each successive
quotient is isomorphic to $\BM_A^\clambda$ for some $\clambda$. 

\end{proof} 

\sssec{} \label{sss:recover from heart}

From \secref{sss:recover from heart mod} it follows that the abelian category $(Z_{\on{Dr},\Rep_q(T)}(A\mod(\Rep_q(T))_{\on{loc.nilp}}))^\heartsuit$
has enough injectives, which are also acyclic in $Z_{\on{Dr},\Rep_q(T)}(A\mod(\Rep_q(T))_{\on{loc.nilp}})$. 

\medskip

This implies that the cohomologically bounded part of $Z_{\on{Dr},\Rep_q(T)}(A\mod(\Rep_q(T))_{\on{loc.nilp}})$, i.e., 
$$(Z_{\on{Dr},\Rep_q(T)}(A\mod(\Rep_q(T))_{\on{loc.nilp}}))^{<\infty ,>-\infty },$$ can be recovered from the abelian category 
$(Z_{\on{Dr},\Rep_q(T)}(A\mod(\Rep_q(T))_{\on{loc.nilp}}))^\heartsuit$ as its bounded derived category.
The entire $Z_{\on{Dr},\Rep_q(T)}(A\mod(\Rep_q(T))_{\on{loc.nilp}})$ can be recovered as the ind-completion of
$(Z_{\on{Dr},\Rep_q(T)}(A\mod(\Rep_q(T))_{\on{loc.nilp}}))_c\subset 
(Z_{\on{Dr},\Rep_q(T)}(A\mod(\Rep_q(T))_{\on{loc.nilp}}))^{<\infty ,>-\infty }$.

\section{Modules over the quantum group}   \label{s:quant grp}

We continue to review the basics of quantum groups. In this section we define the categories of primary interest,
$\Rep_q(G)$ and $\Rep^{\on{mxd}}_q(G)$. The former is (the usual) category of modules over Lusztig's quantum group.
The latter category is less known: it combines Lusztig's version for the positive part and the De Concini-Kac version for
the negative part. 

\ssec{The various version of $U_q(N)$}  \label{ss:q grp}

In this subsection we will introduce some particular Hopf algebras in $\Rep_q(T)$ that correspond to
the several versions of the quantum group that we will consider. 

\sssec{}

Let $U^{\on{free}}_q(N)$ be the free associative algebra in $\Rep_q(T)$ on the generators $e_i$, each in degree 
the simple root $\check\alpha_i$. It has a canonical Hopf algebra structure, defined by the condition that the 
comultiplication sends 
$$e_i\mapsto e_1\otimes 1+1\otimes e_i.$$

\begin{rem}
We emphasize that $U^{\on{free}}_q(N)$ is a Hopf algebra in $\Rep_q(T)$, and not in $\Vect$. The usual formula
$$\nabla(e_i)=e_i\otimes 1+K_i\cdot e_i\otimes 1$$
arises from the braiding on $\Rep_q(T)$, where $K_i$ acts on the $\clambda$-weight space as $b'(\check\alpha_i,\clambda)$.
\end{rem} 

\sssec{}

We introduce the De Concini-Kac version of $U_q(N)$, denoted $U^{\on{DK}}_q(N)$, 
to be the quotient of $U^{\on{free}}_q(N)$ defined as in \cite[Sect. 3.3.6]{Ga4}. 

\medskip

We note that if $q$ satisfies Assumption \eqref{e:estimate} below (which we \emph{will} be assuming in any case),
then the ideal of the projection
$$U^{\on{free}}_q(N)\to U^{\on{DK}}_q(N)$$
is generated by the quantum Serre relations, see \cite[Sect. 3.7.5]{Ga4}.

%
%\medskip
%
%It is known that the ideal generated by the quantum Serre relations is a Hopf ideal; hence
%$U^{\on{DK}}_q(N)$ has a unique structure of Hopf algebra, compatible with the projection
%$$U^{\on{free}}_q(N)\twoheadrightarrow U^{\on{DK}}_q(N).$$

\medskip

Swapping the roles of $\cLambda^+$ and $\cLambda^-$, we obtain the Hopf algebras in $\Rep_q(T)$:
$$U^{\on{free}}_q(N^-)\twoheadrightarrow U^{\on{DK}}_q(N^-).$$

\sssec{}

We define 
$$U^{\on{Lus}}_q(N)$$
to be the Hopf algebra in $\Rep_q(T)$ equal to the graded dual of $U^{\on{DK}}_q(N^-)$.

\sssec{}

Let $U^{\on{cofree}}_q(N)$ be the graded dual of $U^{\on{free}}_q(N^-)$. By duality, we have an injection
$$U^{\on{Lus}}_q(N)\hookrightarrow U^{\on{cofree}}_q(N).$$

\medskip

It is known that $U^{\on{Lus}}_q(N)$ is Lusztig's (i.e., quantum divided power) version of the quantum group.
In what follows we will use the notation
$$\Rep_q(B):=U^{\on{Lus}}_q(N)\mod(\Rep_q(T))_{\on{loc.nilp}}.$$

\sssec{}

Note that we have a canonical identification
$$(U^{\on{free}}_q(N^-))^{\on{rev-comult}}\simeq U^{\on{free}}_{q^{-1}}(N^-),$$
which induces an identification 
$$(U^{\on{DK}}_q(N^-))^{\on{rev-comult}}\simeq U^{\on{DK}}_{q^{-1}}(N^-).$$

This implies that we can think of $U^{\on{DK}}_{q^{-1}}(N^-)$ as obtained from $U^{\on{Lus}}_q(N)$
by the procedure of \secref{sss:dual Hopf}.

\sssec{}

We have a canonical map of Hopf algebras
\begin{equation} \label{e:free to cofree}
U^{\on{free}}_q(N)\to U^{\on{cofree}}_q(N),
\end{equation} 
obtained by extending the identity map on the generators. 

\medskip

It is known that \eqref{e:free to cofree}
factors as
$$U^{\on{free}}_q(N)\twoheadrightarrow U^{\on{DK}}_q(N) \to U^{\on{cofree}}_q(N).$$

By duality, \eqref{e:free to cofree} also factors as
$$U^{\on{free}}_q(N) \to U^{\on{Lus}}_q(N) \hookrightarrow U^{\on{cofree}}_q(N).$$

Hence, the map \eqref{e:free to cofree} actually factors as
\begin{equation} \label{e:free to cofree Lus DK}
U^{\on{free}}_q(N)\twoheadrightarrow U^{\on{DK}}_q(N) \to U^{\on{Lus}}_q(N) \hookrightarrow U^{\on{cofree}}_q(N).
\end{equation} 

\sssec{}

We let $u_q(N)$ denote the image of the above map $U^{\on{DK}}_q(N) \to U^{\on{Lus}}_q(N)$, so that we have

\begin{equation} \label{e:free to cofree Lus DK sm}
U^{\on{free}}_q(N)\twoheadrightarrow U^{\on{DK}}_q(N) \twoheadrightarrow  u_q(N) \hookrightarrow 
U^{\on{Lus}}_q(N) \hookrightarrow U^{\on{cofree}}_q(N).
\end{equation} 

The Hopf algebra $u_q(N)$ is the positive part of the ``small'' quantum group. 

\sssec{}

Let $u_q(N^-)$ be defined similarly. We have:
$$(u_q(N))^\vee\simeq u_q(N^-).$$

In what follows we will denote
$$\Rep^{\on{sml,grd}}_q(B):=u_q(N) \mod(\Rep_q(T))_{\on{loc.nilp}}.$$

\sssec{}

It is known that away from the case of the root of unity (i.e., when $q(\check\alpha_i)$ is not a root of unity for any simple root
$\check\alpha_i$), the maps
$$U^{\on{DK}}_q(N) \twoheadrightarrow  u_q(N) \hookrightarrow U^{\on{Lus}}_q(N)$$
are isomorphisms.

\medskip

By contrast, in the root of unity case (i.e., when $q(\check\alpha_i)$ is a root of unity for all simple roots), it is known that
$u_q(N)$ is finite-dimensional. In fact, its generators $e_i$ satisfy $e_i^{d_i}=0$ for $d_i=\on{ord}(q(\check\alpha_i))$. 

\ssec{Cohomology of the De Concini-Kac algebra}

For what follows we will need to review some cohomological properties of $U_q^{\on{DK}}(N^-)$, viewed as
an associative algebra. They can be summarized by saying that it behaves like the classical universal enveloping
$U(\fn^-)$.

\sssec{}

The following assertion is established in \cite[Theorem 6.4.1]{Geo}:

\begin{thm} \label{t:cohomology DK}
Assume that for every positive root $\alpha$ we have the inequality 
\begin{equation} \label{e:estimate}
\on{ord}(q(\check\alpha))\geq
\langle \check\rho,\alpha\rangle.
\end{equation}
Then we have
$$\on{Tor}_i^{U_q^{\on{DK}}(N^-)}(k,k)^\clambda=
\begin{cases}
&k, \quad \clambda=w(\check\rho)-\rho \text{ for } w\in W \text{ and } i=\ell(w),\\
&0 \text{ otherwise}.
\end{cases}
$$ 
\end{thm}

In other words, this theorem says that the homology of $U_q^{\on{DK}}(N^-)$ looks
exactly the same as that of the classical universal enveloping $U(\fn^-)$. 

\medskip

The estimate in \eqref{e:estimate} says that if the elements $q(\check\alpha)$ are torsion, then
their orders are not too small. 

\medskip

\noindent{\bf Assumption:} {\it In what follows, when discussing quantum groups we will assume that $q$
is such that the estimate \eqref{e:estimate} is satisfied.}

\medskip

We note that in \cite{Geo}, the assertion of \thmref{t:cohomology DK} was established under
a more restrictive assumption than \eqref{e:estimate} below. In a future publication we will
show that \eqref{e:estimate} is actually sufficient for the validity of \thmref{t:cohomology DK}.
Alternatively, the reader should assume that $q$
is such that the conclusion of \thmref{t:cohomology DK} holds. 

\medskip

Note that in terms of \eqref{e:adm}, the estimate \eqref{e:estimate} says that the weight $-\check\rho$ is
admissible. 

\sssec{}

In particular from \thmref{t:cohomology DK}, we obtain:

\begin{cor} \label{c:cohomology DK}
The groups $\on{Tor}_i^{U_q^{\on{DK}}(N^-)}(k,k)$ are non-zero only for finitely many $i$.
\end{cor}

Note that a statement analogous to \corref{c:cohomology DK} would be completely false for the other two versions
of the quantum group: $U^{\on{Lus}}_q(N^-)$ and $u_q(N^-)$. 

\sssec{}

We have the following general assertion:

\begin{lem}
Let an associative algebra $A$ is graded by a monoid isomorphic to $(\BZ^+)^n$ and such that
$k\to A^0$ is an isomorphism. Then the following conditions are equivalent:

\smallskip

\noindent{\em(i)} $\on{Tor}_i^A(k,k)\neq 0$ for finitely many $i$.

\smallskip

\noindent{\em(ii)} $A$ has a finite cohomological dimension.

\end{lem}

Combining this with \corref{c:cohomology DK}, we obtain: 

\begin{cor}  \label{c:UqDK}
The algebra $U^{\on{DK}}_q(N^-)$ has a finite cohomological dimension.
\end{cor}

In particular:

\begin{cor}  \label{c:aug DK}
The augmentation module $k$ is perfect as a $U^{\on{DK}}_q(N^-)$-module.
\end{cor} 

\ssec{The mixed quantum group}  \label{ss:mixed}

In this subsection we will define the second principal actor for this paper: the category of modules over the 
mixed quantum group. The terminology ``mixed'' comes from the fact that this version has Lusztig's quantum
group as its positive part, and the De Concini-Kac one as the negative part. 

\sssec{}

The basic object of study in this paper is the category
$$\Rep^{\on{mxd}}_q(G):=Z_{\on{Dr},\Rep_q(T)}(\Rep_q(B)).$$

\medskip

By \secref{sss:forget from double}, this category is equipped with a pair of adjoint functors
$$\ind_{\on{Lus}^+ \to \on{mxd}}:\Rep_q(B)\rightleftarrows \Rep^{\on{mxd}}_q(G):\oblv_{\on{mxd}\to \on{Lus}^+}.$$ 

\medskip

In addition, we have the following adjoint pair
$$\oblv_{\on{mxd}\to \on{DK}^-}:\Rep^{\on{mxd}}_q(G) \rightleftarrows U^{\on{DK}}_q(N^-) \mod(\Rep_q(T)): 
\coind_{\on{DK}^- \to \on{mxd}}.$$

\sssec{}

For $\clambda\in \cLambda$ we let
$$\BM_{q,\on{mxd}}^\clambda:=\ind_{\on{Lus}^+ \to \on{mxd}}(k^\clambda) \in \Rep^{\on{mxd}}_q(G) \text{ and }
\BM_{q,\on{mxd}}^{\vee,\clambda}:=\coind_{\on{DK}^- \to \on{mxd}}(k^\clambda)\in \Rep^{\on{mxd}}_q(G)$$
be the corresponding standard and costandard objects, respectively.

\medskip

From the commutative diagrams \eqref{e:ind cond1} and \eqref{e:ind cond2}, we obtain that 
$\BM_{q,\on{mxd}}^\clambda$ is free over $U^{\on{DK}}_q(N^-) $ and 
$\oblv_{\on{mxd}\to \on{Lus}^+}(\BM_{q,\on{mxd}}^{\vee,\clambda})$ is cofree as an object of $\Rep_q(B)$, and we have 
$$  
\CHom_{\Rep^{\on{mxd}}_q(G)}(\BM_{q,\on{mxd}}^\clambda,\BM_{q,\on{mxd}}^{\vee,\clambda'})=
\begin{cases}
&k \text{ if } \clambda=\clambda';\\
&0 \text{ otherwise}.
\end{cases}
$$ 

\sssec{}  \label{sss:when compact mixed}

The objects $\BM_{q,\on{mxd}}^\clambda$ are compact and generate $\Rep^{\on{mxd}}_q(G)$.

\medskip

Recall (see \secref{sss:when compact}) that an object of $(\Rep^{\on{mxd}}_q(G))^{<\infty ,>-\infty }$ is compact if and only if its image under 
$\oblv_{\on{mxd}\to \on{DK}^-}$ is compact in $U^{\on{DK}}_q(N^-) \mod(\Rep_q(T))$. 
Note that according to \corref{c:UqDK}, the compactness condition in $U^{\on{DK}}_q(\fn^-) \mod(\Rep_q(T))$
is equivalent to finite generation. 

\sssec{}

According to \secref{sss:t on double}, the category $\Rep^{\on{mxd}}_q(G)$ carries a unique t-structure for which 
the functors $\oblv_{\on{mxd}\to \on{Lus}^+}$ and $\oblv_{\on{mxd}\to \on{DK}^-}$ are t-exact. Moreover, the functors 
$\ind_{\on{Lus}^+ \to \on{mxd}}$ and $\coind_{\on{DK}^- \to \on{mxd}}$
are also t-exact, and in particular, $\BM_{q,\on{mxd}}^\clambda$ and $\BM_{q,\on{mxd}}^{\vee,\clambda}$
belong to $(\Rep^{\on{mxd}}_q(G))^\heartsuit$. 

\begin{rem} \label{r:recover mixed from abelian}
As was explained in \secref{sss:recover from heart}, the abelian category $(\Rep^{\on{mxd}}_q(G))^\heartsuit$ can be described
explicitly as the category of objects of $(\Rep_q(T))^\heartsuit$, endowed with a locally nilpotent action of $U_q^{\on{Lus}}(N)$ and a compatible
action of $U_q^{\on{DK}}(N^-) $. Furthermore,  $\Rep^{\on{mxd}}_q(G)$ can be recovered from $(\Rep^{\on{mxd}}_q(G))^\heartsuit$ 
by the following procedure: 

\medskip

First, we note that the cohomologically bounded part of $\Rep^{\on{mxd}}_q(G)$, i.e., $\left(\Rep^{\on{mxd}}_q(G)\right)^{<\infty,>\infty}$, 
is recovered as
$$\left(\Rep^{\on{mxd}}_q(G)\right)^{<\infty,>\infty}\simeq D^b\left((\Rep^{\on{mxd}}_q(G))^\heartsuit\right).$$ 
Now, all of $\Rep^{\on{mxd}}_q(G)$ is recovered as the
ind-completion of the full subcategory of $\left(\Rep^{\on{mxd}}_q(G)\right)^{<\infty,>\infty}$ generated under finite colimits by the
objects $\BM_{q,\on{mxd}}^\clambda$. 

\end{rem} 

\ssec{The category of modules over the ``big" quantum group}

In this subsection we recall the definition of another object of primary interest: the category of (algebraic=a.k.a. locally finite)
modules over Lusztig's quantum group, denoted $\Rep_q(G)$. 

\medskip

A salient feature of this category is that it does \emph{not} 
arise as Drinfeld's center. The only way we know how to construct $\Rep_q(G)$ is via the underlying abelian category. 

\sssec{}

Let $\Rep_q(G)_{\on{fin.dim}}$ be the category of algebraic representations of Lusztig's quantum group. By definition, this 
is the bounded derived category of the abelian category $(\Rep_q(G)_{\on{fin.dim}})^\heartsuit$ that consists of finite-dimensional
objects of $(\Rep_q(T))^\heartsuit$, endowed with actions of $U_q^{\on{Lus}}(N)$ and $U_q^{\on{Lus}}(N^-)$
that satisfy the usual relations. 

\medskip

The abelian category $(\Rep_q(G)_{\on{fin.dim}})^\heartsuit$ has enough projectives. We let 
\begin{equation} \label{e:perfect in all} 
\Rep_q(G)_{\on{perf}}\subset \Rep_q(G)_{\on{fin.dim}}
\end{equation} 
be the full subcategory consisting of perfect objects (i.e., those represented by finite complexes of projectives).

\medskip

We set
$$\Rep_q(G)_{\on{ren}}:=\on{IndCompl}\left(\Rep_q(G)_{\on{fin.dim}}\right) \text{ and }
\Rep_q(G):=\on{IndCompl}\left(\Rep_q(G)_{\on{perf}}\right).$$

Both categories $\Rep_q(G)_{\on{ren}}$ and $\Rep_q(G)$ carry naturally defined t-structures. 

\sssec{}
 
The inclusion \eqref{e:perfect in all} extends to a fully faithful functor
$$\fr_: \Rep_q(G)\to \Rep_q(G)_{\on{ren}}.$$

The above functor $\fr$ admits a right adjoint, denoted $\fs$, given by ind-extending the inclusion
$$\Rep_q(G)_{\on{fin.dim}}\hookrightarrow \Rep_q(G).$$

Note, however, that the functor $\fs$ is \emph{not} fully faithful (even though its restriction to the subcategory
of compact objects is.) 

\sssec{}   \label{sss:renorm pattern}

The situation of the adjoint pair
\begin{equation} \label{e:ren adj pattern}
\fr: \Rep_q(G)\rightleftarrows \Rep_q(G)_{\on{ren}}: \fs$$
is completely parallel to that of 
$$\QCoh(X) \rightleftarrows \IndCoh(X)
\end{equation} 
of \cite[Sect. 1]{Ga1} for a finite type scheme $X$. In particular, both $\Rep_q(G)_{\on{ren}}$ and $\Rep_q(G)$ have t-structures
such that the following properties hold: 

\begin{itemize}

\item The functor $\fs$ is t-exact and induces an equivalence on eventually coconnective subcategories 
(in particular, on the hearts);

\item The category $\Rep_q(G)$ is left-complete in its t-structure;

\item The kernel of $\fs$ consists of objects that are infinitely connective (i.e., those
objects all of whose cohomologies are zero).

\end{itemize} 

\sssec{}

We have the obvious forgetful functor
$$\oblv_{\on{big}\to \on{Lus}^+}:\Rep_q(G)_{\on{ren}}\to \Rep_q(B).$$

It admits both a left and a right adjoints, denoted $\ind_{\on{Lus}^+\to \on{big}}$ and $\coind_{\on{Lus}^+\to \on{big}}$,
respectively. For $\clambda\in \cLambda$, set
$$\CV^\clambda_q:=\ind_{\on{Lus}^+\to \on{big}}(k^\clambda)\in \Rep_q(G)_{\on{ren}}.$$

\medskip

It is known that for $\clambda\in \cLambda^+$, the object $\CV^\clambda_q$ belongs to $\Rep_q(G)^\heartsuit$;
it is called the Weyl module of highest weight $\clambda$. 

\medskip

Denote:
$$\CV^{\vee,\clambda}_q:=\coind_{\on{Lus}^+\to \on{big}}(k^{w_0(\clambda)})\in \Rep_q(G)_{\on{ren}}.$$
It is known that for $\clambda\in \cLambda^+$, the object $\CV^{\vee,\clambda}_q$ belongs to $\Rep_q(G)^\heartsuit$;
it is called the dual Weyl module of highest weight $\clambda$. 

\ssec{``Big" vs ``mixed"}

\sssec{}  \label{sss:big to mixed}

We have a canonically defined (braided monoidal) functor
$$(\Rep_q(G)_{\on{fin.dim}})^\heartsuit\to (Z_{\on{Dr},\Rep_q(T)}(\Rep_q(B)))^\heartsuit=(\Rep^{\on{mxd}}_q(G))^\heartsuit,$$
which extends to a (braided monoidal) functor
$$\oblv_{\on{big}\to \on{mxd}}:\Rep_q(G)_{\on{ren}}\to \Rep^{\on{mxd}}_q(G).$$

\begin{rem}
Note that the functor $\oblv_{\on{big}\to \on{mxd}}$ does \emph{not} factor through the projection
$$\fs_:\Rep_q(G)_{\on{ren}} \to \Rep_q(G).$$

This is due to the fact that $\Rep^{\on{mxd}}_q(G)$ is not separated in its t-structure. 
\end{rem}

\sssec{}

We claim:

\begin{prop}  \label{p:oblv big to mixed comp}
The functor $\oblv_{\on{big}\to \on{mxd}}$ sends compacts to compacts.
\end{prop}

\begin{proof}

First, by \secref{sss:when compact mixed}, the image of $k\in (\Rep_q(G))^\heartsuit$ under
$\oblv_{\on{big}\to \on{mxd}}(k)$, viewed as an object of $(\Rep^{\on{mxd}}_q(G))^\heartsuit\subset \Rep^{\on{mxd}}_q(G)$, 
is compact. 

\medskip

For any $\CM\in \Rep_q(G)_{\on{fin.dim}}$, we have 
$$\oblv_{\on{big}\to \on{mxd}}(\CM)\simeq \CM\otimes \oblv_{\on{big}\to \on{mxd}}(k).$$

However, it is easy to see that the operation of tensor product by $\CM\in \Rep_q(G)_{\on{fin.dim}}$
on $\Rep^{\on{mxd}}_q(G)$ preserves the subcategory generated under finite colimits by objects
$\BM_{q,\on{mxd}}^\clambda$. Indeed, every
$$\CM \otimes \BM_{q,\on{mxd}}^\clambda$$
admits a finite filtration with subquotients $\BM_{q,\on{mxd}}^{\clambda+\clambda'}$, where $\clambda'$
runs through the set of weights of $\CM$ (with multiplicities). 

\end{proof} 

\sssec{}

By construction, the composite functor
$$\Rep_q(G)_{\on{ren}} \overset{\oblv_{\on{big}\to \on{mxd}}}\longrightarrow 
\Rep^{\on{mxd}}_q(G)\overset{\oblv_{\on{mxd}\to \on{Lus}^+}}\longrightarrow \Rep_q(B)$$
identifies with the functor $\oblv_{\on{big}\to \on{Lus}^+}$.

\medskip

By adjunction, we obtain an identification
$$\ind_{\on{mxd}\to \on{big}} \circ \ind_{\on{Lus}^+\to \on{mxd}}\simeq \ind_{\on{Lus}^+\to \on{big}}.$$

In particular, we have a canonical isomorphism
\begin{equation} \label{e:Weyl from mixed}
\ind_{\on{mxd}\to \on{big}}(\BM_{q,\on{mxd}}^\clambda)\simeq \ind_{\on{Lus}^+\to \on{big}}(k^\clambda)=:\CV_q^\clambda
\end{equation}
for \emph{any} $\clambda\in \cLambda$.

\section{The small quantum group}  \label{s:small}

In this section we will specialize to the case when $q$ takes values in the group of roots of unity in $k^\times$.
We will study the category of representations of the \emph{small quantum group} and its relation to that
of Lusztig's version. 

\ssec{Modules for the graded small quantum group}

In this subsection we define (one of the three versions of) the category of modules for the $\cLambda$-graded version of the 
small quantum group. We will have two more versions versions of this category that differ from each by \emph{renormalization} (i.e., which objects are
declared to be compact). 

\sssec{}

We set
$$\Rep^{\on{sml,grd}}_q(G)_{\on{baby-ren}}:=Z_{\on{Dr},\Rep_q(T)}(\Rep^{\on{sml,grd}}_q(B)).$$

We have the usual adjunctions
$$\ind_{\on{sml}^+\to \on{sml}}: \Rep^{\on{sml,grd}}_q(B) \rightleftarrows 
\Rep^{\on{sml,grd}}_q(G)_{\on{baby-ren}}: \oblv_{\on{sml}\to \on{sml}^+}$$
and
$$\oblv_{\on{sml}\to \on{sml}^-}: \Rep^{\on{sml,grd}}_q(G)_{\on{baby-ren}} \rightleftarrows  
u_q(N^-) \mod(\Rep_q(T)): \coind_{\on{sml}^-\to \on{sml}}.$$

\sssec{}

For $\clambda\in \cLambda$, we let
$$\BM^\clambda_{q,\on{sml}}:=\ind_{\on{sml}^+\to \on{sml}}(k^\clambda) \text{ and }
\BM^{\vee,\clambda}_{q,\on{sml}}:=\coind_{\on{sml}^-\to \on{sml}}(k^\clambda)$$
be the corresponding standard and costandard objects.

\medskip

We have

$$  
\CHom_{\Rep^{\on{sml,grd}}_q(G)_{\on{baby-ren}}}(\BM_{q,\on{sml}}^\clambda,\BM_{q,\on{sml}}^{\vee,\clambda'})=
\begin{cases}
&k \text{ if } \clambda=\clambda';\\
&0 \text{ otherwise}.
\end{cases}
$$ 

\begin{rem}
The objects $\BM^\clambda_{q,\on{sml}}$ and $\BM^{\vee,\clambda}_{q,\on{sml}}$ are sometimes called the \emph{baby Verma}
and \emph{dual baby Verma} modules, respectively. This is the origin of the notation ``$\on{baby-ren}$" in the subscript. 
\end{rem}

\ssec{Renormalized categories}

We will now introduce two more versions of the category of modules over the small quantum group, denoted 
$\Rep^{\on{sml,grd}}_q(G)_{\on{ren}}$ and $\Rep^{\on{sml,grd}}_q(G)$, respectively, that differ from the original
one by a renormalization procedure (i.e., by redefining the class of compact objects). 

\sssec{}

Note that since $u_q(N) $ is finite-dimensional, we have a fully faithful embedding
$$u_q(N) \mod(\Rep_q(T))_{\on{perf}}\hookrightarrow u_q(N) \mod(\Rep_q(T))_{\on{fin.dim}},$$
and by ind-extension a fully faithful embedding
$$\fr_{\on{baby}}:u_q(N) \mod(\Rep_q(T))\hookrightarrow u_q(N) \mod(\Rep_q(T))_{\on{loc.nilp}}.$$
The latter admits a right adjoint, denoted $\fs_{\on{baby}}$; it is the ind-extension of the fully faithful embedding
$$u_q(N) \mod(\Rep_q(T))_{\on{fin.dim}}\hookrightarrow u_q(N) \mod(\Rep_q(T)).$$

The adjoint pair $(\fr,\fs)$ has the same properties as the pair \eqref{e:ren adj pattern}. 
 
\sssec{}

Note that the category $\Rep^{\on{sml,grd}}_q(G)_{\on{baby-ren}}$ can be thought of as modules for the monad 
$$\oblv_{\on{sml}\to \on{sml}^+}\circ \ind_{\on{sml}^+\to \on{sml}}$$
acting on $u_q(N) \mod(\Rep_q(T))_{\on{loc.nilp}}$.
We observe that the action of this monad preserves the subcategory $u_q(N) \mod(\Rep_q(T))_{\on{perf}}$, and hence
also $u_q(N) \mod(\Rep_q(T))$.

\medskip

We define the category $\Rep^{\on{sml,grd}}_q(G)$ to be 
$$\oblv_{\on{sml}\to \on{sml}^+}\circ \ind_{\on{sml}^+\to \on{sml}}\mod(u_q(N) \mod(\Rep_q(T))).$$

\sssec{}
By construction, we have an adjoint pair
$$\ind_{\on{sml}^+\to \on{sml}}: u_q(N) \mod(\Rep_q(T))\rightleftarrows 
\Rep^{\on{sml,grd}}_q(G): \oblv_{\on{sml}\to \on{sml}^+},$$
and a pair of adjoint functors
$$\fr_{\on{baby}}: \Rep^{\on{sml,grd}}_q(G)\rightleftarrows  \Rep^{\on{sml,grd}}_q(G)_{\on{baby-ren}}:\fs_{\on{baby}}$$
that makes all circuits in the following diagram commute:
\begin{equation} \label{e:nonren baby and ind}
\xymatrix{
u_q(N) \mod(\Rep_q(T)) \ar[rr]<2pt> \ar[d]<2pt>^{\fr_{\on{baby}}} &&
\Rep^{\on{sml,grd}}_q(G)^{\on{non-ren}}  \ar[ll]<2pt>  \ar[d]<2pt>^{\fr_{\on{baby}}}  \\
u_q(N) \mod(\Rep_q(T))_{\on{loc.nilp}} \ar[rr]<2pt> \ar[u]<2pt>^{\fs_{\on{baby}}} && 
\Rep^{\on{sml,grd}}_q(G)_{\on{baby-ren}} \ar[ll]<2pt> \ar[u]<2pt>^{\fs_{\on{baby}}}}
\end{equation} 

\medskip

The adjoint pair $(\fr_{\on{baby}},\fs_{\on{baby}})$ has the same properties as the pair 
\eqref{e:ren adj pattern}, specified in \secref{sss:renorm pattern}. In particular, $\Rep^{\on{sml,grd}}_q(G)$
is left-complete in its t-structure. 

\sssec{}  \label{sss:small Verma baby vs non-ren}

By a slight abuse of notation, we will denote by the same symbol $\BM^\clambda_{q,\on{sml}}$ the image
of $\BM^\clambda_{q,\on{sml}}\in \Rep^{\on{sml,grd}}_q(G)_{\on{baby-ren}}$ under the functor
$$\fs_{\on{baby}}: \Rep^{\on{sml,grd}}_q(G)_{\on{baby-ren}}\to \Rep^{\on{sml,grd}}_q(G).$$

Note, however, that $\BM^\clambda_{q,\on{sml}}\in \Rep^{\on{sml,grd}}_q(G)$ is \emph{not} compact. Moreover,
it is \emph{not} true that the image of $\BM^\clambda_{q,\on{sml}}\in \Rep^{\on{sml,grd}}_q(G)$ under the functor
$$\fr_{\on{baby}}: \Rep^{\on{sml,grd}}_q(G)\rightleftarrows  \Rep^{\on{sml,grd}}_q(G)_{\on{baby-ren}}$$
gives back $\BM^\clambda_{q,\on{sml}}\in \Rep^{\on{sml,grd}}_q(G)_{\on{baby-ren}}$ (but we do have a map
from the former to the latter). 

\medskip

Similarly, we define $\BM^{\vee,\clambda}_{q,\on{sml}}\in \Rep^{\on{sml,grd}}_q(G)_{\on{baby-ren}}$, and the above remarks
apply. Since the functor $\fs_{\on{baby}}$ is an equivalence on the eventually coconnective subcategories, we have: 
$$  
\CHom_{\Rep^{\on{sml,grd}}_q(G)}(\BM_{q,\on{sml}}^\clambda,\BM_{q,\on{sml}}^{\vee,\clambda'})=
\begin{cases}
&k \text{ if } \clambda=\clambda';\\
&0 \text{ otherwise}.
\end{cases}
$$

\sssec{}

Let 
$$\Rep^{\on{sml,grd}}_q(G)_{\on{fin.dim}}\subset \Rep^{\on{sml,grd}}_q(G)$$
be the full (but not cocomplete) subcategory consisting of finite-dimensional objects (i.e., those objects that
have non-zero cohomology only in finitely many cohomological degrees, and each of these cohomologies
is finite-dimensional).

\medskip

We define
$$\Rep^{\on{sml,grd}}_q(G)_{\on{ren}}:=\on{IndCompl}(\Rep^{\on{sml,grd}}_q(G)_{\on{fin.dim}}).$$

We have the adjoint pair
$$\fr: 
\Rep^{\on{sml,grd}}_q(G)\rightleftarrows  \Rep^{\on{sml,grd}}_q(G)_{\on{ren}}:\fs$$
that has the same properties as the pair 
\eqref{e:ren adj pattern}, specified in \secref{sss:renorm pattern}.

\sssec{}

Consider the category $\Rep^{\on{sml,grd}}_q(G)_{c,\on{baby-ren}}$ of compact objects in $\Rep^{\on{sml,grd}}_q(G)_{\on{baby-ren}}$.
Note that the functor $\fr_{\on{baby}}$ induces a fully faithful embedding
$$\Rep^{\on{sml,grd}}_q(G)_{c,\on{baby-ren}}\hookrightarrow \Rep^{\on{sml,grd}}_q(G)_{\on{fin.dim}}.$$

Ind-extending, we obtain a fully faithful functor
$$\fr_{\on{baby-ren}\to \on{ren}}: \Rep^{\on{sml,grd}}_q(G)_{\on{baby-ren}}\to \Rep^{\on{sml,grd}}_q(G)_{\on{ren}}$$
so that
$$\fr_\simeq \fr_{\on{baby-ren}\to \on{ren}}\circ \fr_{\on{baby}}.$$

Since the functor $\fr_{\on{baby-ren}\to \on{ren}}$ sends compacts to compacts, it admits a continuous
right adjoint, which we will denote by 
$$\fs_{\on{ren}\to \on{baby-ren}}:\Rep^{\on{sml,grd}}_q(G)_{\on{ren}}\to \Rep^{\on{sml,grd}}_q(G)_{\on{baby-ren}}.$$

By adjunction, we have:
$$\fs \simeq \fs_{\on{baby}}\circ \fs_{\on{ren}\to \on{baby-ren}}.$$

\sssec{}

As in \cite[Corollary 4.4.3]{AriG}, we obtain that the functor $\fs_{\on{ren}\to \on{baby-ren}}$ is t-exact
and induces an equivalence on the eventually coconnective subcategories.

\medskip 

To summarize, we have the following sequence of fully faithful embeddings
$$\Rep^{\on{sml,grd}}_q(G) \overset{\fr_{\on{baby}}}\hookrightarrow 
\Rep^{\on{sml,grd}}_q(G)_{\on{baby-ren}} \overset{\fr_{\on{baby-ren}\to \on{ren}}}\hookrightarrow  
\Rep^{\on{sml,grd}}_q(G)_{\on{ren}}$$
and their right adjoints
$$\Rep^{\on{sml,grd}}_q(G) \overset{\fs_{\on{baby}}}\twoheadleftarrow
\Rep^{\on{sml,grd}}_q(G)_{\on{baby-ren}} \overset{\fs_{\on{ren}\to \on{baby-ren}}}\twoheadleftarrow
\Rep^{\on{sml,grd}}_q(G)_{\on{ren}}.$$

\sssec{}

By a slight abuse of notation, for $\clambda\in \cLambda$ set
$$\BM^\clambda_{q,\on{sml}}:=\fr_{\on{baby-ren}\to \on{ren}}(\BM^\clambda_{q,\on{sml}})\in \Rep^{\on{sml,grd}}_q(G)_{\on{ren}}.$$

By construction, $\BM^\clambda_{q,\on{sml}}$ lies in the heart of $\Rep^{\on{sml,grd}}_q(G)_{\on{ren}}$, and when we apply
the functor $\fs_{\on{ren}\to \on{baby-ren}}$ to it we recover the original $\BM^\clambda_{q,\on{sml}}\in \Rep^{\on{sml,grd}}_q(G)_{\on{baby-ren}}$. 

\medskip

We let
$$\BM^{\vee,\clambda}_{q,\on{sml}}\in \Rep^{\on{sml,grd}}_q(G)_{\on{ren}}$$
be the unique object in the heart that gets sent to $\BM^{\vee,\clambda}_{q,\on{sml}}\in \Rep^{\on{sml,grd}}_q(G)_{\on{baby-ren}}$
by the functor $\fr_{\on{ren}\to \on{baby-ren}}$. 

\medskip

By adjunction, we have:
$$  
\CHom_{\Rep^{\on{sml,grd}}_q(G)_{\on{ren}}}(\BM_{q,\on{sml}}^\clambda,\BM_{q,\on{sml}}^{\vee,\clambda'})=
\begin{cases}
&k \text{ if } \clambda=\clambda';\\
&0 \text{ otherwise}.
\end{cases}
$$ 

\begin{rem}
A variant of the remark \ref{r:recover mixed from abelian} applies to the above three versions of modules over the small
quantum group as well. Namely, in all three cases, the corresponding abelian category
$$\left(\Rep^{\on{sml,grd}}_q(G)\right)^\heartsuit \simeq \left(\Rep^{\on{sml,grd}}_q(G)_{\on{baby-ren}}\right)^\heartsuit\simeq 
\left(\Rep^{\on{sml,grd}}_q(G)_{\on{ren}}\right)^\heartsuit$$
is that of objects of $\Rep_q(T)^\heartsuit$ equipped with an action of $u_q(N)$ (which is automatically locally nilpotent as
$u_q(N)$ is finite-dimensional) and a compatible action of $u_q(N^-) $ (which also 
automatically happens to be locally nilpotent). 

\medskip

Each of the categories
$$\left(\Rep^{\on{sml,grd}}_q(G)\right)^{<\infty,>\infty} \simeq \left(\Rep^{\on{sml,grd}}_q(G)_{\on{baby-ren}}\right)^{<\infty,>\infty} \simeq 
\left(\Rep^{\on{sml,grd}}_q(G)_{\on{ren}}\right)^{<\infty,>\infty}$$
identifies with $D^b\left(\left(\Rep^{\on{sml,grd}}_q(G)\right)^\heartsuit\right)$.

\medskip

Now, $\Rep^{\on{sml,grd}}_q(G)_{\on{baby-ren}}$ (resp., $\Rep^{\on{sml,grd}}_q(G)$)
identifies with the ind-completion of the category generated under finite colimits
by objects of the form $\BM_{q,\on{sml}}^\clambda$ (resp., $\ind_{\on{sml}^+\to \on{sml}}(k^\clambda\otimes u_q(N))$). 

\medskip

For $\Rep^{\on{sml,grd}}_q(G)_{\on{ren}}$ we take as generators all finite-dimensional
objects of $\left(\Rep^{\on{sml,grd}}_q(G)\right)^\heartsuit$. 

\end{rem} 

\ssec{The ungraded small quantum group}

\sssec{}   \label{sss:quant Frob lattice} 

Let $H$ be the reductive group that is the recipient of Lusztig's quantum Frobenius. Let
$$T_H\subset B_H\subset H$$
be the Cartan and Borel subgroups of $H$, respectively. 

\medskip

Denote by $\Lambda_H$ the weight lattice of $T_H$. Explicitly,
$$\Lambda_H=\{\clambda\in \cLambda\,|\, b(\clambda,\clambda')=1 \text{ for all }\clambda'\in \cLambda\}.$$

\sssec{}

Quantum Frobenius for tori defines a map from the category $\Rep(T_H)$ to the $E_3$-center of 
$\Rep_q(T)$. In particular, we obtain an action of $\Rep(T_H)$ on any of the categories of the form 
$$Z_{\on{Dr},\Rep_q(T)}(A\mod(\Rep_q(T))_{\on{loc.nilp}})$$
of \secref{sss:Hopf algebras}, and in particular on $\Rep^{\on{mxd}}_q(G)$ and $\Rep^{\on{sml,grd}}_q(G)_{\on{baby-ren}}$.

\medskip

In addition, by unwinding the constructions, we obtain that we also have an action of $\Rep(T_H)$ on the categories
$\Rep^{\on{sml,grd}}_q(G)$ and $\Rep^{\on{sml,grd}}_q(G)_{\on{ren}}$.

\medskip

Set
$$\Rep^{\on{sml}}_q(G):=\Vect\underset{\Rep(T_H)}\otimes \Rep^{\on{sml,grd}}_q(G),$$
$$\Rep^{\on{sml}}_q(G)_{\on{baby-ren}}:=\Vect\underset{\Rep(T_H)}\otimes \Rep^{\on{sml,grd}}_q(G)_{\on{baby-ren}},$$
$$\Rep^{\on{sml}}_q(G)_{\on{ren}}:=\Vect\underset{\Rep(T_H)}\otimes \Rep^{\on{sml,grd}}_q(G)_{\on{ren}},$$
where $\Rep(T_H)\to \Vect$ is the forgetful functor.

\medskip

These are the three versions of the category of representation of the \emph{ungraded} small quantum group. 
Each of these categories carries a t-structure, uniquely characterized by the property that the forgetful functor
from the corresponding graded version is t-exact.

\sssec{}

Note that the identification
$$\Rep^{\on{sml}}_q(G)_{\on{ren}}:=\Vect\underset{\Rep(T_H)}\otimes \Rep^{\on{sml,grd}}_q(G)_{\on{ren}}$$
gives rise to an action of $T_H$ on the category $\Rep^{\on{sml}}_q(G)_{\on{ren}}$ so that
$$\inv_{T_H}\left(\Rep^{\on{sml}}_q(G)_{\on{ren}}\right)\simeq \Rep^{\on{sm,grd}}_q(G)_{\on{ren}},$$
and similarly for the two other versions. 

\sssec{}

We have the fully faithful embeddings
\begin{equation} \label{e:string of small r}
\Rep^{\on{sml}}_q(G) \overset{\fr_{\on{baby}}}\hookrightarrow 
\Rep^{\on{sml}}_q(G)_{\on{baby-ren}} \overset{\fr_{\on{baby-ren}\to \on{ren}}}\hookrightarrow  
\Rep^{\on{sml}}_q(G)_{\on{ren}}
\end{equation} 
and their right adjoints
\begin{equation} \label{e:string of small s}
\Rep^{\on{sml}}_q(G) \overset{\fs_{\on{baby}}}\twoheadleftarrow
\Rep^{\on{sml}}_q(G)_{\on{baby-ren}} \overset{\fs_{\on{ren}\to \on{baby-ren}}}\twoheadleftarrow
\Rep^{\on{sml}}_q(G)_{\on{ren}},
\end{equation} 
which are t-exact and induce equivalences on eventually coconnective parts. The adjoint pairs
$$\fr_{\on{baby}}:\Rep^{\on{sml}}_q(G)\rightleftarrows \Rep^{\on{sml}}_q(G)_{\on{baby-ren}}:\fs_{\on{baby}}$$
and
$$\fr:\Rep^{\on{sml}}_q(G)\rightleftarrows \Rep^{\on{sml}}_q(G)_{\on{ren}}:\fs$$
have the same properties as the pair 
\eqref{e:ren adj pattern}, specified in \secref{sss:renorm pattern}. 

\sssec{}

By a slight abuse of notation we will denote by 
$$\BM_{q,\on{sml}}^\clambda \in \Rep^{\on{sml}}_q(G)_{?}$$
(for the above three options for the value of ?) 
the image of the standard object $\BM_{q,\on{sml}}^\clambda \in \Rep^{\on{sml,grd}}_q(G)_{?}$
under the forgetful functor
$$\Rep^{\on{sml,grd}}_q(G)_{?}\to \Rep^{\on{sml}}_q(G)_{?}.$$

\medskip

Note, however, that these objects of the ungraded category only depend on $\clambda$ as an element of the quotient group $\cLambda/\Lambda_H$. 

\begin{rem}  \label{r:recover small from heart}
As in the case of the graded version, each of the above three categories can be recovered from its
heart. First off, the abelian category
$$\left(\Rep^{\on{sml}}_q(G)\right)^\heartsuit \simeq \left(\Rep^{\on{sml}}_q(G)_{\on{baby-ren}}\right)^\heartsuit\simeq 
\left(\Rep^{\on{sml}}_q(G)_{\on{ren}}\right)^\heartsuit$$
identifies with modules over the usual  small quantum universal enveloping algebra $u_q(G)$. 

\medskip

The category $\Rep^{\on{sml}}_q(G)$, as defined above, identifies with the (derived) category of modules over $u_q(G)$. 
This is the most commonly used version of the category of modules over the small quantum group. 

\medskip

The categories
$$\Rep^{\on{sml}}_q(G)_{\on{baby-ren}} \text{ and } \Rep^{\on{sml}}_q(G)_{\on{ren}}$$
are obtained as ind-completions of the full (but not cocomplete subcategory) of $\Rep^{\on{sml}}_q(G)$, 
generated under finite colimits by modules of the form $\BM_{q,\on{sml}}^\clambda$ (for the ``baby" version)
and all finite-dimensional modules (for the ``ren" version), respectively. 

\end{rem} 

\ssec{Relation between the ``big" and the ``small" quantum groups}

\sssec{}

As in \secref{sss:big to mixed}, we have a canonically defined braided monoidal functor
\begin{equation} \label{e:big to sml grd}
\oblv_{\on{big}\to \on{sml}}:\Rep_q(G)_{\on{ren}}\to \Rep^{\on{sml,grd}}_q(G)_{\on{ren}}.
\end{equation} 

By construction, this functor sends compacts to compacts.

\medskip

By a slight abuse of notation, we will denote by the same symbol $\oblv_{\on{big}\to \on{sml}}$ the composition
of the above functor with the forgetful functor
$$\Rep^{\on{sml,grd}}_q(G)_{\on{ren}}\to \Rep^{\on{sml}}_q(G)_{\on{ren}}.$$

\sssec{}

Let
$$\on{Frob}^*_q:\Rep(H)\to \Rep_q(G)_{\on{ren}}$$
denote Lusztig's quantum Frobenius.  This functor is in fact an $E_3$-functor from $\Rep(H)$,
viewed as an $E_3$-category, to the $E_3$-center of $\Rep_q(G)_{\on{ren}}$. 

\medskip

We have the following commutative diagram
$$
\CD
\Rep(H)  @>{\on{Frob}^*_q}>> \Rep_q(G)_{\on{ren}}  \\
@V{\oblv_{H\to T_H}}VV   @VV{\oblv_{\on{big}\to \on{sml}}}V   \\
\Rep(T_H)  @>{\on{Frob}^*_q}>> \Rep^{\on{sml,grd}}_q(G)_{\on{ren}}. 
\endCD
$$

In particular, the functor $\oblv_{\on{big}\to \on{sml}}$ of \eqref{e:big to sml grd} canonically factors as
\begin{equation} \label{e:AG grd}
\Rep(T_H)\underset{\Rep(H)}\otimes \Rep_q(G)_{\on{ren}}\to \Rep^{\on{sml,grd}}_q(G)_{\on{ren}}.
\end{equation}

The following theorem was established
in \cite{AG1}: 

\begin{thm} \label{t:AG grd}
The functor \eqref{e:AG grd} is an equivalence.
\end{thm}

\sssec{}

By tensoring \eqref{e:AG grd} with $\Vect$ over $\Rep(T_H)$, we obtain a functor
\begin{equation} \label{e:AG}
\Vect \underset{\Rep(H)}\otimes \Rep_q(G)_{\on{ren}}\to \Rep^{\on{sml}}_q(G)_{\on{ren}}.
\end{equation}

From \thmref{t:AG grd} we obtain: 

\begin{cor} \label{c:AG}  The functor \eqref{e:AG} is an equivalence. In particular, the category 
$\Rep^{\on{sml}}_q(G)_{\on{ren}}$ carries an action of the group $H$, and we have an
identification 
$$\inv_{H}\left(\Rep^{\on{sml}}_q(G)_{\on{ren}}\right)\simeq \Rep_q(G)_{\on{ren}}.$$
\end{cor}

\medskip

Note that, by construction, we have a commutative diagram
$$
\CD
\inv_{H}\left(\Rep^{\on{sml}}_q(G)_{\on{ren}}\right)  @>{\sim}>>  \Rep_q(G)_{\on{ren}} \\
@V{\oblv_{H\to T_H}}VV   @VV{\oblv_{\on{big}\to \on{sml}}}V   \\
\inv_{T_H}\left(\Rep^{\on{sml}}_q(G)_{\on{ren}}\right)  @>{\sim}>>  \Rep^{\on{sml,grd}}_q(G)_{\on{ren}} \\
@V{\oblv_H}VV  @VVV    \\
\inv_{1}\left(\Rep^{\on{sml}}_q(G)_{\on{ren}}\right)    @>{\sim}>>  \Rep^{\on{sml}}_q(G)_{\on{ren}}.
\endCD
$$

\sssec{}

The action of $H$ on $\Rep^{\on{sml}}_q(G)_{\on{ren}}$ given by \corref{c:AG} respects the t-structure. Thus, by viewing
$\Rep^{\on{sml}}_q(G)$ (resp., $\Rep_q(G)$) as a quotient of $\Rep^{\on{sml}}_q(G)_{\on{ren}}$ (resp., $\Rep_q(G)_{\on{ren}}$)
by the subcategory of infinitely connective objects, we obtain that the category $\Rep^{\on{sml}}_q(G)$ also carries an action
of $H$, compatible with the inclusion $\fr$, and we have a canonical identification
$$\on{Inv}_H(\Rep^{\on{sml}}_q(G))\simeq \Rep_q(G),$$
and as a consequence
$$\Rep^{\on{sml}}_q(G)\simeq \Vect \underset{\Rep(H)}\otimes \Rep_q(G).$$

These identifications are compatible with the functors $\fr$ and $\fs$. 

%\sssec{}

%Along with the above identifications, we also obtain the identifications
%\begin{equation} \label{e:big vs graded}
%\Rep^{\on{sml,grd}}_q(G)_{\on{ren}} \simeq \Rep(T_H)\underset{\Rep(H)}\otimes \Rep_q(G)_{\on{ren}}
%\text{ and } 
%\Rep^{\on{sml,grd}}_q(G)\simeq \Rep(T_H)\underset{\Rep(H)}\otimes \Rep_q(G),
%\end{equation} 
%compatible with  
%$$\Rep^{\on{sml}}_q(G)_{\on{ren}}\simeq \Vect\underset{\Rep(T_H)}\otimes \Rep^{\on{sml,grd}}_q(G)_{\on{ren}} \text{ and }
%\Rep^{\on{sml}}_q(G)\simeq \Vect\underset{\Rep(T_H)}\otimes \Rep^{\on{sml,grd}}_q(G).$$

\ssec{Digression: ``big" vs ``small" for the quantum Borel}

The material in this subsection seems not to have references in the literature. 

\sssec{}

Let $\fn_H$ be the Lie algebra of the unipotent radical of $B_H$. 
Quantum Frobenius for $B$ is the map of Hopf algebras
\begin{equation} \label{e:q Frob B}
\on{Frob}_q:U^{\on{Lus}}_q(N)\to U(\fn_H),
\end{equation}
where $U(\fn_H)$ is regarded as a Hopf algebra in $\Rep_q(T)$ via the
$$\on{Frob}_q:\Rep(T_H)\to \Rep_q(T).$$

We have a ``short exact sequence" of Hopf algebras
\begin{equation} \label{e:SOE +}
0\to u_q(N)\to U^{\on{Lus}}_q(N)\to U(\fn_H) \to 0
\end{equation} 

\medskip

A key property of \eqref{e:q Frob B} is that it is \emph{co-central}. 

\medskip

Dually, we have a Hopf algebra homomorphism
$$\on{Sym}(\fn^-_H)\to U_q^{\on{DK}}(N^-),$$
which is \emph{central}, and a short exact sequence of Hopf algebras
\begin{equation} \label{e:SOE -}
0\to \on{Sym}(\fn^-)\to U_q^{\on{DK}}(N^-)\to u_q(N^-)\to 0.
\end{equation} 

\sssec{}

Pullback with respect to \eqref{e:q Frob B} defines a monoidal functor
$$\on{Frob}_q^*:\Rep(B_H)\to \Rep_q(B)$$
so that the diagram
$$
\CD
%\Rep(H) @>{\on{Frob}_q^*}>> \Rep_q(G)_{\on{ren}}  \\
%@V{\oblv_{H\to B_H}}VV   @VV{\oblv_{\on{big}\to \on{Lus}^+}}V  \\
\Rep(B_H)  @>{\on{Frob}_q^*}>>  \Rep_q(B) \\
@V{\oblv_{B_H\to T_H}}VV   @VV{\oblv_{\on{Lus}^+\to \on{sml}^+}}V  \\
\Rep(T_H)  @>{\on{Frob}_q^*}>>  \Rep^{\on{sml,grd}}_q(B)
\endCD
$$
commutes. 

\medskip

In particular, we obtain that the functor 
$$\oblv_{\on{Lus}^+\to \on{sml}^+}: \Rep_q(B) \to \Rep^{\on{sml,grd}}_q(B)$$
canonically factors as 
\begin{equation} \label{e:AG for B grd}
\Rep(T_H) \underset{\Rep(B_H)}\otimes \Rep_q(B)\to  \Rep^{\on{sml,grd}}_q(B).
\end{equation}

\medskip

In the same way as one proves the equivalence in \thmref{t:AG grd}, one also establishes:

\begin{thm} \label{t:AG for B grd} 
The functor \eqref{e:AG for B grd} is an equivalence.
\end{thm} 

\sssec{}

Tensoring \eqref{e:AG for B} with $\Vect$ over $\Rep(T_H)$, we obtain a functor
\begin{equation} \label{e:AG for B}
\Vect \underset{\Rep(H)}\otimes \Rep_q(B) \to \Vect \underset{\Rep(T_H)}\otimes \Rep^{\on{sml,grd}}_q(B)=:\Rep^{\on{sml}}_q(B).
\end{equation}

From \thmref{t:AG for B grd} we obtain: 

\begin{cor} \label{c:AG for B}  The functor \eqref{e:AG for B} is an equivalence. In particular, the category 
$\Rep^{\on{sml}}_q(B)$ carries an action of the group $B_H$, and we have an
identification 
$$\inv_{B_H}(\Rep^{\on{sml}}_q(B))\simeq \Rep_q(B).$$
\end{cor}

\sssec{}

The restriction functor 
$$\oblv_{\on{big}\to \on{Lus}^+}:\Rep_q(G)_{\on{ren}}\to \Rep_q(B)$$
is compatible with the actions of $\Rep(H)$ and $\Rep(B_H)$ via the commutative diagram
$$
\CD
\Rep(H) @>{\oblv_{H\to B_H}}>>  \Rep(B_H) \\
@V{\on{Frob}_q^*}VV   @V{\on{Frob}_q^*}VV  \\
\Rep_q(G)_{\on{ren}} @>>{\oblv_{\on{big}\to \on{Lus}^+}}> \Rep_q(B). 
\endCD
$$

\medskip

We have a commutative diagram
$$
\CD
\Rep_q(G)_{\on{ren}}   @>{\oblv_{\on{big}\to \on{Lus}^+}}>> \Rep_q(B) \\
@V{\oblv_{\on{big}\to \on{sml}}}VV   @VV{\oblv_{\on{Lus}^+\to \on{sml}^+}}V  \\
\Rep^{\on{sml,grd}}_q(G)_{\on{ren}} @>{\oblv_{\on{sm}\to \on{sm}^+}}>> \Rep^{\on{sml,grd}}_q(B),
\endCD
$$
which gives rise to a commutative diagram
\begin{equation} \label{e:com diag B}
\CD
\Rep(T_H)\underset{\Rep(H)}\otimes \Rep_q(G)_{\on{ren}}   @>>>  \Rep(T_H)\underset{\Rep(B_H)}\otimes\Rep_q(B) \\
@VVV   @VVV  \\
\Rep^{\on{sml,grd}}_q(G)_{\on{ren}} @>{\oblv_{\on{sm}\to \on{sm}^+}}>> \Rep^{\on{sml,grd}}_q(B),
\endCD
\end{equation} 
in which the vertical arrows are equivalences, by Theorems \ref{t:AG grd} and \ref{t:AG for B grd}. 

\section{The ``$\frac{1}{2}$" version of the quantum group}  \label{s:1/2}

In this section we continue to assume that $q$ takes values in roots of unity.

\medskip

We will introduce and study yet one more version of the quantum group: it is one that
has Lusztig's version as its positive part and the small quantum group as its negative part. The resulting
category of modules, denoted, $\Rep^{\frac{1}{2}}_q(G)$ will play the role of intermediary between 
$\Rep_q(G)$ and $\Rep^{\on{sml,grd}}_q(G)$. 

\ssec{Definition of the ``$\frac{1}{2}$" version }

\sssec{}

We define the categories
$$\Rep^{\frac{1}{2}}_q(G) \text{ and } \Rep^{\frac{1}{2}}_q(G)_{\on{ren}}$$
to be 
$$\on{Inv}_{B_H}(\Rep^{\on{sml}}_q(G)) \text{ and } \on{Inv}_{B_H}(\Rep^{\on{sml}}_q(G)_{\on{ren}}),$$
respectively.

\medskip

Both categories carry a t-structure, uniquely characterized by the condition
that the corresponding forgetful functors
$$\oblv_{\frac{1}{2}\to \on{sml}}:\Rep^{\frac{1}{2}}_q(G) \to \Rep^{\on{sml}}_q(G) \text{ and } 
\oblv_{\frac{1}{2}\to \on{sml}}:\Rep^{\frac{1}{2}}_q(G)_{\on{ren}} \to \Rep^{\on{sml}}_q(G)_{\on{ren}}$$
are t-exact.  We have an adjoint pair
$$\fr:\Rep^{\frac{1}{2}}_q(G) \rightleftarrows \Rep^{\frac{1}{2}}_q(G)_{\on{ren}}:\fs$$
with $\fs$ t-exact and inducing an equivalence on eventually coconnective parts. 

\sssec{}

By construction, we have:
$$\Rep^{\frac{1}{2}}_q(G)\simeq \Rep(B_H)\underset{\Rep(H)}\otimes \Rep_q(G) \text{ and }
\Rep^{\frac{1}{2}}_q(G)_{\on{ren}}\simeq \Rep(B_H)\underset{\Rep(H)}\otimes \Rep_q(G)_{\on{ren}}.$$

We have the obvious forgetful functors
$$\oblv_{\on{big}\to \frac{1}{2}}:\Rep_q(G)\to \Rep^{\frac{1}{2}}_q(G) 
\text{ and } \oblv_{\on{big}\to \frac{1}{2}}:\Rep_q(G)_{\on{ren}}\to \Rep^{\frac{1}{2}}_q(G) _{\on{ren}}$$
that are \emph{fully faithful} because the forgetful functor 
\begin{equation} \label{e:oblv H to BH}
\oblv_{H\to B_H}:\Rep(H)\to \Rep(B_H)
\end{equation} 
is fully faithful. 

\medskip

The functor $\oblv_{\on{big}\to \frac{1}{2}}$ admits a left and a right adjoints, denoted $\ind_{\frac{1}{2}\to \on{big}}$ and $\coind_{\frac{1}{2}\to \on{big}}$,
respectively. They are related by the formula
\begin{equation} \label{e:coind 1/2}
\coind_{\frac{1}{2}\to \on{big}}(-)\simeq \ind_{\frac{1}{2}\to \on{big}}(-\otimes k^{-2\rho_H})[-d],
\end{equation} 
because this is the case for the corresponding left and right adjoints of the forgetful functor $\oblv_{H\to B_H}$ of \eqref{e:oblv H to BH}
i.e.,
$$\coind_{B_H\to H}(-)\simeq \ind_{B_H\to H}(-\otimes k^{-2\rho_H})[-d]$$
as functors $\Rep(B_H)\to \Rep(H)$. 

\sssec{}

Recall that the restriction functor 
$$\oblv_{\on{big}\to \on{Lus}^+}:\Rep_q(G)_{\on{ren}}\to \Rep_q(B)$$
is compatible with the actions of $\Rep(H)$ and $\Rep(B_H)$, respectively, via the forgetful functor
$\oblv_{H\to B_H}:\Rep(H)\to \Rep(B_H)$.

\medskip

From here we obtain a functor
$$\oblv_{\frac{1}{2}\to \on{Lus}^+}: \Rep^{\frac{1}{2}}_q(G)_{\on{ren}} :=
\Rep(B_H) \underset{\Rep(H)}\otimes \Rep_q(G)_{\on{ren}}\to \Rep_q(B).$$

\medskip

From the diagram \eqref{e:com diag B}, we obtain a commutative diagram
$$
\CD
\Rep^{\frac{1}{2}}_q(G)_{\on{ren}} @>{\oblv_{\frac{1}{2}\to \on{Lus}^+}}>>  \Rep_q(B)  \\
@V{\oblv_{B_H\to T_H}}VV   @VV{\oblv_{\on{Lus}^+\to \on{sml}^+}}V   \\
\Rep_q^{\on{sml,grd}}(G)_{\on{ren}}  @>{\oblv_{\on{sml}\to \on{sml}^+}}>>  \Rep_q^{\on{sml,grd}}(B),
\endCD
$$
in which the bottom row is obtained from the top row by  $\Rep(T_H)\underset{\Rep(B_H)}\otimes -$. 

\medskip

Hence, by passing to left adjoints along the horizontal arrows, we obtain another commutative diagram
$$
\CD
\Rep^{\frac{1}{2}}_q(G)_{\on{ren}} @<{\ind_{\on{Lus}^+\to \frac{1}{2}}}<<  \Rep_q(B)  \\
@V{\oblv_{B_H\to T_H}}VV   @VV{\oblv_{\on{Lus}^+\to \on{sml}^+}}V   \\
\Rep_q^{\on{sml,grd}}(G)_{\on{ren}} @<{\ind_{\on{sml}^+\to \on{sml}}}<<  \Rep_q^{\on{sml,grd}}(B).
\endCD
$$

\sssec{}

For $\clambda\in \cLambda$,  denote
$$\BM^\clambda_{q,\frac{1}{2}}:=\ind_{\on{Lus}^+\to \frac{1}{2}}(k^\clambda)\in \Rep^{\frac{1}{2}}_q(G)_{\on{ren}}.$$

Note that the restriction of $\BM^\clambda_{q,\frac{1}{2}}$ along the functor
$$\oblv_{B_H\to T_H}:\Rep^{\frac{1}{2}}_q(G)_{\on{ren}} \to \Rep^{\on{sml,grd}}_q(G)_{\on{ren}}$$
is the standard object (a.k.a. baby Verma module) $\BM^\clambda_{q,\on{sml}}$. 

\sssec{}

We now claim:

\begin{prop}  \label{p:B preserves baby}
The action of $B_H$ on $\Rep^{\on{sml}}_q(G)_{\on{ren}}$ preserves the full subcategory 
$$\Rep^{\on{sml}}_q(G)_{\on{baby-ren}}\subset \Rep^{\on{sml}}_q(G)_{\on{ren}}.$$
\end{prop} 

\begin{proof}

To prove the proposition, it would suffice to show that the objects $\BM^\clambda_{q,\on{sml}}\in \Rep^{\on{sml}}_q(G)_{\on{ren}}$
could be lifted to $\on{Inv}_{B_H}(\Rep^{\on{sml}}_q(G)_{\on{ren}})$. However, we just saw that $\BM^\clambda_{q,\on{sml}}$
lifts to an object denoted $\BM^\clambda_{q,\frac{1}{2}}\in \Rep^{\frac{1}{2}}_q(G)_{\on{ren}}$. 

\end{proof}

\sssec{}

It follows from \propref{p:B preserves baby} that the category $\Rep^{\on{sml}}_q(G)_{\on{baby-ren}}$ acquires a $B_H$-action, and 
the functors \eqref{e:string of small r}, and hence also the functors 
\eqref{e:string of small s}, are compatible with the $B_H$-actions.

\medskip

We define
$$\Rep^{\frac{1}{2}}_q(G)_{\on{baby-ren}}:=\on{Inv}_{B_H}(\Rep^{\on{sml}}_q(G)_{\on{baby-ren}}).$$ 

We have the corresponding fully faithful embeddings
\begin{equation} \label{e:string of 1/2 r}
\Rep^{\frac{1}{2}}_q(G) \overset{\fr_{\on{baby}}}\hookrightarrow 
\Rep^{\frac{1}{2}}_q(G)_{\on{baby-ren}} \overset{\fr_{\on{baby-ren}\to \on{ren}}}\hookrightarrow  
\Rep^{\frac{1}{2}}_q(G)_{\on{ren}}
\end{equation} 
and their right adjoints
\begin{equation} \label{e:string of 1/2 s}
\Rep^{\frac{1}{2}}_q(G) \overset{\fs_{\on{baby}}}\twoheadleftarrow
\Rep^{\frac{1}{2}}_q(G)_{\on{baby-ren}} \overset{\fs_{\on{ren}\to \on{baby-ren}}}\twoheadleftarrow
\Rep^{\frac{1}{2}}_q(G).
\end{equation} 

Moreover, $\Rep^{\frac{1}{2}}_q(G)_{\on{baby-ren}}$ also carries a t-structure, so that the forgetful functor
$$\Rep^{\frac{1}{2}}_q(G)_{\on{baby-ren}}\to \Rep^{\on{sml}}_q(G)_{\on{baby-ren}}$$
is t-exact.  The functors in \eqref{e:string of 1/2 s} are t-exact and induce equivalences on 
eventually coconnective parts. 

\sssec{}

By construction, the objects $\BM^\clambda_{q,\frac{1}{2}}$ lie in the essential image of 
$\fr_{\on{baby-ren}\to \on{ren}}$. This implies that the functor $\ind_{\on{Lus}^+\to \frac{1}{2}}$ factors as
$$\Rep_q(B)\to \Rep^{\frac{1}{2}}_q(G)_{\on{baby-ren}} \overset{\fr_{\on{baby-ren}\to \on{ren}}}\hookrightarrow  
\Rep^{\frac{1}{2}}_q(G)_{\on{ren}}.$$

\medskip

By adjunction, this implies that the functor
$$\oblv_{\frac{1}{2}\to \on{Lus}^+}:\Rep^{\frac{1}{2}}_q(G)_{\on{ren}}\to \Rep_q(B)$$
factors as
$$\Rep^{\frac{1}{2}}_q(G)_{\on{ren}} \overset{\fs_{\on{ren}\to \on{baby-ren}}}\twoheadrightarrow 
\Rep^{\frac{1}{2}}_q(G)_{\on{baby-ren}} \to \Rep_q(B).$$

By a slight abuse of notation, we will denote the resulting pair of adjoint functors
$$\Rep_q(B)\rightleftarrows \Rep^{\frac{1}{2}}_q(G)_{\on{baby-ren}}$$
by the same symbols $(\ind_{\on{Lus}^+\to \frac{1}{2}},\oblv_{\frac{1}{2}\to \on{Lus}^+})$. 

\begin{rem}
The categories
$$\Rep^{\frac{1}{2}}_q(G),\,\, \Rep^{\frac{1}{2}}_q(G)_{\on{baby-ren}} \text{ and } \Rep^{\frac{1}{2}}_q(G)_{\on{ren}}$$
can be recovered from their common heart by the same procedure as in Remark \ref{r:recover small from heart}. 

\medskip

The abelian category 
$(\Rep^{\frac{1}{2}}_q(G))^\heartsuit$ can be explicitly described as objects of $(\Rep_q(T))^\heartsuit$, equipped
with a locally nilpotent action of $U_q^{\on{Lus}}(N)$ and a compatible action of $u_q(N^-) $.

\end{rem} 

\ssec{``Mixed" vs ``$\frac{1}{2}$"}

\sssec{}

Recall the co-central homomorphism 
$$U_q^{\on{Lus}}(N) \to U(\fn_H)$$
of \eqref{e:q Frob B}. 

\medskip

It induces a functor from $\Rep(B_H)$ to the $E_3$-center of 
$\Rep^{\on{mxd}}_q(G)$. In particular, we obtain a monoidal action of $\Rep(B_H)$ on 
$\Rep^{\on{mxd}}_q(G)$.

\sssec{}

The following conjecture seems to be within easy reach, and in what follows we will assume its validity: 

\begin{conj} \label{c:1/2 vs mixed}
The above action of $\Rep(B_H)$ on $\Rep^{\on{mxd}}_q(G)$ can be promoted to an action
of the monoidal category $\QCoh(\fn_H/\on{Ad}(B_H))$.  Furthermore, we have a canonical 
equivalence
$$\Rep^{\frac{1}{2}}_q(G)_{\on{baby-ren}} \simeq \Rep(B_H)\underset{\QCoh(\fn_H/\on{Ad}(B_H))}\otimes \Rep^{\on{mxd}}_q(G),$$
where the functor $\QCoh(\fn_H/\on{Ad}(B_H))\to \Rep(B_H)$ is given by evaluation at the point $0\overset{\iota}\hookrightarrow \fn_H$. 
\end{conj} 

\begin{rem}
The intuitive meaning behind the second statement in \conjref{c:1/2 vs mixed} is the short exact sequence \eqref{e:SOE -}
of Hopf algebras in $\Rep_q(T)$:
$$0\to \Sym(\fn^-_H) \to U^{\on{DK}}_q(N^-) \to u_q(N^-)\to 0.$$

Namely, the passage from $\Rep^{\on{mxd}}_q(G)$ to $\Rep^{\frac{1}{2}}_q(G)_{\on{baby-ren}}$ is obtained by imposing that the
augmentation ideal in $\Sym(\fn^-_H)$ should act by $0$. Here we identify $\fn^-_H$ with the dual vector space of $\fn_H$, so that
$$\fn_H\simeq \Spec(\Sym(\fn^-_H)).$$
\end{rem}

\sssec{}

Assuming \conjref{c:1/2 vs mixed}, we obtain that there exists a forgetful functor
$$\Rep^{\frac{1}{2}}_q(G)_{\on{baby-ren}}\to \Rep^{\on{mxd}}_q(G),$$ 
denoted 
$$\iota_*\simeq \oblv_{\frac{1}{2}\to \on{mxd}},$$
which admits a left and a right adjoints, denoted 
$$\iota^*\simeq \ind_{\on{mxd}\to \frac{1}{2}} \text{ and } \iota^!\simeq \coind_{\on{mxd}\to \frac{1}{2}},$$
respectively. 

\medskip

In addition, we have
\begin{equation} \label{e:coind mixed 1/2}
\coind_{\on{mxd}\to \frac{1}{2}}(-)\simeq \ind_{\on{mxd}\to \frac{1}{2}}(-)\otimes k^{2\rho_H}[-d],
\end{equation}
because relationship holds for the functors
$$\iota^*,\iota^!:\QCoh(\fn_H/\on{Ad}(B_H))\to \QCoh(\on{pt}/B_H).$$

\sssec{}

The statement of \conjref{c:1/2 vs mixed} should be complemented by the following two additional ones:

\medskip
 
One is that the composite functor
$$\Rep_q(G)_{\on{ren}} \overset{\oblv_{\on{big}\to \frac{1}{2}}}\longrightarrow 
\Rep^{\frac{1}{2}}_q(G)_{\on{ren}} \overset{\fs_{\on{ren}\to \on{baby-ren}}}\longrightarrow \Rep^{\frac{1}{2}}_q(G)_{\on{baby-ren}} 
\overset{\oblv_{\frac{1}{2}\to \on{mxd}}}\longrightarrow \Rep^{\on{mxd}}_q(G)$$ 
identifies with the functor $\oblv_{\on{big}\to \on{mxd}}$.

\medskip

The other is that the composite functor
$$\Rep^{\frac{1}{2}}_q(G)_{\on{baby-ren}}\overset{\oblv_{\frac{1}{2}\to \on{mxd}}}\longrightarrow 
\Rep^{\on{mxd}}_q(G) \overset{\oblv_{\on{mxd}\to \on{Lus}^+}}\longrightarrow \Rep_q(B)$$
identifies with the functor 
$\oblv_{\frac{1}{2}\to \on{Lus}^+}$. 

\medskip

These two identifications of functors must be compatible with the identification
$$\oblv_{\frac{1}{2}\to \on{Lus}^+}\circ \oblv_{\on{big}\to \frac{1}{2}}\simeq \oblv_{\on{big}\to \on{Lus}^+}.$$

\sssec{}

Note that, by adjunction, the second of the above compatibilities implies an isomorphism
$$\ind_{\on{mxd}\to \frac{1}{2}} \circ \ind_{\on{Lus}^+\to \on{mxd}}\simeq \ind_{\on{Lus}^+\to \frac{1}{2}}.$$

In particular, for $\clambda\in \cLambda$, we obtain an isomorphism
$$\ind_{\on{mxd}\to \frac{1}{2}}(\BM_{q,\on{mxd}})\simeq \BM_{q,\frac{1}{2}}.$$

\ssec{More on mixed vs $\frac{1}{2}$}

\sssec{}

We claim:

\begin{prop} \label{p:i_*good good} 
The functor $\oblv_{\frac{1}{2}\to \on{mxd}}\circ \fs_{\on{ren}\to \on{baby-ren}}$ sends compacts to compacts.
\end{prop} 

\begin{proof}

The category $\Rep^{\frac{1}{2}}_q(G)_{\on{ren}}$ is generated by objects of the form
$$k^\mu\otimes \oblv_{\on{big}\to \frac{1}{2}}(\CM), \quad \CM\in \Rep_q(G)_{\on{ren}}, \,\, \mu\in \Lambda.$$

So it suffices to show that the functor $\oblv_{\frac{1}{2}\to \on{mxd}}\circ \fs_{\on{ren}\to \on{baby-ren}}$ 
sends such objects to compact objects in $\Rep^{\on{mxd}}_q(G)$. 

\medskip

Now, since the operation $k^\mu \otimes -$ on $\Rep^{\on{mxd}}_q(G)$ preserves compacts, the assertion 
follows from \propref{p:oblv big to mixed comp}.

\end{proof} 

\begin{rem} 
Note that the above proposition implies the following relationship between the categories 
$\Rep^{\frac{1}{2}}_q(G)_{\on{ren}}$ and $\Rep^{\on{mxd}}_q(G)$:

\medskip

The functor
$$\Rep(B_H)\underset{\QCoh(\fn_H/\on{Ad}(B_H))}\otimes \Rep^{\on{mxd}}_q(G)\to \Rep^{\frac{1}{2}}_q(G)_{\on{ren}}$$
is fully faithful, but not an equivalence; rather it is an equivalence onto the full subcategory 
$\Rep^{\frac{1}{2}}_q(G)_{\on{baby-ren}}\subset \Rep^{\frac{1}{2}}_q(G)_{\on{ren}}$. 

\medskip

Instead, $\Rep^{\frac{1}{2}}_q(G)_{\on{ren}}$ can be obtained from $\Rep(B_H)\underset{\QCoh(\fn_H/\on{Ad}(B_H))}\otimes \Rep^{\on{mxd}}_q(G)$
as follows: it is the ind-completion of the full (but not cocomplete) subcategory of
$\Rep(B_H)\underset{\QCoh(\fn_H/\on{Ad}(B_H))}\otimes \Rep^{\on{mxd}}_q(G)$ spanned by objects
that become compact after applying the forgetful functor 
$$\iota_*:\Rep(B_H)\underset{\QCoh(\fn_H/\on{Ad}(B_H))}\otimes \Rep^{\on{mxd}}_q(G)\to  \Rep^{\on{mxd}}_q(G).$$ 

This is analogous to \cite[Theorem 11.4.2]{FG3}.

\end{rem}

\begin{rem} 

To summarize, we have the following diagram of stacks
$$
\CD
\on{pt}/T_H  @>>>  \on{pt}/B_H @>{\iota}>>  \fn_H/\on{Ad}(B_H)   \\
& & @VVV  \\
& & \on{pt}/H. 
\endCD
$$

And we have the following diagram of categories over these stacks 
$$
\CD
\Rep^{\on{sml,grd}}_q(G)_{\on{ren}} @<<< \Rep^{\frac{1}{2}}_q(G)_{\on{ren}}  @<{\fr_{\on{baby-ren}\to \on{ren}}}<< \Rep^{\frac{1}{2}}_q(G)_{\on{baby-ren}} 
@<{\ind_{\on{mxd}\to \frac{1}{2}}}<<  \Rep^{\on{mxd}}_q(G)  \\
& & @AAA  \\
& & \Rep_q(G)_{\on{ren}}
\endCD
$$

This diagram of categories is \emph{almost} compatible with pullbacks along the above diagram of stacks: the only wrinkle is the functor 
$\fr_{\on{baby-ren}\to \on{ren}}$, which is fully faithful but not an equivalence. 

\end{rem} 

\sssec{}

From \propref{p:i_*good good} we obtain: 

\begin{cor} \label{c:i_*good good} 
The functors 
$$\oblv_{\frac{1}{2}\to \on{mxd}}\circ \fs_{\on{ren}\to \on{baby-ren}}:
 \Rep^{\frac{1}{2}}_q(G)_{\on{ren}}\rightleftarrows  \Rep^{\on{mxd}}_q(G): \fr_{\on{baby-ren}\to \on{ren}}\circ \coind_{\on{mxd}\to \frac{1}{2}}$$
form an adjoint pair.
\end{cor}

\begin{proof}
Given $\CM_1\in \Rep^{\frac{1}{2}}_q(G)_{\on{ren}}$ and $\CM_2\in  \Rep^{\on{mxd}}_q(G)$ we need to construct a canonical 
isomorphism
\begin{multline*} 
\CHom_{\Rep^{\on{mxd}}_q(G)}(\oblv_{\frac{1}{2}\to \on{mxd}}\circ \fs_{\on{ren}\to \on{baby-ren}}(\CM_1),\CM_2)\simeq \\
\simeq \CHom_{\Rep^{\frac{1}{2}}_q(G)_{\on{ren}}}(\CM_1,\fr_{\on{baby-ren}\to \on{ren}}\circ \coind_{\on{mxd}\to \frac{1}{2}}(\CM_2)).
\end{multline*}

With no restriction of generality, we can assume that $\CM_1$ is compact. However, 
\propref{p:i_*good good} implies that we can assume that $\CM_2$ is compact as well. Note that since the functor $\coind_{\on{mxd}\to \frac{1}{2}}$
differs from $\ind_{\on{mxd}\to \frac{1}{2}}$ by a twist, we obtain that $\coind_{\on{mxd}\to \frac{1}{2}}(\CM_2)$ is then also compact.

\medskip

Since the functor $\fs_{\on{ren}\to \on{baby-ren}}$
is fully faithful on compact objects, we have
\begin{multline*} 
\CHom_{\Rep^{\frac{1}{2}}_q(G)_{\on{ren}}}(\CM_1,\fr_{\on{baby-ren}\to \on{ren}}\circ \coind_{\on{mxd}\to \frac{1}{2}}(\CM_2))\simeq \\ \simeq 
\CHom_{\Rep^{\frac{1}{2}}_q(G)_{\on{baby-ren}}}(\fs_{\on{ren}\to \on{baby-ren}}(\CM_1),\fs_{\on{ren}\to \on{baby-ren}}\circ
\fr_{\on{baby-ren}\to \on{ren}}\circ \coind_{\on{mxd}\to \frac{1}{2}}(\CM_2)) \simeq \\
\simeq 
\CHom_{\Rep^{\frac{1}{2}}_q(G)_{\on{baby-ren}}}(\fs_{\on{ren}\to \on{baby-ren}}(\CM_1),\coind_{\on{mxd}\to \frac{1}{2}}(\CM_2))
\end{multline*}
and the assertion follows from the $(\oblv_{\frac{1}{2}\to \on{mxd}},\coind_{\on{mxd}\to \frac{1}{2}})$-adjunction. 

\end{proof} 

From here we obtain:

\begin{cor} \label{c:adj to res} 
The left and right adjoints to $\oblv_{\on{big}\to \on{mxd}}$ are related by the
formula
$$\coind_{\on{mxd}\to \on{big}}\simeq \ind_{\on{mxd}\to \on{big}}[-2d].$$
\end{cor}

\begin{proof}

We have:
$$\ind_{\on{mxd}\to \on{big}}\simeq \ind_{\frac{1}{2}\to \on{big}}\circ \fr_{\on{baby-ren}\to \on{ren}}\circ \ind_{\on{mxd}\to \frac{1}{2}}$$
and by \corref{c:i_*good good}, we have
$$\coind_{\on{mxd}\to \on{big}}\simeq \coind_{\frac{1}{2}\to \on{big}}\circ \fr_{\on{baby-ren}\to \on{ren}}\circ \coind_{\on{mxd}\to \frac{1}{2}}$$

The assertion follows now from \eqref{e:coind mixed 1/2} and \eqref{e:coind 1/2}.

\end{proof} 

\begin{rem}

The assertion of \corref{c:adj to res} holds as-is in the non-root of unity case. This follows from the fact that
in this case the pair of categories
$$\Rep_q(G)\simeq \Rep_q(G)_{\on{ren}} \overset{\oblv_{\on{big}\to \on{mxd}}}\longrightarrow \Rep^{\on{mxd}}_q(G)$$
is equivalent to 
$$\Rep(G)\to \fg\mod^B.$$

\end{rem} 

\ssec{Some consequences of the Frobenius algebra property}

\sssec{}

A remarkable property of $u_q(N^-)$ is that it is a \emph{Frobenius algebra} (this is the case of any finite-dimensional
Hopf algebra, see, e.g., \cite[Corollary 6.4.5]{Et}). We will use this in the following guise:

\begin{lem}  \label{l:Frob ppty}
The object of $u_q(N^-)\mod(\Rep_q(T))$ given by $u_q(N^-)$ itself is cofree:
$$\CHom_{u_q(N^-)\mod(\Rep_q(T))}(\CM,u_q(N^-))\simeq \CHom_{\Rep_q(T)}(\CM,k^{2\check\rho-2\rho_H}).$$ 
\end{lem}

\sssec{}

From \lemref{l:Frob ppty} we will deduce:

\begin{cor}  \label{c:cohomology of UqDK}
For $\clambda\in \cLambda$, we have 
$$\CHom_{U_q^{\on{DK}}(N^-)\mod(\Rep_q(T))}(k^\clambda,U_q^{\on{DK}}(N^-))=
\begin{cases}
&k[-d] \text{ for } \clambda=2\check\rho; \\
&0 \text{ otherwise}.
\end{cases}
$$ 
\end{cor}

\begin{rem}  \label{r:cohomology of UqDK} 
One can show that the assertion of \corref{c:cohomology of UqDK} holds verbatim
also in the non-root of unity case.
\end{rem}

\begin{proof}

For $\CM\in U_q^{\on{DK}}(N^-)\mod(\Rep_q(T))$, we calculate 
$$\CHom_{U_q^{\on{DK}}(N^-)\mod(\Rep_q(T))}(k^\clambda,\CM)\simeq
\CHom_{u_q(N^-)\mod(\Rep_q(T))}(k^\clambda,\CHom_{\Sym(\fn^-_H)}(k,\CM)).$$ 

We note
$$\CHom_{\Sym(\fn^-_H)}(k,\CM)\simeq (k\underset{\Sym(\fn^-_H)}\otimes \CM) \otimes k^{2\rho_H}[-d].$$

We take $\CM=U_q^{\on{DK}}(N^-)$, and we note that
$$k\underset{\Sym(\fn^-_H)}\otimes U_q^{\on{DK}}(N^-)\simeq u_q(N^-).$$

Now the assertion follows from \lemref{l:Frob ppty}. 

\end{proof} 

\begin{rem}
In fact, \lemref{l:Frob ppty} implies that the two functors 
$$u_q(N^-)\mod(\Rep_q(T))\rightrightarrows \Rep_q(T),$$
given by
$$\CM\mapsto k\underset{u_q(N^-)}\otimes \CM \text{ and } \CHom_{u_q(N^-)}(k,\CM)\otimes k^{2\check\rho-2\rho_H}$$
are canonically isomorphic when evaluated on $u_q(N^-)\mod(\Rep_q(T))_{\on{perf}}$.

\medskip

This implies that the two functors
$$U_q^{\on{DK}}(N^-)\mod(\Rep_q(T))\rightrightarrows \Rep_q(T),$$
given by
\begin{equation} \label{e:homology vs cohomology}
\CM\mapsto k\underset{U^{\on{DK}}_q(N^-)}\otimes \CM \text{ and } \CHom_{U^{\on{DK}}_q(N^-)}(k,\CM)\otimes k^{2\check\rho}[-d]
\end{equation} 
are canonically isomorphic when evaluated on $U_q^{\on{DK}}(N^-)\mod(\Rep_q(T))_{\on{perf}}$. 

\medskip

Note, however that \corref{c:UqDK} implies that the
isomorphism between the functors \eqref{e:homology vs cohomology} holds on all of $U_q^{\on{DK}}(N^-)\mod(\Rep_q(T))$. 

\end{rem} 

\sssec{}

Another corollary of \lemref{l:Frob ppty} that we will use is the following:

\begin{cor} \label{c:ind and coind from Borel to small}
The functor 
$$\oblv_{\on{sml}\to \on{sml}^+}:\Rep^{\on{sml,grd}}_q(G)\mod_{\on{ren}}\to \Rep^{\on{sml,grd}}_q(B)\mod$$
admits a \emph{right} adjoint, denoted $\coind_{\on{sml}^+\to \on{sml}}$, and for $\clambda\in \cLambda$ we have:
$$\coind_{\on{sml}^+\to \on{sml}}(k^\clambda)\simeq \ind_{\on{sml}^+\to \on{sml}}(k^{\clambda+2\rho_H-2\check\rho}).$$
\end{cor} 

\begin{proof}

We need to show that $\coind_{\on{sml}^+\to \on{sml}}(k^\clambda)$ is free over $u_q(N^-)$ on one generator of
weight equal to $\clambda+2\rho_H-2\check\rho$. By construction, $\coind_{\on{sml}^+\to \on{sml}}(k^\clambda)$ is cofree 
over $u_q(N^-)$ on one generator of weight $\clambda$. Now the assertion follows from \lemref{l:Frob ppty}.

\end{proof} 

\medskip

The same argument proves also the \emph{ungraded} version of \corref{c:ind and coind from Borel to small}, as well
as the following statement: 

\begin{cor}  \label{c:ind and coind from Borel to small 1/2}
The functor
$$\oblv_{\frac{1}{2}\to \on{Lus}^+}: \Rep^{\frac{1}{2}}_q(G)_{\on{ren}} \to \Rep_q(B)$$
admits a \emph{right} adjoint, denoted $\coind_{\on{Lus}^+\to \frac{1}{2}}$, and for $\clambda\in \cLambda$ we have:
$$\coind_{\on{Lus}^+\to \frac{1}{2}}(k^\clambda)\simeq \ind_{\on{Lus}^+\to \frac{1}{2}}(k^{\clambda+2\rho_H-2\check\rho}).$$
\end{cor} 

\sssec{}

One can use \corref{c:ind and coind from Borel to small 1/2} to reproduce the result of \cite[Theorem 7.3]{APW}:

\begin{cor}  \label{c:big ind an coind} 
The objects
$$\CV^\clambda_q:=\ind_{\on{Lus}^+\to \on{big}}(k^\clambda) \text{ and } \CV^{\vee,w_0(\clambda)}_q:=\coind_{\on{Lus}^+\to \on{big}}(k^\clambda)$$
are related by the formula
$$\coind_{\on{Lus}^+\to \on{big}}(k^\clambda)\simeq \ind_{\on{Lus}^+\to \on{big}}(k^{\clambda-2\check\rho})[-d].$$
\end{cor}

\begin{proof}

We have:
\begin{multline*} 
\coind_{\on{Lus}^+\to \on{big}}(k^\clambda)\simeq 
\coind_{\frac{1}{2}\to \on{big}}\circ \coind_{\on{Lus}^+\to \frac{1}{2}}(k^\clambda)\simeq \\
\simeq 
\coind_{\frac{1}{2}\to \on{big}}\circ \ind_{\on{Lus}^+\to \frac{1}{2}}(k^{\clambda+2\rho_H-2\check\rho})\simeq
\ind_{\frac{1}{2}\to \on{big}}\circ \ind_{\on{Lus}^+\to \frac{1}{2}}(k^{\clambda-2\check\rho})[-d].$$
\end{multline*}

\end{proof} 

\begin{rem}
The assertion of \corref{c:big ind an coind} is also valid in the non-root of unity case.
\end{rem} 

\section{Duality for quantum groups}  \label{s:q duality}

In this section we will show that the effect of replacing $q$ by $q^{-1}$ is \emph{duality for DG categories} 
for the various versions of the category of modules over the quantum group.

\ssec{Framework for duality}

\sssec{}  \label{sss:duality from mon}

Let $\CC$ be a monoidal DG category, which is compactly generated, and such that the following conditions are satisfied:

\begin{itemize}

\item The monoidal operation preserves compactness;

\item Every compact object admits a (right) monoidal dual.

\end{itemize}

Let $\CC^{\on{rev}}$ denote the monoidal category obtained by reversing the monoidal operation. Then the functor
\begin{equation} \label{e:monoidal pairing}
\CC\otimes \CC^{\on{rev}}\to \Vect
\end{equation}
given by ind-extending 
$$c_1,c_2 \mapsto \CHom_\CC({\bf 1}_\CC,c_1\otimes c_2), \quad c_1\in \CC_c,\,c_2\in \CC^{\on{rev}}_c$$
defines a perfect pairing and in itself is a right-lax monoidal functor. 

\medskip

In particular, we obtain an identification of DG categories
$$\CC^\vee\simeq \CC^{\on{rev}}.$$

\sssec{}

Assume now that $\CC$ is \emph{braided monoidal}. We define a braiding on $\CC^{\on{rev}}$ by letting the map
$$c_1\overset{\on{rev}}\otimes c_2\to c_2\overset{\on{rev}}\otimes c_1$$
be the map
$$c_2\otimes c_1 \overset{R^{-1}_{c_1,c_2}}\longrightarrow c_1\otimes c_2.$$

Then the above functor \eqref{e:monoidal pairing} has a natural right-lax braided monoidal structure. 

\sssec{}  \label{sss:Hopf rev}

Let $\CC$ be braided monoidal, and let $A$ be a Hopf algebra in $\CC$. Let $A^{\on{rev-mult}}$ be the
Hopf algebra in $\CC^{\on{rev}}$ defined as follows:

\medskip

As an object, $A^{\on{rev-mult}}$ is the same as $A$. The co-multiplication map
$$A^{\on{rev-mult}}\to A^{\on{rev-mult}}\overset{\on{rev}}\otimes A^{\on{rev-mult}}$$
is the initial comultiplication map
$$A\to A\otimes A.$$

The multiplication map
$$A^{\on{rev-mult}}\overset{\on{rev}}\otimes A^{\on{rev-mult}}\to A^{\on{rev-mult}}$$
is set to be the map
$$A\otimes A \overset{R^{-1}_{A,A}}\longrightarrow A\otimes A\to A.$$

\medskip

In a similar way we define the Hopf algebra $A^{\on{rev-comult}}$ in $\CC^{\on{rev}}$. We note, however, that the antipode
defines an isomorphism of Hopf algebras 
$$A^{\on{rev-mult}}\simeq A^{\on{rev-comult}}.$$

\begin{rem}
Recall that in \secref{sss:rev algebras} the notation $A^{\on{rev-mult}}$ (resp., $A^{\on{rev-comult}}$) had a different meaning:
in {\it loc.cit.} they denoted Hopf algebras in $\CC^{\on{rev-br}}$. 

\medskip

The two notions are related as follows: we have a canonical defined functor
\begin{equation} \label{e:rev vs br}
\CC^{\on{rev}}\to \CC^{\on{rev-br}}
\end{equation}
that \emph{reverses} the product but preserves the braiding. Hence this functor maps Hopf algebras to Hopf algebras.

\medskip

Now, the functor \eqref{e:rev vs br} sends $A^{\on{rev-mult}}$ (resp., $A^{\on{rev-comult}}$) in $\CC^{\on{rev}}$ to
$A^{\on{rev-mult}}$ (resp., $A^{\on{rev-comult}}$) in $\CC^{\on{rev-br}}$.

\end{rem}

\sssec{}

Let us be in the setting of \secref{sss:Hopf rev}. We have a natural monoidal equivalence 
\begin{equation} \label{e:A-mod rev}
(A\mod(\CC))^{\on{rev}}\simeq A^{\on{rev-mult}}\mod(\CC^{\on{rev}}).
\end{equation}

\medskip

It sends an $A$-module $M$ in $\CC$ to an $A^{\on{rev-mult}}$-module in $\CC^{\on{rev}}$, whose underlying 
object of $\CC^{\on{rev}}$ is the same $M$, and where the action map
$$A^{\on{rev-mult}} \overset{\on{rev}}\otimes M\to M$$
is set to be
$$M\otimes A \overset{R^{-1}_{A,M}}\longrightarrow A\otimes M\to M.$$

\sssec{}

The equivalence \eqref{e:A-mod rev} induces a \emph{braided monoidal} equivalence
\begin{equation} \label{e:opposite centers}
\left(Z_{\on{Dr},\CC}(A\mod(\CC))\right)^{\on{rev}}\simeq Z_{\on{Dr},\CC^{\on{rev}}}(A^{\on{rev-mult}}\mod(\CC^{\on{rev}})).
\end{equation} 

At the level of the underlying DG categories, we obtain an equivalence
$$Z_{\on{Dr},\CC}(A\mod(\CC))\simeq Z_{\on{Dr},\CC^{\on{rev}}}(A^{\on{rev-mult}}\mod(\CC^{\on{rev}})),$$
which makes the following diagrams commute:
\begin{equation} \label{e:duality1}
\CD 
Z_{\on{Dr},\CC}(A\mod(\CC))   @>{\sim}>>  Z_{\on{Dr},\CC^{\on{rev}}}(A^{\on{rev-mult}}\mod(\CC^{\on{rev}})) \\
@VVV    @VVV   \\
A\mod(\CC) @>{\sim}>> A^{\on{rev-mult}}\mod(\CC^{\on{rev}})
\endCD
\end{equation} 
and
\begin{equation} \label{e:duality2}
\CD 
Z_{\on{Dr},\CC}(A\mod(\CC))   @>{\sim}>>  Z_{\on{Dr},\CC^{\on{rev}}}(A^{\on{rev-mult}}\mod(\CC^{\on{rev}})) \\
@VVV    @VVV   \\
A\comod(\CC) @>{\sim}>> A^{\on{rev-comult}}\comod(\CC^{\on{rev}})\simeq A^{\on{rev-mult}}\comod(\CC^{\on{rev}}). 
\endCD
\end{equation} 

\medskip

In addition, the diagrams
$$
\CD 
Z_{\on{Dr},\CC}(A\mod(\CC))   @>{\sim}>>  Z_{\on{Dr},\CC}(A^{\on{rev-mult}}\mod(\CC^{\on{rev}})) \\
@AAA    @AAA   \\
A\mod(\CC) @>{\sim}>> A^{\on{rev-mult}}\mod(\CC^{\on{rev}}),
\endCD
$$
(obtained from \eqref{e:duality1} by passing to \emph{right} adjoints along the vertical arrows) and 
$$
\CD 
Z_{\on{Dr},\CC}(A\mod(\CC))   @>{\sim}>>  Z_{\on{Dr},\CC}(A^{\on{rev-mult}}\mod(\CC^{\on{rev}})) \\
@AAA    @AAA   \\
A\comod(\CC) @>{\sim}>> A^{\on{rev-comult}}\comod(\CC^{\on{rev}})\simeq A^{\on{rev-mult}}\comod(\CC^{\on{rev}}).
\endCD
$$
(obtained from \eqref{e:duality2} by passing to \emph{left} adjoints along the vertical arrows) also commute.

\ssec{The case of quantum groups}

\sssec{}  \label{sss:duality big}

We start by discussing duality for the big quantum group. We note that the braided monoidal abelian categories 
$(\Rep_q(G))^\heartsuit$ and $(\Rep_{q^{-1}}(G))^\heartsuit$ are related by
$$((\Rep_q(G))^\heartsuit)^{\on{rev}}\simeq (\Rep_{q^{-1}}(G))^\heartsuit, \quad \CM\mapsto \CM^\sigma,$$
induced by the canonical algebra isomorphism
$$\sigma:U^{\on{Lus}}_q(G)\to U^{\on{Lus}}_{q^{-1}}(G),$$
which reverses the comultiplication.

\medskip

The above equivalence induces an equivalence 
\begin{equation} \label{e:big invert q}
(\Rep_q(G)_{\on{ren}})^{\on{rev}}\simeq \Rep_{q^{-1}}(G)_{\on{ren}}.
\end{equation}

\medskip

By \secref{sss:duality from mon}, the equivalence \eqref{e:big invert q} induces an identification 
\begin{equation} \label{e:big duality}
(\Rep_q(G)_{\on{ren}})^\vee\simeq  \Rep_{q^{-1}}(G)_{\on{ren}},
\end{equation}
with the pairing
$$\Rep_q(G)_{\on{ren}}\otimes \Rep_{q^{-1}}(G)_{\on{ren}}\to \Vect$$
given by
$$\CM_1,\CM_2\mapsto \CHom_{\Rep_q(G)_{\on{ren}}}(k,\CM_1\otimes \CM_2^\sigma).$$

\medskip

The corresponding contravariant functor on compact objects
$$\BD:(\Rep_{q^{-1}}(G)_{\on{fin.dim}})\to (\Rep_q(G)_{\on{ren}})_{\on{fin.dim}}$$
is 
\begin{equation} \label{e:big dualization}
\CM\mapsto (\CM^\sigma)^\vee,
\end{equation} 
where $(-)^\vee$ is monoidal dualization. 

\begin{rem}
The same discussion applies to $\Rep_q(G)$. Here we use the fact that contragredient duality on 
$(\Rep_q(G)_{\on{fin.dim}})^\heartsuit$ sends projective to projectives; in other words, in the abelian category 
$(\Rep_q(G)_{\on{fin.dim}})^\heartsuit$, the classes of projective and injective objects coincide (this is the case
for any monoidal abelian category which is rigid). 
\end{rem}

\sssec{}  \label{sss:coh A}

We now consider the braided monoidal category $\Rep_q(T)$. Note that the corresponding category $(\Rep_q(T))^{\on{rev}}$
identifies with $\Rep_{q^{-1}}(T)$. 

\medskip

If $A$ is a Hopf algebra as in \secref{sss:Hopf algebras}, we obtain a canonical identification
%\begin{multline} \label{e:Z rev}
$$\left(Z_{\on{Dr},\Rep_q(T)}(A\mod(\Rep_q(T))_{\on{loc.nilp}}\right)^{\on{rev}} \simeq
Z_{\on{Dr},\Rep_{q^{-1}}(T)}(A^{\on{rev-mult}}\mod(\Rep_{q^{-1}}(T))_{\on{loc.nilp}}),$$
denoted $\CM\mapsto \CM^\sigma$, 
%\end{multline}
and in particular, an identification
$$\left(Z_{\on{Dr},\Rep_q(T)}(A\mod(\Rep_q(T))_{\on{loc.nilp}})\right)^\vee \simeq
Z_{\on{Dr},\Rep_{q^{-1}}(T)}(A^{\on{rev-mult}}\mod(\Rep_{q^{-1}}(T))_{\on{loc.nilp}}),$$
with the pairing 
$$Z_{\on{Dr},\Rep_q(T)}(A\mod(\Rep_q(T))_{\on{loc.nilp}})\otimes Z_{\on{Dr},\Rep_{q^{-1}}(T)}(A^{\on{rev-mult}}\mod(\Rep_{q^{-1}}(T))_{\on{loc.nilp}})\to
\Vect$$
given by ind-extending
\begin{multline*}
\CM_1,\CM_2\mapsto \CHom_{Z_{\on{Dr},\Rep_q(T)}(A\mod(\Rep_q(T))_{\on{loc.nilp}})}(k,\CM_1\otimes \CM_2^\sigma), \\
\CM_1\in Z_{\on{Dr},\Rep_q(T)}(A\mod(\Rep_q(T))_{\on{loc.nilp},c}),\\
\CM_2\in Z_{\on{Dr},\Rep_{q^{-1}}(T)}(A^{\on{rev-mult}}\mod(\Rep_{q^{-1}}(T))_{\on{loc.nilp},c}).
\end{multline*}

\sssec{}

Let us take $A=u_q(N)$. Note that by construction, we can identify
$$\left(u_q(N)\right)^{\on{rev-mult}}\simeq u_{q^{-1}}(N)$$
as Hopf algebras in $\Rep_{q^{-1}}(T)$.

\medskip

From here we obtain an equivalence 
$$(\Rep_q^{\on{sml,grd}}(G)_{\on{baby-ren}})^{\on{rev}}\simeq \Rep_{q^{-1}}^{\on{sml,grd}}(G)_{\on{baby-ren}}, \quad \CM\mapsto \CM^\sigma$$

This equivalence induces an equivalence 
$$(\Rep_q^{\on{sml,grd}}(G))^{\on{rev}}\simeq \Rep_{q^{-1}}^{\on{sml,grd}}(G),$$
and further, by restriction an ind-extension, an equivalence
$$(\Rep_q^{\on{sml,grd}}(G)_{\on{ren}})^{\on{rev}}\simeq \Rep_{q^{-1}}^{\on{sml,grd}}(G)_{\on{ren}}.$$

\sssec{}

Thus, we obtain an equivalence
\begin{equation} \label{e:duality small}
(\Rep_q^{\on{sml,grd}}(G)_{\on{ren}})^\vee\simeq \Rep_{q^{-1}}^{\on{sml,grd}}(G)_{\on{ren}},
\end{equation}
with the pairing 
\begin{equation} \label{e:pairing sml}
\Rep^{\on{sml,grd}}_q(G)_{\on{ren}}\otimes \Rep^{\on{sml,grd}}_{q^{-1}}(G)_{\on{ren}}\to \Vect
\end{equation}
given by 
\begin{equation} \label{e:pairing small}
\CM_1,\CM_2\mapsto \CHom_{\Rep^{\on{sml,grd}}_q(G)_{\on{ren}}}(k,\CM_1\otimes \CM^\sigma_2),\quad 
\CM_1\in \Rep^{\on{sml.grd}}_q(G)_{\on{fin.dim}},\, \CM_2\in \Rep^{\on{sml,grd}}_{q^{-1}}(G)_{\on{fin.dim}}
\end{equation}

Note, however, that formula \eqref{e:pairing small} with \emph{any} $\CM_1\in \Rep^{\on{sml,grd}}_q(G)$ 
and $\CM_2\in \Rep^{\on{sml,grd}}_{q^{-1}}(G)$ defines the pairing \eqref{e:pairing mixed pre} because
$1\in \Rep^{\on{sml,grd}}_q(G)$ is compact. 

\sssec{}

The corresponding contravariant functor
\begin{equation} \label{e:dualization small}
\BD:\Rep^{\on{sml,grd}}_{q^{-1}}(G)_{\on{fin.dim}}\to \Rep^{\on{sml.grd}}_q(G)_{\on{fin.dim}}
\end{equation}
is 
$$\CM\mapsto (\CM^\sigma)^\vee,$$
where $(-)^\vee$ is monoidal dualization. 

\sssec{}

The duality \eqref{e:duality small} induces the dualities
$$(\Rep_q^{\on{sml,grd}}(G)_{\on{baby-ren}})^\vee\simeq (\Rep_{q^{-1}}^{\on{sml,grd}}(G)_{\on{baby-ren}})^{\on{rev}}
\text{ and } (\Rep_q^{\on{sml,grd}}(G))^\vee\simeq (\Rep_{q^{-1}}^{\on{sml,grd}}(G))^{\on{rev}},$$
with the corresponding contravariant functors
$$\BD:\Rep^{\on{sml,grd}}_{q^{-1}}(G)_{\on{baby-ren},c}\to \Rep^{\on{sml.grd}}_q(G)_{\on{baby-ren},c} \text{ and } 
\BD:\Rep^{\on{sml,grd}}_{q^{-1}}(G)_{c}\to \Rep^{\on{sml.grd}}_q(G)_{c}$$
obtained by restriction.

\ssec{Cohomological duality for the mixed category}

\sssec{}

In the context of \secref{sss:coh A}, let us take $A=U^{\on{Lus}}_q(N)$. Note that by construction, we can identify
$$\left(U^{\on{Lus}}_q(N)\right)^{\on{rev-mult}}\simeq U^{\on{Lus}}_{q^{-1}}(N)$$
as Hopf algebras in $\Rep_{q^{-1}}(T)$.

\medskip

From here we obtain an equivalence
$$\Rep^{\on{mxd}}_q(G))^{\on{rev}}\simeq \Rep^{\on{mxd}}_{q^{-1}}(G),\quad \CM\mapsto \CM^\sigma,$$
which induces an equivalence 
\begin{equation} \label{e:duality mixed}
(\Rep^{\on{mxd}}_q(G))^\vee \simeq \Rep^{\on{mxd}}_{q^{-1}}(G),
\end{equation}
with the pairing
\begin{equation} \label{e:pairing mixed pre}
\Rep^{\on{mxd}}_q(G)\otimes \Rep^{\on{mxd}}_{q^{-1}}(G)\to \Vect
\end{equation}
given by ind-extending 
\begin{equation} \label{e:pairing mixed}
\CM_1,\CM_2\mapsto \CHom_{\Rep^{\on{mxd}}_q(G)}(k,\CM_1\otimes \CM^\sigma_2),\quad 
\CM_1\in \Rep^{\on{mxd}}_q(G)_c,\, \CM_2\in \Rep^{\on{mxd}}_{q^{-1}}(G)_c.
\end{equation}

Note, however, that formula \eqref{e:pairing mixed} with \emph{any} $\CM_1\in \Rep^{\on{mxd}}_q(G)$ 
and $\CM_2\in \Rep^{\on{mxd}}_{q^{-1}}(G)$ defines the pairing \eqref{e:pairing mixed pre} because
$1\in \Rep^{\on{mxd}}_q(G)$ is compact. 

\medskip

Let $\BD^{\on{can}}$ denote corresponding contravariant dualization functor
$$\Rep^{\on{mxd}}_{q^{-1}}(G)_c\to \Rep^{\on{mxd}}_q(G)_c$$

\sssec{}

Note that the resulting identifications
$$(\Rep_q(G)_{\on{ren}})^\vee\simeq \Rep_{q^{-1}}(G)_{\on{ren}}$$
and
$$(\Rep^{\on{mxd}}_q(G))^\vee\simeq \Rep^{\on{mxd}}_{q^{-1}}(G)$$
are compatible in the following way:

\begin{prop}  \label{p:duality big and mixed}
The following diagram, in which the vertical arrows are contravariant functors, commutes:
$$
\CD
\Rep_{q^{-1}}(G)_{\on{fin.dim}}  @>{\oblv_{\on{big}\to \on{mxd}}}>>  (\Rep^{\on{mxd}}_{q^{-1}}(G))_{c} \\
@V{\BD}VV   @VV{\BD^{\on{can}}}V  \\
(\Rep_q(G)_{\on{ren}})_{\on{fin.dim}}  @>{\oblv_{\on{big}\to \on{mxd}}}>>  (\Rep^{\on{mxd}}_q(G))_{c}
\endCD
$$
commutes.
\end{prop} 

\begin{proof} 

We need to show that for $\CM_1\in \Rep^{\on{mxd}}_q(G)_{c}$ and 
$\CM_2\in \Rep_{q^{-1}}(G)_{\on{fin.dim}}$, we have a canonical isomorphism
$$\CHom_{\Rep^{\on{mxd}}_q(G)}(\oblv_{\on{big}\to \on{mxd}}(\BD(\CM_2)),\CM_1)\simeq 
\CHom_{\Rep^{\on{mxd}}_q(G)}(k,\CM_1\otimes (\oblv_{\on{big}\to \on{mxd}}(\CM_2))^\sigma).$$

This follows from the commutativity of the next diagram, which in turn follows from the construction
$$
\CD
\Rep_{q^{-1}}(G)_{\on{fin.dim}}  @>{\oblv_{\on{big}\to \on{mxd}}}>>  \Rep^{\on{mxd}}_{q^{-1}}(G) \\
@V{\sigma}VV   @VV{\sigma}V  \\
(\Rep_q(G)_{\on{fin.dim}})^{\on{rev}}  @>{\oblv_{\on{big}\to \on{mxd}}}>>  (\Rep^{\on{mxd}}_q(G)))^{\on{rev}}.
\endCD
$$

\end{proof}

\sssec{}

The goal of this section is to prove the following result:

\begin{thm} \label{t:quantum duality}  
For $\clambda\in \cLambda$, we have
$$\BD^{\on{can}}(\BM^\clambda_{q^{-1},\on{mxd}})\simeq \BM^{-\clambda-2\check\rho}_{q,\on{mxd}}[d].$$
\end{thm}

The proof of \thmref{t:quantum duality} will use a tool which will be of independent interest: the long intertwining 
functor. 

\begin{rem}
The definition of the pairing \eqref{e:pairing mixed} and \thmref{t:quantum duality} are direct quantum analogs
of the corresponding assertions in the classical situation: we have the self-duality of $\fg\mod^B$ defined 
by the formula
$$\CM_1,\CM_2\mapsto \CHom_{\fg\mod^B}(k,\CM_1\otimes \CM_2),$$
and the corresponding contravraiant dualization functor 
$$\BD^{\on{can}}:\fg\mod^B_c\to \fg\mod^B_c$$ 
is known to send the Verma module $M^\clambda$ to $M^{-\clambda-2\rho}[d]$. 
\end{rem} 

\ssec{The long intertwining functor}

\sssec{}

Let us consider the following ``opposite" version of the category $\Rep^{\on{mxd}}_q(G)$. 

\medskip

Namely, set: 
$$\Rep^{\wt{\on{mxd}}}_q(G):=Z_{\on{Dr},\Rep_q(T)}(U^{\on{DK}}_q(N^-)\mod(\Rep_q(T))_{\on{loc.nilp}}).$$
equipped with pairs of adjoint functors
$$\ind_{\on{DK}^-\to \wt{\on{mxd}}}: U^{\on{DK}}_q(N^-) \mod(\Rep_q(T))_{\on{loc.nilp}} \rightleftarrows  \Rep^{\wt{\on{mxd}}}_q(G):
\oblv_{\wt{\on{mxd}}\to \on{DK}^-}$$
and
$$\oblv_{\wt{\on{mxd}}\to \on{Lus}^+}:\Rep^{\wt{\on{mxd}}}_q(G) \rightleftarrows U^{\on{Lus}}_q(N) \mod(\Rep_q(T)): 
\coind_{\on{Lus}^+ \to \wt{\on{mxd}}}.$$

For $\clambda\in \cLambda$ we will denote by
$$\BM_{q,\wt{\on{mxd}}}^\clambda, \BM_{q,\wt{\on{mxd}}}^{\vee,\clambda}\in \Rep^{\wt{\on{mxd}}}_q(G)$$
the corresponding standard and costandard objects, respectively. 

\medskip

The category $\Rep^{\wt{\on{mxd}}}_q(G)$ carries a t-structure and can be recovered from its heart by the procedure
of \secref{sss:recover from heart}. 

\begin{rem}
Note that when $q$ is not a root of unity, the action of $w_0\in \on{Aut}(\cLambda)$ defines an equivalence
$$\Rep^{\on{mxd}}_q(G)\to \Rep^{\wt{\on{mxd}}}_q(G).$$

However, for $q$ a root of unity, the two versions are truly different: in the former the Lusztig algebra is supposed to act locally
nilpotently, and in the latter the De Concini-Kac one.

\end{rem}

\sssec{}

Let $U_q(G)^{\on{mxd}}\mod$ be the derived category of the abelian category 
$(U_q(G)^{\on{mxd}}\mod)^\heartsuit$, where the latter consists of objects of $\Rep_q(T)$, equipped 
with an action of $U_q^{\on{Lus}}(N)$ and a compatible action of $U_q^{\on{DK}}(N^-)$. We have
the natural forgetful functors
$$U_q^{\on{Lus}}(N)\mod(\Rep_q(T)) \overset{\oblv_{\on{mxd}\to \on{Lus}^+}}\longleftarrow U_q(G)^{\on{mxd}}\mod
\overset{\oblv_{\on{mxd}\to \on{DK}^-}}\longrightarrow U_q^{\on{DK}}(N^-)\mod(\Rep_q(T)),$$
which admit left adjoints, denoted $\ind_{\on{Lus}^+\to \on{mxd}}$ and $\ind_{\on{DK}^-\to \on{mxd}}$, respectively. 

\medskip

We have the natural forgetful functors
$$\Rep^{\on{mxd}}_q(G)\overset{j^+}\longrightarrow U_q(G)^{\on{mxd}}\mod \overset{j^-}\longleftarrow \Rep^{\wt{\on{mxd}}}_q(G)$$
that make the following diagrams commute:
$$
\xymatrix{
\Rep^{\on{mxd}}_q(G) \ar[rr] \ar[d]<2pt> && U_q(G)^{\on{mxd}}\mod  \ar[d]<2pt>  \\
U_q^{\on{Lus}}(N)\mod(\Rep_q(T))_{\on{loc.nilp}} \ar[rr] \ar[u]<2pt> && U_q^{\on{Lus}}(N)\mod(\Rep_q(T)) \ar[u]<2pt>}
$$
and
$$
\xymatrix{
U_q(G)^{\on{mxd}}\mod \ar[d]<2pt> && \Rep^{\wt{\on{mxd}}}_q(G) \ar[ll] \ar[d]<2pt>  \\ 
U_q^{\on{DK}}(N^-)\mod(\Rep_q(T)) \ar[u]<2pt> && U_q^{\on{DK}}(N^-)\mod(\Rep_q(T))_{\on{loc.nilp}} \ar[ll] \ar[u]<2pt>}
$$

\sssec{}

Note that it follows from \corref{c:aug DK} that the functor $j^-$ sends compacts to compacts. Let 
$(j^-)^R$ denote its right adjoint. 

\medskip

We define the functor
$$\Upsilon: \Rep^{\on{mxd}}_q(G)\to \Rep^{\wt{\on{mxd}}}_q(G)$$
as 
$$\Upsilon:= (j^-)^R\circ j^+.$$

\begin{rem}
We will refer to $\Upsilon$ as the \emph{long intertwining functor}. It is a direct analog of the
functor in the classical situation given by 
$$\on{Av}^{N^-}_*:\fg\mod^B\to \fg\mod^{B^-}.$$

\medskip

The analog of \propref{p:Upsilon Verma} below in the classical situation holds as well: the proof given below
applies. Alternatively, one can deduce it from the Beilinson-Bernstein localization theory. 
\end{rem} 

\sssec{}

We claim:

\begin{prop}  \label{p:Upsilon Verma}
For $\clambda\in \cLambda$,
$$\Upsilon(\BM^\clambda_{q,\on{mxd}})\simeq \BM_{q,\wt{\on{mxd}}}^{\vee,\clambda+2\check\rho}[-d].$$
\end{prop}

\begin{proof}
It suffices to show that
$$\CHom_{\Rep^{\wt{\on{mxd}}}_q(G)}(\BM_{q,\wt{\on{mxd}}}^{\clambda'},\Upsilon(\BM^\clambda_{q,\on{mxd}}))\simeq
\begin{cases}
&k[-d] \text{ if } \clambda'=\clambda+2\check\rho \\
&0 \text{ otherwise}.
\end{cases}
$$

By definition, 
$$\CHom_{\Rep^{\wt{\on{mxd}}}_q(G)}(\BM_{q,\wt{\on{mxd}}}^{\clambda'},\Upsilon(\BM^\clambda_{q,\on{mxd}}))\simeq
\CHom_{U_q(G)^{\on{mxd}}\mod}(j^-(\BM_{q,\wt{\on{mxd}}}^{\clambda'}),j^+(\BM^\clambda_{q,\on{mxd}})),$$
which we further rewrite as
$$\CHom_{U_q^{\on{DK}}(N^-)\mod}(k^{\clambda'},\oblv_{\on{mxd}\to \on{DK}^-}(\BM^\clambda_{q,\on{mxd}})).$$

Now, since $\oblv_{\on{mxd}\to \on{DK}^-}(\BM^\clambda_{q,\on{mxd}})$ is free over $U_q^{\on{DK}}(N^-)$,
the assertion follows from \corref{c:cohomology of UqDK}.

\end{proof} 

\ssec{Cohomological vs contragredient duality}

\sssec{}

Let 
$$\Rep_q^{\on{mxd}}(G)_{\on{loc.fin.dim}} \subset \Rep_q^{\on{mxd}}(G) \text{ and } 
\Rep_q^{\wt{\on{mxd}}}(G)_{\on{loc.fin.dim}} \subset \Rep_q^{\wt{\on{mxd}}}(G)$$
be the full subcategories corresponding to the condition that in each cohomological degree 
each graded component is finite-dimensional. Note that the standard and costandard objects
belong to the corresponding subcategories; hence so do all compact objects. 

\medskip

Component-wise dualization, combined with the identifications
$$U_q^{\on{Lus}}(N)\simeq (U_q^{\on{DK}}(N^-))^\vee \text{ and } U_q^{\on{DK}}(N)\simeq (U_q^{\on{Lus}}(N^-))^\vee$$
defines a contravariant equivalence
$$\Rep_q^{\on{mxd}}(G)_{\on{loc.fin.dim}} \to \Rep_{q^{-1}}^{\wt{\on{mxd}}}(G)_{\on{loc.fin.dim}}$$
that we will refer to as \emph{contragredient duality} and denote by $\BD^{\on{contr}}$.

\medskip

By construction, we have
\begin{equation} \label{e:dual Verma tilde}
\BD^{\on{contr}}(\BM^\clambda_{q^{-1},\on{mxd}})\simeq \BM^{\vee,-\clambda}_{q,\wt{\on{mxd}}} \text{ and }
\BD^{\on{contr}}(\BM^{\vee,\clambda}_{q^{-1},\on{mxd}})\simeq \BM^{-\clambda}_{q,\wt{\on{mxd}}}. 
\end{equation} 

\sssec{}

We claim:

\begin{prop} \label{p:dualities and Ups}
There is a canonical isomorphism
$$\BD^{\on{can}}\simeq (\BD^{\on{contr}})^{-1}\circ \Upsilon.$$
\end{prop}

\begin{rem}
\propref{p:dualities and Ups} is a direct quantum analog of a similar assertion in the classical situation,
see \cite[Corollary 3.2.3]{GY}.
\end{rem} 

\begin{proof}

First, we note that $\Upsilon$ sends $(\Rep^{\on{mxd}}_{q^{-1}}(G))_{c}$ to $(\Rep_{q^{-1}}^{\wt{\on{mxd}}}(G))_{\on{loc.fin.dim}}$:
indeed, it suffices to check this for the standard objects, and the result follows from \propref{p:Upsilon Verma}. 

\medskip

To prove the proposition, it suffices to show that for $\CM_1\in \Rep^{\on{mxd}}_q(G)_{c}$ and 
$\CM_2\in \Rep^{\on{mxd}}_{q^{-1}}(G)_{c}$, we have a
canonical identification
$$\CHom_{\Rep^{\on{mxd}}_q(G)}(k,\CM_1\otimes \CM^\sigma_2)\simeq 
\CHom_{\Rep^{\on{mxd}}_q(G)}((\BD^{\on{contr}})^{-1}\circ \Upsilon(\CM_2),\CM_1).$$

We rewrite the RHS as
$$\CHom_{\Rep^{\wt{\on{mxd}}}_{q^{-1}}(G)}(\BD^{\on{contr}}(\CM_1), \Upsilon(\CM_2)),$$
and further as
$$\CHom_{U_{q^{-1}}(G)^{\on{mxd}}\mod}(j_-(\BD^{\on{contr}}(\CM_1)),j_+(\CM_2)).$$

Now, the assertion follows from the (tautological) identification
$$\CHom_{\Rep^{\on{mxd}}_q(G)}(k,\CM_1\otimes \CM^\sigma_2)\simeq 
\CHom_{U_{q^{-1}}(G)^{\on{mxd}}\mod}(j_-(\BD^{\on{contr}}(\CM_1)),j_+(\CM_2)).$$

\end{proof} 

\sssec{}

From \propref{p:dualities and Ups} we obtain:

\begin{cor} \label{c:dualities and Ups}
There are canonical isomorphisms of functors 
$$\Upsilon\simeq \BD^{\on{contr}}\circ \BD^{\on{can}}: (\Rep^{\on{mxd}}_q(G))_c\to \Rep_q^{\wt{\on{mxd}}}(G)_{\on{loc.fin.dim}}$$
and 
$$\BD^{\on{contr}}\simeq \Upsilon\circ \BD^{\on{can}}: (\Rep^{\on{mxd}}_{q^{-1}}(G))_c\to (\Rep_q^{\wt{\on{mxd}}}(G))_{\on{loc.fin.dim}}.$$
\end{cor} 

\sssec{Proof of \thmref{t:quantum duality}}

By \propref{p:dualities and Ups}, we have
$$\BD^{\on{can}}(\BM^\clambda_{q^{-1},\on{mxd}})\simeq (\BD^{\on{contr}})^{-1}\circ \Upsilon(\BM^\clambda_{q^{-1},\on{mxd}}),$$
while the RHS identifies, according to \propref{p:Upsilon Verma} with 
$$(\BD^{\on{contr}})^{-1}(\BM_{q^{-1},\wt{\on{mxd}}}^{\vee,\clambda+2\check\rho}[-d]),$$
and the latter identifies, according to \eqref{e:dual Verma tilde}, with 
$$\BM^{-\clambda-2\check\rho}_{q,\on{mxd}}[d],$$
as required.

\qed

\bigskip

\centerline{\bf Part III: Kac-Moody vs quantum group representations}

\bigskip

\section{A conjectural extension of the Kazhdan-Lusztig equivalence}  \label{s:conj}

In this section we take our field of coefficients $k$ to be $\BC$, and we let $b$ and $\kappa$ be related by formula
\eqref{e:q and kappa}. 

\medskip

We will formulate the main conjecture of this paper (\conjref{c:main}), which compares the categories
$\Rep_q^{\on{mxd}}(G)$ and $\hg\mod_{-\kappa}^I$. 

\medskip

\noindent {\it Notational remark}: in order to unburden the notation, from this section on we will choose a uniformizer $t\in \CO$,
thereby trivializing the line $\omega_x$ (see \secref{sss:omega}).

\ssec{The Kazhdan-Lusztig equivalence}

In this subsection we take $-\kappa$ to be a negative integral level (although whatever we will say goes
through for negative rational levels, see Remark \ref{r:Frob lattice}). 

\medskip

We will recall the statement of the Kazhdan-Lusztig equivalence of \cite{KL}. 

\sssec{}

Consider the category
$$\KL(G,-\kappa):=\hg\mod_{-\kappa}^{G(\CO)}.$$

We recall from \secref{sss:g-modK} that the category $\KL(G,-\kappa)$ is \emph{by definition} compactly generated by objects of the form
\begin{equation} \label{e:compact in KL}
\Ind_{\fg(\CO)}^{\hg_{-\kappa}}(V),
\end{equation} 
where $V$ is a \emph{finite-dimensional} representation of $G(\CO)$. 

\medskip

The category $\KL(G,-\kappa)$ carries a t-structure for which the forgetful functor $\KL(G,-\kappa)\to \Vect$
is t-exact. Note, however, that it is \emph{not true} that $\KL(G,-\kappa)$ is left-separated complete in its t-structure. 

\begin{rem} \label{r:recover KL from heart}
Let us emphasize again the relationship between $\KL(G,-\kappa)$ and the abelian category $\KL(G,-\kappa)^\heartsuit$:

\medskip

The former is the ind-completion of the full subcategory of $D^b(\KL(G,-\kappa)^\heartsuit)$ 
generated under finite colimits by objects \eqref{e:compact in KL}.

\medskip

Note that this relationship is the same as that between $\Rep_q(G)_{\on{ren}}$ and the abelian category $(\Rep_q(G))^\heartsuit$. 

\end{rem}

\sssec{}

For $\clambda\in \cLambda^+$ recall the notation for the Weyl modules
$$\BV^\clambda_{-\kappa}:=\Ind_{\fg(\CO)}^{\hg_{-\kappa}}(V^\clambda)\in \KL(G,-\kappa),$$
see \eqref{e:Weyl mod}. The objects $\BV^\clambda_{-\kappa}$ are compact and they generate $\KL(G,-\kappa)$. 

\sssec{}

Let $\Dmod_{-\kappa}(\Gr_G)^{G(\CO)}$ denote the monoidal category of spherical D-modules on the affine 
Grassmannian $\Gr_G$. We have a monoidal action of $\Dmod_{-\kappa}(\Gr_G)^{G(\CO)}$ on $\KL(G,-\kappa)$,
$$\CF,\CM\mapsto \CF\star \CM.$$

\medskip

A basic piece of structure that we need is the (classical) Geometric Satake, which is a monoidal functor
$$\on{Sat}:\Rep(\cG)\to \Dmod_{-\kappa}(\Gr_G)^{G(\CO)}.$$

We normalize it so that for $\mu\in \Lambda^+$, we have a canonical map
$$\BV^\mu_{-\kappa} \to \on{Sat}(V^\mu)\star \BV^0_{-\kappa}.$$

Here $V^\mu$ is the irreducible object of $\Rep(\cG)^\heartsuit$ of highest weight $\mu$;
when we write $\BV^\mu_{-\kappa}$, we consider $\mu$ as an element of $\cLambda$ via
$\kappa(\lambda,-)$. 

\medskip

For a pair of elements $\mu\in \Lambda^+$ and $\clambda^+\in \cLambda^+$ we have a canonical map
\begin{equation} \label{e:sph with Weyl}
\BV^{\mu+\clambda}_{-\kappa} \to \on{Sat}(V^\mu)\star \BV^\clambda_{-\kappa}.
\end{equation} 

\sssec{}

Let $q$ be the $\BC^\times$-valued 
quadratic form on $\cLambda$ corresponding to $\kappa$. Note that the assumption that $\kappa$ is integral implies
that the lattice $\Lambda_H$ (see \secref{sss:quant Frob lattice}) identifies with $\Lambda$, and the group $H$
identifies with the Langlands dual $\cG$ of $G$. 

\medskip

The theorem of Kazhdan and Lusztig of \cite{KL} states the existence of an equivalence of abelian categories
\begin{equation} \label{e:KL equiv ab}
\sF_{-\kappa}:\KL(G,-\kappa)^\heartsuit \simeq \Rep_q(G)^\heartsuit,
\end{equation}
with the property that for $\clambda\in \cLambda^+$ there exists a canonical isomorphism
\begin{equation} \label{e:Weyl to Weyl}
\sF_{-\kappa}(\BV^\clambda_{-\kappa})\simeq \CV_q^\clambda.
\end{equation}

In what follows we will use two more properties of the equivalence \eqref{e:KL equiv ab}, which are not
stated in the paper, but can be deduced from it:

\medskip

\begin{itemize}

\item The action of $\Rep(\cG)^\heartsuit$ on $\KL(G,-\kappa)^\heartsuit$ corresponds under $\sF_{-\kappa}$ to
the action of $\Rep(\cG)^\heartsuit$ on $\Rep_q(G)^\heartsuit$ via pullback by quantum Frobenius.

\medskip

\item Under the above identification, the image under $\sF_{-\kappa}$ of the map \eqref{e:sph with Weyl}
is the canonical map
\begin{equation} \label{e:Frob with Weyl}
\CV^{\mu+\clambda}_q \to \on{Frob}_q^*(V^\mu)\star \CV^\clambda_q.
\end{equation} 

\end{itemize}

\sssec{}

From \eqref{e:KL equiv ab}, by passing to the bounded derived categories and ind-completing the corresponding full subcategories, we obtain an
equivalence of DG categories:  

\begin{equation} \label{e:KL equiv}
\sF_{-\kappa}:\KL(G,-\kappa) \simeq \Rep_q(G)_{\on{ren}}.
\end{equation}

\medskip

The compatibility of the equivalence \eqref{e:KL equiv ab} with the action of $\Rep(\cG)^\heartsuit$ at the abelian level implies
the compatibility of the equivalence \eqref{e:KL equiv} with the action of $\Rep(\cG)$.

\begin{rem}
The equivalence \eqref{e:KL equiv ab} satisfying \eqref{e:Weyl to Weyl}
exists for irrational levels as well (but in this case we have neither the Hecke action nor quantum Frobenius).

\medskip

Note, however, that such an equivalence taken ``as-is" (i.e., without taking into account the braided monoidal
structures) is not very interesting as both categories are semi-simple.

\end{rem}

\ssec{Conjectural extension to the Iwahori case}

\sssec{}

We now consider the category $\hg\mod_{-\kappa}^I$.  We propose the following extension of the Kazhdan-Lusztig equivalence:

\begin{conj}  \label{c:main}
There exists an equivalence 
\begin{equation} \label{e:main functor}
\sF_{-\kappa}:\hg\mod_{-\kappa}^I\simeq \Rep_q^{\on{mxd}}(G)
\end{equation} 
with the following properties:

\medskip

\noindent{\em(i)} The square 
\begin{equation} \label{e:KL vs I}
\CD
\KL(G,-\kappa):=\hg\mod_{-\kappa}^{G(\CO)}   @>{\sF_{-\kappa}}>>   \Rep_q(G)_{\on{ren}}  \\
@V{\oblv_{G(\CO)/I}}VV   @VV{\oblv_{\on{big}\to \on{mxd}}}V  \\
\hg\mod_{-\kappa}^I   @>{\sF_{-\kappa}}>>  \Rep_q^{\on{mxd}}(G)
\endCD
\end{equation} 
commutes. 

\medskip

\noindent{\em(ii)} For $\mu\in \Lambda$, the action of $J_\mu$ on $\hg\mod_{-\kappa}^I$ corresponds to the
action of $k^\mu\in \Rep(\cT)$ on $\Rep_q^{\on{mxd}}(G)$.

\medskip

\noindent{\em(iii)} For $\clambda\in \cLambda$, we have
$$\sF_{-\kappa}(\BW_{-\kappa}^\clambda)\simeq \BM_{q,\on{mxd}}^\clambda.$$

\end{conj}

There is one more expected property of the functor \eqref{e:main functor} that has to do with the structure
on both sides of categories over the stack $\cn/\on{Ad}(\cB)$; we will discuss it in \secref{sss:AB compat}.

\begin{rem}
\conjref{c:main}, satisfying (i) and (iii) applies also for $\kappa$ irrational. However, in this case it reduces to the known
statement that both categories identify with the finite-dimensional $\fg\mod^B$ so that
$$\BM^\clambda_{-\kappa}\, \leftrightarrow\, M^\lambda\, \leftrightarrow\,  \BM_{q,\on{mxd}}^\clambda.$$

Here we are using \propref{p:Wak irrational}, which says that for $\kappa$ irrational $\BW_{-\kappa}^\clambda\simeq \BM_{-\kappa}^\clambda$.
\end{rem}

\sssec{}

We will now run some consistency checks on \conjref{c:main}.

\medskip

From the commutativity of the square \eqref{e:KL vs I}, we deduce the commutativity of the following two squares:
\begin{equation} \label{e:KL vs I left}
\CD
\KL(G,-\kappa):=\hg\mod_{-\kappa}^{G(\CO)}   @>{\sF_{-\kappa}}>>   \Rep_q(G)_{\on{ren}}  \\
@A{\on{Av}^{G(\CO)/I}_!}AA   @AA{\ind_{\on{mxd}\to \on{big}}}A  \\
\hg\mod_{-\kappa}^I   @>{\sF_{-\kappa}}>>  \Rep_q^{\on{mxd}}(G)  
\endCD
\end{equation} 
and
$$
\CD
\KL(G,-\kappa):=\hg\mod_{-\kappa}^{G(\CO)}   @>{\sF_{-\kappa}}>>   \Rep_q(G)_{\on{ren}}  \\
@A{\on{Av}^{G(\CO)/I}_*}AA   @AA{\coind_{\on{mxd}\to \on{big}}}A  \\
\hg\mod_{-\kappa}^I   @>{\sF_{-\kappa}}>>  \Rep_q^{\on{mxd}}(G),  
\endCD
$$
where the vertical arrows are obtained by passing to left and right adjoints in \eqref{e:KL vs I}, respectively.
Note, however, that since $G(\CO)/I\simeq G/B$ is proper, we have a canonical isomorphism
$$\on{Av}^{G(\CO)/I}_*\simeq \on{Av}^{G(\CO)/I}_![-2d],$$
see \secref{sss:proper averaging}.

\medskip

We note that this is consistent with \conjref{c:adj to res}, which says that
$$\coind_{\on{mxd}\to \on{big}}\simeq \ind_{\on{mxd}\to \on{big}}[-2d].$$

\sssec{}

Points (ii) and (iii) of \conjref{c:main} are compatible due to the relations
$$J_\mu \star \BW_{-\kappa}^{\clambda}\simeq \BW_{-\kappa}^{\mu+\clambda} \text{ and }
k^\mu\otimes \BM_{q,\on{mxd}}^\clambda\simeq \BM_{q,\on{mxd}}^{\mu+\clambda}.$$

(Note that for the first isomorphism we use the trivialization of $\omega_x$.)

\sssec{}

Let us now investigate the compatibility between points (i) and (iii) of \conjref{c:main}. Namely,
evaluating both circuits of diagram \eqref{e:KL vs I left} on $\BW_{-\kappa}^{\clambda}$
we obtain an isomorphism:
$$\sF_{-\kappa}(\on{Av}^{G(\CO)/I}_!(\BW_{-\kappa}^{\clambda}))\simeq \ind_{\on{mxd}\to \on{big}}(\BM^\clambda_{q,\on{mxd}}),$$
where $\sF_{-\kappa}$ is the original Kazhdan-Lusztig equivalence.

\medskip

Recall that according to \eqref{e:Weyl from mixed}, we have
$$\ind_{\on{mxd}\to \on{big}}(\BM^\clambda_{q,\on{mxd}})\simeq \CV_q^\clambda.$$

Recall also that if $\clambda\in \cLambda^+$, then $\BW_{-\kappa}^{\clambda}\simeq \BM_{-\kappa}^{\clambda}$
(see \corref{c:Wak=Ver}), and since
$$\on{Av}^{G/B}_!(M^\clambda)=V^\clambda\in \Rep(G),$$
we obtain $\on{Av}^{G(\CO)/I}_!(\BM_{-\kappa}^{\clambda})\simeq \BV^\clambda_{-\kappa}$, and hence 
\begin{equation} \label{e:all Weyl}
\on{Av}^{G(\CO)/I}_!(\BW_{-\kappa}^{\clambda})\simeq \BV^\clambda_{-\kappa},
\end{equation}
establishing the desired consistency.

\sssec{}

We take \eqref{e:all Weyl} as the \emph{definition} of $\BV^\clambda_{-\kappa}$ for $\clambda\in \cLambda$ that is not necessarily dominant.

\medskip

To summarize, we obtain the following isomorphism that references only the \emph{original Kazhdan-Lusztig functor} $\sF_{-\kappa}$ of \eqref{e:KL equiv}
\begin{equation} \label{e:Weyl}
\sF_{-\kappa}(\BV^\clambda_{-\kappa})\simeq \CV_q^\clambda.
\end{equation}

The isomorphism \eqref{e:Weyl} is a basic property of $\sF_{-\kappa}$ for $\clambda$ dominant,
and it follows from \conjref{c:main} for all $\clambda$. However, in \secref{sss:proof of Weyl}
we will prove:

\begin{thm} \label{t:Weyl}
The isomorphism \eqref{e:Weyl} holds (unconditionally) for all $\clambda\in \cLambda$.
\end{thm} 

\ssec{Digression: dual Weyl modules}

\sssec{}

For $\clambda\in \cLambda^+$ we introduce the dual Weyl module $\BV_{-\kappa}^{\vee,\clambda}\in \KL(G,-\kappa)$
by requiring
$$\CHom_{\hg\mod_{-\kappa}^{G(\CO)}}(\BV_{-\kappa}^{\clambda'},\BV_{-\kappa}^{\vee,\clambda})=
\begin{cases}
&k \text{ for } \clambda'=\clambda,\\
&0 \text{ for } \clambda'\neq \clambda,\, \clambda'\in \cLambda^+.
\end{cases}$$

It is known that $\BV_{-\kappa}^{\vee,\clambda}$ actually lies in $\KL(G,-\kappa)^\heartsuit$.

\medskip

We have the following assertion:

\begin{lem}  \label{l:dual Weyl as av}
For $\clambda\in \cLambda^+$ there exists a canonical isomorphism
$$\BV_{-\kappa}^{\vee,\clambda}\simeq \on{Av}^{G(\CO)/I}_*(\BM_{-\kappa}^{\vee,\clambda}).$$
\end{lem}

\begin{proof}

We have:
$$\CHom_{\hg\mod_{-\kappa}^{G(\CO)}}(\BV_{-\kappa}^{\clambda'},\on{Av}^{G(\CO)/I}_*(\BM_{-\kappa}^{\vee,\clambda}))\simeq
\CHom_{\hg\mod_{-\kappa}^I}(\BV_{-\kappa}^{\clambda'},\BM_{-\kappa}^{\vee,\clambda}).$$

The BGG resolution of $V^\clambda$ with terms $M^{w(\clambda'+\check\rho)-\check\rho}$ implies that 
$\BV_{-\kappa}^{\clambda'}$ admits a resolution with terms of the form $\BM_{-\kappa}^{w(\clambda'+\check\rho)-\check\rho}$.
Note that for all $w\neq 1$, the weight $w(\clambda'+\check\rho)-\check\rho$ is non-dominant, and hence not equal to $\clambda$. 
Hence, we obtain that
$$\CHom_{\hg\mod_{-\kappa}^I}(\BV_{-\kappa}^{\clambda'},\BM_{-\kappa}^{\vee,\clambda})\simeq 
\CHom_{\hg\mod_{-\kappa}^I}(\BM_{-\kappa}^{\clambda'},\BM_{-\kappa}^{\vee,\clambda}),$$
and the assertion follows.

\end{proof}

\sssec{}

We now define
\begin{equation} \label{e:dual Weyl general}
\BV_{-\kappa}^{\vee,\clambda}:=\on{Av}^{G(\CO)/I}_!(\BW^{w_0(\clambda)-2\check\rho}_{-\kappa})[-d]
\end{equation} 
for \emph{any} $\clambda\in \cLambda$. 

\medskip

We claim:

\begin{lem}  \label{l:general dual Weyl}
For $\clambda\in \cLambda^+$, the definition of $\BV_{-\kappa}^{\vee,\clambda}$ agrees with the initial one.
\end{lem}

\begin{proof} 

First, we claim that for $\clambda$ dominant, we have
$$j_{w_0,*}\star \BM^{\vee,w_0(\clambda)-2\check\rho}_{-\kappa}\simeq  \BM^{\vee,\clambda}_{-\kappa}.$$

Indeed, it suffices to show that
$$\CHom_{\hg\mod_{-\kappa}^I}(\BM^{\clambda'}_{-\kappa}, j_{w_0,*}\star \BM^{\vee,w_0(\clambda)-2\check\rho}_{-\kappa})\simeq 
\CHom_{\hg\mod_{-\kappa}^I}(j_{w_0,!}\star \BM^{\clambda'}_{-\kappa}, \BM^{\vee,w_0(\clambda)-2\check\rho}_{-\kappa})=
\begin{cases}
&k \text{ for } \clambda'=\clambda,\\
&0 \text{ otherwise.}
\end{cases}
$$

For this it suffices to show that the cone of the canonical map
$$j_{w_0,!}\star \BM^{\clambda'}_{-\kappa}\to \BM^{w_0(\clambda')-2\check\rho}_{-\kappa}$$
is an extension of Verma modules of highest weights different from $w_0(\clambda)-2\check\rho$. 
For the latter, it suffices to show the corresponding assertion for the map for the finite-dimensional $\fg$:
$$j_{w_0,!}\star M^{\clambda'}\to M^{w_0(\clambda')-2\check\rho},$$
which in turn follows from the (valid) isomorphism
$$j_{w_0,*}\star M^{\vee,w_0(\clambda)-2\check\rho}\simeq M^{\vee,\clambda}, \quad \clambda\in \cLambda^+.$$

\medskip

Now, using \corref{c:Wakimoto for neg}, we have
\begin{multline*} 
\on{Av}^{G(\CO)/I}_!(\BW^{\vee,w_0(\clambda)-2\check\rho}_{-\kappa})[-d]\simeq
\on{Av}^{G(\CO)/I}_!(\BM^{\vee,w_0(\clambda)-2\check\rho}_{-\kappa})[-d]\simeq
\on{Av}^{G(\CO)/I}_*(\BM^{\vee,w_0(\clambda)-2\check\rho}_{-\kappa})[d]\simeq \\
\\ \simeq \on{Av}^{G(\CO)/I}_*(j_{w_0,*}\star \BM^{\vee,w_0(\clambda)-2\check\rho}_{-\kappa})\simeq
\on{Av}^{G(\CO)/I}_*(\BM^{\vee,\clambda}_{-\kappa})\simeq \BV_{-\kappa}^{\vee,\clambda},
\end{multline*}
as required. 

\end{proof}

\sssec{}

Note that taking into account \corref{c:big ind an coind}, we obtain that \thmref{t:Weyl} is equivalent to the following
statement: 

\begin{cor} \label{c:dual Weyl}
For any $\clambda\in \cLambda$, we have
$\sF_{-\kappa}(\BV^{\vee,\clambda}_{-\kappa})\simeq \CV_q^{\vee,\clambda}$.
\end{cor}

Finally, note that for $\clambda\in \cLambda^+$, the assertion of \corref{c:dual Weyl} follows formally from
\lemref{l:general dual Weyl}: indeed, it is clear that the original Kazhdan-Lusztig functor satisfies: 
$$\sF_{-\kappa}(\BV_{-\kappa}^{\vee,\clambda})\simeq \CV_q^{\vee,\clambda}.$$

\ssec{Digression: Kazhdan-Lusztig equivalence for positive level}

Let $\kappa$ be the positive level, opposite of the negative level $-\kappa$. In this subsection we will
discuss some consequences that the original Kazhdan-Lusztig equivalence has for the category $\KL(G,\kappa)$. 

\sssec{}

Recall that we have the duality identifications
$$\KL(G,\kappa)\simeq (\KL(G,-\kappa))^\vee \text{ and  } \Rep_{q^{-1}}(G)_{\on{ren}} \simeq (\Rep_q(G)_{\on{ren}})^\vee,$$
the former given by \propref{p:duality K} and the latter by \eqref{e:big duality}. 

\medskip

Thus, starting from the equivalence
$$\sF_{-\kappa}:\KL(G,-\kappa)\simeq \Rep_q(G)_{\on{ren}}$$
by duality, we obtain an equivalence
\begin{equation} \label{e:KL pos level}
\sF_\kappa:\KL(G,\kappa)\simeq \Rep_{q^{-1}}(G)_{\on{ren}}.
\end{equation} 

\sssec{}

Note that the compatibility of the equivalence $\sF_\kappa$ of \eqref{e:KL pos level} with the Hecke action reads as follows:
$$\sF_\kappa(\on{Sat}(V)\star \CM)\simeq \on{Frob}_q^*(V^\tau)\otimes M,$$
where $\tau$ is the Cartan involution on $\cG$.

\medskip

This is due to the fact that Verdier duality 
$$\BD:(\Dmod_{-\kappa}(\Gr_G)^{G(\CO)})_c\to (\Dmod_{\kappa}(\Gr_G)^{G(\CO)})_c$$
satisfies
\begin{equation} \label{e:Verdier and Sat}
\BD(\on{Sat}(V))\simeq \on{Sat}((V^\tau)^\vee).
\end{equation} 

\sssec{}

Note that the equivalence $\sF_\kappa:\KL(G,\kappa)\simeq \Rep_{q^{-1}}(G)_{\on{ren}}$
satisfies:
$$\sF_\kappa(\BV_{\kappa}^{\lambda})\simeq \CV_{q^{-1}}^{\vee,\clambda}, \quad \clambda\in \cLambda^+.$$

\medskip

This follows from the identifications 
$$\BD(\BV^\clambda_{-\kappa})\simeq \BV^{-w_0(\clambda)}_\kappa \text{ and } 
\BD(\CV_q^{-w_0(\clambda)})\simeq \CV_{q^{-1}}^{\vee,\clambda}.$$

\sssec{}

For any $\clambda\in \cLambda$, set
$$\BV_{\kappa}^{\lambda}:=\on{Av}^{G(\CO)/I}_*(\BD(\BW^{-w_0(\clambda)}_{-\kappa})).$$
Note that for $\clambda\in \cLambda^+$, this is consistent with the definition of $\BV_\kappa^\clambda$
as $\Ind_{\fg(\CO)}^{\hg_\kappa}(V^\clambda)$. Indeed,
$$\on{Av}^{G(\CO)/I}_*(\BD(\BW^{-w_0(\clambda)}_{-\kappa}))\simeq  
\BD\left(\on{Av}^{G(\CO)/I}_!(\BW^{-w_0(\clambda)}_{-\kappa})\right)\simeq
\BD(\BV_{-\kappa}^{-w_0(\clambda)})\simeq \BV_\kappa^\clambda.$$

\medskip 

Note that \thmref{t:Weyl} is equivalent to the following:

\begin{thm} \label{t:Weyl positive}
For any $\clambda\in \cLambda$, we have 
$\sF_k(\BV^\clambda_\kappa)\simeq \CV_{q^{-1}}^{\vee,\clambda}$.
\end{thm}

We emphasize again that the assertion of \thmref{t:Weyl positive} follows from the usual
Kazhdan-Lusztig equivalence for $\clambda\in \cLambda^+$. 

\ssec{Extension to the Iwahori case for positive level}

In this subsection we will assume \conjref{c:main} and deduce some consequences for the category $\hg\mod_\kappa^I$.

\sssec{}

Recall now that we have the equivalences
$$\hg\mod_{\kappa}^I \simeq (\hg\mod_{-\kappa}^I)^\vee  \text{ and } \Rep^{\on{mxd}}_{q^{-1}}(G) \simeq (\Rep^{\on{mxd}}_q(G))^\vee,$$
see \propref{p:duality K} for the former and \thmref{t:quantum duality} for the latter. 

\medskip

Thus, starting from the conjectural equivalence
$$\sF_{-\kappa}:\hg\mod_{-\kappa}^I\simeq \Rep^{\on{mxd}}_q(G),$$
by duality we obtain an equivalence 
$$\sF_\kappa:\hg\mod_{\kappa}^I \simeq \Rep^{\on{mxd}}_{q^{-1}}(G).$$

\sssec{}

Let us explore the properties of this equivalence that follow from the properties of \eqref{e:main functor}.

\medskip

First, the diagram 
\begin{equation} \label{e:KL vs I pos}
\CD
\KL(G,\kappa)   @>{\sF_{\kappa}}>>   \Rep_{q^{-1}}(G)_{\on{ren}}  \\
@V{\oblv_{G(\CO)/I}}VV   @VV{\oblv_{\on{big}\to \on{mxd}}}V  \\
\hg\mod_{\kappa}^I   @>{\sF_{\kappa}}>>  \Rep_{q^{-1}}^{\on{mxd}}(G). 
\endCD
\end{equation} 
is commutative. This follows from the commutativity of \eqref{e:KL vs I} by duality, see \secref{sss:Av adj} and \propref{p:duality big and mixed}. 

\sssec{}

Let us define the object 
$$J^{\BD}_\mu\in \Dmod_{\kappa}(\on{Fl}^{\on{aff}}_G)^I$$
as the Verdier dual of $J_\mu$. For example, for $\mu$ dominant, $J^{\BD}_\mu\simeq j_{\mu,*}$ and $J^{\BD}_{-\mu}=j_{-\mu,!}$. 

\medskip

Then the functor $\sF_{\kappa}$ intertwines the convolution action of $J^{\BD}_\mu$ on $\hg\mod_{\kappa}^I$ with the functor
$k^{-\mu}\otimes-$ on $\Rep_{q^{-1}}^{\on{mxd}}(G)$. 

\sssec{}

Assuming the existence of $\sF_{-\kappa}$ and hence that of $\sF_\kappa$, we obtain the following: 

\begin{prop}  \label{p:at pos level}
Under the equivalence 
$$\sF_{\kappa}:\hg\mod_{\kappa}^I\simeq \Rep^{\on{mxd}}_{q^{-1}}(G),$$
for $\clambda\in \cLambda$ the functor 
$$\on{C}^\semiinf(\fn(\CK),-)^\clambda:\hg\mod_{\kappa}^I \to \Vect$$
corresponds to the functor
$$\on{C}^\cdot(U_{q^{-1}}^{\on{Lus}}(N),-)^\clambda:\Rep_{q^{-1}}^{\on{mxd}}(G) \to \Vect.$$
\end{prop}

\begin{proof}

By \thmref{t:quantum duality}, the functor
$$\on{C}^\cdot(U_{q^{-1}}^{\on{Lus}}(N),-)^\clambda:\Rep_{q^{-1}}^{\on{mxd}}(G) \to \Vect$$
is given by the pairing with $\BM^{-\clambda-2\check\rho}_{q,\on{mxd}}[d]$,
while by \corref{c:semiinf with Wak I}, the functor 
$$\on{C}^\semiinf(\fn(\CK),-)^\clambda:\hg\mod_{\kappa}^I \to \Vect$$
is given by the pairing with $\BW^{-\clambda-2\check\rho}_{-\kappa}[d]$.  Now the assertion
follows from the isomorphism $$\sF_{-\kappa}(\BW^\cmu_{-\kappa})\simeq \BM_{q,\on{mxd}}^\cmu.$$

\end{proof} 

\sssec{}

Bu juxtaposing \propref{p:at pos level} with the diagram \eqref{e:KL vs I pos}, we obtain the following assertion,
which, however, can be proved unconditionally:

\begin{cor}  \label{c:Weyl pos}
Under the equivalence 
$$\sF_{\kappa}:\KL(G,\kappa)\simeq \Rep_{q^{-1}}(G)_{\on{ren}},$$
for $\clambda\in \cLambda$ the functor 
$$\on{C}^\semiinf(\fn(\CK),-)^\clambda:\KL(G,\kappa) \to \Vect$$
corresponds to the functor
$$\on{C}^\cdot(U_{q^{-1}}^{\on{Lus}}(N),-)^\clambda:\Rep_{q^{-1}}(G)_{\on{ren}} \to \Vect.$$
\end{cor}

\begin{rem}

\corref{c:Weyl pos} reproduces the result of \cite[Theorem 5.3.1]{Liu}. The proof that we will give 
is close in spirit to one in {\it loc.cit.}, but is \emph{logically inequivalent} to it.  

\end{rem} 

\begin{proof}

The functor
$$\on{C}^\cdot(U_{q^{-1}}^{\on{Lus}}(N),-)^\clambda:\Rep_{q^{-1}}(G)_{\on{ren}} \to \Vect$$
is given by $\CHom_{\Rep_{q^{-1}}(G)_{\on{ren}}}(\CV^\clambda_{q^{-1}},-)$, i.e., by the pairing with the object
$$(\CV^\clambda_{q^{-1}})^\vee \simeq \CV^{\vee,-w_0(\clambda)}_q.$$

\medskip

The functor $\on{C}^\semiinf(\fn(\CK),-)^\clambda:\hg\mod_{\kappa}^I  \to \Vect$ is given by
the pairing with $\BW^{-\clambda-2\check\rho}_{-\kappa}[d]$, and hence as a functor
$\on{C}^\semiinf(\fn(\CK),-)^\clambda:\KL(G,\kappa) \to \Vect$
by the pairing by the pairing with
$$\on{Av}^{G(\CO)/I}_*(\BW^{-\clambda-2\check\rho}_{-\kappa}[d])\simeq 
\on{Av}^{G(\CO)/I}_!(\BW^{-\clambda-2\check\rho}_{-\kappa}[-d])=:\BV^{\vee,-w_0(\clambda)}_{-\kappa}.$$

Now the assertion follows from \corref{c:dual Weyl}. 

\end{proof}

\section{Cohomology of the small quantum group via Kac-Moody algebras}   \label{s:small cohomology}

The goal of this section is to prove \thmref{t:Weyl}. The idea is to bootstrap the assertion for
any $\clambda$ from the case when $\clambda$ is dominant. In the process of doing so
we will need to describe the counterpart on the Kac-Moody side of \emph{baby Verma modules}
for quantum groups. 

\medskip

As a byproduct we will describe the functor on the Kazhdan-Lusztig category that corresponds to
the functor of cohomology with respect to $u_q(N)$ for quantum groups. 

\ssec{The Drinfeld-Pl\"ucker formalism}  \label{ss:Dr-Pl}

The thrust of Sects. \ref{ss:Dr-Pl}-\ref{ss:Dr-Pl baby} is to give an expression of the baby Verma module
$\BM^\clambda_{q,\on{small}}$ in terms of restrictions of (dual) Weyl modules for the big quantum group (we need such a description
in order to be able to transfer it to the Kazhdan-Lusztig category via the functors $\sF_{-\kappa}$ or $\sF_\kappa$).
The framework for doing so is provided by the Drinfeld-Pl\"ucker formalism\footnote{The terminology ``Drinfeld-Pl\"ucker formalism"
as well as the abstract framework for this formalism was suggested by S.~Raskin.}. 

\medskip

In this subsection we recall some material from \cite[Sect. 6]{Ga2}.

\sssec{}  \label{sss:DrPl pattern}

We start with the following observation. 

\medskip

Let $\CC$ be a DG category equipped with an action of $\Rep(\cG)$. Let 
$$\{c^\mu,\,\mu\in \Lambda\}\in \CC$$ be a collection of objects
equipped with a \emph{Drinfeld-Pl\"ucker data}, i.e., a homotopy-coherent system of tensor-compatible maps
\begin{equation} \label{e:Plucker maps}
(V^{\mu_1})^\vee\star c^{\mu_2}\to c^{-\mu_1+\mu_2}.
\end{equation}

To the family $\{c^\mu\}$ we can attach the object
$$c:=\underset{\nu\in \Lambda}\oplus\, \underset{\mu\in \Lambda^+}{\on{colim}}\,  V^\mu \star c^{-\nu-\mu}\in \CC.$$
In the formation of the colimit the transition maps are as follows: for $\mu_2=\mu_1+\mu$, the corresponding map is 
$$V^{\mu_1} \star c^{-\mu_1-\nu} \to V^{\mu_1} \star V^{\mu_2}\star (V^{\mu_2})^\vee \star c^{-\mu_1-\nu} \to V^{\mu_1+\mu_2}\star c^{-\mu_1-\mu_2-\nu}.$$

\medskip

Let us denote the tautological maps $V^\mu \star c^{-\mu-\nu}\to c$ by $\phi_\mu$. In particular, we have the maps 
$$\phi_0: c^{-\nu}\to c.$$

\medskip

According to \cite[Sect. 6]{Ga2}, the above object $c$ has the following pieces of structure. 

\sssec{}  \label{sss:Hecke structure}

First, $c$ is a Hecke eigen-object, i.e., it carries a tensor-compatible system of maps
\begin{equation} \label{e:Hecke}
V \star c\simeq c\otimes \ul{V}, \quad V\in \Rep(\cG),
\end{equation} 
where $\ul{V}\in \Vect$ is the vector space underlying $V$. 

\medskip

Explicitly, the maps \eqref{e:Hecke} are given as follows: for $\eta\in \Lambda^+$ large compared to $V$, we have
$$V \star V^\eta \star c^{-\nu-\eta} \simeq (\underset{\epsilon}\oplus \, V^{\eta+\epsilon} \otimes \ul{V}(\epsilon))\star c^{-\nu-\eta}\simeq 
\underset{\epsilon}\oplus \, V^{\eta+\epsilon} \star c^{(-\nu+\epsilon)-\eta-\epsilon} \otimes \ul{V}(\epsilon).$$

\medskip

Note that in terms of the Hecke structure, the map $\phi_\mu:V^\mu \star c^{-\mu-\nu}\to c$ can be expressed via $\phi_0$ as
\begin{equation} \label{e:identify map}
V^\mu \star c^{-\mu-\nu}\overset{\on{Id}\otimes \phi_0}\longrightarrow V^\mu \otimes c \simeq c \otimes \ul{V}^\mu \to c,
\end{equation}
where the last arrow is given by the \emph{projection} onto the highest weight line. 

\sssec{}

We will think of the category of Hecke eigen-objects in $\CC$ as
$$\Rep(\cG)\underset{\Rep(\cG)}\otimes \CC,$$
and we will denote the resulting object by
$$c^{\on{H}}\in \Rep(\cG)\underset{\Rep(\cG)}\otimes \CC.$$

The original $c$ is recovered from $c^{\on{H}}$ via the forgetful functor
$$\Vect\underset{\Rep(\cG)}\otimes \CC \overset{\coind_\cG}\longrightarrow \Rep(\cG)\underset{\Rep(\cG)}\otimes \CC\simeq 
\CC.$$

\sssec{}  \label{sss:B-action}

Another piece of structure on $c$ is an action of the algebraic group $\cB$. This action is compatible 
with the Hecke structure, i.e., the isomorphisms \eqref{e:Hecke} are $\cB$-equivariant, where the
$\cB$-action on the LHS is induced by that on $c$ and on the RHS it is diagonal action.  

\sssec{}

The above compatibility means
that $cc^{\on{H}}$ lifts to an object $c^{\on{enh}}$ of the category
$$\Rep(\cB)\underset{\Rep(\cG)}\otimes \CC$$
along the forgetful functor 
$$\Rep(\cB)\underset{\Rep(\cG)}\otimes \CC \overset{\oblv_\cB}\longrightarrow 
\Vect\underset{\Rep(\cG)}\otimes \CC.$$

\medskip

Moreover:

\begin{itemize}

\item The direct summand $\underset{\mu\in \Lambda^+}{\on{colim}}\,  V^\mu \star c^{-\mu-\nu}$ of $c$ corresponds to the
weight $\nu$-component of $c$ with respect to the action of $\cT\subset \cB$, i.e., $\inv_{\cT}(k^{-\nu} \otimes c)$;

\medskip

\item The map $\phi_0:c^{-\nu}\to c$ factors through $c^{-\nu} \to \inv_{\cB}(k^{-\nu} \otimes c)\to c$.

\end{itemize}  

\sssec{} \label{sss:B action det}

We note that the latter property combined with \eqref{e:identify map} determines the $\cB$-action on all of $c$. 

\medskip

Indeed, for a point $b\in \cB$, the composite map
$$V^\mu\star c^{-\nu-\mu} \overset{\phi_\mu}\longrightarrow c \overset{b}\to c$$ identifies with
$$V^\mu \star c^{-\mu-\nu}\overset{\on{Id}\otimes \phi_0}\longrightarrow V^\mu \star c \simeq
c \otimes \ul{V}^\mu \overset{(-\nu-\mu)(b)\cdot \otimes b\cdot}\longrightarrow c \otimes \ul{V}^\mu \to c,$$
where $(-\nu-\mu)(b)\cdot$ stands for the operation of multiplication by the scalar $(-\nu-\mu)(b)$.

\ssec{Digression: a conceptual explanation}

We will now present, following S.~Raskin, a conceptual meaning of Drinfeld-Pl\"ucker structures, and
of the construction
$$\{c^\mu\}\rightsquigarrow c^{\on{enh}}$$
of the previous subsection. 

\sssec{}

Consider the category
$$\CC\otimes \Rep(\cT)$$
as acted on by $\Rep(\cG)\otimes \Rep(\cT)$.

\medskip

Consider the base affine space $\ol{\cG/\cN}$ of $\cG$, as acted on by $\cG\times \cT$. We can
view $\CO_{\ol{\cG/\cN}}$ as a commutative algebra object in $\Rep(\cG)\otimes \Rep(\cT)$.

\medskip

By definition, the category of families $\{c^\mu\}$ in \secref{sss:DrPl pattern}, denoted $\on{DrPl}(\CC)$, is
$$\CO_{\ol{\cG/\cN}}\mod(\CC\otimes \Rep(\cT)).$$
 
\sssec{}  \label{sss:jmath}

In terms of this interpretation, the construction
$$\{c^\mu\}\rightsquigarrow c^{\on{enh}}$$
is the functor
$$\jmath^*:\on{DrPl}(\CC) \to \Rep(\cB)\underset{\Rep(\cG)}\otimes \CC$$
given by
\begin{multline*}
\CO_{\ol{\cG/\cN}}\mod(\CC\otimes \Rep(\cT))\simeq 
\CO_{\ol{\cG/\cN}}\mod(\Rep(\cG)\otimes \Rep(\cT))\underset{\Rep(\cG)\otimes \Rep(\cT)}\otimes (\CC\otimes \Rep(\cT))\simeq \\
\simeq \QCoh((\ol{\cG/\cN})/(\cG\times \cT))\underset{\Rep(\cG)\otimes \Rep(\cT)}\otimes (\CC\otimes \Rep(\cT))\to \\
\to \QCoh((\cG/\cN)/(\cG\times \cT))\underset{\Rep(\cG)\otimes \Rep(\cT)}\otimes (\CC\otimes \Rep(\cT))\simeq
\Rep(\cB)\underset{\Rep(\cG)\otimes \Rep(\cT)}\otimes (\CC\otimes \Rep(\cT))\simeq \\
\simeq \Rep(\cB)\underset{\Rep(\cG)}\otimes \CC,
\end{multline*}
where the arrow
$$\QCoh((\ol{\cG/\cN})/(\cG\times \cT))\to \QCoh((\cG/\cN)/(\cG\times \cT))$$ 
is pullback with respect to the open embedding 
$$\jmath:\cG/\cN\to \ol{\cG/\cN}.$$

\sssec{}

The functor $\jmath^*$ is the left adjoint to a \emph{fully faithful} functor denoted $\jmath_*$, corresponding to the direct image functor
$$\QCoh((\cG/\cN)/(\cG\times \cT))\to \QCoh((\ol{\cG/\cN})/(\cG\times \cT)).$$

Explicitly, the functor $\jmath_*$ sends to an object $c'\in \Rep(\cB)\underset{\Rep(\cG)}\otimes \CC$
to the family $\{c^\mu\}$ with 
$$c^\mu:=\coind_{\cB\to \cG}(k^\mu\otimes c').$$

\medskip

The transition maps in this family are given by the canonical identifications
$$(V^\mu)^\vee\simeq \coind_{\cB\to \cG}(k^{-\mu}), \quad \mu\in \Lambda^+,$$
where we identify, by definition,
$$V^\mu:=\ind_{\cB\to \cG}(k^\mu).$$

\ssec{The baby Verma object via the Drinfeld-Pl\"ucker formalism}  \label{ss:Dr-Pl baby}

We will now show how the object 
$$\coind_{\on{Lus}^+\to \frac{1}{2}}(k^\clambda)\in \Rep_{q}^{\frac{1}{2}}(G)_{\on{ren}}\simeq \Rep(\cB) 
\underset{\Rep(\cG)}\otimes\Rep_{q}(G)_{\on{ren}}$$
arises following the pattern of \secref{sss:DrPl pattern}.

\sssec{}   \label{sss:baby via DrPl}

We take
\begin{equation} \label{c:quantum c}
c^\mu:=\coind_{\on{Lus}^+\to \on{big}}(k^{\clambda+\mu}),
\end{equation} 
with the maps \eqref{e:Plucker maps} given by the maps 
\begin{multline} \label{e:transition maps quantum}
\on{Frob}_{q}^*((V^{\mu_1})^\vee)\otimes \coind_{\on{Lus}^+\to \on{big}}(k^{\clambda+\mu_2})
\simeq \\
\simeq \coind_{\on{Lus}^+\to \on{big}}\left(\on{Frob}_{q}^*((V^{\mu_1})^\vee)|_{\Rep_q(B)}\otimes k^{\clambda+\mu_2}\right)\to 
\coind_{\on{Lus}^+\to \on{big}}(k^{\clambda-\mu_1+\mu_2})
\end{multline} 
that come from the natural projections
$$\on{Frob}_{q}^*((V^{\mu})^\vee)|_{\Rep_q(B)}\to k^{-\mu}.$$

We claim that the resulting object 
$$c^{\on{enh}}\in \Rep(\cB) \underset{\Rep(\cG)}\otimes\Rep_{q}(G)_{\on{ren}}=\Rep_{q}^{\frac{1}{2}}(G)_{\on{ren}}$$
identifies canonically with $\coind_{\on{Lus}^+\to \frac{1}{2}}(k^\clambda)$. 

\sssec{}

First, note that the functor
\begin{equation} \label{e:big forget}
\Rep(\cB) \underset{\Rep(\cG)}\otimes\Rep_{q}(G)_{\on{ren}} \overset{\oblv_{\cB}}\longrightarrow 
\Vect \underset{\Rep(\cG)}\otimes\Rep_{q}(G)_{\on{ren}} \overset{\coind_\cG}\longrightarrow \Rep_{q}(G)_{\on{ren}}
\end{equation} 
identifies with the composite
$$\Rep_{q}^{\frac{1}{2}}(G)_{\on{ren}} \overset{\oblv_{\frac{1}{2}\to \on{sml}}}\longrightarrow 
\Rep_{q}^{\on{sml}}(G)_{\on{ren}} \overset{\coind_{\on{sml}\to \on{big}}}\longrightarrow \Rep_{q}(G)_{\on{ren}}.$$

\medskip

We have:
\begin{multline}   \label{e:express baby}
\coind_{\on{sml}\to \on{big}}\circ \oblv_{\frac{1}{2}\to \on{sml}}\circ  \coind_{\on{Lus}^+\to \frac{1}{2}}(k^\clambda) \simeq
\coind_{\on{sml}\to \on{big}}\circ  \coind_{\on{sml}^+\to \on{sml}}(k^\clambda) \simeq \\
\simeq \coind_{\on{sml}^+\to \on{big}}(k^{\clambda}) \simeq 
\coind_{\on{Lus}^+\to \on{big}}\circ \coind_{\on{sml}^+\to \on{Lus}^+}(k^{\clambda})  \simeq \\
\simeq \coind_{\on{Lus}^+\to \on{big}}\left(\on{Frob}_{q}^*(\CO_\cB)\otimes k^\clambda\right). 
\end{multline} 

\medskip

We now note that $\CO_\cB\in \Rep(\cB)$ can be written as 
\begin{equation} \label{e:O B}
\underset{\nu\in \Lambda}\oplus\, \underset{\mu\in \Lambda^+}{\on{colim}}\, \, 
\oblv_{\cG\to \cB}(V^\mu) \otimes k^{-\nu-\mu}.
\end{equation} 

Hence, the RHS in \eqref{e:express baby} can be rewritten as 
$$\underset{\nu\in \Lambda}\oplus\, \underset{\mu\in \Lambda^+}{\on{colim}}\,\,  
\coind_{\on{Lus}^+\to \on{big}}\left(\on{Frob}_{q}^*(V^\mu) \otimes k^{\clambda-\nu-\mu}\right),$$
which we finally rewrite as
\begin{equation} \label{e:baby Verma as colimit}
\underset{\nu\in \Lambda}\oplus\, \underset{\mu\in \Lambda^+}{\on{colim}}\,\,  \on{Frob}_{q}^*(V^\mu) \otimes 
\coind_{\on{Lus}^+\to \on{big}}(k^{\clambda-\nu-\mu}).
\end{equation} 

Thus, we have identified the image of $\coind_{\on{Lus}^+\to \frac{1}{2}}(k^\clambda)$ under the functor
\eqref{e:big forget} with the object $c\in \Rep_{q}(G)_{\on{ren}}$ corresponding to the family \eqref{c:quantum c}. 

\sssec{} \label{sss:baby via DrPl verify}

We claim that the Hecke property and the $\cB$-action on 
\begin{equation} \label{e:baby induced}
\coind_{\on{sml}\to \on{big}}\circ \oblv_{\frac{1}{2}\to \on{sml}}\circ  \coind_{\on{Lus}^+\to \frac{1}{2}}(k^\clambda)
\end{equation} 
coincide with those on the colimit \eqref{e:baby Verma as colimit}, specified by the procedures in Sects. \ref{sss:Hecke structure} and \ref{sss:B-action}. 

\medskip

First off, in the formation of the colimit \eqref{e:baby Verma as colimit}, for each $\nu$, we can replace the index set $\{\mu\in \cLambda^+\}$
by its cofinal coset consisting of those $\mu$, for which $\clambda-\nu-\mu$ belongs to $-\cLambda^+\subset \cLambda$. In this case,
the terms appearing in the colimit belong to $(\Rep_{q}(G)_{\on{ren}})^\heartsuit$.  

\medskip

Hence, it suffices to check the corresponding assertions at the level of homotopy categories (i.e., homotopy-coherence is automatic). 

\sssec{}

Now, the fact that Hecke structure on \eqref{e:baby induced} is given, in terms of its presentation as \eqref{e:baby Verma as colimit},
by the procedure of \secref{sss:Hecke structure} follows from the corresponding property of $\CO_\cB$: the isomorphisms
$$\oblv_{\cG\to \cB}(V)\otimes \CO_{\cB} \simeq  \CO_{\cB}\otimes \ul{V}$$
are given by
\begin{multline*} 
\oblv_{\cG\to \cB}(V) \otimes \oblv_{\cG\to \cB}(V^\eta) \otimes k^{-\nu-\eta} 
\simeq \left(\underset{\epsilon}\oplus \, \oblv_{\cG\to \cB}(V^{\eta+\epsilon}) \otimes \ul{V}(\epsilon)\right) \otimes k^{-\nu-\eta} \simeq \\
\simeq \underset{\epsilon}\oplus \, \oblv_{\cG\to \cB}(V^{\eta+\epsilon})  \otimes k^{(-\nu+\epsilon)-\eta-\epsilon}\otimes \ul{V}(\epsilon).
\end{multline*}

\sssec{}

The fact that the $\cB$-actions agree follows from the fact that (at the level of homotopy categories), the Hecke structure determines
the $\cB$-action (see \secref{sss:B action det}), combined with the fact that the maps 
\begin{multline*} 
\coind_{\on{Lus}^+\to \on{big}}(k^{\clambda-\nu}) \to  \coind_{\on{Lus}^+\to \on{big}}\circ \coind_{\on{sml}^+\to \on{Lus}^+}(k^{\clambda}) \simeq \\
\simeq \coind_{\on{sml}\to \on{big}}\circ \oblv_{\frac{1}{2}\to \on{sml}}\circ  \coind_{\on{Lus}^+\to \frac{1}{2}}(k^{\clambda})
\end{multline*}
identify with 
\begin{multline*} 
\inv_{\cB}\left(k^{-\nu} \otimes \coind_{\on{sml}\to \on{big}}\circ \oblv_{\frac{1}{2}\to \on{sml}}\circ  \coind_{\on{Lus}^+\to \frac{1}{2}}(k^\clambda)\right)\to \\
\to \coind_{\on{sml}\to \on{big}}\circ \oblv_{\frac{1}{2}\to \on{sml}}\circ  \coind_{\on{Lus}^+\to \frac{1}{2}}(k^\clambda).
\end{multline*}

\ssec{The semi-infinite IC sheaf}  \label{ss:semiinf IC}

We will now start the process of transferring the baby Verma module to the Kac-Moody side, i.e., we would like to describe the
object of
$$\Rep(\cB)\underset{\Rep(\cG)}\otimes \KL(G,-\kappa)$$ 
corresponding under $\sF_{-\kappa}$ to
$$\coind_{\on{Lus}^+\to \frac{1}{2}}(k^\clambda)\in \Rep_{q}^{\frac{1}{2}}(G)_{\on{ren}}\simeq \Rep(\cB) 
\underset{\Rep(\cG)}\otimes\Rep_{q}(G)_{\on{ren}}.$$

A key tool for this will be a certain geometric object, introduced in \cite{Ga2} under the same ``semi-infinite IC sheaf". 

\sssec{}

Let $\ICs_{-\kappa}\in \Dmod_{-\kappa}(\Gr_G)^{T(\CO)}$ be the object introduced in \cite[Sect. 2.3]{Ga2}. Explicitly,
$$\underset{\mu\in \Lambda^+}{\on{colim}}\, t^{-\mu}\cdot \on{Sat}(V^\mu)[\langle \mu,2\check\rho\rangle].$$

\medskip

Let $\ICsd_{-\kappa}$ be its graded version, i.e.,
$$\underset{\nu\in \Lambda}\oplus\, t^{-\nu}\cdot\ICs_\kappa.$$

\medskip

It was shown in \cite[Sect. 6]{Ga2}, the object $c=\ICsd_{-\kappa}$ can be obtained by the Drifeld-Pl\"ucker formalism of \secref{sss:DrPl pattern}
starting from the collection of objects
$$c^\mu:=\delta_{t^\mu,\Gr}[\langle -\mu,2\check\rho\rangle]\in \Dmod_{-\kappa}(\Gr_G)^{T(\CO)},$$
and the maps
\begin{equation} \label{e:transition maps semiinf Gr}
\delta_{t^{\mu_2},\Gr}[\langle -\mu_2,2\check\rho\rangle] \star \on{Sat}((V^{\mu_1})^\vee)\to \delta_{t^{-\mu_1+\mu_2},\Gr}[\langle \mu_1-\mu_2,2\check\rho\rangle],
\end{equation} 
that come by adjunction from the canonical maps
$$\delta_{t^\mu,\Gr}[\langle -\mu,2\check\rho\rangle]\to \Sat(V^\mu).$$

\sssec{}

In particular, $\ICsd_{-\kappa}$ carries a Hecke structure and a $\cB$-action. Let $(\ICsd_{-\kappa})^{\on{enh}}$ denote the resulting object of 
$$\Rep(\cB)\underset{\Rep(\cG)}\otimes\Dmod_{-\kappa}(\Gr_G)^{T(\CO)}.$$

%\medskip
%
%The following is established in {\it loc.cit.}:
%
%\begin{thm}  \label{t:inv in semiinf}
%The tautological map $\delta_{1,\Gr}\to \ICsd_{-\kappa}$ identifies 
%$$\delta_{1,\Gr}\simeq \inv_{\cB}(\ICsd_{-\kappa}).$$
%\end{thm} 

\sssec{}

A key observation is that $\ICsd_{-\kappa}$ is $N(\CK)$-equivariant, see \cite[Proposition 2.3.7(a)]{Ga2}. 
Hence, we obtain that
$\ICsd_{-\kappa}$ is naturally an object of $\Dmod_{-\kappa}(\Gr_G)^{N(\CK)\cdot T(\CO)}$. 

\medskip

Since the forgetful functor
$$\Dmod_{-\kappa}(\Gr_G)^{N(\CK)\cdot T(\CO)}\to \Dmod_{-\kappa}(\Gr_G)^{T(\CO)}$$
is fully faithful, we obtain that $(\ICsd_{-\kappa})^{\on{enh}}$ is naturally an object of
$$\Rep(\cB)\underset{\Rep(\cG)}\otimes\Dmod_{-\kappa}(\Gr_G)^{N(\CK)\cdot T(\CO)}.$$

\ssec{The semi-infinite IC sheaf, Iwahori version}

\sssec{}

Recall now that according to \cite[Sect. 5]{Ga2}, the functor
$$\on{Av}^{I^-}_*: \Dmod_{-\kappa}(\Gr_G)^{T(\CO)}\to \Dmod_{-\kappa}(\Gr_G)^{I^-\cdot T(\CO)},$$
when restricted to
$$\Dmod_{-\kappa}(\Gr_G)^{N(\CK)\cdot T(\CO)}\subset  \Dmod_{-\kappa}(\Gr_G)^{T(\CO)}$$
defines an equivalence onto
$$\Dmod_{-\kappa}(\Gr_G)^I\subset \Dmod_{-\kappa}(\Gr_G)^{I^-\cdot T(\CO)}.$$

\medskip

Set
$$\overset{\bullet}\CF{}^\semiinf_{-\kappa}:=\on{Av}^{I^-}_*(\ICsd_{-\kappa}) \in \Dmod_{-\kappa}(\Gr_G)^I.$$

It is equipped with a Hecke structure and a compatible action of $\cB$. 
Let $(\overset{\bullet}\CF{}^\semiinf_{-\kappa})^{\on{enh}}$ denote the resulting object of 
$$\Rep(\cB)\underset{\Rep(\cG)}\otimes\Dmod_{-\kappa}(\Gr_G)^I.$$

\medskip

The following results from the construction of $\overset{\bullet}\CF{}^\semiinf_{-\kappa}$ in \cite[Sect. 4.5.1]{Ga2}:

\begin{thm} \label{t:inv in semiinf}
The tautological map $\delta_{1,\Gr}\to \overset{\bullet}\CF{}^\semiinf_{-\kappa}$ defines an isomorphism
$$\delta_{1,\Gr}\simeq \inv_{\cB}(\overset{\bullet}\CF{}^\semiinf_{-\kappa}).$$
\end{thm} 

\sssec{}

Note that by construction, $(\overset{\bullet}\CF{}^\semiinf_{-\kappa})^{\on{enh}}$, viewed as an object of 
$$\Rep(\cB)\underset{\Rep(\cG)}\otimes\Dmod_{-\kappa}(\Gr_G)^{T(\CO)},$$
can be obtained by the procedure of \secref{sss:DrPl pattern} 
applied to the family of objects
$$c^\mu:=\on{Av}^{I^-}_*(\delta_{t^{\mu,\Gr}})[-\langle \mu,2\check\rho\rangle]\in \Dmod_{-\kappa}(\Gr_G)^{T(\CO)},$$
and the maps 
$$\on{Av}^{I^-}_*(\delta_{t^{\mu_2},\Gr})[\langle -\mu_2,2\check\rho\rangle]\star 
\on{Sat}((V^{\mu_1})^\vee)\to \on{Av}^{I^-}_*(\delta_{t^{\mu_1+\mu_2},\Gr})[\langle \mu_1-\mu_2,2\check\rho\rangle],$$  
induced by \eqref{e:transition maps semiinf Gr}.

\sssec{}

Consider now another family, namely,
$$'\!c^\mu:=J_\mu\star \delta_{1,\Gr},$$ 
and the maps
$$J_{\mu_2}\star \delta_{1,\Gr} \star \on{Sat}((V^{\mu_1})^\vee)\to J_{-\mu_1+\mu_2}\star \delta_{1,\Gr}$$
that come by adjunction from the canonical maps
$$J_{\mu}\star  \delta_{1,\Gr} \simeq j_{\mu,!}\star \delta_{1,\Gr}\simeq \pi_*(j_{\mu,!})
 \to \on{Sat}(V^\mu), \quad \mu\in \cLambda^+,$$
where $\pi:\Fl^{\on{aff}}_G\to \Gr_G$. 

\medskip

We claim that the resulting object 
$$'\!c^{\on{enh}}\in \Rep(\cB)\underset{\Rep(\cG)}\otimes\Dmod_{-\kappa}(\Gr_G)^{T(\CO)}$$
will be canonically isomorphic to
$$c^{\on{enh}}=(\overset{\bullet}\CF{}^\semiinf_{-\kappa})^{\on{enh}}.$$

\medskip

This follows from the fact that if $\mu\in \Lambda^+$, we have an evident identification, 
$$'\!c^{-\mu}\simeq  c^{-\mu},$$  
whereas $\Lambda^+\subset \Lambda$ is cofinal. 

\medskip

Since the forgetful functor $ \Dmod_{-\kappa}(\Gr_G)^I\to  \Dmod_{-\kappa}(\Gr_G)^{T(\CO)}$ is fully faithful, we obtain 
that the isomorphism $c^{\on{enh}}\simeq {}'\!c^{\on{enh}}$ holds also in 
$$\Rep(\cB)\underset{\Rep(\cG)}\otimes\Dmod_{-\kappa}(\Gr_G)^I$$

\sssec{}

In particular, we obtain an isomorphism 
$$\overset{\bullet}\CF{}^\semiinf_{-\kappa} \simeq \underset{\nu}\oplus\, \underset{\mu}{\on{colim}}\, j_{-\nu-\mu,*}\star \on{Sat}(V^\mu),$$
where the colimit is taken over the set of those $\mu$ for which $\nu+\mu\in \Lambda^+$. 

\medskip

The latter presentation makes it clear that $\overset{\bullet}\CF{}^\semiinf_{-\kappa}$ lies in $(\Dmod_{-\kappa}(\Gr_G)^I)^\heartsuit[d]$:
indeed, for $\nu\in \Lambda^{++}$, we have
$$J_{-\nu}\star \delta_{1,\Gr}\simeq j_{-\mu,*}\star \delta_{1,\Gr}\in (\Dmod_{-\kappa}(\Gr_G)^I)^\heartsuit[d],$$
while convolution with objects of the form $\on{Sat}(V)$ for $V\in \Rep(\cG)^\heartsuit$ is a t-exact endo-functor of $\Dmod_{-\kappa}(\Gr_G)^I$. 

\sssec{}

Let $\on{invol}$ denote the equivalence
\begin{equation} \label{e:inversion}
\Dmod_{-\kappa}(\Gr_G)^I\simeq \Dmod_{\kappa}(\Fl^{\on{aff}}_G)^{G(\CO)}
\end{equation}
given by the inversion on $G(\CK)$. We normalize it so that
$$\on{invol}(\delta_{1,\Gr})\simeq \pi^*(\delta_{1,\Gr}).$$

\medskip

Note that
$$\on{invol}(\CF\star \Sat(V))\simeq \Sat(V^\tau) \star \on{invol}(\CF), \quad V\in \Rep(\cG),$$
where $V^\tau$ is obtained by the action of the Cartan involution $\tau$ on $V$
(in particular, $(V^\mu)^\tau\simeq V^{-w_0(\tau)}$).

\medskip

Denote
$$\overset{\bullet}\CF{}^{\semiinf,\on{invol}}_\kappa:=\on{invol}(\overset{\bullet}\CF{}^\semiinf_{-\kappa})\in \Dmod_{\kappa}(\Fl^{\on{aff}}_G)^{G(\CO)}.$$

This object has a natural Hecke structure, and a compatible $\cB$-action. Explicitly,
$$\overset{\bullet}\CF{}^{\semiinf,\on{invol}}_\kappa \simeq \underset{\nu}\oplus\, \underset{\mu}{\on{colim}}\, 
\on{Sat}(V^{-w_0(\mu)})\star j_{\nu+\mu,*},$$
where the colimit is taken over the set of those $\mu$ for which $\nu+\mu\in \Lambda^+$.

\medskip

Denote by 
$(\overset{\bullet}\CF{}^{\semiinf,\on{invol}}_\kappa)^{\on{enh}}$ the corresponding object of 
$$\Rep(\cB)\underset{\Rep(\cG)}\otimes\Dmod_{\kappa}(\Fl^{\on{aff}}_G)^{G(\CO)}.$$

\medskip

From \thmref{t:inv in semiinf} we obtain:

\begin{cor} \label{c:inv in semiinf tau}
The tautological map $\pi^*(\delta_{1,\Gr})\to \overset{\bullet}\CF{}^{\semiinf,\on{invol}}_\kappa$ defines an isomorphism
$$\pi^*(\delta_{1,\Gr})\simeq \inv_{\cB}(\overset{\bullet}\CF{}^{\semiinf,\on{invol}}_\kappa).$$
\end{cor} 

\sssec{}   \label{sss:semiinf tau family}

By construction, $\overset{\bullet}\CF{}^{\semiinf,\on{invol}}_\kappa$,
equipped with the above pieces of structure, arises by the procedure of \secref{sss:DrPl pattern} 
from the collection of objects
$$c^\mu:=\on{Av}_*^{G(\CO)/I}(J^{\BD}_{-\mu})\in \Dmod_{\kappa}(\Fl^{\on{aff}}_G)^{G(\CO)}$$
where we regard $\Dmod_{\kappa}(\Fl_G)^{G(\CO)}$ as acted on by $\Rep(\cG)$ via $\on{Sat}\circ \tau$. 
The corresponding maps
\begin{equation} \label{e:tau maps}
\on{Sat}((V^{-w_0(\mu_1)})^\vee)\star \on{Av}_*^{G(\CO)/I}(J^{\BD}_{-\mu_2}) \to \on{Av}_*^{G(\CO)/I}(J^{\BD}_{\mu_1-\mu_2}), \quad \mu_1\in \Lambda^+
\end{equation} 
are obtained by adjunction from the canonical maps
$$\on{Av}_*^{G(\CO)/I}(J^{\BD}_{-\mu})=\on{Av}_*^{G(\CO)/I}(j_{-\mu,!})\to \pi^*(\on{Sat}(V^{-w_0(\mu)})), \quad \mu\in \Lambda^+.$$

\ssec{The baby Verma object via Wakimoto modules: positive level case}

In this subsection we will be able to carry out the program indicated in the preamble to \secref{ss:semiinf IC} but for the Kazhdan-Lusztig
category at the positive level. 

\medskip

This will lead to the proof of \thmref{t:Weyl positive}, and by duality, to that of \thmref{t:Weyl}.   

\sssec{}

Recall that the (dual) Kazhdan-Lusztig equivalence
$$\sF_{\kappa}:\KL(G,\kappa)\simeq \Rep_{q^{-1}}(G)_{\on{ren}}$$
respects the action of $\Rep(\cG)$, where $\Rep(\cG)$ acts on $\KL(G,\kappa)$ via $\on{Sat}\circ \tau$.

\medskip

Hence, it induces an equivalence
\begin{equation} \label{e:KL B}
\Rep(\cB) \underset{\Rep(\cG)}\otimes\KL(G,\kappa)\simeq  
\Rep(\cB) \underset{\Rep(\cG)}\otimes \Rep_{q^{-1}}(G)_{\on{ren}},
\end{equation}
where we identify the right-hand side with $\Rep_{q^{}-1}^{\frac{1}{2}}(G)_{\on{ren}}$.

\medskip

Recall that for $\clambda\in \cLambda$, we have the object
$$\coind_{\on{Lus}^+\to \frac{1}{2}}(k^\clambda)\in \Rep_{q^{-1}}^{\frac{1}{2}}(G)_{\on{ren}}.$$

\medskip

In this subsection we will identify the image of this object under the equivalence \eqref{e:KL B}. 

\sssec{}

Consider the object
$$(\overset{\bullet}\CF{}^{\semiinf,\on{invol}}_\kappa)^{\on{enh}}\star \BD(\BW_{-\kappa}^{-\clambda})\in 
\Rep(\cB) \underset{\Rep(\cG)}\otimes\KL(G,\kappa).$$

We claim:

\begin{thm}  \label{t:identify baby pos}
The object $(\overset{\bullet}\CF{}^{\semiinf,\on{invol}}_\kappa)^{\on{enh}}\star \BD(\BW_{-\kappa}^{-\clambda})$ 
corresponds to $\coind_{\on{Lus}^+\to \frac{1}{2}}(k^\clambda)$
under the equivalence \eqref{e:KL B}.
\end{thm} 

\begin{rem}
The proof of \thmref{t:identify baby pos} that we give below will be ``artificial" in that it will result from some 
explicit calculation. However, later on, in \secref{s:AB action}, we will show ``why" \thmref{t:identify baby pos} 
should hold, based on an additional property of the conjectural equivalence \eqref{e:main functor}.
\end{rem} 

\sssec{}  \label{sss:proof of Weyl}

Before we prove \thmref{t:identify baby pos}, let us show that it implies \thmref{t:Weyl positive}, and hence \thmref{t:Weyl}. Indeed, 
consider the objects
$$\overset{\bullet}\CF{}^{\semiinf,\on{invol}}_\kappa \star \BD(\BW_{-\kappa}^{-\clambda})\in \KL(G,\kappa)$$
and
$$\coind_{\on{sml}\to \on{big}}\circ \oblv_{\frac{1}{2}\to \on{sml}}\circ  \coind_{\on{Lus}^+\to \frac{1}{2}}(k^\clambda)\in 
\Rep_{q^{-1}}(G)_{\on{ren}},$$
regarded as equipped with $\cB$-actions.

\medskip

According to \thmref{t:identify baby pos}, these two objects correspond to one another under the equivalence $\sF_\kappa$.
Hence, so do the corresponding objects obtained by applying $\inv_\cB$.

\medskip

However,
$$\inv_\cB\left(\coind_{\on{sml}\to \on{big}}\circ \oblv_{\frac{1}{2}\to \on{sml}}\circ  \coind_{\on{Lus}^+\to \frac{1}{2}}(k^\clambda)\right)\simeq
\coind_{\on{Lus}^+\to \on{big}}(k^\clambda)=:\CV^{\vee,w_0(\clambda)}_{q^{-1}},$$
while according to \corref{c:inv in semiinf tau}, we have
\begin{multline*} 
\inv_\cB\left(\overset{\bullet}\CF{}^{\semiinf,\on{invol}}_\kappa \star \BD(\BW_{-\kappa}^{-\clambda})\right)\simeq
\inv_\cB\left(\overset{\bullet}\CF{}^{\semiinf,\on{invol}}_\kappa\right) \star \BD(\BW_{-\kappa}^{-\clambda})\simeq \simeq \\
\simeq \pi^*(\delta_{1,\Gr})\star \BD(\BW_{-\kappa}^{-\clambda})\simeq \on{Av}^{G(\CO)/I}_*(\BD(\BW_{-\kappa}^{-\clambda}))=:
\BV_{\kappa}^{w_0(\clambda)},
\end{multline*} 
as desired.

\qed

\sssec{}

The rest of this subsection is devoted to the proof of \thmref{t:identify baby pos}. We will show that the image of 
$(\overset{\bullet}\CF{}^{\semiinf,\on{invol}}_\kappa)^{\on{enh}}\star \BD(\BW_{-\kappa}^{-\clambda})$ under the functor 
\eqref{e:KL B} and $\coind_{\on{Lus}^+\to \frac{1}{2}}(k^\clambda)$ are obtained by the Drinfeld-Pl\"ucker formalism of \secref{sss:DrPl pattern} 
from equivalent families of objects $c^\mu$ and $'\!c^\mu$. 

\medskip

On the one hand, by \secref{sss:semiinf tau family}, the object $(\overset{\bullet}\CF{}^{\semiinf,\on{invol}}_\kappa)^{\on{enh}}\star \BD(\BW_{-\kappa}^{-\clambda})$
is obtained from the family of objects
$$c^\mu:=\on{Av}_*^{G(\CO)/I}(J^{\BD}_{-\mu})\star \BD(\BW_{-\kappa}^{-\clambda})=:\BV^{w_0(\clambda+\mu)}_\kappa,$$
where the transition maps
\begin{equation} \label{e:transition map KL}
\on{Sat}((V^{-w_0(\mu_1)})^\vee)\star \BV^{w_0(\clambda+\mu_2)}_\kappa\to \BV^{w_0(\clambda-\mu_1+\mu_2)}_\kappa
\end{equation} 
are obtained from the maps \eqref{e:tau maps} by convolution.

\medskip

On the other hand, by Sects. \ref{sss:baby via DrPl}-\ref{sss:baby via DrPl verify}, the object $\coind_{\on{Lus}^+\to \frac{1}{2}}(k^\clambda)$
is obtained from the collection of objects
$$'\!c^\mu:=\coind_{\on{Lus}^+\to \on{big}}(k^{\clambda+\mu})=:\CV_{q^{-1}}^{\vee,w_0(\clambda+\mu)},$$
and the transtition maps 
\begin{multline}  \label{e:tran map quantum}
\on{Frob}_{q^{-1}}^*((V^{\mu_1})^\vee)\otimes \coind_{\on{Lus}^+\to \on{big}}(k^{\clambda+\mu_2}) \simeq \\
\simeq \coind_{\on{Lus}^+\to \on{big}}\left(\on{Frob}_{q}^*((V^{\mu_1})^\vee)|_{\Rep_q(B)}\otimes k^{\clambda+\mu_2}\right)
\to \coind_{\on{Lus}^+\to \on{big}}(k^{\clambda-\mu_1+\mu_2}).
\end{multline} 

\medskip

It suffices to show that the families $\sF_\kappa(c^{-\mu})$ and $'\!c^{-\mu}$ can be identified 
for $\mu$ running over a subset cofinal in $\Lambda$. 
We take this subset to consist of those $\mu$ for which $\clambda-\mu$ is anti-dominant. 

\sssec{}  \label{sss:verify pos 1}

Indeed, if $\clambda-\mu$ is anti-dominant, we do know that
$$\sF_\kappa(\BV^{w_0(\clambda-\mu)}_\kappa)\simeq \CV_{q^{-1}}^{\vee,w_0(\clambda-\mu)}.$$

Since these objects lie in $(\Rep_{q^{-1}}(G)_{\on{ren}})^\heartsuit$, the compatibility with the
transition maps is sufficient to check at the level of 1-morphisms (i.e., the compatibility with
homotopy-coherence is automatic). 

\sssec{}

On the one hand, we note that the maps \eqref{e:tran map quantum} are obtained by duality from the maps \eqref{e:Frob with Weyl}
$$\CV_q^{\mu_1-(\clambda+\mu_2)} \to \on{Frob}_q^*(V^{\mu_1})\otimes \CV_q^{-(\clambda+\mu_2)},$$
which are the images under $\sF_{-\kappa}$ of the maps \eqref{e:sph with Weyl}
\begin{equation} \label{e:transition map KL again}
\BV_{-\kappa}^{\mu_1-(\clambda+\mu_2)} \to \Sat(V^{\mu_1}) \star \BV_{-\kappa}^{-(\clambda+\mu_2)}.
\end{equation}

\medskip

On the other hand, the transition maps \eqref{e:transition map KL} are obtained by duality from the maps
\begin{equation} \label{e:transition map KL again again}
\on{Av}^{G(\CO)/I}_!(\BW_{-\kappa}^{\mu_1-\clambda-\mu_2}) \simeq 
\on{Av}^{G(\CO)/I}_!(J_{\mu_1}\star \BW_{-\kappa}^{-\clambda-\mu_2}) \to 
\on{Sat}(V^{\mu_1})\star \on{Av}^{G(\CO)/I}_!(\BW_{-\kappa}^{-\clambda-\mu_2}),
\end{equation}
induced by the maps 
\begin{equation} \label{e:st to sph}
\on{Av}^{G(\CO)/I}_!(J_{\mu})=\on{Av}^{G(\CO)/I}_!(j_{\mu,!}) \to \pi^!(\on{Sat}(V^{\mu})), \quad \mu\in \Lambda^+.
\end{equation}

\sssec{} \label{sss:verify pos 3}

Recall that for $\clambda'\in \cLambda^+$ we identify $\on{Av}^{G(\CO)/I}_!(\BW_{-\kappa}^{\clambda'})$ with $\BV_{-\kappa}^{\clambda'}$ via
$$\BW_{-\kappa}^{\clambda'}\simeq \BM_{-\kappa}^{\clambda'}$$
and
\begin{multline*} 
\on{Av}^{G(\CO)/I}_!(\BM_{-\kappa}^{\clambda'})=
\on{Av}^{G(\CO)/I}_!\left(\on{Ind}_{\fg(\CO)}^{\hg_{-\kappa}}(M^{\clambda'})\right)\simeq
\on{Ind}_{\fg(\CO)}^{\hg_{-\kappa}}\left(\on{Av}^{G(\CO)/I}_!(M^{\clambda'})\right)\simeq \\
\simeq \on{Ind}_{\fg(\CO)}^{\hg_{-\kappa}}\left(\on{Av}^{G/B}_!(M^{\clambda'})\right)\simeq 
\on{Ind}_{\fg(\CO)}^{\hg_{-\kappa}}(V^{\clambda'})\simeq \BV_{-\kappa}^{\clambda'}.
\end{multline*} 

So we need to show that for $\clambda'\in \cLambda^+$ and $\mu\in \Lambda^+$, the map
\begin{multline*} 
\BV_{-\kappa}^{\mu+\clambda'}\simeq \on{Av}^{G(\CO)/I}_!(\BM_{-\kappa}^{\mu+\clambda'})\simeq
\on{Av}^{G(\CO)/I}_!(j_{\mu,!} \star j_{-\mu,*}\star \BM_{-\kappa}^{\mu+\clambda'}) \overset{\text{\eqref{e:twist Verma}}}\longrightarrow \\
\to \on{Av}^{G(\CO)/I}_!(j_{\mu,!} \star \BM_{-\kappa}^{\clambda'})  \overset{\text{\eqref{e:st to sph}}}\longrightarrow 
\on{Sat}(V^\mu) \star \on{Av}^{G(\CO)/I}_!(\BM_{-\kappa}^{\clambda'})\simeq \on{Sat}(V^\mu) \star  \BV_{-\kappa}^{\clambda'}
\end{multline*}
coincides with the canonical map
$$\BV_{-\kappa}^{\mu+\clambda'} \to \on{Sat}(V^\mu) \star  \BV_{-\kappa}^{\clambda'}.$$

However, this follows by unwinding the definitions by tracking the image of the highest weight vector. 

\ssec{The baby Verma object via Wakimoto modules: the negative level case}

Even though we have already proved \thmref{t:Weyl}, we would now like it to prove it more directly,
without appealing to positive vs negative level duality. 

\medskip

For this we will need to carry out the program
indicated in the preamble to \secref{ss:semiinf IC} directly in the negative level case. 

\sssec{}

Consider the original Kazhdan-Lusztig equivalence
$$\sF_{-\kappa}:\KL(G,-\kappa)\simeq \Rep_q(G)_{\on{ren}},$$
and the induced equivalence 

\medskip

It induces an equivalence
\begin{equation} \label{e:KL B neg}
\Rep(\cB) \underset{\Rep(\cG)}\otimes\KL(G,-\kappa)\simeq  
\Rep(\cB) \underset{\Rep(\cG)}\otimes \Rep_q(G)_{\on{ren}}.
\end{equation}

In this subsection we will identify the image of
$$\ind_{\on{Lus}^+\to \frac{1}{2}}(k^\clambda)\in \Rep_q^{\frac{1}{2}}(G)_{\on{ren}}\simeq 
\Rep(\cB) \underset{\Rep(\cG)}\otimes \Rep_q(G)_{\on{ren}}$$
as an object of $\Rep(\cB) \underset{\Rep(\cG)}\otimes\KL(G,-\kappa)$ under the equivalence \eqref{e:KL B neg}. 

\sssec{}  \label{sss:F-}

Consider the following variant of the object $(\overset{\bullet}\CF{}^{\semiinf,\on{invol}}_\kappa)^{\on{enh}}$, denoted
$$(\overset{\bullet}\CF{}^{-,\semiinf,\on{invol}}_{-\kappa})^{\on{enh}}\in \Rep(\cB)\underset{\Rep(\cG)}\otimes\Dmod_{-\kappa}(\Fl^{\on{aff}}_G)^{G(\CO)}.$$

Namely, 
$$(\overset{\bullet}\CF{}^{\semiinf,\on{invol}}_{-\kappa})^{\on{enh}}:=\left((\overset{\bullet}\CF{}^{\semiinf,\on{invol}}_{-\kappa})^{\on{enh}}\star j_{w_0,*}\right)^\tau[d],$$
where $\tau$ is the Cartan involution on $\cG$ (normalized so that it preserves $\cB$ and acts as $\mu\mapsto -w_0(\mu)$ on the weights).  

\medskip

Explicitly, $(\overset{\bullet}\CF{}^{-,\semiinf,\on{invol}}_{-\kappa})^{\on{enh}}$ is obtained by the Drinfeld-Pl\"ucker formalism of \secref{sss:DrPl pattern} from
the collection of objects
$$c^\mu:=\on{Av}^{G(\CO)/I}_!(J_\mu),$$
and the transition maps 
\begin{equation} \label{e:trans maps neg}
\Sat((V^{\mu_1})^\vee) \star \on{Av}^{G(\CO)/I}_!(J_{\mu_2})\to \on{Av}^{G(\CO)/I}_!(J_{-\mu_1+\mu_2})
\end{equation}
coming by adjunction from the maps
$$\on{Av}^{G(\CO)/I}_!(J_\mu) =\on{Av}^{G(\CO)/I}_!(j_{\mu,!}) \to \pi^!(\Sat(V^\mu)), \quad \mu\in \Lambda^+.$$

The corresponding object $\overset{\bullet}\CF{}^{-,\semiinf,\on{invol}}_{-\kappa}\in \Dmod_{-\kappa}(\Fl^{\on{aff}}_G)^{G(\CO)}$ is given by
$$\underset{\nu\in \Lambda}\oplus\, \underset{\mu\in \Lambda^+}{\on{colim}}\, \on{Sat}(V^\mu)\star \on{Av}^{G(\CO)/I}_!(J_{-\mu-\nu}).$$

\sssec{}

We claim:

\begin{thm} \label{t:identify baby neg}
Under the equivalence \eqref{e:KL B neg}, the object $\ind_{\on{Lus}^+\to \frac{1}{2}}(k^\clambda)$ corresponds to
$$(\overset{\bullet}\CF{}^{-,\semiinf,\on{invol}}_{-\kappa})^{\on{enh}}\star J_{-2\rho}[-d]\star \BW_{-\kappa}^{\clambda}.$$
\end{thm}

\sssec{}

Before giving the proof, let us note that \thmref{t:identify baby neg} gives 
another proof for \thmref{t:Weyl}.

\medskip

Indeed, let us apply $\inv_{\cB}$ to 
$$\coind_{\cG}\circ \oblv_{\cB}(\ind_{\on{Lus}^+\to \frac{1}{2}}(k^\clambda))$$
and
$$\overset{\bullet}\CF{}^{-,\semiinf,\on{invol}}_{-\kappa}\star J_{-2\rho}[-d]\star \BW_{-\kappa}^{\clambda}.$$

\medskip

On the quantum group side we get 
\begin{multline*}
\coind_{\frac{1}{2}\to \on{big}}\circ \ind_{\on{Lus}^+\to \frac{1}{2}}(k^\clambda)  \simeq 
\coind_{\frac{1}{2}\to \on{big}}\circ \coind_{\on{Lus}^+\to \frac{1}{2}}(k^{\clambda+2\check\rho-2\rho})  \simeq  \\
\simeq \coind_{\on{Lus}^+\to \on{big}}(k^{\clambda+2\check\rho-2\rho})\simeq \ind_{\on{Lus}^+\to \on{big}}(k^{\clambda-2\rho})[-d].
\end{multline*}

\medskip

On the Kac-Moody side, by \corref{c:inv in semiinf tau}, we get
\begin{multline*}
\on{Av}^{G(\CO)/I}_*(j_{w_0,*} \star J_{-2\rho}[-d]\star \BW_{-\kappa}^{\clambda})[d]\simeq \on{Av}^{G(\CO)/I}_*(j_{w_0,*} \star J_{-2\rho}\star \BW_{-\kappa}^{\clambda})
\simeq \\
\simeq \on{Av}^{G(\CO)/I}_*(j_{w_0,*} \star \BW_{-\kappa}^{\clambda-2\rho})
\simeq \pi^!(\delta_{1,\Gr})\star \BW_{-\kappa}^{\clambda-2\rho}[-d]
\simeq \on{Av}^{G(\CO)/I}_!(\BW_{-\kappa}^{\clambda-2\rho})[-d].
\end{multline*}

Thus, we obtain
$$\sF_{-\kappa}(\on{Av}^{G(\CO)/I}_!(\BW_{-\kappa}^{\clambda-2\rho}))[-d]\simeq \ind_{\on{Lus}^+\to \on{big}}(k^{\clambda-2\rho})[-d],$$
as desired.

\qed

\sssec{}  \label{sss:duality Plucker}

The rest of the subsection is devoted to the proof of \thmref{t:identify baby neg}. We will deduce it from 
\thmref{t:identify baby neg} using an idea involving duality. On the first pass, this will look like an 
artificial procedure, but in \secref{ss:Dr-Pl and duality} we will explain a conceptual framework that it fits in. 

\medskip

Let us start with a family of objects $\{c^\mu\}\in \CC$ as in \secref{sss:DrPl pattern}; let us assume $\CC$
is compactly generated, and that the objects $c^\mu$ are compact.

\medskip

We define the dual family $\{c^{\vee,\mu}\}\in \CC^\vee$ by setting
$$c^{\vee,\mu}:=\BD(c^{-\mu+2\rho})[-d],$$
where $\BD$ denotes the canonical contravariant equivalence $\CC^c\to (\CC^\vee)^c$. 
Let $c^{\vee,\on{enh}}$ denote the resulting object of $\Rep(\cB)\underset{\Rep(\cG)}\otimes \CC^\vee$. 

\sssec{}

We apply the above procedure first to $\CC=\KL(G,\kappa)$ and $c^\mu$ being 
$$\on{Av}_*^{G(\CO)/I}(J^{\BD}_{-\mu})\star \BD(\BW_{-\kappa}^{-\clambda})$$
and the transition maps obtained from the maps \eqref{e:tau maps} by convolution.

\medskip

Note that the corresponding dual family $c^{\vee,\mu}$ is
$$\on{Av}^{G(\CO)/I}_!(J_{\mu-2\rho})\star \BW_{-\kappa}^{-\clambda}[-d]$$
and the transition maps obtained from the maps \eqref{e:trans maps neg} by convolution.
Hence the resulting object $c^{\vee,\on{enh}}$ identifies with 
$$(\overset{\bullet}\CF{}^{-,\semiinf,\on{invol}}_{-\kappa})^{\on{enh}}\star J_{-2\rho}[-d]\star \BW_{-\kappa}^{-\clambda}.$$

\sssec{}  \label{sss:identify dual system}

We now perform the same procedure on the quantum group side. Take $\CC:=\Rep_{q^{-1}}(G)_{\on{ren}}$. 
We start with the family $c^\mu$ 
$$\coind_{\on{Lus}^+\to \on{big}}(k^{\clambda+\mu})$$
and the transition maps \eqref{e:transition maps quantum}. This family corresponds to the one on the Kac-Moody
side under $\sF_\kappa$, by Sects. \ref{sss:verify pos 1}-\ref{sss:verify pos 3}. 

\medskip

Consider the dual family $c^{\vee,\mu}$. We obtain that in order to prove \thmref{t:identify baby neg} it suffices to show that
the corresponding object 
$$c^{\vee,\on{enh}}\in \Rep(\cB) \underset{\Rep(\cG)}\otimes \Rep_q(G)_{\on{ren}}\simeq \Rep_q^{\frac{1}{2}}(G)_{\on{ren}}$$
identifies with $\ind_{\on{Lus}^+\to \frac{1}{2}}(k^{-\clambda})$.

\medskip

We will give two proofs of this fact. One, more direct but less conceptual, right below, and an essentially equivalent
but more conceptual one in \secref{sss:Dr-Pl and duality}. 

\sssec{}

Recall that according to \corref{c:ind and coind from Borel to small}, 
$$\ind_{\on{Lus}^+\to \frac{1}{2}}(k^{-\clambda})\simeq \coind_{\on{Lus}^+\to \frac{1}{2}}(k^{-\clambda+2\check\rho-2\rho}),$$
and according to \ref{sss:baby via DrPl}-\ref{sss:baby via DrPl verify} it corresponds to 
$'\!c^{\on{enh}}$ given by the family
$$'\!c^\mu:=\coind_{\on{Lus}^+\to \on{big}}(k^{-\clambda+\mu+2\check\rho-2\rho})$$
and transition maps \eqref{e:transition maps quantum}.

\sssec{}

Hence, it suffices to show that the families $c^{\vee,-\mu}$ and $'\!c^{-\mu}$ are equivalent for $\mu$ belonging to a cofinal
subset in $\Lambda$. We take the subset in question to consist of those $\mu$ for which $-\clambda-\mu+2\check\rho-2\rho$
is anti-dominant.

\medskip

First, for an individual $\mu$, we have
$$c^{\vee,-\mu}=(\coind_{\on{Lus}^+\to \on{big}}(k^{\clambda+\mu+2\rho}))^\vee[-d]\simeq 
\ind_{\on{Lus}^+\to \on{big}}(k^{-\clambda-\mu-2\rho})[-d],$$
which by \corref{c:big ind an coind} identifies with
$$\coind_{\on{Lus}^+\to \on{big}}(k^{-\clambda-\mu+2\check\rho-2\rho})\simeq {}'\!c^{-\mu}.$$

Now, when $-\clambda-\mu+2\check\rho-2\rho\in -\cLambda^+$, these objects belong to $(\Rep_q(G)_{\on{ren}})^\heartsuit$.
Hence, in order to show that the two coincide as systems, it is enough to do so at the level of $1$-morphisms 
(homotopy-coherence is automatic). 

\sssec{}

Thus, we have to show that for $\clambda'\in \cLambda$ and $\mu\in \Lambda^+$ the following diagram commutes:
$$
\CD
\on{Frob}^*_q((V^\mu)^\vee) \otimes  \ind_{\on{Lus}^+\to \on{big}}(k^{\mu+\clambda'})[-d]  @>>>  \ind_{\on{Lus}^+\to \on{big}}(k^{\clambda'})[-d] \\
@V{\sim}VV  @VV{\sim}V   \\
\on{Frob}^*_q((V^\mu)^\vee) \otimes \coind_{\on{Lus}^+\to \on{big}}(k^{\mu+\clambda'+2\check\rho}) @>>> \coind_{\on{Lus}^+\to \on{big}}(k^{\clambda'+2\check\rho}).
\endCD
$$

This follows by juxtaposing the following two commutative diagrams
$$
\CD
\on{Frob}^*_q((V^\mu)^\vee) \otimes  \ind_{\on{Lus}^+\to \frac{1}{2}}(k^{\mu+\clambda'})  @>>>  \ind_{\on{Lus}^+\to \frac{1}{2}}(k^{\clambda'}) \\
@V{\sim}VV  @VV{\sim}V   \\
\on{Frob}^*_q((V^\mu)^\vee) \otimes \coind_{\on{Lus}^+\to \frac{1}{2}}(k^{\mu+\clambda'+2\check\rho-2\rho}) @>>> 
\coind_{\on{Lus}^+\to \frac{1}{2}}(k^{\clambda'+2\check\rho-2\rho}),
\endCD
$$
which follows from the construction, and
$$
\CD
(V^\mu)^\vee \otimes \ind_{\cB\to \cG}(k^\mu \otimes V)[-d] @>>>  \ind_{\cB\to \cG}(V)[-d]   \\
@VVV   @VVV   \\
(V^\mu)^\vee \otimes \coind_{\cB\to \cG}(k^{\mu+2\rho} \otimes V) @>>>  \coind_{\cB\to \cG}(k^{2\rho}\otimes V),
\endCD
$$ 
which is a property of Serre duality on $\cG/\cB$.

\qed

\ssec{Drinfeld-Pl\"ucker formalism and duality}  \label{ss:Dr-Pl and duality}

In this subsection we will give a conceptual explanation of the duality procedure of \secref{sss:duality Plucker}. 

\sssec{}  \label{sss:duality B}

Since $\CC$ was assumed compactly generated, we obtain that 
$\Rep(\cB)\underset{\Rep(\cG)}\otimes \CC$
is also compactly generated, and hence, dualizable. Moreover, we have a canonical identification
\begin{equation} \label{e:B dual}
\left(\Rep(\cB)\underset{\Rep(\cG)}\otimes \CC\right)^\vee \simeq \Rep(\cB)\underset{\Rep(\cG)}\otimes \CC^\vee,
\end{equation}
so that that the functor
$$\oblv_{\cG\to \cB}:\Rep(\cB)\underset{\Rep(\cG)}\otimes \CC^\vee\to \CC^\vee$$
is the dual of the functor 
$$\coind_{\cB\to \cG}:\Rep(\cB)\underset{\Rep(\cG)}\otimes \CC\to \CC.$$

Explicitly, for $c\in \CC^c$ and $V\in \Rep(\cB)^c$, we have
$$\BD(V\otimes c)\simeq V^\vee\otimes \BD(c)$$
as objects in the two sides of \eqref{e:B dual}, respectively. 

\sssec{}  \label{sss:Serre and B}

Let $c'$ be a compact object of $\Rep(\cB)\underset{\Rep(\cG)}\otimes \CC$. Tautologically, we have
$$\coind_{\cB\to \cG}(\BD(c'))\simeq \BD(\ind_{\cB\to \cG}(c'))$$
as objects in $\CC^\vee$.

\medskip

Recall also that 
$$\ind_{\cB\to \cG}(c')\simeq \coind_{\cB\to \cG}(k^{2\rho}\otimes c')[d],$$
and recall the functor 
$$\jmath^*: \Rep(\cB)\underset{\Rep(\cG)}\otimes \CC\to 
\on{DrPl}(\CC).$$

Hence, we obtain the following expression for $\jmath_*(\BD(c'))$:
\begin{equation} \label{e:Serre and B}
\jmath_*(\BD(c'))^\mu\simeq \BD(\jmath_*(c)^{-\mu+2\rho})[-d].
\end{equation}

\medskip

This is the origin for the duality procedure in \secref{sss:duality Plucker}. 

\sssec{}

We now consider the situation when
$$\CC:=\Rep_{q^{-1}}(G)_{\on{ren}},$$
so that $\CC^\vee=\Rep_q(G)_{\on{ren}}$ and
$$\Rep(\cB)\underset{\Rep(\cG)}\otimes \CC=\Rep_{q^{-1}}^{\frac{1}{2}}(G)_{\on{ren}},$$
\begin{equation} \label{e:1/2 dual}
(\Rep(\cB)\underset{\Rep(\cG)}\otimes \CC)^\vee\simeq \Rep(\cB)\underset{\Rep(\cG)}\otimes \CC^\vee\simeq 
\Rep_q^{\frac{1}{2}}(G)_{\on{ren}}.
\end{equation}

\medskip

By unwinding the constructions we obtain that the resulting pairing
$$\Rep_q^{\frac{1}{2}}(G)_{\on{ren}}\otimes \Rep_{q^{-1}}^{\frac{1}{2}}(G)_{\on{ren}}\to \Vect$$
is given by
$$\CM_1,\M_2\mapsto \CHom_{\Rep_q^{\frac{1}{2}}(G)_{\on{ren}}}(k,\CM_1\otimes \CM^\sigma_2),$$
where $\sigma$ is a canonical equivalence
$$(\Rep_q^{\frac{1}{2}}(G)_{\on{ren}})^{\on{rev}}\to \Rep_{q^{-1}}^{\frac{1}{2}}(G)_{\on{ren}}.$$

\medskip

In particular, we obtain that with respect to the duality \eqref{e:1/2 dual} we have:
\begin{equation} \label{e:dual half}
\BD(\coind_{\on{Lus}^+\to \on{big}}(k^\clambda))\simeq \ind_{\on{Lus}^+\to \on{big}}(k^{-\clambda}).
\end{equation}

\sssec{}   \label{sss:Dr-Pl and duality}

We will now give a conceptual proof of the identification stated \secref{sss:identify dual system}, i.e., that
for the family 
$$c^{\vee,\mu}:=(\coind_{\on{Lus}^+\to \on{big}}(k^{\clambda-\mu+2\rho}))^\vee[-d]$$
the resulting object $c^{\vee,\on{enh}}$ identifies with $\ind_{\on{Lus}^+\to \on{big}}(k^{-\clambda})$.

\medskip

Indeed, this follows from \eqref{e:Serre and B} and \eqref{e:dual half} using the fact that the functor 
$\jmath_*$ is fully faithful and that the initial family
$$c^\mu:=\coind_{\on{Lus}^+\to \on{big}}(k^{\clambda+\mu})$$
identifies with
$\jmath_*(\coind_{\on{Lus}^+\to \on{big}}(k^\clambda))$.

\ssec{Cohomology of the small quantum group}   \label{ss:small cohomology}

In this subsection we will show how Theorems \ref{t:identify baby neg} and \ref{t:identify baby pos}
allow us to express the functor of cohomology with respect to $u_q(N)$ on the Kazhdan-Lusztig side. 

\sssec{}

Consider the functor
$$\Rep_q(G)_{\on{ren}}\to \Vect,$$
given by
$$\CM\mapsto \on{C}^\cdot(u_q(N),\CM)^\clambda.$$

In this subsection we will explain what functor it corresponds to under the equivalences
$$\sF_{-\kappa}:\KL(G,-\kappa)\simeq \Rep_q(G)_{\on{ren}} \text{ and }
\sF_\kappa:\KL(G,\kappa)\simeq \Rep_{q^{-1}}(G)_{\on{ren}}.$$

\sssec{}

Recall the object
$$\overset{\bullet}\CF{}^\semiinf_{-\kappa}:=\on{Av}^{I^-}_*(\ICsd_{-\kappa}) \in \Dmod_{-\kappa}(\Gr_G)^I.$$

It is acted on by $\cB$, and in particular by $\cT$. Set
$$\CF^\semiinf_{-\kappa}:=\inv_{\cT}(\overset{\bullet}\CF{}^\semiinf_{-\kappa}).$$

Explicitly,
$$\CF^\semiinf_{-\kappa}\simeq \underset{\mu}{\on{colim}}\, j_{-\mu,*}\star \on{Sat}(V^\mu).$$

\sssec{}

Recall also the object
$$\overset{\bullet}\CF{}^{-,\semiinf,\on{invol}}_{-\kappa}\in \Dmod_{-\kappa}(\Fl^{\on{aff}}_G)^{G(\CO)},$$
see \secref{sss:F-}. It is also acted on by $\cB$ and hence by $\cT$. 

\medskip

Let
$$\overset{\bullet}\CF{}^{-,\semiinf}_{\kappa}\in \Dmod_{\kappa}(\Gr_G)^I$$
be the image of $\overset{\bullet}\CF{}^{-,\semiinf,\on{invol}}_{-\kappa}$ under the inversion equivalence \eqref{e:inversion}.
Set
$$\CF^{-,\semiinf}_\kappa:=\inv_{\cT}(\overset{\bullet}\CF{}^{-,\semiinf}_\kappa).$$

Explicitly,
$$\CF^{-,\semiinf}_\kappa\simeq \underset{\mu}{\on{colim}}\,  J^{\BD}_\mu \star \on{Sat}(V^{-w_0(\mu)})[2d].$$

\sssec{}

We will prove:

\begin{thm} \label{t:small cohomology}  \hfill

\smallskip

\noindent{\em(a)} Under the equivalence $\sF_{-\kappa}:\KL(G,-\kappa)\simeq \Rep_q(G)_{\on{ren}}$, the functor 
$$\CM\mapsto \on{C}^\cdot(u_q(N),\CM)^\clambda, \quad \Rep_q(G)_{\on{ren}}\to \Vect$$
corresponds to 
$$\CM\mapsto \CHom_{\hg\mod^I_{-\kappa}}(\BW^\clambda_{-\kappa},\CF^\semiinf_{-\kappa}\star \CM).$$

\smallskip

\noindent{\em(b)} Under the equivalence $\sF_{\kappa}:\KL(G,\kappa)\simeq \Rep_{q^{-1}}(G)_{\on{ren}}$, the functor 
$$\CM\mapsto \on{C}^\cdot(u_{q^{-1}}(N),\CM)^\clambda, \quad \Rep_{q^{-1}}(G)_{\on{ren}}\to \Vect$$
corresponds to 
$$\CM\mapsto \on{C}^\semiinf(\fn(\CK),\CF^{-,\semiinf}_{\kappa}\star \CM)^\clambda[-2d].$$
\end{thm}

\begin{rem}

Note that we can rewrite the RHS in \thmref{t:small cohomology}(a) also as
$$\CM\mapsto \CHom_{\hg\mod^{T(\CO)}_{-\kappa}}(\BW^\clambda_{-\kappa},\ICs_{-\kappa}\star \CM).$$

\medskip 

Similarly, set
$$\ICsm_{\kappa}=w_0\cdot \ICs_{\kappa}\in \Dmod_{\kappa}(\Gr_G)^{N^-(\CK)\cdot T(\CO)},$$
and note that
$$\CF^{-,\semiinf}_\kappa\simeq \on{Av}^{I/T(\CO)}_*(\ICsm_{\kappa})[2d]\simeq \on{Av}^{N(\CO)}_*(\ICsm_{\kappa})[2d].$$

Then the RHS in \thmref{t:small cohomology}(b) can be rewritten as
$$\CM\mapsto \on{C}^\semiinf(\fn(\CK),\ICsm_\kappa\star \CM)^\clambda.$$

\end{rem}

\sssec{}

The rest of this section is devoted to the proof of \thmref{t:small cohomology}.

\medskip

First, we note that the functor 
$$\CM\mapsto \on{C}^\cdot(u_q(N),\CM)^\clambda, \quad \Rep_q(G)_{\on{ren}}\to \Vect$$ 
can be interpreted as the functor of pairing with the object
$$\coind_{\on{sml,grd}\to \on{big}}\circ \coind_{\on{sml,grd}^+\to \on{sml,grd}}(k^{-\clambda})\in \Rep_{q^{-1}}(G)_{\on{ren}}.$$

Note also that the above object 
$$\coind_{\on{sml,grd}\to \on{big}}\circ \coind_{\on{sml,grd}^+\to \on{sml,grd}}(k^{-\clambda})\simeq 
\coind_{\on{sml,grd}^+\to \on{big}}(k^{-\clambda})$$
can be canonically identified with
$$\inv_{\cT}\left(\coind_{\on{sml}^+\to \on{big}}(k^{-\clambda})\right)\simeq
\inv_{\cT}\left(\coind_{\on{sml}\to \on{big}}\circ  \coind_{\on{sm}^+\to \on{sm}}(k^{-\clambda})\right).$$

\sssec{}

Hence, applying Theorems \ref{t:identify baby pos} and \thmref{t:identify baby neg}, respectively,
we rewrite the corresponding functors on the Kazhdan-Lusztig side as
$$\CM\mapsto \langle 
\CF^{\semiinf,\on{invol}}_\kappa\star \BD(\BW_{\kappa}^{-\clambda}),\CM\rangle$$
(for point (a)), and
$$\CM\mapsto \langle
\CF^{-,\semiinf,\on{invol}}_{-\kappa}\star J_{-2\rho}[-d]\star \BW_{-\kappa}^{-\clambda+2\rho-2\check\rho},\CM\rangle$$
(for point (b)), respectively (here for point (b) we have also used \corref{c:ind and coind from Borel to small} to pass from
$\coind_{\on{sm}^+\to \on{sm}}(k^{-\clambda})$ to $\ind_{\on{sm}^+\to \on{sm}}(k^{-\clambda+2\rho-2\check\rho})$). 

\medskip

We have:
$$\langle 
\CF^{\semiinf,\on{invol}}_\kappa\star \BD(\BW_{-\kappa}^{\clambda}),\CM\rangle\simeq
\CHom_{\hg\mod^I_{-\kappa}}(\BW^\clambda_{-\kappa},\CF^\semiinf_{-\kappa}\star \CM),$$
thereby proving point (a). 

\medskip

For point (b) we have:
$$\langle
\CF^{-,\semiinf,\on{invol}}_{-\kappa}\star J_{-2\rho}[-d]\star \BW_{-\kappa}^{-\clambda+2\rho-2\check\rho},\CM\rangle\simeq
\langle 
\CF^{-,\semiinf,\on{invol}}_{-\kappa}\star \BW_{-\kappa}^{-\clambda-2\check\rho}[-d],\CM\rangle\simeq
\langle \BW_{-\kappa}^{-\clambda-2\check\rho}[d],\CF^{-,\semiinf}_{\kappa}\star \CM\rangle[-2d].$$

Finally, we recall that for $\CM'\in \hg\mod_\kappa^I$ we have
$$\langle \BW_{-\kappa}^{-\clambda-2\check\rho}[d],\CM'\rangle \simeq
\on{C}^\semiinf(\fn(\CK),\CM')^\clambda,$$
by \corref{c:semiinf with Wak I}. 

\qed

\section{Compatibility with the Arkhipov-Bezrukavnikov action action}  \label{s:AB action}

In this section we introduce one more requirement on the conjectural equivalence 
\eqref{e:main functor}. This will provide a conceptual explanation of the isomorphisms of
Theorems \ref{t:identify baby neg} and \ref{t:identify baby pos}.  

\medskip

In the next section 
we will use it to give an interpretation of the functor of cohomology with respect to
$U^{\on{DK}}_q(N)$ on the Kac-Moody side. 

\ssec{The Arkhipov-Bezrukavnikov action: a reminder}

\sssec{}

Recall that in \cite[Sect. 3]{AB} a monoidal functor
\begin{equation} \label{e:AB functor}
\QCoh(\cn/\on{Ad}(\cB))\to \Dmod_{-\kappa}(\Fl^{\on{aff}}_G)^I
\end{equation} 
was constructed.

\medskip

We will need the following pieces of information regarding this functor:

\begin{itemize}

\item For $\mu\in \Lambda$, 
$$\sfq^*(k^\mu)\mapsto J_\mu,$$
where $\sfq$ denotes the projection $$\cn/\on{Ad}(\cB)\to \on{pt}/\cB.$$

\item The resulting action of $\Rep(\cG)$ on $\Dmod_{-\kappa}(\Gr_G)^I$ obtained via
\begin{equation} \label{e:G to n/B}
\Rep(\cG) \overset{\oblv_{\cG\to \cB}}\longrightarrow \Rep(\cB) \overset{\sfq^*}\to \QCoh(\cn/\on{Ad}(\cB))
\end{equation} 
and the action of $\Dmod_{-\kappa}(\Fl^{\on{aff}}_G)^I$ on $\Dmod_{-\kappa}(\Gr_G)^I$, is given by
$$\CF\mapsto \CF \star \on{Sat}(-).$$

%\item For $\mu\in \Lambda^+$ the map
%$$j_{\mu,!}\star  \delta_{1,\Gr}\simeq 
%J_\mu\star \delta_{1,\Gr}\overset{\sfq^*(k^\mu)\to \sfq^*(\oblv_{\cG\to \cB}(V^\mu))}\longrightarrow \delta_{1,\Gr} \star \on{Sat}(V^\mu) \simeq \on{Sat}(V^\mu)$$
%is the canonical map $j_{\mu,!}\star  \delta_{1,\Gr}\to \on{Sat}(V^\mu)$.

\end{itemize}

\sssec{}   \label{sss:G(O) and I}

Thus, if $\CC$ is a category equipped with an action of $G(\CK)$ at level $-\kappa$ (see \secref{sss:action at level kappa}), 
then the category $\CC^I$ acquires an action of $\QCoh(\cn/\on{Ad}(\cB))$ with the following properties:

\begin{itemize}

\item For $c\in \CC$ and $\mu\in \Lambda$, we have 
$$\sfq^*(k^\mu)\otimes c:=J_\mu\star c.$$

\item The forgetful functor $\oblv_{G(\CO)/I}:\CC^{G(\CO)}\to \CC^I$ intertwines the $\Rep(\cG)$-action on $\CC^{G(\CO)}$
coming from $\on{Sat}$ and the $\Rep(\cG)$-action on $\CC^I$ coming from \eqref{e:G to n/B}. 

%\item For $\mu\in \Lambda^+$ and $c\in \CC^{G(\CO)}$, the resulting map
%\begin{multline*} 
%J_\mu \star \oblv_{G(\CO)/I}(c)\simeq \sfq^*(k^\mu)\otimes (\oblv_{G(\CO)/I}(c)) \overset{\sfq^*(k^\mu)\to \sfq^*(\oblv_{\cG\to \cB}(V^\mu))}\longrightarrow  \\
%\to V^\mu \otimes (\oblv_{G(\CO)/I}(c))  \simeq \oblv_{G(\CO)/I}(\on{Sat}(V^\mu)\star c)
%\end{multline*} 
%comes from the canonical map $j_{\mu,!}\star  \delta_{1,\Gr}\to \on{Sat}(V^\mu)$.

\end{itemize} 

\sssec{}

It follows that the functor $\oblv_{G(\CO)/I}:\CC^{G(\CO)}\to \CC^I$ canonically factors as
$$\CC^{G(\CO)}  \overset{\oblv_{\cG\to \cB}}\longrightarrow
\Rep(\cB)\underset{\Rep(\cG)}\otimes \CC^{G(\CO)} \overset{(\oblv_{G(\CO)/I})^{\on{enh}}}\longrightarrow \CC^I,$$
and its left adjoint $\on{Av}^{G(\CO)/I}_!$ factors as 
$$\CC^{G(\CO)}  \overset{\ind_{\cG\to \cB}}\longleftarrow
\Rep(\cB)\underset{\Rep(\cG)}\otimes \CC^{G(\CO)} \overset{(\on{Av}^{G(\CO)/I}_!)^{\on{enh}}}\longleftarrow \CC^I,$$
where $(\on{Av}^{G(\CO)/I}_!)^{\on{enh}}$ is the left adjoint of $(\oblv_{G(\CO)/I})^{\on{enh}}$, and is a functor
of $\Rep(\cB)$-module categories. 

\sssec{}

Similarly, $\on{Av}^{G(\CO)/I}_*$ factors as 
$$\CC^{G(\CO)}  \overset{\coind_{\cG\to \cB}}\longleftarrow
\Rep(\cB)\underset{\Rep(\cG)}\otimes \CC^{G(\CO)} \overset{(\on{Av}^{G(\CO)/I}_*)^{\on{enh}}}\longleftarrow \CC^I,$$
where $(\on{Av}^{G(\CO)/I}_*)^{\on{enh}}$ is the right adjoint of $(\oblv_{G(\CO)/I})^{\on{enh}}$, and is a functor of 
of $\Rep(\cB)$-modules categories. 

\medskip

Recall now that 
$$\on{Av}^{G(\CO)/I}_!\simeq \on{Av}^{G(\CO)/I}_*[2d] \text{ and }\ind_{\cG\to \cB}(-)\simeq \coind_{\cG\to \cB}(k^{2\rho}\otimes -)[d].$$ 
It follows formally that we have a canonical isomorphism of functors of $\Rep(\cB)$-modules categories
$$(\on{Av}^{G(\CO)/I}_!)^{\on{enh}}\simeq (\on{Av}^{G(\CO)/I}_*)^{\on{enh}} \circ (J_{-2\rho}\star -)[d].$$

\sssec{}

The following result was established in \cite[Theorem 7.3.1]{FG3}:

\begin{thm}  \label{t:identify baby geom}
The functor 
$$(\on{Av}^{G(\CO)/I}_!)^{\on{enh}}: \Dmod_{-\kappa}(\Fl^{\on{aff}}_G)^I\to \Rep(\cB)\underset{\Rep(\cG)}\otimes\Dmod_{-\kappa}(\Fl^{\on{aff}}_G)^{G(\CO)}$$
is given by
$$\CF'\mapsto (\overset{\bullet}\CF{}^{-,\semiinf,\on{invol}}_{-\kappa})^{\on{enh}}\star J_{-2\rho}\star \CF'[-d].$$
\end{thm} 

\begin{cor} \label{c:identify baby geom}
For any $\CC$ with an action of $G(\CK)$ at level $-\kappa$, the functors 
$$(\on{Av}^{G(\CO)/I}_!)^{\on{enh}} \text{ and } (\on{Av}^{G(\CO)/I}_*)^{\on{enh}},$$
are given by convolution with the objects 
$$(\overset{\bullet}\CF{}^{-,\semiinf,\on{invol}}_{-\kappa})^{\on{enh}}\star J_{-2\rho}[-d] \text{ and }
(\overset{\bullet}\CF{}^{-,\semiinf,\on{invol}}_{-\kappa})^{\on{enh}}[-2d],$$
respectively. 
\end{cor}

\sssec{}  \label{sss:AB pos}

Let us now consider the situation at the positive level. We define the monoidal functor
\begin{equation} \label{e:AB functor pos}
\QCoh(\cn/\on{Ad}(\cB))\to \Dmod_{\kappa}(\Fl^{\on{aff}}_G)^I
\end{equation} 
by applying monoidal duality 
$$\QCoh(\cn/\on{Ad}(\cB))_c\to \QCoh(\cn/\on{Ad}(\cB))_c$$
(i.e., the naive duality on perfect complexes) 
and Verdier duality 
$$\BD:(\Dmod_{-\kappa}(\Fl^{\on{aff}}_G)^I)_c\to (\Dmod_{\kappa}(\Fl^{\on{aff}}_G)^I)_c.$$

\medskip

Thus, for a category $\CC$ with an action of $G(\CK)$ at level $\kappa$, we obtain an action 
of $\QCoh(\cn/\on{Ad}(\cB))$ on $\CC^I$ with the following properties:

\begin{itemize}

\item For $c\in \CC$ and $\mu\in \Lambda$, we have 
$$\sfq^*(k^\mu)\otimes c:=J^{\BD}_{-\mu}\star c.$$

\item The forgetful functor $\oblv_{G(\CO)/I}:\CC^{G(\CO)}\to \CC^I$ intertwines the $\Rep(\cG)$-action on $\CC^{G(\CO)}$
coming from $\on{Sat}\circ \tau$ and the $\Rep(\cG)$-action on $\CC^I$ coming from \eqref{e:G to n/B}. 

\end{itemize}

\sssec{}

We still have the pair of functors
$$(\oblv_{G(\CO)/I})^{\on{enh}}:\Rep(\cB)\underset{\Rep(\cG)}\otimes \CC^{G(\CO)} \rightleftarrows \CC^I:
(\on{Av}^{G(\CO)/I}_!)^{\on{enh}},$$
and it follows formally from \thmref{t:identify baby geom} that the functor $(\on{Av}^{G(\CO)/I}_!)^{\on{enh}}$ is given by 
convolution with the object
$$(\overset{\bullet}\CF{}^{\semiinf,\on{invol}}_{\kappa})^{\on{enh}}\star J^{\BD}_{2\rho}[d]\in 
\Rep(\cB)\underset{\Rep(\cG)}\otimes\Dmod_{\kappa}(\Fl^{\on{aff}}_G)^{G(\CO)}.$$

\ssec{Compatibility with the equivalences $\sF_{-\kappa}$ and $\sF_{\kappa}$}

\sssec{}  \label{sss:AB compat}

We are now ready state one more expected property of the conjectural equivalence $\sF_{-\kappa}$:

\medskip

The equivalence 
$$\sF_{-\kappa}:\hg\mod^I_{-\kappa}\to \Rep^{\on{mxd}}_q(G)$$
intertwines the $\QCoh(\cn/\on{Ad}(\cB))$-action on $\hg\mod^I_{-\kappa}$
and on $\Rep^{\on{mxd}}_q(G)$ (the latter is from \conjref{c:1/2 vs mixed}). Moreover, \eqref{e:KL vs I} is a 
commutative diagram of categories acted on by $\Rep_q(G)$.

\medskip

As a formal consequence we obtain the following commutative diagram 
\begin{equation}  \label{e:KL vs I B}
\CD
\Rep(\cB)\underset{\Rep(\cG)}\otimes \KL(G,-\kappa)  @>{\sF_{-\kappa}}>>   \Rep(\cB)\underset{\Rep(\cG)}\otimes \Rep_q(G)_{\on{ren}}  \\
@A{(\on{Av}^{G(\CO)/I}_!)^{\on{enh}}}AA   @AA{\fr_{\on{baby-ren}\to \on{ren}}\circ \ind_{\on{mxd}\to \frac{1}{2}}}A  \\
\hg\mod^I_{-\kappa}   @>{\sF_{-\kappa}}>>  \Rep^{\on{mxd}}_q(G)
\endCD
\end{equation}
of categories acted on by $\Rep(\cB)$. 

\sssec{}

Let us show how commutative diagram \eqref{e:KL vs I B} explains the result of 
\thmref{t:identify baby neg}: indeed, we apply both
circuits to $\BW_{-\kappa}^\clambda$, noting that
$$\ind_{\on{mxd}\to \frac{1}{2}}(\BM_{q,\on{mxd}}^\clambda)\simeq \BM_{q,\frac{1}{2}}^\clambda\simeq
\ind_{\on{Lus}^+\to \frac{1}{2}}(k^\clambda),$$
and use \corref{c:identify baby geom}. 

\sssec{}

By duality, we obtain that the equivalence 
$$\sF_{\kappa}:\hg\mod^I_{\kappa}\to \Rep^{\on{mxd}}_{q^{-1}}(G)$$
intertwines the $\QCoh(\cn/\on{Ad}(\cB))$-action on $\hg\mod^I_{\kappa}$
from \secref{sss:AB pos}, and on $\Rep^{\on{mxd}}_{q^{-1}}(G)$, and we obtain a 
commutative diagram 
\begin{equation}  \label{e:KL vs I B pos}
\CD
\Rep(\cB)\underset{\Rep(\cG)}\otimes \KL(G,\kappa)  @>{\sF_{\kappa}}>>   \Rep(\cB)\underset{\Rep(\cG)}\otimes \Rep_{q^{-1}}(G)_{\on{ren}}  \\
@A{(\on{Av}^{G(\CO)/I}_!)^{\on{enh}}}AA   @AA{\fr_{\on{baby-ren}\to \on{ren}}\circ \ind_{\on{mxd}\to \frac{1}{2}}}A  \\
\hg\mod^I_{\kappa}   @>{\sF_{\kappa}}>>  \Rep^{\on{mxd}}_{q^{-1}}(G)
\endCD
\end{equation}
of categories acted on by $\Rep(\cB)$. 

\sssec{}

Note that \eqref{e:KL vs I B pos} gives a conceptual explanation of \thmref{t:identify baby pos}. Indeed, let us apply both circuits of the diagram
to $\BD(\BW_{-\kappa}^{-\clambda-2\rho})[d]$. 

\medskip

On the one hand,
$$\sF_{\kappa}(\BD(\BW_{-\kappa}^{-\clambda-2\rho})[d])\simeq
\BD^{\on{can}}(\BM^{-\clambda-2\rho}_{q,\on{mxd}})[-d]\simeq \BM^{\clambda+2\rho-2\check\rho}_{q^{-1},\on{mxd}},$$
and
$$\ind_{\on{mxd}\to \frac{1}{2}}(\BM^{\clambda+2\rho-2\check\rho}_{q^{-1},\on{mxd}})\simeq 
\ind_{\on{Lus}^+\to \frac{1}{2}}(k^{\clambda+2\rho-2\check\rho}_{q^{-1},\on{mxd}})\simeq
\coind_{\on{Lus}^+\to \frac{1}{2}}(k^{\clambda}_{q^{-1},\on{mxd}}).$$

On the other hand,
\begin{multline*}
(\on{Av}^{G(\CO)/I}_!)^{\on{enh}}(\BD(\BW_{-\kappa}^{-\clambda-2\rho})[d])\simeq
(\overset{\bullet}\CF{}^{\semiinf,\on{invol}}_{\kappa})^{\on{enh}}\star J^{\BD}_{2\rho}[d]\star 
\BD(\BW_{-\kappa}^{-\clambda-2\rho}[d]) \simeq \\
\simeq (\overset{\bullet}\CF{}^{\semiinf,\on{invol}}_{\kappa})^{\on{enh}}\star 
\BD(J_{2\rho}\star \BW_{-\kappa}^{-\clambda-2\rho}) \simeq (\overset{\bullet}\CF{}^{\semiinf,\on{invol}}_{\kappa})^{\on{enh}}\star 
\BD(\BW_{-\kappa}^{-\clambda}),
\end{multline*}
as required. 

\ssec{More on spherical vs Iwahori}

The material in this subsection is included for the sake of completeness. We will not need in the sequel. 

\sssec{}

We return to the general setting of \secref{sss:G(O) and I}. The following is established in 
\cite[Main Theorem 4, Sect. 5]{FG3}\footnote{The statement of \cite[Main Theorem 4, Sect. 5]{FG3} contains a typo: 
the functor $\Upsilon$ goes in the opposite direction.}:

\begin{thm}  \label{t:restr thm}
The functor of $\Rep(\cB)$-module categories
$$(\on{Av}^{G(\CO)/I}_*)^{\on{enh}}:\Dmod_{-\kappa}(\Fl^{\on{aff}}_G)^I \to
\Rep(\cB)\underset{\Rep(\cG)}\otimes \Dmod_{-\kappa}(\Fl^{\on{aff}}_G)^{G(\CO)}$$
factors canonically as
\begin{multline*}
\Dmod_{-\kappa}(\Fl^{\on{aff}}_G)^I \overset{\iota^*}\longrightarrow
\QCoh(\on{pt}/\cB)\underset{\QCoh(\cn/\on{Ad}(\cB))}\otimes \Dmod_{-\kappa}(\Fl^{\on{aff}}_G)^I \overset{\fr_{\on{geom}}}\to  \\
\to \Rep(\cB)\underset{\Rep(\cG)}\otimes \Dmod_{-\kappa}(\Fl^{\on{aff}}_G)^{G(\CO)},
\end{multline*}
where the functor $\fr_{\on{geom}}$ is fully faithful.
\end{thm} 

\begin{rem}
The assertion of \thmref{t:restr thm} is in fact a formal consequence of Bezrukavnikov's theorem in \cite{Bez}, which we will
review in \secref{ss:review of Bez}. The corresponding assertion on the coherent side is that the functor
$$\IndCoh((\cn\underset{\cg}\times \wt{\check\CN})/\on{Ad}(\cB)) \overset{\iota^*}\longrightarrow
\IndCoh((\on{pt}\underset{\cg}\times \wt{\check\CN})/\on{Ad}(\cB))$$
factors as
\begin{multline*}
\IndCoh((\cn\underset{\cg}\times \wt{\check\CN})/\on{Ad}(\cB)) \overset{\iota^*}\longrightarrow
\QCoh(\on{pt}/\cB)\underset{\QCoh(\cn/\on{Ad}(\cB))}\otimes  \IndCoh((\cn\underset{\cg}\times \wt{\check\CN})/\on{Ad}(\cB))\to \\
\to \IndCoh((\on{pt}\underset{\cg}\times \wt{\check\CN})/\on{Ad}(\cB)),
\end{multline*}
where the second arrow is fully faithful (the latter can be seen from the theory of singular support of \cite{AriG}).
\end{rem} 

\sssec{}

It follows formally from \thmref{t:restr thm} that for $\CC$ as in \secref{sss:G(O) and I}, the functor 
$$(\on{Av}^{G(\CO)/I}_*)^{\on{enh}}:\CC^I\to \Rep(\cB)\underset{\Rep(\cG)}\otimes \CC^{G(\CO)}$$
factors canonically as
$$\CC^I \overset{\iota^*}\longrightarrow
\QCoh(\on{pt}/\cB)\underset{\QCoh(\cn/\on{Ad}(\cB)}\otimes \CC^I \overset{\fr_\CC}\to 
\Rep(\cB)\underset{\Rep(\cG)}\otimes \CC^{G(\CO)},$$
where the functor $\fr_{\CC}$ is fully faithful.

\begin{rem}
The essential image of the functor 
\begin{equation} \label{e:r geom} 
\fr_{\CC}:\QCoh(\on{pt}/\cB)\underset{\QCoh(\cn/\on{Ad}(\cB)}\otimes \CC^I \to \Rep(\cB)\underset{\Rep(\cG)}\otimes \CC^{G(\CO)}
\end{equation} 
can be described explicitly via the \emph{derived Satake equivalence}. 

\medskip

Namely, the category $\CC^{G(\CO)}$ is acted on by the monoidal category $\IndCoh((\on{pt}\underset{\cg}\times \on{pt})/\cG)$, and hence 
the category $\Rep(\cB)\underset{\Rep(\cG)}\otimes \CC^{G(\CO)}$ is acted on by the monoidal category $\IndCoh((\on{pt}\underset{\cg}\times \on{pt})/\cB)$. 
Hence, objects in $\Rep(\cB)\underset{\Rep(\cG)}\otimes \CC^{G(\CO)}$ admit a singular support, which is a conical $\on{Ad}(\cB)$-invariant Zariski-closed
subset in $\cg^\vee$ (which is automatically contained in the nilpotent cone). It follows from \cite[Proposition7.4.3]{AriG} 
and Bezrukavnikov's theory \cite{Bez} that the image of the functor $\fr_{\CC}$ of \eqref{e:r geom} is the full subcategory consisting of objects whose singular 
support is contained in $(\cg/\cn)^\vee\subset \cg^\vee$.
\end{rem}  

\sssec{}

Let us apply this to $\CC=\hg\mod_{-\kappa}$. We obtain that the functor
$$(\on{Av}^{G(\CO)/I}_!)^{\on{enh}}:\hg\mod_{-\kappa}^I \to
\Rep(\cB)\underset{\Rep(\cG)}\otimes \KL(G,-\kappa)$$ factors via a fully faithful functor
$$\fr_{\hg\mod_{-\kappa}}:\QCoh(\on{pt}/\cB)\underset{\QCoh(\cn/\on{Ad}(\cB)}\otimes \hg\mod_{-\kappa}^I \to
\Rep(\cB)\underset{\Rep(\cG)}\otimes \KL(G,-\kappa).$$

Recall that on the quantum group side we have
$$\QCoh(\on{pt}/\cB)\underset{\QCoh(\cn/\on{Ad}(\cB)}\otimes \Rep_q^{\on{mxd}}(G)\simeq
\Rep_q^{\frac{1}{2}}(G)_{\on{baby-ren}}.$$

Thus, the functor $\fr_{\hg\mod_{-\kappa}}$ corresponds to the fully faithful embedding
$$\fr_{\on{baby-ren}\to \on{ren}}:\Rep_q^{\frac{1}{2}}(G)_{\on{baby-ren}}\to \Rep_q^{\frac{1}{2}}(G)_{\on{ren}}\simeq \Rep(\cB)\underset{\Rep(\cG)}\otimes \Rep_q(G)_{\on{ren}}.$$

\section{Cohomology of the De Concini-Kac quantum group via Kac-Moody algebras}   \label{s:DK cohomology}

In this section we will show how \conjref{c:main} allows to interpret the functor of cohomology with respect to
$U_q^{\on{DK}}(N)$ on the Kac-Moody side of the equivalences $\sF_{-\kappa}$ and $\sF_\kappa$. 

\ssec{The statement}

\sssec{}

Let us assume the existence of the equivalence 
\begin{equation} \label{e:KLI neg}
\sF_{-\kappa}:\hg\mod_{-\kappa}^I \simeq \Rep^{\on{mxd}}_q(G)
\end{equation} 
that satisfies the additional compatibility of \secref{sss:AB compat},
and hence by duality also of the equivalence
\begin{equation} \label{e:KLI pos}
\sF_\kappa:\hg\mod_{\kappa}^I \simeq \Rep^{\on{mxd}}_{q^{-1}}(G)
\end{equation} 
with the corresponding additional compatibility. 

\medskip

For $\clambda\in \cLambda$, on the quantum group side we consider the functor
$$\CM \mapsto \on{C}^\cdot(U_q^{\on{DK}}(N^-),-)^\clambda:\Rep_q^{\on{mxd}}(G) \to \Vect.$$

We now wish to describe what this functor corresponds to on the Kac-Moody side.

\sssec{}

We will prove: 

\begin{thm} \label{t:DK}  Assume \conjref{c:main}. Then: 

\smallskip

\noindent{\em(a)} Under the equivalence \eqref{e:KLI neg}, the functor 
$$\CM \mapsto \on{C}^\cdot(U_q^{\on{DK}}(N^-),-)^\clambda$$
corresponds to the functor 
$$\on{C}^\semiinf(\fn^-(\CK),-)^\clambda:\hg\mod_{-\kappa}^I \to \Vect.$$

\smallskip

\noindent{\em(b)} Under the equivalence \eqref{e:KLI pos}, the functor 
$$\CM \mapsto \on{C}^\cdot(U_{q^{-1}}^{\on{DK}}(N^-),-)^\clambda$$
corresponds to the functor
$$\CM \mapsto \on{C}^\semiinf\left(\fn^-(\CK), \underset{\mu\in \Lambda^+}{\on{colim}}\, j_{-\mu,*}\star j_{\mu,*}\star \CM\right)^\clambda.$$

\end{thm}

\sssec{}

As a plausibility check for \thmref{t:DK}(a), let us check that both sides of the theorem evaluate in the same way on 
$$\sF_{-\kappa}(\BW^\clambda_{-\kappa})\simeq \BM^\clambda_{q,\on{mxd}}.$$

Indeed,
$$\on{C}^\semiinf(\fn^-(\CK),\BW_{-\kappa}^{\clambda})^{\clambda'}=
\begin{cases}
&k[-d] \text{ if } \clambda'=\clambda+2\check\rho \\
&0 \text{ otherwise},
\end{cases}
$$
by \propref{p:neg semiinf of Wak}, and 
$$\on{C}^\cdot(U_q^{\on{DK}}(N^-),\BM^\clambda_{q,\on{mxd}})^{\clambda'}=
\begin{cases}
&k[-d] \text{ if } \clambda'=\clambda+2\check\rho \\
&0 \text{ otherwise},
\end{cases}
$$
by \corref{c:cohomology of UqDK}.

\sssec{}

As a formal consequence of \thmref{t:DK} we obtain the following statement (which is conjectural, since \thmref{t:DK} assumes \conjref{c:main})
pertaining to the original Kazhdan-Lusztig equivalence: 

\begin{conj}   \hfill

\smallskip

\noindent{\em(a)} Under the equivalence $\sF_{-\kappa}:\KL(G,-\kappa)\simeq \Rep_q(G)$, the functor
$$\CM \mapsto \on{C}^\cdot(U_q^{\on{DK}}(N^-),-)^\clambda$$
corresponds to the functor 
$$\on{C}^\semiinf(\fn^-(\CK),-)^\clambda.$$

\smallskip

\noindent{\em(b)} Under the equivalence $\sF_\kappa:\KL(G,\kappa)\simeq \Rep_{q^{-1}}(G)$, the functor
$$\CM \mapsto \on{C}^\cdot(U_{q^{-1}}^{\on{DK}}(N^-),-)^\clambda$$
corresponds to the functor
$$\CM \mapsto \on{C}^\semiinf\left(\fn^-(\CK), \underset{\mu\in \Lambda^+}{\on{colim}}\, j_{-\mu,*}\star j_{\mu,*}\star \CM\right)^\clambda.$$

\end{conj} 

The rest of this section is devoted to the proof of \thmref{t:DK}. 

\ssec{Cohomology of $U_q^{\on{DK}}(N^-)$ via the cohomology of $U_q^{\on{Lus}}(N)$}

We begin the proof of \thmref{t:DK} by showing that there is a certain categorical procedure that allows to
express the cohomology of $U_q^{\on{DK}}(N^-)$ in terms of the cohomology of $U_q^{\on{Lus}}(N)$. 

\medskip

The rest of the proof will essentially consist of applying the same procedure on the Kac-Moody side. 

\sssec{}   \label{sss:P functors}

Let $\CC$ be a category equipped with an action of $\Rep(\cG)$. For $\mu\in \Lambda^+$ consider the following endo-functor of 
$\Rep(\cB)\underset{\Rep(\cG)}\otimes \CC$,
$$c\mapsto P_\mu(c):=k^{-\mu}\otimes \left(\oblv_{\cG\to \cB}\circ \coind_{\cB\to \cG}(k^{w_0(\mu)}\otimes c)\right).$$

\medskip

Pick a representative $w'_0\in \on{Norm}_{\cG}(\cT)$ of $w_0\in W$. The action of $w'_0$ trivializes the lowest weight line
in every $V^\mu$; in particular we obtain a canonical map of $\cB$-modules
\begin{equation} \label{e:lowest weight}
V^{-w_0(\mu)}\to k^{-\mu}
\end{equation} 
or, equivalently, an identification
$$V^{-w_0(\mu)}\simeq (V^\mu)^\vee.$$

\medskip

The choice of $w'_0$ defines on the family of functors $\mu\rightsquigarrow P_\mu$ a structure of directed family 
with the transition maps being 
\begin{multline*}
k^{-\mu_2}\otimes \left(\oblv_{\cG\to \cB}\circ \coind_{\cB\to \cG}(k^{w_0(\mu_2)}\otimes c)\right) \to \\
\to k^{-\mu_2}\otimes V^{-w_0(\mu_1)}\otimes (V^{-w_0(\mu_1)})^\vee 
\otimes \left(\oblv_{\cG\to \cB}\circ \coind_{\cB\to \cG}(k^{w_0(\mu_2)}\otimes c)\right) \overset{\text{\eqref{e:lowest weight}}}\longrightarrow \\
\to k^{-\mu_1-\mu_2} \otimes (V^{-w_0(\mu_1)})^\vee 
\otimes \left(\oblv_{\cG\to \cB}\circ \coind_{\cB\to \cG}(k^{w_0(\mu_2)}\otimes c)\right) \to \\
\to k^{-\mu_1-\mu_2} \otimes \left(\oblv_{\cG\to \cB}\circ \coind_{\cB\to \cG}(k^{w_0(\mu_1+\mu_2)}\otimes c)\right),
\end{multline*}
where the last arrow comes from the canonical map 
$$(V^\mu)^\vee \to \coind_{\cB\to \cG}(k^{-\mu}), \quad \mu\in \Lambda^+$$
dual to $\ind_{\cB\to \cG}(k^\mu)\simeq V^\mu$. 

\medskip

Denote
$$P:=\underset{\mu\in \Lambda^+}{\on{colim}}\, P_\mu.$$

\sssec{}

Take $\CC:=\Rep_q(G)$, and consider the above family of endo-functors of 
$$\Rep_q^{\frac{1}{2}}(G)_{\on{ren}}\simeq \Rep(\cB)\underset{\Rep(\cG)}\otimes \Rep_q(G).$$

Consider now the following family of endo-functors of $\Rep_q^{\on{mxd}}(G)$:
$$\CM\mapsto \wt{P}_\mu(\CM):=\oblv_{\frac{1}{2}\to \on{mxd}}\circ \fs_{\on{ren}\to \on{baby-ren}} \circ P_\mu \circ 
\fr_{\on{baby-ren}\to \on{ren}}\circ \coind_{\on{mxd}\to \frac{1}{2}}(\CM),$$
which we can also rewrite as
$$\CM\mapsto k^{-\mu}\otimes \left(\oblv_{\on{big}\to \on{mxd}} \circ \coind_{\on{mxd}\to \on{big}}(k^{w_0(\mu)}\otimes \CM)\right).$$

%\medskip

%Denote 
%$$\wt{P}:=\underset{\mu\in \Lambda^+}{\on{colim}}\, \wt{P}_\mu.$$

\medskip

We claim:

\begin{prop}  \label{p:DK via Lus}
The functor
$$\CM\mapsto \on{C}^\cdot(U_q^{\on{DK}}(N^-),\CM)^\clambda,\quad \Rep_q^{\on{mxd}}(G)\to \Vect$$
identifies canonically with
$$\underset{\mu\in \Lambda^+}{\on{colim}}\, \on{C}^\cdot(U_q^{\on{Lus}}(N),\wt{P}_\mu(\CM))^{w_0(\clambda)}.$$
\end{prop} 

The rest of this subsection is devoted to the proof of \propref{p:DK via Lus}.

\sssec{}

We will deduce \propref{p:DK via Lus} from the following statement in the abstract context of 
\secref{sss:P functors}. 

\medskip

Note that the action of $w_0$ defines an endo-functor of $\Rep(\cT)$. A choice of a lift of $w_0$ to an 
element $w'_0\in \on{Norm}_\cG(\cT)$ endows this endo-functor with a structure of endo-functor of
$\Rep(\cT)$ as a module over $\Rep(\cG)$. In particular, we have a well-defined endo-functor
$$\Rep(\cT)\underset{\Rep(\cG)}\otimes \CC \overset{w'_0}\to \Rep(\cT)\underset{\Rep(\cG)}\otimes \CC.$$ 

\medskip

We claim:

\begin{prop}  \label{p:P functors}
The functor $P$ identifies canonically with the composition
$$\Rep(\cB)\underset{\Rep(\cG)}\otimes \CC \overset{\oblv_{\cB\to \cT}}\longrightarrow
\Rep(\cT)\underset{\Rep(\cG)}\otimes \CC \overset{w'_0}\to \Rep(\cT)\underset{\Rep(\cG)}\otimes \CC
\overset{\coind_{\cT\to \cB}}\longrightarrow \Rep(\cB)\underset{\Rep(\cG)}\otimes \CC.$$
\end{prop} 

\sssec{}

Let us first show how \propref{p:P functors} implies \propref{p:DK via Lus}. 

\medskip

First, we note that both sides are continuous functors out of $\Rep_q^{\on{mxd}}(G)$ (for the LHS this follows from
the definition of $\Rep_q^{\on{mxd}}(G)$, and for the RHS from \corref{c:UqDK}). Hence, we can assume that $\CM$ is compact. 
Denote $\CM':=\coind_{\on{mxd}\to \frac{1}{2}}(\CM)$. This is a compact object 
in $\Rep_q^{\frac{1}{2}}(G)_{\on{baby-ren}}$. 

\medskip

We rewrite

\begin{multline*} 
\on{C}^\cdot(U_q^{\on{Lus}}(N),\wt{P}_\mu(\CM))^{w_0(\clambda)}\simeq \\
\simeq 
\CHom_{\Rep_q^{\on{mxd}}(G)}\left(\BM_{q,\on{mxd}}^{w_0(\clambda)}, 
\oblv_{\frac{1}{2}\to \on{mxd}}\circ \fs_{\on{ren}\to \on{baby-ren}} \circ P_\mu \circ 
\fr_{\on{baby-ren}\to \on{ren}}\circ \coind_{\on{mxd}\to \frac{1}{2}}(\CM)\right)\simeq \\
\simeq  \CHom_{\Rep_q^{\frac{1}{2}}(G)_{\on{baby-ren}}}\left(\BM_{q,\frac{1}{2}}^{w_0(\clambda)}, 
\fs_{\on{ren}\to \on{baby-ren}} \circ P_\mu \circ \fr_{\on{baby-ren}\to \on{ren}}(\CM')\right),
\end{multline*}
and hence
\begin{multline*} 
\underset{\mu\in \Lambda^+}{\on{colim}}\, \on{C}^\cdot(U_q^{\on{Lus}}(N),\wt{P}_\mu(\CM))^{w_0(\clambda)}\simeq \\
\simeq 
\underset{\mu\in \Lambda^+}{\on{colim}}\, 
\CHom_{\Rep_q^{\frac{1}{2}}(G)_{\on{baby-ren}}}\left(\BM_{q,\frac{1}{2}}^{w_0(\clambda)}, 
\fs_{\on{ren}\to \on{baby-ren}} \circ P_\mu \circ \fr_{\on{baby-ren}\to \on{ren}}(\CM')\right)\simeq  \\
\simeq \CHom_{\Rep_q^{\frac{1}{2}}(G)_{\on{baby-ren}}}\left(\BM_{q,\frac{1}{2}}^{w_0(\clambda)}, \fs_{\on{ren}\to \on{baby-ren}} \circ P \circ 
\fr_{\on{baby-ren}\to \on{ren}}(\CM')\right)\simeq \\
\simeq  \CHom_{\Rep_q^{\frac{1}{2}}(G)_{\on{ren}}}\left(\BM_{q,\frac{1}{2}}^{w_0(\clambda)},P \circ 
\fr_{\on{baby-ren}\to \on{ren}}(\CM')\right)
\end{multline*}

Now using \propref{p:P functors} we rewrite the latter expression as
$$\CHom_{\Rep_q^{\frac{1}{2}}(G)_{\on{ren}}}\left(\BM_{q,\frac{1}{2}}^{w_0(\clambda)}, 
\coind_{\cT\to \cB}\circ w'_0\circ \oblv_{\cB\to \cT} \circ \fr_{\on{baby-ren}\to \on{ren}}(\CM')\right)$$
and further by adjunction as
$$\CHom_{\Rep_q^{\on{sml,grd}}(G)_{\on{ren}}}\left(\BM_{q,{\on{sml}}}^{w_0(\clambda)}, 
w'_0\circ \oblv_{\cB\to \cT} \circ \fr_{\on{baby-ren}\to \on{ren}}(\CM')\right).$$

Since both sides are compact as objects in $\Rep_q^{\on{sml,grd}}(G)_{\on{ren}}$, and since the functor $\fs$
is fully faithful on compact objects, the latter expression maps isomorphically to
\begin{multline*} 
\CHom_{\Rep_q^{\on{sml,grd}}}(G)\left(\BM_{q,\on{sml}}^{w_0(\clambda)}, 
\fs \circ w'_0\circ \oblv_{\cB\to \cT} \circ \fr_{\on{baby-ren}\to \on{ren}}(\CM')\right) \simeq \\
\simeq \CHom_{\Rep_q^{\on{sml,grd}}}(G)\left(\BM_{q,\on{sml}}^{w_0(\clambda)}, 
w'_0\circ \oblv_{\cB\to \cT} \circ \fs \circ \fr_{\on{baby-ren}\to \on{ren}}(\CM')\right)\simeq \\
\simeq \CHom_{\Rep_q^{\on{sml,grd}}}(G)\left(\BM_{q,\on{sml}}^{w_0(\clambda)}, 
w'_0\circ \oblv_{\cB\to \cT} \circ \fs_{\on{baby}}(\CM')\right).
\end{multline*} 

\medskip

Next, we rewrite 
\begin{multline*} 
\on{C}^\cdot(U_q^{\on{DK}}(N^-),\CM)^\clambda \simeq
\on{C}^\cdot\left(u_q(N^-),\oblv_{\frac{1}{2}\to \on{sml,grd}}\circ \fs_{\on{baby}}\circ \coind_{\on{mxd}\to \frac{1}{2}}(\CM)\right)^\clambda\simeq \\
\simeq \on{C}^\cdot\left(u_q(N^-),\oblv_{\cB\to \cT}\circ \fs_{\on{baby}}(\CM')\right)^\clambda,
\end{multline*} 
which we further rewrite as
\begin{equation} \label{e:rewrite DK}
\CHom_{\Rep_q^{\on{sml,grd}}(G)}\left(\ind_{\on{sml}^-\to \on{sml,grd}}(k^\clambda), 
\oblv_{\cB\to \cT }\circ \fs_{\on{baby}}(\CM')\right).
\end{equation} 

We have
$$\Rep_q^{\on{sml,grd}}(G)\simeq \Rep(\cT)\underset{\Rep(\cG)}\otimes \Rep_q(G).$$
Let us apply the functor $w'_0$ and use the identification
$$w'_0(\ind_{\on{sml}^-\to \on{sml,grd}}(k^\clambda))\simeq \ind_{\on{sml}^+\to \on{sml,grd}}(k^{w_0(\clambda)}).$$

Thus, we rewrite the expression in \eqref{e:rewrite DK} as
$$\CHom_{\Rep_q^{\on{sml,grd}}(G)}\left(\ind_{\on{sml}^+\to \on{sml,grd}}(k^{w_0(\clambda)}), w'_0\cdot 
\oblv_{\cB\to \cT }\circ \fs_{\on{baby}}(\CM')\right),$$
as desired. 

\sssec{Proof of \propref{p:P functors}}

In order to prove \propref{p:P functors} it suffices to consider the universal case, i.e., $\CC=\Rep(\cG)$.
We need to establish an isomorphism between the two endo-functors of $\Rep(\cB)$ 
as functors between $\Rep(\cG)$-module categories. 

\medskip

We first construct a natural transformation
\begin{equation} \label{e:P nat trans}
P\to \coind_{\cT\to \cB}\circ w'_0 \circ \oblv_{\cB\to \cT}.
\end{equation}

To do so, we need to construct a compatible family of natural transformations
$$P_\mu\to \coind_{\cT\to \cB}\circ w'_0 \circ \oblv_{\cB\to \cT}, \quad  \mu\in \Lambda^+.$$

By adjunction, the latter amounts to a compatible system of natural transformations
$$\oblv_{\cG\to \cT}\circ \coind_{\cB\to \cG}(k^{w_0(\mu)}\otimes -)\to w'_0 \circ (k^{w_0(\mu)}\otimes \oblv_{\cB\to \cT}(-)).$$

These natural transformations are obtained from the natural transformation
\begin{multline*} 
\oblv_{\cG\to \cT}\circ \coind_{\cB\to \cG} \simeq w'_0\circ \oblv_{\cG\to \cT}\circ \coind_{\cB\to \cG} \simeq \\
\simeq w'_0\circ \oblv_{\cB\to \cT}\circ \oblv_{\cG\to \cB}\circ \coind_{\cB\to \cG} \to w'_0\circ \oblv_{\cB\to \cT}.
\end{multline*} 

\medskip

In order to check that \eqref{e:P nat trans} is an isomorphism, it is enough to do so on the objects $k^\nu\in \Rep(\cB)$. 
We have
$$P(k^\nu)\simeq \underset{\mu\in \Lambda^+}{\on{colim}}\, k^{w_0(\mu)}\otimes \oblv_{\cG\to \cB}((V^{\mu-\nu})^\vee)
\overset{w'_0}\simeq \underset{\mu\in \Lambda^+}{\on{colim}}\, k^{w_0(\mu)}\otimes \oblv_{\cG\to \cB}(V^{w_0(-\mu+\nu)}),$$
which maps isomorphically to $\coind_{\cT\to \cB}(k^{w_0(\nu)})$ (see \eqref{e:O B}), as required. 

\ssec{The $P_\mu$ functors on the geometric side}

\sssec{}

Let $\CC$ be a category with an action of $G(\CK)$ at level $-\kappa$. For $\mu\in \Lambda^+$, consider the 
endo-functor of $\CC^I$, denoted $\wt{P}''_\mu$ of $\CC^I$, given by convolution with the object $j_{-\mu,*}\star j_{\mu,*}\star j_{w_0,*}[-d]$. 
These functors form a directed family under the transitions maps specified in \corref{c:semiinf neg via Wak}. 

\medskip

Denote
$$\wt{P}'':=\underset{\mu\in \Lambda^+}{\on{colim}}\, \wt{P}''_\mu.$$

\sssec{}  \label{sss:P'}

Consider now the object of $\Dmod_{-\kappa}(\Fl^{\on{aff}}_G)^I$ equal to 
\begin{equation} \label{e:P' objects}
j_{-\mu,*}\star p^*(\delta_{1,\Gr})\star j_{w_0(\mu),*}, \quad \mu\in \Lambda^+.
\end{equation}

We note that for $\mu\in \Lambda^{++}$, we have 
$$j_{-\mu,*}\star p^*(\delta_{1,\Gr})\star j_{w_0(\mu),*}\in 
(\Dmod_{-\kappa}(\Fl^{\on{aff}}_G)^I)^{\heartsuit}[d].$$ This follows from the fact that
$$p_*(j_{-\mu,*})\in (\Dmod_{-\kappa}(\Gr^{\on{aff}}_G)^I)^{\heartsuit}[d] \text{ and } 
\on{Av}^{G(\CO)/I}_*(j_{w_0(\mu),*})\in \Dmod_{-\kappa}(\Fl^{\on{aff}}_G)^{G(\CO)},$$
combined with  \cite[Lemma 9.1.4]{FG3}. 

\medskip

We consider the submonoid $\{0\}\cup \Lambda^{++}\subset \Lambda^+$; note that
the resulting subcategory $\Lambda^{++}\subset \Lambda^+$ is cofinal. 

\medskip

We claim that the objects \eqref{e:P' objects} form a directed family indexed by $\Lambda^{++}$. Namely, the transition maps are induced by the maps
\begin{equation} \label{e:trans map prime}
p^*(\delta_{1,\Gr}) \to j_{-\mu,*}\star p^*(\delta_{1,\Gr})\star j_{w_0(\mu),*}, \quad \mu\in \Lambda^{++}
\end{equation}
constructed as follows: 

\medskip

By adjunction, the datum of a map \eqref{e:trans map prime} is equivalent to that of a map
$$\delta_{1,\Gr} \to j_{-\mu,*}\star p^*(\delta_{1,\Gr})\star p_*(j_{w_0(\mu),*})$$
and further to that of a map
\begin{equation} \label{e:trans map prime next}
p_*(j_{\mu,!}) \to p^*(\delta_{1,\Gr})\star p_*(j_{w_0(\mu),*}).
\end{equation}

We note:
$$p_*(j_{w_0(\mu),*})\simeq p_*(j_{w_0(\mu)\cdot w_0,*})[d]=p_*(j_{w_0\cdot \mu,*})[d].$$
The sought-for map \eqref{e:trans map prime next} is 
\begin{multline*} 
p_*(j_{\mu,!}) \simeq 
j_{w_0,!}\star p_*(j_{w_0\cdot \mu,!})\to p^!(\delta_{1,\Gr})[-d]\star  p_*(j_{w_0\cdot \mu,!})\simeq \\
\simeq p^*(\delta_{1,\Gr})[d]\star  p_*(j_{w_0\cdot \mu,!})
\to p^*(\delta_{1,\Gr})[d]\star  p_*(j_{w_0\cdot \mu,*}).
\end{multline*} 

\begin{rem}
Consider the object
$$\underset{\mu\in \Lambda^{++}}{\on{colim}}\, j_{-\mu,*}\star p^*(\delta_{1,\Gr})\star j_{w_0(\mu),*} \in \Dmod_{-\kappa}(\wt\Fl^{\on{aff}}_G)^I.$$
In \secref{ss:big Schubert} we will see that under Bezrukavnikov's equivalence
$$\Dmod_{-\kappa}(\Fl^{\on{aff}}_G)^I\simeq \IndCoh((\wt{\check\CN}\underset{\cg}\times \wt{\check\CN})/\cG)$$
the above object corresponds (up to a twist) to the dualizing sheaf on the big Schubert cell in 
$$(\cG/\cB\times \cG/\cB)^{o}/\cG\subset (\cG/\cB\times \cG/\cB)/\cG\subset (\wt{\check\CN}\underset{\cg}\times \wt{\check\CN})/\cG.$$

\end{rem}

\sssec{}  \label{sss:P' to P''}

Let $\wt{P}'_\mu$ denote the endo-functor of $\CC^I$ given by convolution with $j_{-\mu,*}\star p^*(\delta_{1,\Gr})\star j_{w_0(\mu),*}$.
Denote 
$$\wt{P}':=\underset{\mu\in \Lambda^{++}}{\on{colim}}\, \wt{P}'_\mu.$$

Note that we have a natural transformation between the families 
$$\wt{P}'_\mu\to \wt{P}''_\mu,$$
namely
\begin{multline*} 
j_{-\mu,*}\star p^*(\delta_{1,\Gr})\star j_{w_0(\mu),*} \to j_{-\mu,*}\star j_{w_0,*}[-d]\star j_{w_0(\mu),*} \simeq 
j_{-\mu,*}\star j_{w_0\cdot w_0(\mu),*}[-d]\simeq \\
\simeq j_{-\mu,*}\star j_{\mu\cdot w_0,*}[-d]\simeq 
j_{-\mu,*}\star j_{\mu,*}\star j_{w_0,*}[-d].
\end{multline*}

\medskip

We will prove:

\begin{thm} \label{t:P geom}
The resulting natural transformation $\wt{P}'\to \wt{P}''$ is an isomorphism. 
\end{thm} 

\ssec{Relation to semi-infinite cohomology on the Kac-Moody side: negative level case}  

In this subsection we will use the formalism of the functors $P_\mu$ to prove point (a) of
\thmref{t:DK}. 

\sssec{}  \label{sss:P functors neg}

Let $P_\mu$ be the endo-functor of 
$$\Rep(\cB)\underset{\Rep(\cG)}\otimes \CC^{G(\CO)}$$
defined in \secref{sss:P functors}.

\medskip

Define the endo-functor $\wt{P}_\mu$ of $\CC^I$ by 
$$\wt{P}_\mu:=(\oblv_{G(\CO)/I})^{\on{enh}} \circ P_\mu \circ (\on{Av}^{G(\CO)/I}_*)^{\on{enh}}.$$

Denote
$$\wt{P}:=\underset{\mu\in \Lambda^+}{\on{colim}}\, \wt{P}_\mu.$$

Note that we have a term-wise isomorphism
\begin{equation} \label{e:P vs P'}
\wt{P}_\mu\simeq \wt{P}'_\mu.
\end{equation}

We will prove:

\begin{prop} \label{p:P vs P'}
The term-wise isomorphism \eqref{e:P vs P'} lifts to an isomorphism of directed families for $\mu\in \Lambda^{++}$.
\end{prop}

\begin{rem}
Recall that the transition maps for the family $\{\wt{P}_\mu\}$ came from the transition maps for the family $\{P_\mu\}$,
and those depended on a choice of a lift of $w_0\in W$ to an element $w'_0\in \on{Norm}_\cG(\cT)$. 
However, in the course of the proof of \propref{p:P vs P'} we will see that Geometric Satake provides a particular
choice of such a lift.
\end{rem}

\begin{cor} \label{c:P vs P'}
There exists an isomorphism of endo-functors $\wt{P}\simeq \wt{P}'$.
\end{cor} 

\sssec{}

Let us assume \propref{p:P vs P'} (and hence \corref{c:P vs P'}) and prove \thmref{t:DK}(a):

\medskip 

According to \corref{c:semiinf neg via Wak}, we have:
\begin{multline*}
\on{C}^\semiinf(\fn^-(\CK),\CM)^\clambda \simeq 
\underset{\mu\in \Lambda^+}{\on{colim}}\,
\CHom_{\hg\mod_{-\kappa}^I}(\BW_{-\kappa}^{w_0(\clambda)}, j_{-\mu,*}\star j_{\mu,*}\star j_{w_0,*}[-d]\star \CM)\simeq \\
\simeq \CHom_{\hg\mod_{-\kappa}^I}(\BW_{-\kappa}^{w_0(\clambda)}, \wt{P}''(\CM)),
\end{multline*}
which according to \thmref{t:P geom} identifies with 
$$\CHom_{\hg\mod_{-\kappa}^I}(\BW_{-\kappa}^{w_0(\clambda)}, \wt{P}'(\CM)),$$
and further, by \corref{c:P vs P'} with
$$\CHom_{\hg\mod_{-\kappa}^I}(\BW_{-\kappa}^{w_0(\clambda)}, \wt{P}(\CM)).$$

\medskip

Now, by \secref{sss:AB compat}, the equivalence $\sF_{-\kappa}$ intertwines the endo-functor $\wt{P}$ on $\hg\mod_{-\kappa}^I$
with the functor $\wt{P}$ on $\Rep_q^{\on{mxd}}(G)$. Hence, the assertion of \thmref{t:DK}(a) follows from \propref{p:DK via Lus}. 

\qed 

\ssec{Relation to semi-infinite cohomology on the Kac-Moody side: positive level case}

In this subsection we will use the formalism of the functors $P_\mu$ to prove point (b) of
\thmref{t:DK}. 

\sssec{}   

Let $\CC$ be a category acted on (strongly) by $G(\CK)$ at level $\kappa$. We define the endo-functors
$\wt{P}''_\mu$ of $\CC^I$ to be given by convolution with the objects
$$j_{w_0,*}[-d]\star j_{w_0(\mu),*}\star j_{-w_0(\mu),*}, \quad \mu\in \Lambda^+.$$

Set
$$\wt{P}'':=\underset{\mu\in \Lambda^+}{\on{colim}}\, \wt{P}''_\mu.$$

\medskip

Define the endo-functors $\wt{P}'_\mu$ to be given by convolutions with the objects
$$j_{\mu,*}\star p^*(\delta_{1,\Gr})\star j_{-w_0(\mu),*}.$$

These functors form a directed family by a procedure similar to that in \secref{sss:P'}. Set
$$\wt{P}':=\underset{\mu\in \Lambda^{++}}{\on{colim}}\, \wt{P}'_\mu.$$

As in \secref{sss:P' to P''} we have a natural transformation between the families 
$$\wt{P}'_\mu\to \wt{P}''_\mu.$$

From \thmref{t:P geom} we obtain:

\begin{cor} \label{c:P geom pos}
The resulting natural transformation $\wt{P}'\to \wt{P}''$ is an isomorphism.
\end{cor}

\begin{proof}
The assertion of \corref{c:P geom pos} follows from the fact that the situation at the positive level is
obtained from that at the negative level by applying the inversion equivalence
$$\Dmod_{-\kappa}(\Fl^{\on{aff}}_G)^I\to \Dmod_{\kappa}(\Fl^{\on{aff}}_G)^I.$$
\end{proof}

\sssec{}

Let $P_\mu$ be the directed family of endo-functors of $$\Rep(\cB)\underset{\Rep(\cG)}\otimes \CC^{G(\CO)}$$
defined in \secref{sss:P functors}. We have a term-wise identification
\begin{equation} \label{e:P vs P' pos}
\wt{P}_\mu\simeq \wt{P}'_\mu.
\end{equation}

As in \propref{p:P vs P'} we prove:

\begin{prop} \label{p:P vs P' pos}
The term-wise isomorphism \eqref{e:P vs P' pos} lifts to an isomorphism of directed families for $\mu\in \Lambda^+$. 
\end{prop}

Hence:

\begin{cor} \label{c:P vs P' pos}
There exists an isomorphism of endo-functors $\wt{P}\simeq \wt{P}'$.
\end{cor} 

\sssec{}

We are now ready to deduce \thmref{t:DK}(b):

\medskip

We have:
\begin{multline*}
\on{C}^\semiinf\left(\fn^-(\CK), \underset{\mu\in \Lambda^+}{\on{colim}}\, j_{-\mu,*}\star j_{\mu,*}\star \CM\right)^\clambda\simeq \\
\simeq \on{C}^\semiinf\left(\fn(\CK), j_{w_0,*}[-d]\star \left(\underset{\mu\in \Lambda^+}{\on{colim}}\, j_{-\mu,*}\star j_{\mu,*}\right)\star \CM\right)^{w_0(\clambda)}\simeq \\
\simeq \on{C}^\semiinf(\fn(\CK),\wt{P}''(\CM))^{w_0(\clambda)},
\end{multline*}
which according to \corref{c:P geom pos} identifies with
$$\on{C}^\semiinf(\fn(\CK),\wt{P}'(\CM))^{w_0(\clambda)},$$
and further, according to \corref{c:P vs P' pos}, with
$$\on{C}^\semiinf(\fn(\CK),\wt{P}(\CM))^{w_0(\clambda)}.$$

Now, by \secref{sss:AB compat}, the equivalence $\sF_{\kappa}$ intertwines the endo-functor $\wt{P}$ on $\hg\mod_{\kappa}^I$
with the functor $\wt{P}$ on $\Rep_{q^{-1}}^{\on{mxd}}(G)$. Hence, the assertion of \thmref{t:DK}(a) follows from \propref{p:DK via Lus}. 

\qed 

\ssec{Proof of \thmref{t:P geom}}

\sssec{}

Although the assertion of \thmref{t:P geom} is geometric (talks about D-modules on the affine flag scheme), we will
use representation theory to prove it. Namely, choose an integral weight $\clambda\in \cLambda$ which is $\kappa$-admissible
(see \secref{sss:adm}) and \emph{regular}, which means that the inequalities in \eqref{e:adm} are strict. 

\medskip

Then the Kashiwara-Tanisaki localization theorem says that the functor
\begin{equation} \label{e:Gamma lambda}
\CF \mapsto \CF\star \BM^{\clambda}_{-\kappa}, \quad \Dmod_{-\kappa}(\on{Fl}^{\on{aff}}_G)^I\to \hg\mod_{-\kappa}^I
\end{equation}
is conservative (in addition to being t-exact). 

\medskip

Indeed, under the equivalence
$$\Dmod_{-\kappa}(\on{Fl}^{\on{aff}}_G)^I\simeq \Dmod_{(-\kappa,\lambda)}(\on{Fl}^{\on{aff}}_G)^I,$$
the functor \eqref{e:Gamma lambda} corresponds to the functor 
$$\Gamma(\on{Fl}^{\on{aff}}_G,-):\Dmod_{(-\kappa,\lambda)}(\on{Fl}^{\on{aff}}_G)^I\to \hg\mod_{-\kappa}^I.$$

\sssec{}

Hence, it suffices to show that the resulting map
\begin{equation}   \label{e:two versions of orth}
\underset{\mu\in \Lambda^+}{\on{colim}}\,  j_{-\mu,*}\star p^*(\delta_{1,\Gr})\star j_{w_0(\mu),*}\star \BM^\clambda_{-\kappa}[d]\to
\underset{\mu\in \Lambda^+}{\on{colim}}\,  j_{-\mu,*}\star j_{\mu,*}\star j_{w_0,*}\star \BM^\clambda_{-\kappa}
\end{equation} 
is an isomorphism.

\medskip

We will show that both sides of \eqref{e:two versions of orth} yield an object $\CM\in \hg\mod_{-\kappa}^I$ that satisfies: 
$$
\CHom_{\hg\mod_{-\kappa}^I}(\BW^{\clambda'}_{-\kappa},\CM)=
\begin{cases}
&k \text{ if } \clambda'=w_0(\clambda)-2\check\rho; \\
&0 \text{ otherwise}.
\end{cases}
$$

This would imply that the map \eqref{e:two versions of orth} is an isomorphism: indeed it is easy to see that it is non-zero,
while the objects $\BW^{\clambda'}_{-\kappa}$ generate $\hg\mod_{-\kappa}^I$.

\begin{rem}
Note that the above object $\CM$ is \emph{not} one of the affine dual Verma modules: the latter are right-orthogonal
to affine Verma modules, whereas our $\CM$ is right-orthogonal to the Wakimoto modules. 
\end{rem}

\sssec{}

The fact that the functor \eqref{e:Gamma lambda} is t-exact and conservative implies that $\BM^\clambda_{-\kappa}$
is irreducible. 

\medskip

This, in turn implies that the canonical map
$$\BM^\clambda_{-\kappa}\to \BW^\clambda_{-\kappa}$$
is an isomorphism (indeed, the two sides have equal formal characters). 

\sssec{}

Now the fact that the RHS in \eqref{e:two versions of orth} yields an object with the desired orthogonality property follows from 
\corref{c:right orth to Wak}. 

\sssec{}

For the LHS we have:
\begin{multline*} 
\CHom_{\hg\mod_{-\kappa}^I}\left(\BW_{-\kappa}^{\clambda'},
\underset{\mu\in \Lambda^+}{\on{colim}}\,  j_{-\mu,*}\star p^*(\delta_{1,\Gr})\star j_{w_0(\mu),*}\star \BW^\clambda_{-\kappa}[d]\right)\simeq \\
\simeq \underset{\mu\in \Lambda^+}{\on{colim}}\,  \CHom_{\hg\mod_{-\kappa}^I}(\BW_{-\kappa}^{\clambda'}, j_{-\mu,*}\star 
p^*(\delta_{1,\Gr})\star j_{w_0(\mu),*}\star \BW^\clambda_{-\kappa}[d]).
\end{multline*}

For an individual $\mu$, we have:
\begin{multline*} 
\CHom_{\hg\mod_{-\kappa}^I}(\BW_{-\kappa}^{\clambda'}, j_{-\mu,*}\star 
p^*(\delta_{1,\Gr})\star j_{w_0(\mu),*}\star \BW^\clambda_{-\kappa}[d])\simeq  \\
\simeq \CHom_{\hg\mod_{-\kappa}^I}(j_{\mu,!}\star \BW_{-\kappa}^{\clambda'},p^*(\delta_{1,\Gr})\star j_{w_0(\mu),*}\star \BW^\clambda_{-\kappa}[d])\simeq \\
\simeq \CHom_{\hg\mod_{-\kappa}^I}(\BW_{-\kappa}^{\mu+\clambda'},p^*(\delta_{1,\Gr})\star \BW^{w_0(\mu)+\clambda}_{-\kappa}[d])\simeq \\
\simeq \CHom_{\hg\mod_{-\kappa}^I}\left(\BW_{-\kappa}^{\mu+\clambda'},\oblv_{G(\CO)/I}\circ \on{Av}^{G(\CO)/I}_*(\BW_{-\kappa}^{w_0(\mu)+\clambda})[d]\right)\simeq \\
\simeq \CHom_{\KL(G,-\kappa)}(\on{Av}^{G(\CO)/I}_!(\BW_{-\kappa}^{\mu+\clambda'}),\on{Av}^{G(\CO)/I}_*(\BW_{-\kappa}^{w_0(\mu)+\clambda})[d]).
\end{multline*}

Assume now that $\mu$ is large enough so that $\mu+\lambda'$ is dominant and $w_0(\mu)+\clambda+2\check\rho$ is anti-dominant. In this case
$$\on{Av}^{G(\CO)/I}_!(\BW_{-\kappa}^{\mu+\clambda'})\simeq \BV_{-\kappa}^{\mu+\clambda'} \text{ and }
\on{Av}^{G(\CO)/I}_*(\BW_{-\kappa}^{w_0(\mu)+\clambda})[d]\simeq \BV_{-\kappa}^{\vee,\mu+w_0(\clambda)-2\check\rho},$$
and the desired orthogonality is manifest. 

\qed 

\ssec{Proof of \propref{p:P vs P'}}

\sssec{}

Since the objects  
$$j_{-\mu,*}\star p^*(\delta_{1,\Gr})\star j_{w_0(\mu),*}, \quad \mu\in \Lambda^{++}$$ lie in the heart of the t-structure, it suffices to show
that the transition maps coincide for individual 1-morphisms (i.e., higher homotopy coherence is automatic). 

\sssec{}

Recall that the term-wise isomorphism 
$$\wt{P}_\mu\simeq \wt{P}'_\mu$$
comes from the (tautological) identifications
$$\on{Av}^{G(\CO)/I}_*\simeq \coind_{\cB\to \cG}\circ (\on{Av}^{G(\CO)/I}_*)^{\on{enh}} \text{ and }
\oblv_{G(\CO)/I}\simeq (\oblv_{G(\CO)/I})^{\on{enh}}\circ \oblv_{\cG\to \cB}.$$

\medskip

Recall also that according to \corref{c:identify baby geom}, the functor $(\on{Av}^{G(\CO)/I}_*)^{\on{enh}}$
is given by convolution with the object
$$(\overset{\bullet}\CF{}^{-,\semiinf,\on{invol}}_{-\kappa})^{\on{enh}}[-2d]\in \Dmod_{-\kappa}(\Fl^{\on{aff}}_G)^{G(\CO)}.$$

\sssec{}

Let us now recall the construction of the transition maps for the family $\{\wt{P}_\mu\}$.   

\medskip

First, let us note that Geometric Satake defines an identification
\begin{equation} \label{e:geom Sat and w_0}
\on{Sat}(V^{-w_0(\mu)})\simeq \on{Sat}((V^\mu)^\vee),
\end{equation}
thereby fixing a choice of a lift $w'_0\in \on{Norm}_\cG(\cT)$.  

\medskip

Indeed, 
\begin{equation} \label{e:geom Sat and IC}
\on{Sat}(V^\mu)\simeq \on{IC}_\mu,
\end{equation}
so \eqref{e:geom Sat and w_0} is the assertion that the objects
$$\on{IC}_\mu \text{ and } \on{IC}_{-w_0(\mu)}$$
of $\Dmod_{-\kappa}(\Gr_G)^{G(\CO)}$ are monoidal duals of each other. However, this
follows from the fact that the operation of passage to the monoidal dual in $\Dmod_{-\kappa}(\Gr_G)^{G(\CO)}$
is given by
$$\CF\mapsto \BD(\CF^{\on{invol}}),$$
where $$\on{invol}:\Dmod_{-\kappa}(\Gr_G)^{G(\CO)}\to \Dmod_{\kappa}(\Gr_G)^{G(\CO)}$$ is induced by inversion,
and $\BD: \Dmod_{\kappa}(\Gr_G)^{G(\CO)}\to \Dmod_{-\kappa}(\Gr_G)^{G(\CO)}$ is Verdier duality.

\sssec{}

Thus, the assertion of \propref{p:P vs P'} amounts to the verification of the commutativity of the following diagram
$$
\CD
p^*(\delta_{1,\Gr})  @>>>  \coind_{\cB\to \cG}\left((\overset{\bullet}\CF{}^{-,\semiinf,\on{invol}}_{-\kappa})^{\on{enh}}[-2d]\right) \\
& &  @VVV  \\
@VVV   \on{IC}_{-w_0(\mu)}\star \on{IC}_\mu \star  \coind_{\cB\to \cG}\left((\overset{\bullet}\CF{}^{-,\semiinf,\on{invol}}_{-\kappa})^{\on{enh}}[-2d]\right) \\
& &  @VVV  \\
j_{-\mu,*} \star  p^*(\delta_{1,\Gr})  \star j_{w_0(\mu),*}   @>>>
j_{-\mu,*}\star  \coind_{\cB\to \cG}\left((\overset{\bullet}\CF{}^{-,\semiinf,\on{invol}}_{-\kappa})^{\on{enh}}[-2d]\right)\star j_{w_0(\mu),*} 
\endCD
$$
where the lower right vertical arrow comes from the canonical maps
$$\on{IC}_{-w_0(\mu)}\to p_*(j_{-\mu,*})$$
and
\begin{equation} \label{e:mu map}
\on{IC}_\mu \star  \coind_{\cB\to \cG}\left((\overset{\bullet}\CF{}^{-,\semiinf,\on{invol}}_{-\kappa})^{\on{enh}}\right)  \to 
\coind_{\cB\to \cG}\left((\overset{\bullet}\CF{}^{-,\semiinf,\on{invol}}_{-\kappa})^{\on{enh}}\star j_{w_0(\mu),*}\right) 
\end{equation} 

Let us rewrite the map \eqref{e:mu map} more explicitly. First, we have
$$\coind_{\cB\to \cG}\left((\overset{\bullet}\CF{}^{-,\semiinf,\on{invol}}_{-\kappa})^{\on{enh}}\right)\simeq
\inv_{\cB}\left(\overset{\bullet}\CF{}^{-,\semiinf,\on{invol}}_{-\kappa}\right).$$

By unwinding the definitions, we obtain that the map
$$\on{IC}_\mu \star  \overset{\bullet}\CF{}^{-,\semiinf,\on{invol}}_{-\kappa} \to 
\overset{\bullet}\CF{}^{-,\semiinf,\on{invol}}_{-\kappa}\star j_{w_0(\mu),*},$$
obtained from \eqref{e:mu map}, and viewed as a map of objects of $\Dmod_{-\kappa}(\Gr_G)^{G(\CO)}$
equipped with an action of $\cB$, equals the composition
\begin{multline*}
\on{IC}_\mu \star  \overset{\bullet}\CF{}^{-,\semiinf,\on{invol}}_{-\kappa} =
\on{Sat}(V^\mu) \star  \overset{\bullet}\CF{}^{-,\semiinf,\on{invol}}_{-\kappa}
\simeq \overset{\bullet}\CF{}^{-,\semiinf,\on{invol}}_{-\kappa} \otimes \ul{V^\mu} \to \\
\to \overset{\bullet}\CF{}^{-,\semiinf,\on{invol}}_{-\kappa} \otimes k^{w_0(\mu)} \simeq 
\overset{\bullet}\CF{}^{-,\semiinf,\on{invol}}_{-\kappa} \star  j_{w_0(\mu),*},
\end{multline*} 
where the projection $V^\mu\to k^{w_0(\mu)}$ uses the specified choice of the lift $w'_0\in \on{Norm}_\cG(\cT)$,
and the isomorphism
$$\overset{\bullet}\CF{}^{-,\semiinf,\on{invol}}_{-\kappa} \otimes k^{w_0(\mu)} \simeq 
\overset{\bullet}\CF{}^{-,\semiinf,\on{invol}}_{-\kappa} \star  j_{w_0(\mu),*}$$
follows from the construction of $\overset{\bullet}\CF{}^{-,\semiinf,\on{invol}}_{-\kappa}$ as
$$\underset{\nu}\oplus\, \underset{\mu'\in \Lambda^+}{\on{colim}}\, \on{Sat}(V^{\mu'})\star \on{Av}^{G(\CO)/I}_!(j_{-\nu-\mu',*}).$$

To summarize, we need to establish the commutativity of the following diagram
$$
\CD
p^*(\delta_{1,\Gr})  @>>>  \overset{\bullet}\CF{}^{-,\semiinf,\on{invol}}_{-\kappa}[-2d]  \\
& & @VVV  \\
@VVV  \on{IC}_{-w_0(\mu)}\star \on{IC}_\mu \star \overset{\bullet}\CF{}^{-,\semiinf,\on{invol}}_{-\kappa}[-2d]  \\
& & @VVV  \\
j_{-\mu,*} \star  p^*(\delta_{1,\Gr})  \star j_{w_0(\mu),*}   @>>> j_{-\mu,*} \star \overset{\bullet}\CF{}^{-,\semiinf,\on{invol}}_{-\kappa}[-2d] \star  j_{w_0(\mu),*}. 
\endCD
$$

\sssec{}

Let us now describe explicitly the map 
$$\on{IC}_\mu \star \overset{\bullet}\CF{}^{-,\semiinf,\on{invol}}_{-\kappa} \to \overset{\bullet}\CF{}^{-,\semiinf,\on{invol}}_{-\kappa} \star  j_{w_0(\mu),*}.$$

By unwinding the constructions, we obtain that this map fits into the commutative diagrams
$$
\CD 
\on{IC}_\mu \star \IC_{-w_0(\mu)}\star \on{Av}^{G(\CO)/I}_!(j_{-\nu+w_0(\mu),*})  @>>>  \on{Av}^{G(\CO)/I}_!(j_{-\nu+w_0(\mu),*})  \\
@V{\sim}VV   @VV{\sim}V  \\
\on{IC}_\mu \star \on{Sat}(V^{-w_0(\mu)})\star \on{Av}^{G(\CO)/I}_!(j_{-\nu+w_0(\mu),*})  & &  \on{Av}^{G(\CO)/I}_!(j_{-\nu,*}) \star j_{w_0(\mu),*} \\
@VVV  @VVV   \\
\on{IC}_\mu \star \overset{\bullet}\CF{}^{-,\semiinf,\on{invol}}_{-\kappa} @>>> \overset{\bullet}\CF{}^{-,\semiinf,\on{invol}}_{-\kappa} \star  j_{w_0(\mu),*}
\endCD
$$
for $\nu\in \Lambda$.

\medskip

Hence, we are reduced to showing the commutativity of the next diagram
$$
\CD
p^*(\delta_{1,\Gr})  @>>> \IC_{-w_0(\mu)}\star \on{Av}^{G(\CO)/I}_!(j_{w_0(\mu),*})[-2d] \\
& & @VVV \\
@VVV  \IC_{-w_0(\mu)}\star \on{IC}_\mu \star \IC_{-w_0(\mu)}\star \on{Av}^{G(\CO)/I}_!(j_{w_0(\mu),*})[-2d]  \\
& & @VVV  \\
j_{-\mu,*} \star  p^*(\delta_{1,\Gr})  \star j_{w_0(\mu),*}   @>>> j_{-\mu,*}  \star \on{Av}^{G(\CO)/I}_!(j_{w_0(\mu),*})[-2d].
\endCD
$$

\sssec{}

We note, however, that the composite right vertical arrow coincides with the map
$$\IC_{-w_0(\mu)}\star \on{Av}^{G(\CO)/I}_!(j_{w_0(\mu),*})\to  j_{-\mu,*}  \star \on{Av}^{G(\CO)/I}_!(j_{w_0(\mu),*})$$
induced by the map $\on{IC}_{-w_0(\mu)}\to p_*(j_{-\mu,*})$. 

\medskip

Identifying $\on{Av}^{G(\CO)/I}_!(-)[-2d]\simeq \on{Av}^{G(\CO)/I}_*$, we obtain that it suffices to show that the map 
$$\delta_{1,\Gr} \to j_{-\mu,*}\star p^*(\delta_{1,\Gr})\star p_*(j_{w_0(\mu),*})$$
of \eqref{e:trans map prime} equals the composition
$$p^*(\delta_{1,\Gr})  \to \IC_{-w_0(\mu)}\star \on{Av}^{G(\CO)/I}_*(j_{w_0(\mu),*})\to 
j_{-\mu,*}  \star \on{Av}^{G(\CO)/I}_*(j_{w_0(\mu),*}),$$
which is an elementary verification. 

\qed

\section{Interpretation in terms of coherent sheaves}   \label{s:coh}

In this final section we will show that the (conjectural) equivalence \eqref{e:main functor} 
allows to compare a \emph{regular block} of $\Rep^{\on{mxd}}_q(G)$ with the category
of ind-coherent sheaves on the Steinberg variety for $\cG$, 

\ssec{Bezrukavnikov's theory: recollections}   \label{ss:review of Bez}

\sssec{}

Bezrukavnikov's theory of \cite{Bez} states the existence of an equivalence of monoidal categories
\begin{equation} \label{e:Bezr1}
\Dmod_{-\kappa}(\Fl^{\on{aff}}_G)^I\simeq \IndCoh((\wt{\check\CN}\underset{\cg}\times \wt{\check\CN})/\cG)
\end{equation}
and of their module categories
\begin{equation} \label{e:Bezr2}
\Dmod_{-\kappa}(\wt\Fl^{\on{aff}}_G)^I\simeq \IndCoh((\wt{\check\CN}\underset{\cg}\times \wt\cg)/\cG),
\end{equation}
where $\wt\Fl^{\on{aff}}_G:=G(\CK)/\overset{\circ}{I}$.

\medskip

Under the above equivalences, the direct image functor along
$$\IndCoh((\wt{\check\CN}\underset{\cg}\times \wt{\check\CN})/\cG\to \IndCoh((\wt{\check\CN}\underset{\cg}\times \wt\cg)/\cG)$$
corresponds to the pullback functor along $\wt\Fl^{\on{aff}}_G\to \Fl^{\on{aff}}_G$ shifted by $[\dim(T)]$.

\sssec{}

We normalize equivalence \eqref{e:Bezr1} so that the object $J_\mu\in \Dmod_{-\kappa}(\Fl^{\on{aff}}_G)^I$ goes over to the image
if $k^\mu\in \Rep(\cB)$ under the composite functor
\begin{multline*}
\Rep(\cB)\simeq \QCoh(\on{pt}/\cB) \overset{\sfq^*}\to \QCoh(\cn/\on{Ad}(\cB)) \overset{-\otimes \omega_{\cn/\on{Ad}(\cB)}}\longrightarrow 
\IndCoh(\cn/\on{Ad}(\cB)) \simeq \\
\simeq \IndCoh(\wt{\check\CN}/\cG)  \overset{*\on{-pshfwd}}\longrightarrow
\IndCoh((\wt{\check\CN}\underset{\cg}\times \wt{\check\CN})/\cG).
\end{multline*}

In particular the monoidal unit $\delta_{1,\Fl}\in \Dmod_{-\kappa}(\Fl^{\on{aff}}_G)^I$ corresponds to the image of $k\in \Rep(\cB)$ under
the above functor (which is a monoidal unit in $\IndCoh((\wt{\check\CN}\underset{\cg}\times \wt{\check\CN})/\cG)$, as it should be).

\sssec{}

Along with the equivalences \eqref{e:Bezr1} and \eqref{e:Bezr2} one proves their spherical counterparts:
\begin{equation} \label{e:Bezr1 sph}
\Dmod_{-\kappa}(\Fl^{\on{aff}}_G)^{G(\CO)}\simeq \IndCoh((\on{pt}\underset{\cg}\times \wt{\check\CN})/\cG)
\end{equation}
and
\begin{equation} \label{e:Bezr2 sph}
\Dmod_{-\kappa}(\wt\Fl^{\on{aff}}_G)^{G(\CO)}\simeq \IndCoh((\on{pt}\underset{\cg}\times \wt\cg)/\cG).
\end{equation}

The equivalence \eqref{e:Bezr1 sph} has the following features:

\begin{itemize}

\item For $\mu\in \Lambda$, it sends $\on{Av}^{G(\CO)/I}_!(J_\mu)[-d] \in \Dmod_{-\kappa}(\Fl^{\on{aff}}_G)^{G(\CO)}$ to the image of $k^\mu\in \Rep(\cB)$
under the functor
\begin{equation} \label{e:B to spring}
\Rep(\cB) \simeq \QCoh(\on{pt}/\cB) \overset{-\otimes \omega_{\on{pt}/\cB}}\longrightarrow \IndCoh(\on{pt}/\cB) \to 
\IndCoh((\on{pt}\underset{\cg}\times \wt{\check\CN})/\cG),
\end{equation}
where the last arrow is $*$-pushforward along 
$$\on{pt}/\cB \to (\on{pt}\underset{\cg}\times \cn)/\on{Ad}(\cB)\simeq (\on{pt}\underset{\cg}\times \wt{\check\CN})/\cG.$$

\item For $\mu\in \Lambda^+$ it sends $p^!(\on{Sat}(V^\mu))[-d]\in \Dmod_{-\kappa}(\Fl^{\on{aff}}_G)^{G(\CO)}$ to the image of $V^\mu\in \Rep(\cG)$ under 
the composition of $\oblv_{\cG\to \cB}: \Rep(\cG) \to \Rep(\cB)$ and the functor \eqref{e:B to spring}.

\end{itemize}

\sssec{}

The equivalences \eqref{e:Bezr1} and \eqref{e:Bezr1 sph} are related via the commutative diagram 
$$
\CD
\Dmod_{-\kappa}(\Fl^{\on{aff}}_G)^I   @>>>    \IndCoh((\wt{\check\CN}\underset{\cg}\times \wt{\check\CN})/\cG)  \\
@V{\on{Av}^{G(\CO)/I}_![-d]}VV     @VVV  \\
\Dmod_{-\kappa}(\Fl^{\on{aff}}_G)^{G(\CO)} @>>> \IndCoh((\on{pt}\underset{\cg}\times \wt{\check\CN})/\cG),
\endCD
$$
where the right vertical arrow is
\begin{multline*}
\IndCoh((\wt{\check\CN}\underset{\cg}\times \wt{\check\CN})/\cG)\simeq
\IndCoh((\on{\cn}\underset{\cg}\times \wt{\check\CN})/\cB) \overset{\text{!-plbck}}\longrightarrow
\IndCoh((\on{pt}\underset{\cg}\times \wt{\check\CN})/\cB)  \overset{*\on{-pshfwd}}\longrightarrow 
\IndCoh((\on{pt}\underset{\cg}\times \wt{\check\CN})/\cG)
\end{multline*}
or, by passing to right adjoints, via the diagram 
\begin{equation}  \label{e:sph Bez comp}
\CD
\Dmod_{-\kappa}(\Fl^{\on{aff}}_G)^I   @>>>    \IndCoh((\wt{\check\CN}\underset{\cg}\times \wt{\check\CN})/\cG)  \\
@A{\oblv_{G(\CO)/I}}AA     @AAA  \\
\Dmod_{-\kappa}(\Fl^{\on{aff}}_G)^{G(\CO)} @>>> \IndCoh((\on{pt}\underset{\cg}\times \wt{\check\CN})/\cG),
\endCD
\end{equation}
where the right vertical arrow is the functor
%$$\IndCoh((\on{pt}\underset{\cg}\times \wt{\check\CN})/\cG)\overset{\text{!-plbck}}\longrightarrow 
%\IndCoh((\on{pt}\underset{\cg}\times \wt{\check\CN})/\cB)\overset{*\on{-pshfwd}}\longrightarrow \IndCoh((\on{\cn}\underset{\cg}\times \wt{\check\CN})/\cB)\simeq 
%\IndCoh((\wt{\check\CN}\underset{\cg}\times \wt{\check\CN})/\cG).$$
\begin{multline*}
\IndCoh((\on{pt}\underset{\cg}\times \wt{\check\CN})/\cG)\overset{\text{!-plbck}}\longrightarrow
\IndCoh((\on{pt}\underset{\cg}\times \wt{\check\CN})/\cB) \overset{k^{-2\rho}\otimes -}\longrightarrow
\IndCoh((\on{pt}\underset{\cg}\times \wt{\check\CN})/\cB) \overset{*\on{-pshfwd}}\longrightarrow \\
\to \IndCoh((\on{\cn}\underset{\cg}\times \wt{\check\CN})/\cB)\simeq \IndCoh((\wt{\check\CN}\underset{\cg}\times \wt{\check\CN})/\cG).
\end{multline*}

The functors \eqref{e:Bezr2} and \eqref{e:Bezr2 sph} are related in a similar way. 

\sssec{}

Note that it follows that the image of
$$J_{2\rho}\star p^*(\delta_{1,\Gr})[d]\in \Dmod_{-\kappa}(\Fl^{\on{aff}}_G)^I$$
under the equivalence \eqref{e:Bezr1} is the image of the dualizing sheaf on $\QCoh((\cG/\cB\times \cG/\cB)/\cG$ under
%$$\QCoh((\cG/\cB\times \cG/\cB)/\cG)\overset{-\otimes \omega_{(\cG/\cB\times \cG/\cB)/\cG)}} 
%\longrightarrow
$$\IndCoh((\cG/\cB\times \cG/\cB)/\cG)\overset{*\on{-pshfwd}}\longrightarrow
\IndCoh((\wt{\check\CN}\underset{\cg}\times \wt{\check\CN})/\cG).$$

\ssec{Regular block categories}

\sssec{}

The linkage principle for $\hg\mod_{-\kappa}^I$ says that 
$$\CHom_{\hg\mod_{-\kappa}^I}(\BM_{-\kappa}^{\clambda_1},\BM_{-\kappa}^{\clambda_2})=0$$
unless $\clambda_1,\clambda_2\in \cLambda$ lie in the same orbit of the ``dotted" action of $W^{\on{aff,ext}}:=W\ltimes \Lambda$ on $\cLambda$,
i.e., when $\clambda_2$ is not of the form
$$(\mu\cdot w)\cdot \clambda_1:=\mu+w(\clambda_1+\check\rho)-\rho, \quad \mu\in \Lambda,\, w\in W.$$

\medskip

Fix a \emph{regular} $\kappa$-admissible weight $\clambda_0$. Let 
$$\Bl(\hg\mod_{-\kappa}^I)\subset \hg\mod_{-\kappa}^I$$ 
be the full subcategory generated by affine Verma modules $\BM_{-\kappa}^\clambda$ for $\clambda\in W^{\on{aff,ext}}\cdot \clambda_0$.
By the above, this subcategory is actually a direct summand of $\hg\mod_{-\kappa}^I$.

\medskip

Since the objects $\BM_{-\kappa}^\clambda$ and $\BW_{-\kappa}^\clambda$ have the same decomposition series, we have
$$\BW_{-\kappa}^\clambda \in \Bl(\hg\mod_{-\kappa}^I)\, \Leftrightarrow\, \clambda\in W^{\on{aff,ext}}\cdot \clambda_0.$$

\sssec{}

The Kashiwara-Tanisaki localization theorem of \cite{KT} (see also \cite[Theorem 5.5]{FG1}) says that there exists a canonical t-exact equivalence

\begin{equation}  \label{e:KT} 
\Dmod_{(-\kappa,\clambda_0)}(\wt\Fl^{\on{aff}}_G)^I\simeq \Bl(\hg\mod_{-\kappa}^I),
\end{equation}
given by taking sections on $\wt\Fl^{\on{aff}}_G$ and then $T$-invariants. 

\medskip

The composite functor
$$\Dmod_{-\kappa}(\Fl^{\on{aff}}_G)^I\to \Dmod_{-\kappa}(\wt\Fl^{\on{aff}}_G)^I\simeq \Bl(\hg\mod_{-\kappa}^I)$$
is given by
$$\CF\mapsto \CF\star \BM_{-\kappa}^{\clambda_0}.$$

\medskip

Composing with \eqref{e:Bezr2} we thus obtain an equivalence 
\begin{equation}  \label{e:coh to KM}
\IndCoh((\wt{\check\CN}\underset{\cg}\times \wt\cg)/\cG)\simeq \Bl(\hg\mod_{-\kappa}^I)
\end{equation}

\sssec{}

The linkage principle for the small quantum group implies that
$$\CHom_{\Rep^{\on{sml,grd}}_q(G)}(\BM^{\clambda_1}_{q,\on{sml}},\BM^{\clambda_2}_{q,\on{sml}})=0$$
unless unless $\clambda_1,\clambda_2\in \cLambda$ lie in the same orbit of the dotted action of $W^{\on{aff,ext}}$ on $\cLambda$.

\medskip

Using the grading, it follows that we also have
$$\CHom_{\Rep^{\on{mxd}}_q(G)}(\BM^{\clambda_1}_{q,\on{mxd}},\BM^{\clambda_2}_{q,\on{mxd}})=0$$
unless $\clambda_2\in W^{\on{aff,ext}}\cdot \clambda_1$. 

\medskip

Let 
$$\Bl(\Rep^{\on{mxd}}_q(G))\subset \Rep^{\on{mxd}}_q(G)$$ 
be the full subcategory generated by the standard objects $\BM^\clambda_{q,\on{mxd}}$ for 
$\clambda\in W^{\on{aff,ext}}\cdot \clambda_0$.
By the above, this subcategory is actually a direct summand of $\Rep^{\on{mxd}}_q(G)$.

\sssec{}

Thus, from \conjref{c:main} we deduce:

\begin{conj} \label{c:main block} 
There exists an equivalence
$$\Bl(\hg\mod_{-\kappa}^I)\simeq \Bl(\Rep^{\on{mxd}}_q(G)).$$
\end{conj}

Combining with \eqref{e:coh to KM}, we obtain:

\begin{conj} \label{c:coh vs} 
There exists an equivalence
$$\IndCoh((\wt{\check\CN}\underset{\cg}\times \wt\cg)/\cG)\simeq \Bl(\Rep^{\on{mxd}}_q(G)).$$
\end{conj}

\sssec{}  \label{sss:coh norm}

Thus, we obtain a string of equivalences
\begin{equation} \label{e:string of equiv}
\IndCoh((\wt{\check\CN}\underset{\cg}\times \wt\cg)/\cG)\simeq \Dmod_{-\kappa}(\wt\Fl^{\on{aff}}_G)^I\simeq 
\Bl(\hg\mod_{-\kappa}^I) \simeq \Bl(\Rep^{\on{mxd}}_q(G)).
\end{equation}

If we also assume the compatibility of $\sF_{-\kappa}$ specified in \secref{sss:AB compat}, then these equivalences 
respect the action of $\QCoh(\cn/\on{Ad}(\cB))\simeq \QCoh(\wt{\check\CN}/\cG)$. 

\medskip

Under the equivalences \eqref{e:string of equiv} we have the following correspondence of objects
\begin{equation} \label{e:mu objects}
\Delta_*^{\IndCoh}\circ (-\otimes \omega_{\cn/\cB}) \circ \sfq^*(k^\mu) 
\, \leftrightarrow\, J_\mu \, \leftrightarrow\, \BW_{-\kappa}^{\mu+\clambda_0}
\, \leftrightarrow\, \BM_q^{\mu+\clambda_0},
\end{equation} 

\sssec{}

We obtain that for $\mu\in \Lambda\subset \cLambda$, the functor
$$\on{C}^\cdot(U^{\on{Lus}}_q(N),-)^{\mu+\clambda_0}:\Bl(\Rep^{\on{mxd}}_q(G))\to \Vect$$
corresponds to the functor on $\IndCoh((\wt{\check\CN}\underset{\cg}\times \wt\cg)/\cG)$ given by
\begin{multline*}
\IndCoh((\wt{\check\CN}\underset{\cg}\times \wt\cg)/\cG) \overset{\Delta^!}\longrightarrow 
\IndCoh(\wt{\check\CN}/\cG)\simeq \IndCoh(\cn/\cB)\overset{-\otimes \omega^{-1}_{\cn/\cB}}\longrightarrow
\QCoh(\cn/\cB)\overset{\sfq_*}\to \\
\to \QCoh(\on{pt}/\cB) \overset{\CHom(k^\mu,-)}\longrightarrow k.
\end{multline*}

\begin{rem}
The composite conjectural equivalence 
\begin{equation} \label{e:U chi fam}
\IndCoh(\cn\underset{\cg}\times \wt\cg)/\cB)\simeq \Bl(\Rep^{\on{mxd}}_q(G)),
\end{equation}
viewed as categories acted on by $\QCoh(\cn/\on{Ad}(\cB))$ implies the following:

\medskip

Take an element $\chi\in \cn$, and let us tensor both sides of \eqref{e:U chi fam} with $\Vect$ over $\QCoh(\cn/\on{Ad}(\cB))$, where 
$$\QCoh(\cn/\on{Ad}(\cB))\to \Vect$$
corresponds to the evaluation at $\chi$. We obtain that the resulting category
$$\Bl(\Rep^\chi_q(G)):=\Vect\underset{\QCoh(\cn/\on{Ad}(\cB))}\otimes \Bl(\Rep^{\on{mxd}}_q(G))$$ 
identifies with
\begin{equation} \label{e:almost Spr}
\Vect\underset{\QCoh(\cn/\on{Ad}(\cB))}\otimes \IndCoh(\cn\underset{\cg}\times \wt\cg)/\cB).
\end{equation}

Note that \eqref{e:almost Spr} is a full subcategory of 
$$\IndCoh(\on{pt}\underset{\cg}\times \wt\cg),$$
where $\on{pt}\to \cg$ is the image of $\chi$ under $\cn\to \cg$. Here the (derived) scheme $$\on{pt}\underset{\cg}\times \wt\cg$$
is the derived Springer fiber over $\chi$. 

\end{rem}

\begin{rem}

One actually expects something a little stronger than the above equivalence
\begin{equation} \label{e:old chi}
\Vect\underset{\QCoh(\cn/\on{Ad}(\cB))}\otimes \IndCoh(\cn\underset{\cg}\times \wt\cg)/\cB)\simeq \Bl(\Rep^\chi_q(G))_\chi.
\end{equation}

Namely, let $\Rep^\chi_q(G)_{\on{ren}}$ be the following renormalized version of
$$\Rep^\chi_q(G):=\Vect\underset{\QCoh(\cn/\on{Ad}(\cB))}\otimes \Rep^{\on{mxd}}_q(G).$$

Namely, $\Rep^\chi_q(G)_{\on{ren}}$ is the ind-completion of the full (but not cocomplete) subcategory of 
comprised of objects whose image under the forgetful functor $\Rep^\chi_q(G)\to \Rep^{\on{mxd}}_q(G)$
is compact. 

\medskip

Let $\Bl(\Rep^\chi_q(G)_{\on{ren}})$ be the corresponding direct summand of $\Rep^\chi_q(G)_{\on{ren}}$. Then we expect an equivalence
\begin{equation} \label{e:better chi}
\IndCoh(\on{pt}\underset{\cg}\times \wt\cg)\simeq \Bl(\Rep^\chi_q(G)_{\on{ren}}).
\end{equation}

Note that for $\chi=0$, we have
$$\Rep^0_q(G)\simeq \Rep^{\on{sml}}_q(G)_{\on{baby-ren}} \text{ and } \Rep^0_q(G)_{\on{ren}}\simeq \Rep^{\on{sml}}_q(G)_{\on{ren}},$$
so in this case \eqref{e:better chi} recovers the equivalence
$$\IndCoh(\on{pt}\underset{\cg}\times \wt\cg)  \simeq   \Bl(\Rep^{\on{sml}}_q(G)_{\on{ren}}),$$
see \eqref{e:ABG all} below.

\end{rem}

\ssec{Spherical case and relation to the \cite{ABG} equivalence}

\sssec{}

Let $\Bl(\KL(G,-\kappa))$ denote the premiage of $\Bl(\hg\mod_{-\kappa}^I)$ under the forgetful functor
$$\oblv_{G(\CO)/I}:\Bl(\KL(G,-\kappa))\to \Bl(\hg\mod_{-\kappa}^I).$$

This is a direct summand of $\Bl(\KL(G,-\kappa))$ generated by the Weyl modules
$$\BV^\clambda_{-\kappa}, \quad \clambda\in \cLambda^+\cap (W^{\on{aff,ext}}\cdot \clambda_0).$$

\medskip

The Kashiwara-Tanisaki equivalence induces an equivalence
$$\Dmod_{-\kappa}(\wt\Fl^{\on{aff}}_G)^{G(\CO)} \to \Bl(\KL(G,-\kappa))$$
that makes the following diagram commute: 
$$
\CD
\Dmod_{-\kappa}(\wt\Fl^{\on{aff}}_G)^I   @>>> \Bl(\hg\mod_{-\kappa}^I) \\
@AAA   @AAA  \\
\Dmod_{-\kappa}(\wt\Fl^{\on{aff}}_G)^{G(\CO)}  @>>>  \Bl(\KL(G,-\kappa)).
\endCD
$$ 

\sssec{}

Let 
$$\Bl(\Rep_q(G)_{\on{ren}})\subset \Rep_q(G)_{\on{ren}}$$ be the full subcategory generated by the Weyl modules
$$\CV^\clambda_q, \quad \clambda\in \cLambda^+\cap (W^{\on{aff,ext}}\cdot 0).$$

The linkage principle for $\Rep_q(G)_{\on{ren}}$ says that $\Bl(\Rep_q(G)_{\on{ren}})$ is actually a direct summand of
$\Rep_q(G)_{\on{ren}}$. 

\medskip

The Kazhdan-Lusztig equivalence $\sF_{-\kappa}$ induces an equivalence
$$\Bl(\KL(G,-\kappa))\to  \Bl(\Rep_q(G)_{\on{ren}}).$$

\sssec{}

Composing, we obtain a string of equivalences
\begin{equation} \label{e:string of equiv sph}
\IndCoh((\on{pt}\underset{\cg}\times \wt\cg)/\cG) \simeq \Dmod_{-\kappa}(\wt\Fl^{\on{aff}}_G)^{G(\CO)}\simeq 
\Bl(\KL(G,-\kappa)) \simeq \Bl(\Rep_q(G)_{\on{ren}}),
\end{equation}
compatible with those in \eqref{e:string of equiv} under the functors specified earlier.

\medskip

The composite equivalence 
\begin{equation} \label{e:ABG}
\IndCoh((\on{pt}\underset{\cg}\times \wt\cg)/\cG)  \simeq \Bl(\Rep_q(G)_{\on{ren}})
\end{equation}
is the equivalence established in \cite{ABG}.

\sssec{}

Note that under the equivalence \eqref{e:ABG}, the image of 
$$k^\mu\in \Rep(\cB), \quad \mu\in \Lambda$$
under the functor 
%\begin{equation} \label{e:B to spring tilde}
$$\Rep(\cB) \simeq \QCoh(\on{pt}/\cB) \overset{-\otimes \omega_{\on{pt}/\cB}}\longrightarrow \IndCoh(\on{pt}/\cB) 
\overset{*\on{-pshfwd}}\longrightarrow \IndCoh((\on{pt}\underset{\cg}\times \cb)/\cB)\simeq 
\IndCoh((\on{pt}\underset{\cg}\times \wt\cg)/\cG)$$
%\end{equation}
corresponds to 
$$\CV^{\mu+\clambda_0}_q[-d]\simeq \ind_{\on{Lus}^+\to \on{big}}(k^{\mu+\clambda_0})[-d]\simeq 
\coind_{\on{Lus}^+\to \on{big}}(k^{\mu+\clambda_0+2\check\rho})\simeq \CV_q^{\vee,\mu+\clambda_0+2\check\rho}
\in \Bl(\Rep_q(G)_{\on{ren}}).$$

\medskip

The latter identification seems well-known for $\mu\in -\Lambda^+$, but may be new for more general $\mu$.

\sssec{}

The equivalence \eqref{e:ABG} induces a diagram of equivalences:
\begin{equation} \label{e:ABG all}
\CD
\IndCoh(\on{pt}\underset{\cg}\times \wt\cg)  @>>>   \Bl(\Rep^{\on{sml}}_q(G)_{\on{ren}})  \\
@A{{\text{!-plbck}}}AA   @AA{\oblv_{\on{sml.grd}\to \on{sml}}}A  \\
\IndCoh((\on{pt}\underset{\cg}\times \wt\cg)/\cT)   @>>> \Bl(\Rep^{\on{sml,grd}}_q(G)_{\on{ren}})  \\
@A{{\text{!-plbck}}}AA   @AA{\oblv_{\frac{1}{2}\to \on{sml.grd}}}A  \\
\IndCoh((\on{pt}\underset{\cg}\times \wt\cg)/\cB)   @>>> \Bl(\Rep^{\frac{1}{2}}_q(G)_{\on{ren}}) \\
@A{{\text{!-plbck}}}AA   @AA{\oblv{\on{big}\to \frac{1}{2}}}A  \\
\IndCoh((\on{pt}\underset{\cg}\times \wt\cg)/\cG)   @>>> \Bl(\Rep_q(G)_{\on{ren}}).
\endCD
\end{equation}

\sssec{}

Let $\sfi$ denote the embedding of the unit point
$$\on{pt}\to \cG/\cB.$$

We obtain that with respect to the equivalence 
\begin{equation} \label{e:ABG small}
\IndCoh((\on{pt}\underset{\cg}\times \wt\cg)/\cT)   \simeq \Bl(\Rep^{\on{sml,grd}}_q(G)_{\on{ren}}), 
\end{equation}
the image of $k^\mu\in \Rep(\cT)$ under the functor
$$\Rep(\cT) \simeq \QCoh(\on{pt}/\cT) \overset{-\otimes \omega_{\on{pt}/\cT}}\longrightarrow \IndCoh(\on{pt}/\cT) \overset{\sfi^{\IndCoh}_*}\longrightarrow
\IndCoh((\on{pt}\underset{\cg}\times \wt\cg)/\cT)$$
corresponds to the object
$$\ind_{\on{sml}^+\to \on{sml.grd}}(k^{\mu+\clambda_0})[-d] \in \Bl(\Rep^{\on{sml,grd}}_q(G)_{\on{ren}}).$$
%\simeq \coind_{\on{sml}^+\to \on{sml.grd}}(k^{\mu-2\rho})[-d]\in \Bl(\Rep^{\on{sml,grd}}_q(G)_{\on{ren}}).$$

\sssec{}

Let $\sfi^-$ denote the embedding of the point $w_0$
$$\on{pt}\to \cG/\cB.$$

We obtain that under the equivalence \eqref{e:ABG small} the image of $k^\mu\in \Rep(\cT)$ under the functor
$$\Rep(\cT) \simeq \QCoh(\on{pt}/\cT) \overset{-\otimes \omega_{\on{pt}/\cT}}\longrightarrow \IndCoh(\on{pt}/\cT) \overset{(\sfi^-)^{\IndCoh}_*}\longrightarrow
\IndCoh((\on{pt}\underset{\cg}\times \wt\cg)/\cT)$$
corresponds to the object
$$\ind_{\on{sml}^-\to \on{sml.grd}}(k^{\mu+\clambda_0})[-d]\in \Bl(\Rep^{\on{sml,grd}}_q(G)_{\on{ren}}).$$
%\simeq \coind_{\on{sml}^-\to \on{sml.grd}}(k^{\mu-2\rho})[-d]\in \Bl(\Rep^{\on{sml,grd}}_q(G)_{\on{ren}}).$$

\sssec{}  \label{sss:small and coh}

From the above identifications, we obtain that for $\mu\in \Lambda\subset \cLambda$, the functor
$$\IndCoh((\on{pt}\underset{\cg}\times \wt\cg)/\cT)\to \Vect$$
given by
$$\IndCoh((\on{pt}\underset{\cg}\times \wt\cg)/\cT) \overset{\sfi^!}\longrightarrow \IndCoh(\on{pt}/\cT) 
\overset{-\otimes \omega^{-1}_{\on{pt}/\cT}}\longrightarrow \QCoh(\on{pt}/\cT) \simeq \Rep(\cT)
\overset{\CHom(k^\mu,-)}\longrightarrow \Vect$$
corresponds to the functor 
$$\on{C}^\cdot(u_q(N),-)^{\mu+\clambda_0}[d]: \Bl(\Rep^{\on{sml,grd}}_q(G)_{\on{ren}})\to \Vect.$$

\medskip

Similarly, the functor 
$$\IndCoh((\on{pt}\underset{\cg}\times \wt\cg)/\cT) \overset{(\sfi^-)^!}\longrightarrow \IndCoh(\on{pt}/\cT) 
\overset{-\otimes \omega^{-1}_{\on{pt}/\cT}}\longrightarrow \QCoh(\on{pt}/\cT) \simeq \Rep(\cT)
\overset{\CHom(k^\mu,-)}\longrightarrow \Vect$$
corresponds to the functor 
$$\on{C}^\cdot(u_q(N^-),-)^{\mu+\clambda_0}[d]:\Bl(\Rep^{\on{sml,grd}}_q(G)_{\on{ren}})\to \Vect.$$

\ssec{Cohomology of the DK quantum group via coherent sheaves}  \label{ss:DK cohomology}

We will now give an interpretation of the functors $\on{C}^\cdot(U^{\on{DK}}_q(N^-),-)^\clambda$ 
on $\Bl(\Rep^{\on{mxd}}_q(G))$ and $\Bl(\Rep_q(G)_{\on{ren}})$ (for some specific values of $\clambda$)
in terms of coherent sheaves.

\medskip

For the duration of this subsection we will assume \conjref{c:main block}.

\sssec{}

We obtain that for $\clambda$ of the form $\mu+\clambda_0$, the functor
\begin{equation} \label{e:DK coh coh}
\on{C}^\cdot(U^{\on{DK}}_q(N^-),-)^{\mu+\clambda_0}[d]: \Bl(\Rep^{\on{mxd}}_q(G))\to \Vect
\end{equation} 
corresponds to the functor
\begin{multline*}
\IndCoh((\wt{\check\CN}\underset{\cg}\times \wt\cg)/\cG) \simeq 
\IndCoh((\cn\underset{\cg}\times \wt\cg)/\cB) \overset{!\on{-plbck}}\longrightarrow \\
\to \IndCoh((\on{pt}\underset{\cg}\times \wt\cg)/\cB) \overset{!\on{-plbck}}\longrightarrow 
\IndCoh((\cG/\cB)/\cB) \overset{!\on{-plbck}}\longrightarrow  \\
\to \IndCoh((\cG/\cB)/\cT) \overset{(\sfi^-)^!}\longrightarrow \IndCoh(\on{pt}/\cT) 
\overset{-\otimes \omega^{-1}_{\on{pt}/\cT}}\longrightarrow \QCoh(\on{pt}/\cT) \simeq \Rep(\cT)
\overset{\CHom(k^\mu,-)}\longrightarrow \Vect.
\end{multline*}

Note, however, that the composite map
$$\on{pt}/\cT\overset{\sfi^-}\to (\cG/\cB)/\cT\to (\cG/\cB)/\cB\to (\on{pt}\underset{\cg}\times \wt\cg)/\cB\to
(\cn\underset{\cg}\times \wt\cg)/\cB\simeq (\wt{\check\CN}\underset{\cg}\times \wt\cg)/\cG$$
appearing in the above formula identifies with
$$\on{pt}/\cT\simeq ((\cG/\cB)\times (\cG/\cB))^o/\cG\subset (\wt{\check\CN}\underset{\cg}\times \wt\cg)/\cG,$$
where $((\cG/\cB)\times (\cG/\cB))^o\subset (\cG/\cB)\times (\cG/\cB)$ is the open Bruhat cell, which is naturally
an open subset in $(\wt{\check\CN}\underset{\cg}\times \wt\cg)/\cG$. 

\medskip

Hence, we obtain that the functor \eqref{e:DK coh coh} corresponds to 
\begin{multline*}
\IndCoh((\wt{\check\CN}\underset{\cg}\times \wt\cg)/\cG) \overset{!\on{-plbck}}\longrightarrow 
\IndCoh( ((\cG/\cB)\times (\cG/\cB))^o/\cG)\simeq \IndCoh(\on{pt}/\cT)\simeq\\ 
\simeq  \QCoh(\on{pt}/\cT) \simeq\Rep(\cT) \overset{\CHom(k^\mu,-)}\longrightarrow \Vect.
\end{multline*}

\sssec{}

Note now that we have a Cartesian diagram
$$
\CD
(\on{pt}\underset{\cg}\times \cb^-)/\cT  @>>>  \on{pt}/\cT   \\
@VVV   @VVV   \\
(\on{pt}\underset{\cg}\times \wt\cg)/\cB @>>>  (\cn\underset{\cg}\times \wt\cg)/\cB
\endCD
$$
and also a Cartesian diagram
$$
\CD
(\on{pt}\underset{\cg}\times \cb^-)/\cT   @>>>  \on{pt}/\cT   \\ 
@VVV  @VVV  \\
(\on{pt}\underset{\cg}\times \wt\cg)/\cG  @>{r}>>  (\cG/\cB)/\cG,
\endCD
$$
where the right vertical arrow is the composite
$$\on{pt}/\cT \overset{\sfi^-}\longrightarrow (\cG/\cB)/\cT\to (\cG/\cB)/\cG,$$
and the bottom horizontal arrow, denoted $r$, is 
$$(\on{pt}\underset{\cg}\times \wt\cg)/\cG\to \wt\cg/\cG\to (\cG/\cB)/\cG.$$

\medskip

From here, combined with \conjref{c:coh vs} (and taking into account \eqref{e:sph Bez comp}), we obtain:

\begin{conj}  \label{c:ABG DK}
Under the \cite{ABG} equivalence
$$\IndCoh((\on{pt}\underset{\cg}\times \wt\cg)/\cG)  \simeq \Bl(\Rep_q(G)_{\on{ren}}),$$
the functor
$$\on{C}^\cdot(U^{\on{DK}}_q(N^-),-)^{\mu+\clambda_0}[d]: \Bl(\Rep_q(G)_{\on{ren}})\to \Vect$$
corresponds to the composition 
\begin{multline*}
\IndCoh((\on{pt}\underset{\cg}\times \wt\cg)/\cG)  \overset{r^\IndCoh_*}\longrightarrow 
\IndCoh((\cG/\cB)/\cG)  \overset{!\on{-plbck}}\longrightarrow  \IndCoh((\cG/\cB)/\cT) 
\overset{(\sfi^-)^!}\longrightarrow \\
\to \IndCoh(\on{pt}/\cT) \overset{-\otimes \omega^{-1}_{\on{pt}/\cT}}\longrightarrow \QCoh(\on{pt}/\cT) \simeq \Rep(\cT)
\overset{\CHom(k^{\mu+2\rho},-)}\longrightarrow \Vect.$$
\end{multline*}
\end{conj}

\sssec{}

Equivalently, we obtain that the functor
$$\on{C}^\cdot(U^{\on{DK}}_q(N),-)^{\mu+\clambda_0}[d]: \Bl(\Rep_q(G)_{\on{ren}})\to \Vect$$
corresponds to the composition 
\begin{multline*}
\IndCoh((\on{pt}\underset{\cg}\times \wt\cg)/\cG)  \overset{r^\IndCoh_*}\longrightarrow 
\IndCoh((\cG/\cB)/\cG)  \overset{!\on{-plbck}}\longrightarrow  \IndCoh((\cG/\cB)/\cT) 
\overset{\sfi^!}\longrightarrow \\
\to \IndCoh(\on{pt}/\cT) \overset{-\otimes \omega^{-1}_{\on{pt}/\cT}}\longrightarrow \QCoh(\on{pt}/\cT) \simeq \Rep(\cT)
\overset{\CHom(k^{\mu+2\rho},-)}\longrightarrow \Vect.$$
\end{multline*}

\sssec{}

As a reality check, let us compare the latter expression for $\on{C}^\cdot(U^{\on{DK}}_q(N),-)^{\mu+\clambda_0}[d]$ with one for
$\on{C}^\cdot(u_q(N),-)^{\mu+\clambda_0}[d]$ given in \secref{sss:small and coh}. 

\medskip

Namely, we note that for $\CM\in u_q(N)\mod(\Rep_q(T))$ the object
$$\coind_{\on{DK}^+\to \on{sm}^+}\circ \oblv_{\on{sm}^+\to \on{DK}^+}(\CM)$$
admits a filtration with associated graded isomorphic to
$$\CM\otimes \Sym(\cn[-1]).$$

Similarly, the object 
$$r^*\circ \sfi^{\IndCoh}_*(k^\mu)$$ has a canonical filtration with the associated graded being the image under
$$\Rep(\cT) \simeq \QCoh(\on{pt}/\cT) \overset{-\otimes \omega_{\on{pt}/\cT}}\longrightarrow \IndCoh(\on{pt}/\cT)
\overset{\sfi_*^{\IndCoh}}\longrightarrow \IndCoh((\on{pt}\underset{\cg}\times \wt\cg)/\cG)$$
of the objects 
$$k^\mu \otimes \Sym((\cg/\cb)^\vee[1])\simeq k^{\mu+2\rho} \otimes \Sym(\cn^\vee[1]).$$

\ssec{The big Schubert cell on the D-modules side}  \label{ss:big Schubert}

\sssec{}

In this subsection, for completeness, we will prove the following result:

\begin{thm}
Under the equivalence \eqref{e:Bezr1}, the object 
$$J_{2\rho}\star \left(\underset{\mu\in \Lambda^{++}}{\on{colim}}\, j_{-\mu,*}\star p^*(\delta_{1,\Gr})\star j_{w_0(\mu),*} \in \Dmod_{-\kappa}(\Fl^{\on{aff}}_G)^I\right)[d]$$
corresponds to the direct image of the dualizing sheaf along the map 
$$(\cG/\cB\times \cG/\cB)^{o}/\cG\to  (\wt{\check\CN}\underset{\cg}\times \wt{\check\CN})/\cG.$$
\end{thm} 

The rest of this subsection is devoted to the proof of this theorem.

\sssec{}

Consider the setting of \secref{sss:P functors neg} with $\CC=\Dmod_{-\kappa}(\on{Fl}^{\on{aff}}_G)^I$. We have the diagram
\begin{equation} \label{e:enh diag}
\CD
\Rep(\cB)\underset{\Rep(\cG)}\otimes \Dmod_{-\kappa}(\on{Fl}^{\on{aff}}_G)^{G(\CO)}  @>{(\oblv_{G(\CO)/I})^{\on{enh}}}>> \Dmod_{-\kappa}(\on{Fl}^{\on{aff}}_G)^I \\
@A{\oblv_{\cG\to \cB}}AA  \\
\Dmod_{-\kappa}(\on{Fl}^{\on{aff}}_G)^{G(\CO)}
\endCD
\end{equation} 

By \propref{p:P vs P'}, the $\wt{P}$ endofunctor of $\Dmod_{-\kappa}(\on{Fl}^{\on{aff}}_G)^I$ associated with this diagram is given by
convolution with the object
$$\underset{\mu\in \Lambda^{++}}{\on{colim}}\, j_{-\mu,*}\star p^*(\delta_{1,\Gr})\star j_{w_0(\mu),*}.$$

\medskip

By \eqref{e:sph Bez comp}, under Bezrukavnikov's equivalence, diagram \eqref{e:enh diag} corresponds to the following diagram:
$$
\CD
\IndCoh((\on{pt}\underset{\cg}\times \wt{\check\CN})/\cB)  @>{\sfq^*(k^{-2\check\rho})\otimes \iota^\IndCoh_*}>> 
\IndCoh(\cn\underset{\cg}\times \wt{\check\CN})/\cB)  @>{\sim}>> (\wt{\check\CN}\underset{\cg}\times \wt{\check\CN})/\cG  \\
@A{!\on{-plbck}}AA   \\
\IndCoh((\on{pt}\underset{\cg}\times \wt{\check\CN})/\cG). 
\endCD
$$

Thus, we need to show that the value of the corresponding $\wt{P}$ endofunctor of $\IndCoh(\wt{\check\CN}\underset{\cg}\times \wt{\check\CN})/\cB)$
on $\Delta^\IndCoh_*(\omega_{\wt{\check\CN}/\cG})$ 
produces the object isomorphic to the direct image of the dualizing sheaf along the map 
$$(\cG/\cB\times \cG/\cB)^{o}/\cG\to  (\wt{\check\CN}\underset{\cg}\times \wt{\check\CN})/\cG\simeq (\cn\underset{\cg}\times \wt{\check\CN})/\cB$$
tensored by $\sfq^*(k^{-2\check\rho})[-d]$. 

\sssec{}

The right adjoint of $\sfq^*(k^{-2\check\rho})\otimes \iota^\IndCoh_*$ identifies with $k^{2\check\rho}\otimes \iota^!$. 
And we have a Cartesian diagram
$$
\CD
\on{pt}/\cB   @>>> \cn/\cB  @>{\sim}>>   \wt{\check\CN}/\cG \\
@V{\iota}VV  @VVV @VV{\Delta}V  \\
(\on{pt}\underset{\cg}\times \wt{\check\CN})/\cB @>{\iota}>> (\cn\underset{\cg}\times \wt{\check\CN})/\cB @>{\sim}>> (\wt{\check\CN}\underset{\cg}\times \wt{\check\CN})/\cG,
\endCD
$$
hence, 
$$\iota^!(\Delta^\IndCoh_*(\omega_{\wt{\check\CN}/\cG})) \simeq \iota_*^\IndCoh(\omega_{\on{pt}/\cB}).$$

Since the map $\iota:\on{pt}/\cB\to (\on{pt}\underset{\cg}\times \wt{\check\CN})/\cB$ factors as
$$\on{pt}/\cB \overset{\sfi}\longrightarrow (\cG/\cB)/\cB \to (\on{pt}\underset{\cg}\times \wt{\check\CN})/\cB,$$
it suffices to show that the value of the endofunctor of $\IndCoh((\cG/\cB)/\cB)$ given by $P$ on the object
\begin{equation} \label{e:rho delta}
k^{2\check\rho}\otimes \sfi^\IndCoh_*(\omega_{\on{pt}/\cB})\in \IndCoh((\cG/\cB)/\cB)
\end{equation}
produces the dualizing sheaf on the big Schubert cell in $(\cG/\cB)/\cB$, shifted by $[-d]$. 

\sssec{}

Consider the diagram
$$
\CD
\QCoh((\cG/\cB)/\cB) @>{-\otimes\omega_{(\cG/\cB)/\cB}}>> \IndCoh((\cG/\cB)/\cB)  \\
@A{*\on{-plbck}}AA   @AA{!\on{-plbck}}A   \\
\QCoh(\on{pt}/\cG)  @>{-\otimes\omega_{\on{pt}/\cG}}>>   \IndCoh(\on{pt}/\cG).
\endCD
$$

The object \eqref{e:rho delta} identifies with 
$$\sfi_*(\CO_{\on{pt}/\cB})[-d]\otimes \omega_{(\cG/\cB)/\cB}.$$

Hence, it suffices to show that the endofunctor of $\QCoh((\cG/\cB)/\cB)$ given by $P$ on $\sfi_*(\CO_{\on{pt}/\cB})$
produces the structure sheaf of the big Schubert cell. However, the latter is manifest.

\end{document}